\documentclass{article}
\usepackage[utf8]{inputenc}
\usepackage{float}
\usepackage{amsmath,amssymb,amsthm} 
\usepackage{graphicx}
\usepackage{fancybox}
\usepackage{caption}
\usepackage{array, comment}
\usepackage{mwe}
\usepackage{bm}
\usepackage{booktabs}
\newcolumntype{C}[1]{>{\centering\let\newline\\\arraybackslash\hspace{0pt}}m{#1}}
\usepackage{subcaption}
\usepackage{hyperref} 
\usepackage[dvipsnames, table]{xcolor} 
\usepackage{tikz} 
\usepackage{mathdots} 
\usepackage[normalem]{ulem}
\usepackage{xfrac}
\usepackage{multirow} 
\usepackage{graphicx}

\hypersetup{
  colorlinks = true, 
  urlcolor = blue, 
  linkcolor = blue, 
  citecolor = red 
}

\definecolor{Dgreen}{RGB}{2,100,64}

\newcommand{\calP}{\mathcal{P}}
\newcommand{\canakci}{{\c{C}}anak{\c{c}}{\i} }
\newcommand{\gb}{\mathbf{g}}
\newcommand{\eb}{\mathbf{e}}
\newcommand{\degx}{\deg_{\boldsymbol{x}}}
\newcommand{\low}{\text{lower}}
\newcommand{\upper}{\text{upper}}
\newcommand{\cont}{\text{cont}}
\newcommand{\highlight}[1]{\colorbox{gray!15}{\small{#1}}}

\newtheorem{definition}{Definition}
\newtheorem{lemma}{Lemma}
\newtheorem{theorem}{Theorem}
\newtheorem{cor}{Corollary}

\newtheorem{prop}{Proposition}
\theoremstyle{remark}

\newtheorem{remark}{Remark}
\newtheorem*{remark*}{Remark}

\newcommand{\innerthmname}{}

\theoremstyle{definition}
\newtheorem*{innerdfn}{\innerthmname}
\newenvironment{upstatement}[1]
 {\renewcommand{\innerthmname}{#1}\innerdfn}
 {\endinnerdfn}

\newcommand{\brac}[2]{\textrm{Brac}_{#1}{#2}}
\newcommand{\bang}[2]{\textrm{Bang}_{#1}{#2}}
\newcommand{\calB}{\mathcal{B}}
\newcommand{\calBcirc}{\mathcal{B}^{\circ}}

\title{Skein relations for punctured surfaces}
\author{Esther Banaian,  Wonwoo Kang, Elizabeth Kelley}
\date{}

\begin{document}

\maketitle

\abstract{
We investigate skein relations in cluster algebras from punctured surfaces, extending the work of \canakci-Schiffler and Musiker-Williams on unpunctured surfaces. Using a combinatorial expansion formula by O{\u{g}}uz-Y{\i}ld{\i}r{\i}m and Pilaud-Reading-Schroll, we provide explicit formulas for these relations. This work demonstrates that the punctured analogues of the bangle and bracelet functions form spanning sets for cluster algebras associated with a punctured surfaces.  For surfaces with boundary and closed surfaces of genus 0, we further show that the bangles and bracelets form bases.}

\tableofcontents

\vspace{5mm}

\section{Introduction}\label{sec:introduction}

Fomin and Zelevinsky introduced the notion of cluster algebras in 2002, in their study of dual canonical bases \cite{FZ2002}. In the relatively short intervening time, mathematicians have discovered connections between cluster algebras and a wide range of areas of mathematics, including total positivity, quiver representations, Teichm\"{u}ller theory, tropical geometry, Lie theory, and Poisson geometry.

An important class of cluster algebras are those arising from surfaces \cite{FST-I}. Musiker, Schiffler, and Williams gave a combinatorial expansion formula for elements of cluster algebras of surface type in \cite{musiker2011positivity}, building on work for unpunctured surfaces in \cite{musiker2010cluster}.
When $\gamma$ has only plain endpoints,  the expansion of $x_\gamma$ with respect to the initial cluster determined by $T$ is given by a generating function of perfect matchings or dimer covers of a certain planar, bipartite graph $\mathcal{G}_{\gamma,T}$ called a \emph{snake graph}. When $\gamma$ has a notched endpoint, the formula in \cite{musiker2011positivity} uses a subset of the perfect matchings of the snake graph of a related plain arc.  Wilson gave an alternate construction using \emph{loop graphs} in \cite{wilson2020surface}. See Section \ref{subsec:snakegraph} for more details. 

Musiker, Schiffler, and Williams used their snake graph construction to exhibit two bases for cluster algebras coming from unpunctured surfaces \cite{musiker2013bases}. An important step of their proof involved studying \emph{skein relations}, which was done by Musiker and Williams in \cite{musiker2013}. The skein relations are multiplication formulas for products of cluster variables which correspond to curves with a point of intersection. The algebraic resolution of these products involves the variables associated to the geometric resolution of the incompatibility. While skein relations were already known in the coefficient-free case, as in \cite{fock2006moduli}, Musiker and Williams work in the context of principal coefficients.

Musiker and Williams work only with plain arcs in \cite{musiker2013}, so their formulas are not able to completely determine skein relations on punctured surfaces. Recently, partial progress has been made in extending these relations to punctured surfaces. In \cite{Mandel2023}, relations were given for punctured surfaces in the coefficient-free case while specific forms of skein relations in the principal coefficient case, referred to as ``tidy exchange relations'', were given in \cite{pilaud2023posets}. Our goal in this work is to exhibit all skein relations on punctured surfaces in the context of principal coefficients. 

It is not clear how to easily extend the methods from \cite{musiker2013} to tagged arcs, so we will use a different perspective by utilizing an expansion formula for cluster variables which is in terms of distributive lattices. The set of perfect matchings of a snake graph or loop graph can be given a partial order, and the resulting poset turns out to be a distributive lattice \cite{musiker2013bases, wilson2020surface}. By Birkhoff's Fundamental Theorem of Finite Distributive of Lattices, there exists another poset $\mathcal{P}$ such that the lattice of perfect matchings of a snake graph or loop graph is isomorphic to the lattice of order ideals of $\mathcal{P}$ \cite{birkhoff1937rings}. These underlying posets are exactly the induced partial order on the set of join-irreducibles. This allows one to give a new expansion formula for cluster variables in terms of the posets of join-irreducibles, as in \cite{ouguz2024cluster, pilaud2023posets}. We note that in \cite{pilaud2023posets}, the authors prove their formula using a geometric perspective instead, utilizing the connection to snake graphs, which allows them to work in a more general setting. Weng also uses such posets to study Donaldson-Thomas transformations in \cite{weng2023f}.

A poset of join-irreducibles in a lattice of perfect matchings of snake graph will always be a \emph{fence poset}. A fence poset is characterized by the fact that its Hasse diagram is a path graph. Fence posets were originally defined to only be posets of the form $a_1 < a_2 > a_3 < \cdots$, where $a_i$'s are elements in the poset, but we will use the term more generally, as has become common. Fence posets have been of great interest recently. Some of this activity is from a conjecture by \cite{MORIER-GENOUD_OVSIENKO_2020} that the rank generating function of a lattice of order ideals of a fence poset (or equivalently of perfect matchings of a snake graph) is unimodal. In \cite{MORIER-GENOUD_OVSIENKO_2020}, these rank generating functions are used define $q$-deformed rational numbers. The conjecture was solved in \cite{ouguz2023rank} and inspired several other articles studying fence posets, including \cite{elizalde2023partial, elizalde2023rowmotion}. 

Fence posets appear in several other contexts in mathematics. They resemble diagrams associated to \emph{strings}, a combinatorial object used to describe a set of indecomposable modules in a string algebra; see \cite{butler1987auslander} for definitions. In this setting, order ideals correspond to canonical submodules. The relationship between these two perspectives has been useful in comparing the snake graph expansion formula and the Caldero-Chapoton function for algebras arising from triangulated surfaces \cite{geiss2022schemes, geiss2023bangle} and orbifolds \cite{banaian2023snake}. Fence posets also can be viewed as heaps of Coxeter elements in the symmetric group, and order ideals of such a heap correspond to permutations less than the given Coxeter element in weak order. \canakci and Schroll unify these three perspectives of fence posets in \cite{ccanakcci2021lattice}.

Our methods are inspired by an alternate proof of skein relations in the unpunctured case by \canakci and Schiffler through ``snake graph calculus'' \cite{snakegraphcalculus1, snakegraphcalculus2, snakegraphcalculus3}. Here, the authors give bijections between sets of perfect matchings of graphs from the geometric resolution of intersecting arcs. These methods have proven useful to exhibit bases in new cases \cite{ccanakci2015cluster,ccanakcci2019bases} and to study extensions in Jacobian algebras and cluster categories associated to triangulated surfaces \cite{canakci2017extensions}. We are hopeful our proof methods can similarly be used for further study concerning punctured surfaces.

An advantage of our result is that we can explicitly describe the $y$-monomial appearing in the relations. In particular, in our proofs, it becomes immediately apparent that exactly one of the terms in the resolution is multiplied by a $y$-monomial. This is also the case in the unpunctured setting and is a key fact used to exhibit bases for cluster algebras from unpunctured surfaces \cite{musiker2013bases}. 

The $y$-monomials that appear in skein relations can be described by two (multi)sets of arcs from an initial triangulation $T$. We denote these sets $R$ and $S$; we could have $R = \emptyset$ or $S = \emptyset$, but never $R \cup S = \emptyset$.  The ``$R$'' stands for ``share'' as it records arcs that $\gamma_1$ and $\gamma_2$ both cross near the designated intersection point. The ``$S$'' stands for ``sweep'' as it records (some) arcs which the resolution sweeps past at the endpoint. 

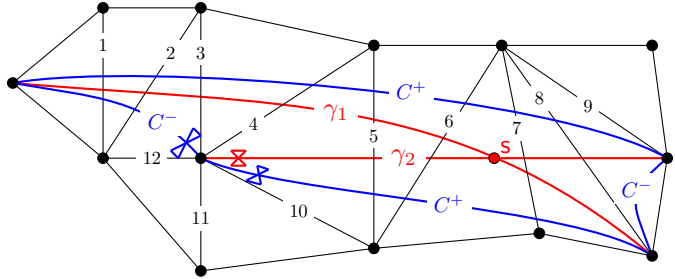
\begin{figure}[H]
\captionsetup{width=0.9\textwidth}
    \centering
    \begin{tikzpicture}[transform shape]
    
    \draw[out=-5,in=140, looseness=1, red,thick] (-2.5,1) to (6,-1.3);
    \draw[red,thick] (0,0) to (6.2,0);
    
    \draw[red,thick] (0.4,-0.1) to (0.6,-0.1);
    \draw[red,thick] (0.4,0.1) to (0.6,0.1);
    \draw[red,thick] (0.4,-0.1) to (0.6,0.1);
    \draw[red,thick] (0.4,0.1) to (0.6,-0.1);

    \node[red,fill=white] at (1.8,0.6) {$\gamma_1$};
    \node[red,fill=white] at (2.7,0) {$\gamma_2$};

    \draw (6,1.5) to (6.2,0);
    \draw (6,1.5) to (4,1.5);
    \draw (2.3,1.5) to (4,1.5);
    \draw (6,-1.3) to (4.5,-1);
    \draw (2.3,-1.2) to (4.5,-1);
    \draw (6,-1.3) to (6.2,0);
    \draw (4,1.5) to (6.2,0);
    \draw (4.5,-1) to (4,1.5);
    \draw (6,-1.3) to (4,1.5);
    \draw (2.3,-1.2) to (2.3,1.5);
    \draw (2.3,-1.2) to (4,1.5);
    \draw (0,0) to (2.3,1.5);
    \draw (0,2) to (2.3,1.5);
    \draw (0,2) to (-1.3,0);
    \draw (0,0) to (-1.3,0);
    \draw (-1.3,0) to (-1.3,2);
    \draw (-2.5,1) to (-1.3,2);
    \draw (-2.5,1) to (-1.3,0);
    \draw (0,2) to (-1.3,2);
    \draw (0,2) to (0,0);
    \draw (0,0) to (2.3,-1.2);
    \draw (0,0) to (0,-1.5);
    \draw (2.3,-1.2) to (0,-1.5);
    \draw (-1.3,0) to (0,-1.5);

    
    \node[fill=white,scale=0.7] at (-1.3,1.5) {$1$};
    \node[fill=white,scale=0.7] at (-0.4,1.4) {$2$};
    \node[fill=white,scale=0.7] at (0,1.4) {$3$};
    \node[fill=white,scale=0.7] at (0.7,0.45) {$4$};
    \node[fill=white,scale=0.7] at (2.3,0.3) {$5$};
    \node[fill=white,scale=0.7] at (3.3,0.5) {$6$};
    \node[fill=white,scale=0.7] at (4.2,0.4) {$7$};
    \node[fill=white,scale=0.7] at (4.5,0.8) {$8$};
    \node[fill=white,scale=0.7] at (5.15,0.7) {$9$};
    \node[fill=white,scale=0.7] at (1.3,-0.7) {${10}$};
    \node[fill=white,scale=0.7] at (0,-0.8) {${11}$};
    \node[fill=white,scale=0.7] at (-0.65,0) {${12}$};

    \draw[out=-10,in=130, looseness=1, blue, thick] (-2.5,1) to (0,0);
    \draw[out=-140,in=110, looseness=1.5, blue, thick] (6.2,0) to (6,-1.3);
    \draw[out=150,in=10, looseness=0.5, blue, thick] (6.2,0) to (-2.5,1);
    \draw[out=145,in=-25, looseness=0.5, blue, thick] (6,-1.3) to (0,0);
    
    \draw[blue,thick] (0.8,-0.35) to (0.75,-0.1);
    \draw[blue,thick] (0.9,-0.15) to (0.75,-0.1);
    \draw[blue,thick] (0.9,-0.15) to (0.6,-0.3);
    \draw[blue,thick] (0.8,-0.35) to (0.6,-0.3);
    
    \draw[blue,thick] (0,0.3) to (-0.1,0.4);
    \draw[blue,thick] (-0.3,0) to (-0.4,0.1);
    \draw[blue,thick] (-0.3,0) to (-0.1,0.4);
    \draw[blue,thick] (0,0.3) to (-0.4,0.1);

    \node[blue, fill=white, scale=0.8] at (2.8,0.9) {$C^{+}$};
    \node[blue, fill=white, scale=0.8] at (3.3,-0.6) {$C^{+}$};
    \node[blue, fill=white, scale=0.8] at (5.8,-0.4) {$C^{-}$};
    \node[blue, fill=white, scale=0.8] at (-0.5,0.5) {$C^{-}$};
    
    \draw[fill=red] (3.9,0) circle [radius=2pt];
    \node[red] at (4.05,0.15) {$\mathsf{s}$};

    \draw[fill=black] (6.2,0) circle [radius=2pt];
    \draw[fill=black] (6,1.5) circle [radius=2pt];
    \draw[fill=black] (6,-1.3) circle [radius=2pt];
    \draw[fill=black] (4,1.5) circle [radius=2pt];
    \draw[fill=black] (4.5,-1) circle [radius=2pt];
    \draw[fill=black] (2.3,-1.2) circle [radius=2pt];
    \draw[fill=black] (2.3,1.5) circle [radius=2pt];
    \draw[fill=black] (0,2) circle [radius=2pt];
    \draw[fill=black] (0,0) circle [radius=2pt];
    \draw[fill=black] (-1.3,0) circle [radius=2pt];
    \draw[fill=black] (-2.5,1) circle [radius=2pt];
    \draw[fill=black] (-1.3,2) circle [radius=2pt];
    \draw[fill=black] (0,-1.5) circle [radius=2pt];
    
    \end{tikzpicture}
    \caption{A pair of intersecting tagged arcs, $\gamma_1$ and $\gamma_2$, and their resolution, $C^+ \cup C^-$}
    \label{fig:ExampleY}
\end{figure}

We describe the sets precisely. Consider two arcs $\gamma_1$ and $\gamma_2$ with a point of intersection $\mathsf{s}$. For ease, we lift to a larger surface so that each arc crosses an arc in the triangulation at most once. All arcs are considered up to isotopy, and as we range over the isotopy classes of $\gamma_1$ and $\gamma_2$, the point $\mathsf{s}$ moves. We range over the subset of each isotopy class such that at every point, intersections between $\gamma_i$ and $T$ are minimized. Let $R$ denote the set of all arcs from $T$ which $\mathsf{s}$ can lie on as we range over these subsets of the isotopy classes of $\gamma_1$ and $\gamma_2$. Note both $\gamma_1$ and $\gamma_2$ necessarily cross all arcs in $R$. We resolve the intersection $\mathsf{s}$ by replacing the pair $\{\gamma_1,\gamma_2\}$ with two pairs $\{\gamma_3,\gamma_4\}$ and $\{\gamma_5,\gamma_6\}$ which avoid $\mathsf{s}$ in two different ways. Notice that, up to indexing, $\gamma_3$ and $\gamma_4$ both cross all arcs in $R$ while $\gamma_5$ and $\gamma_6$ do not cross any arcs in $R$. In Figure \ref{fig:ExampleY}, $R = \{\tau_5,\tau_6,\tau_7\}$.

To describe $S$, we begin with $\gamma_1$ and $\gamma_2$ and replace their point of incompatibility $\mathsf{s}$ with  $\asymp$ and \rotatebox[origin=c]{90}{$\asymp$}, forming the resolution. Then, we take the resulting curves and vary through their isotopy class until we reach representatives which intersect the arcs from $T$ a minimal number of times; moreover, we do not allow ourselves to increase the number of intersections at any intermediate point in this process. During this process, we may ``sweep'' past arcs from the triangulation which share an endpoint with $\gamma_1$ or $\gamma_2$. To compute $S$, we look at each endpoint of $\gamma_1$ and $\gamma_2$ and the set of arcs swept as we reach our final representatives. If $\gamma_1$ or $\gamma_2$ is plain at this endpoint, we record any arcs which we crossed in the clockwise direction; if $\gamma_1$ or $\gamma_2$ is notched, we instead record arcs which we crossed in the counterclockwise direction. We do this for both resolutions of $\mathsf{s}$, but we only get a nonzero set $S$ for at most one of the resolutions. In Figure \ref{fig:ExampleY}, when forming the set $C^-$, we sweep past $\tau_8$ in clockwise direction at a plain endpoint and past $\tau_4$ and $\tau_3$ in counterclockwise direction at a tagged endpoint, so $S = \{\tau_3,\tau_4,\tau_8\}$.

One can also see that, if $R$ is nonempty, then the set of arcs which avoid the arcs in $R$ also are those which contribute to $S$. Thus, we can label the two sets in the resolution by $C^+$ and $C^-$ so that the arcs in $C^+$ cross all arcs in $R$ (if they exist) and do not contribute to $S$. We also note the descriptions of $R$ and $S$ can be repeated for an arc with self-intersection, for an intersection of an arc and a closed curve, or an intersection of a pair of closed curves. Naturally, in the case of a pair of closed curves with intersection, $S$ would always be trivial. 

We can now state one of our main theorems. We will set $x_{C^+} := \Pi_{\gamma \in C^+} x_\gamma, x_{C^-} := \Pi_{\gamma \in C^-} x_\gamma, Y_R := \prod_{\gamma \in R}y_\gamma$ and $Y_S := \prod_{\gamma \in S}y_\gamma$.

\begin{theorem}\label{thm:main}
Let $(S,M)$ be a possibly punctured surface, and let $T$ be a triangulation of $(S,M)$. 
\begin{enumerate}
    \item Let $\gamma_1$ and $\gamma_2$ each be an arc or closed curve such that the multicurve $\{\gamma_1,\gamma_2\}$ is incompatible. Let $C^+$ and $C^-$ be the resolution of one point of incompatibility between $\gamma_1$ and $\gamma_2$. Then, \[
    x_{\gamma_1} x_{\gamma_2} = x_{C^+} + Y_RY_S x_{C^-}
    \]
    where $Y$ is a monomial in $y$-variables.
    \item Let $\gamma$ be an arc or closed curve which is incompatible with itself. Let $C^+$ and $C^-$ be the resolution of one point of self-incompatibility. Then, \[
    x_\gamma = x_{C^+} + Y_RY_S x_{C^-}
    \]
    where $Y$ is a monomial in $y$-variables.
\end{enumerate}   
\end{theorem}

For example, with notation from Figure \ref{fig:ExampleY}, we have $x_{\gamma_1}x_{\gamma_2} = x_{C^+} + y_3y_4y_5y_6y_7y_8 x_{C^-}$.

\begin{remark}
Some care should be taken in the definition of the symbol $x_{\gamma}$ in our Theorem statement. In our proofs, we will write $x_\gamma = \textbf{x}^{\gb_\gamma} \sum_{I \in J(\calP_\gamma)} wt(I)$, which is analogous to separating the $\gb$-vector and $F$-polynomial of a cluster variable, and show that such expressions satisfy skein relations. When $\gamma$ is an ordinary arc, $x_\gamma$ indeed is the cluster variable associated to $\gamma$, as can be shown using \cite{pilaud2023posets}. When $\gamma$ has self-intersection or is a closed curve, then it cannot be associated to a cluster variable and we are defining the expression $x_\gamma$. Since our expression in terms of order ideals is combinatorially equivalent to using the dimer partition function of a snake graph, as shown in \cite{ouguz2024cluster}, our definition of the expression $x_\gamma$ matches that used in \cite{musiker2013bases} for the unpunctured case. When $\gamma$ is a generalized arc, then by Theorem 5.1 of \cite{pilaud2023posets}, $x_\gamma$ is the \emph{principal coefficients laminated lambda length} of a geodesic. However, when $\gamma$ is a not necessarily simple closed curve, it is not clear if our $x_\gamma$ would have the same interpretation.
\end{remark}

\begin{remark}
Musiker and Williams provide skein relations for plain arcs on punctured surfaces in \cite{musiker2013}, and our formulas align with theirs when all involved arcs are plain. By exclusively considering plain arcs, Musiker and Williams could describe the $y$-monomials that appear by counting the number of intersections between the arcs in the resolution and the \emph{elementary laminations} associated with the initial triangulation. However, when one also examines notched arcs, it becomes impossible to simplify the calculations to counting intersection numbers; instead, one would need to operate within an \emph{opened surface} context, as defined by Fomin-Thurston \cite{fomin2018cluster}.
\end{remark}

\begin{remark}
Skein relations for punctured surfaces will also appear in the work of Michael Tsironis whose proofs will be in terms of quiver representations over Jacobian algebras (by utilizing the bijection between loop graphs and indecomposable representations) \cite{tsironis}.
\end{remark}

In Sections 8.4-8.6 of \cite{musiker2013bases}, Musiker, Schiffler, and Williams show that \emph{bracelets} $\mathcal{B}^{\circ}$ and \emph{bangles} $\mathcal{B}$ form bases for cluster algebras from unpunctured surfaces and suggest that this result may be extensible to the punctured setting. As part of their discussion about this extension, the authors describe fifteen types of new skein relations that can arise on a punctured surface. In this paper, we group those fifteen cases into four categories:
\begin{enumerate}
    \item \textbf{Incompatibility at a puncture:} an ordinary arc and a singly-notched arc incident to the same puncture, an ordinary arc and a doubly-notched arc with one or two common endpoints; two singly-notched arcs with one or two incompatible tagging(s) at a puncture; and, a singly-notched arc and doubly-notched arc with an incompatible tagging at a puncture.
    \item \textbf{Transverse crossings:}  an ordinary arc and a singly-notched arc; an ordinary arc and a doubly-notched arc; two singly-notched arcs; a singly-notched arc and a doubly-notched arc; or two doubly-notched arcs.
    \item \textbf{Self-crossings:} a singly-notched generalized arc or a doubly-notched generalized arc.
    \item \textbf{Closed curves:}  a loop and a singly-notched or a loop and a doubly-notched arc.
\end{enumerate}

Using properties of the explicit skein relations produced via this case analysis, we obtain our other main theorem:

\begin{theorem}\label{thm:Bases}
    Let $(S,M)$ be such that $S$ has a non-empty boundary, or $S$ has genus $0$ and $|M|>3$. Then $\mathcal{B}^{\circ}$ and $\mathcal{B}$ are bases for $\mathcal{A}(S,M)$.
\end{theorem}

\begin{remark}
In surfaces with at least two marked points on the boundary, one can deduce Theorem~\ref{thm:Bases} by combining the results of Mandel-Qin~\cite{Mandel2023} and Greg Muller~\cite{Muller2016}, even though it wasn’t explicitly mentioned in either paper. Gei{\ss}, Labardini-Fragoso, and Wilson, who also only considered surfaces with non-empty boundary, showed that bangles form a basis for $\mathcal{A}(S, M)$ in the coefficient-free case in \cite{geiss2023bangle}.
\end{remark}

The structure of the paper is as follows. Section~\ref{sec:cluster_algebra} provides brief background information on cluster algebras from surfaces and snake graphs. In Section~\ref{sec:expansion_formula}, we present the expansion formula using order ideals from fence posets, as found in \cite{ouguz2024cluster, pilaud2023posets}, and compare our language here with the language used in the snake graph calculus papers \cite{snakegraphcalculus1, snakegraphcalculus2, snakegraphcalculus3}. Our proof of Theorem \ref{thm:main} is conducted by cases. Each proof follows a similar format, outlined in Subsection \ref{subsec:proofStrategy}. Sections 4 through 7 cover various cases: Section~\ref{sec:puncture_incompatibility} addresses resolutions of arcs with incompatible taggings at a marked point (Case 1); Section~\ref{sec:transverse_crossings} covers arcs with transverse crossings (Case 2); Section~\ref{sec:self_crossings} deals with arcs with self-intersections (Case 3); and, Section~\ref{sec:closed_curves} explores intersections involving a closed curve (Case 4). Sections \ref{subsec:Type0} and \ref{subsec:IncAtTwo} establish notation which will be used frequently in the following sections (see, for example, Figure \ref{fig:Indexing} and Table \ref{table:typeII_wind}).  We conclude by discussing further implications of our work in Section~\ref{sec:implication}. These implications include observing that the analogues of the bangles and bracelets defined in \cite{musiker2013bases} have some of the same desirable properties as in the unpunctured case and form bases for many surface-type cluster algebras. Additionally, we discuss progress towards providing expansion formulas for generalized cluster algebras from punctured orbifolds, building on the work of \cite{banaian2020snake}.

\section{Cluster Algebras from Surfaces}
\label{sec:cluster_algebra}

We will let $S$ be a connected, oriented 2-dimensional Riemann surface with (possibly empty) surface, and let $M$ be a finite set of marked points on $S$ such that there is at least one marked point on each boundary component of $S$. For technical reasons, we do not allow $(S,M)$ to be a sphere with less than four punctures, a monogon with less than two punctures, an unpunctured disc with less than two marked points on the boundary, or a closed surface with two marked points. 

We begin by briefly reviewing some relevant notions about \emph{cluster algebras from surfaces}. We use the standard definition of a cluster algebra from \cite{FZ2002} and follow the surface model of Fomin, Shapiro, and D. Thurston \cite{FST-I}. Although we will highlight a few necessary definitions, we refer the reader to \cite{FST-I} for a much more detailed treatment if desired or to expository works such as \cite{glick2017introduction,williams2014cluster}. In this paper, we work with the subclass of cluster algebra that can be modeled by triangulations of marked surfaces. In particular, we focus on \emph{punctured surfaces} where those marked points may appear in the interior of the surface. Such marked points are referred to as \emph{punctures}.

\begin{definition}\label{def:arc}
An \emph{arc} $\gamma$ on a surface $(S,M)$ is a non-self-intersecting curve in $S$ whose endpoints lie in $M$, but is otherwise disjoint from $M$ and $\partial S$. Arcs are not allowed to be contractible to $\partial S$, or to cut out an unpunctured monogon or bigon. Arcs are considered up to isotopy class. All arcs will be regarded with an orientation and we denote  $s(\gamma)$ as the \emph{starting point} of $\gamma$ and $t(\gamma)$ as the \emph{terminal point}.
\end{definition}

An \emph{ideal triangulation} of $(S,M)$ is a maximal collection of pairwise compatible plain arcs. It is possible that some triangles formed by an ideal triangulation do not have three distinct sides. A triangle with two sides identified is referred to as a \emph{self-folded triangle}. Since an arc enclosed by a self-folded triangle cannot be flipped, in order to allow mutation at every cluster variable (equivalently, in order to make all arcs flippable), punctured surfaces require the slightly more nuanced notion of \emph{tagged arcs}.

\begin{definition}
A \emph{tagged arc} is obtained from an ordinary arc by ``tagging'' each of its endpoints either \emph{plain} or \emph{notched}. In diagrams, notched endpoints are decorated with the $\bowtie$ symbol. Endpoints on $\partial S$ must be tagged plain and both endpoints of loops must have the same tagging.
\end{definition}

A tagged triangulation is a maximal collection of pairwise compatible tagged arcs of $(S,M)$, where loops cutting out once-punctured monogons are forbidden and compatibility of tagged arcs is defined as follows.

\begin{definition}\label{def:CompatibleTaggedArcs}
Given a tagged arc $\gamma$, let $\gamma^0$ denote the underlying plain arc. Tagged arcs $\alpha$ and $\beta$ in $(S,M)$ are \emph{compatible} if and only if the following conditions are satisfied:
\begin{itemize}
    \item the isotopy classes of $\alpha^{0}$ and $\beta^{0}$ contain non-intersecting representatives;
    \item if $\alpha^{0} \neq \beta^{0}$ but $\alpha$ and $\beta$ share an endpoint $p$, then both arcs must have the same tagging at $p$;
    \item and if $\alpha^{0} = \beta^{0}$, then at least one end of $\alpha$ has the same tagging as the corresponding end of $\beta$.
\end{itemize}
\end{definition}

In this paper, we consider only tagged triangulations with no self-folded triangles. Two consequences of this assumption are that (1) our triangulations contain only plain arcs and (2)  for every puncture $p$ on our marked surface, there will be at least two distinct arcs from the triangulation which are incident to $p$. We often refer to these arcs as \emph{spokes}. 

To each ideal triangulation $T$ we may uniquely assign a tagged triangulation $\iota(T)$ by replacing a loop that cuts out a once-punctured monogon with a singly notched arc where the notch is incident to the enclosed puncture.

When thinking about the intersection of arcs, we always consider the representatives of the isotopy classes of each arc with the minimum number of crossings. As helpful notation, let $e(\gamma,\gamma')$ denote the minimum number of crossings between representatives of the isotopy classes of two arcs $\gamma$ and $\gamma'$. Then the first condition of the previous definition can be restated as requiring that $e(\gamma,\gamma') = 0$ in order for $\gamma$ and $\gamma'$ to be compatible.

In the standard surface model dictionary, a (possibly punctured) surface $(S,M)$ defines a surface-type cluster $\mathcal{A}$ via the following correspondences: clusters of $\mathcal{A}$ correspond to tagged triangulations of $(S,M)$; individual tagged arcs in a triangulation $T$ correspond to cluster variables of $\mathcal{A}$, and mutations in $\mathcal{A}$ correspond to flips of tagged arcs.

It will often be convenient for us to instead represent tagged arcs using \emph{hooks}. This notion was used by Wilson in \cite{wilson2020surface} to construct \emph{loop graphs} and produce cluster expansion formulas for tagged arcs and by Labardini-Fragoso as well as by Dom{\'i}nguez to define modules associated to tagged arcs in the Jacobian algebra from a punctured surface  \cite{dominguez2017arc, Labardini-Fragoso-thesis}.

\begin{definition}[\cite{wilson2020surface}]
Let $T$ be an ideal triangulation of a surface $S$ and $\gamma$ be a directed arc on $S$ with an endpoint at puncture $p$. A \emph{hook} at $p$ is a curve that either:
\begin{itemize}
    \item winds around $p$ clockwise or counterclockwise, intersecting all spokes incident to $p$ exactly once, and then follows $\gamma$, if $\gamma$ begins at $p$,
    \item or follows $\gamma$ and then winds around $p$ clockwise or counterclockwise, intersecting all spokes incident to $p$ exactly once, if $\gamma$ ends at $p$.
\end{itemize}
\end{definition}

An example of a tagged arc $\gamma^{(p)}$ and the two possible replacements of its notched endpoint with a hook is shown below.

\begin{figure}[H]
\captionsetup{width=0.9\textwidth}
    \centering
    \begin{tikzpicture}[scale=1.5]
    \draw (-2,0) to (-2.5,1);
    \draw (-2,0) to (-2.5,-1);
    \draw (-2.5,1) to (-2,0);
    \draw (-2.5,-1) to (-2,0);
    \draw (-2,0) to (-1.5,0.5);
    \draw (-2,0) to (-1.5,-0.5);
    \draw[red,thick] (-3.3,0) to node[midway,below]{$\gamma^{(p)}$} (-2,0);
    \draw[red,thick] (-2.1,0.1) to (-2.2,-0.1);
    \draw[red,thick] (-2.2,0.1) to (-2.1,-0.1);
    \draw[red,thick] (-2.1,0.1) to (-2.2,0.1);
    \draw[red,thick] (-2.1,-0.1) to (-2.2,-0.1);
    \draw[fill=black] (-2,0) circle [radius=1pt];
    \node at (-1.95,0.2) {$p$};

    \draw (1,0) to (0.5,1);
    \draw (1,0) to (0.5,-1);
    \draw (0.5,1) to (1,0);
    \draw (0.5,-1) to (1,0);
    \draw (1,0) to (1.5,0.5);
    \draw (1,0) to (1.5,-0.5);
    \draw[red,thick] (-0.3,0) to node[midway,below]{$\tilde{\gamma}$} (0.75,0);
    \draw[red, thick, out=90,in=90,looseness=2] (0.75,0) to (1.3,0);
    \draw[red, thick, out=-90,in=-70,looseness=2] (1.3,0) to (0.8,-0.2);
    \draw[fill=black] (1,0) circle [radius=1pt];
    \node at (1.05,0.15) {$p$};

    \draw (4,0) to (3.5,1);
    \draw (4,0) to (3.5,-1);
    \draw (3.5,1) to (4,0);
    \draw (3.5,-1) to (4,0);
    \draw (4,0) to (4.5,0.5);
    \draw (4,0) to (4.5,-0.5);
    \draw[red,thick] (2.7,0) to node[midway,below]{$\tilde{\gamma}$} (3.75,0);
    \draw[red, thick, out=-90,in=-90,looseness=2] (3.75,0) to (4.3,0);
    \draw[red, thick, out=90,in=70,looseness=1.5] (4.3,0) to (3.8,0.2);
    \draw[fill=black] (4,0) circle [radius=1pt];
    \node at (4.05,0.2) {$p$};

    \end{tikzpicture}
    \label{fig:hook}
\end{figure}

We will consider several other families of curves on a surface here, namely closed curves and arcs with self-intersection. The latter are sometimes referred to as \emph{generalized arcs}. When a closed curve has no self-intersections, it is called \emph{essential}.

\subsection{Snake Graphs}\label{subsec:snakegraph}

\emph{Snake graphs} were introduced by Musiker, Schiffler, and Williams in \cite{musiker2011positivity} to give a proof of positivity for cluster algebras of surface type and produce cluster expansion formulas. Later, in \cite{wilson2020surface}, Wilson gave a simplified construction, called a \emph{loop graph}, for notched arcs. We will briefly summarize these constructions, but refer the reader to the original papers for further details and examples.

Throughout this section, we will use the following triangulation as a running source of examples.

\begin{center}
\begin{tikzpicture}

    \draw (4.5,-1) to (5,0);
    \draw (4,1.5) to (5,0);
    \draw (6.5,0) to (5,0);
    \draw (3,-1.2) to (4.5,-1);
    \draw (4.5,-1) to (4,1.5);
    \draw (3,-1.2) to (4,1.5);
    \draw (3,-1.2) to (3,1.5);
    \draw (0,1) to (1.5,0);
    \draw (0,1) to (0,-1);
    \draw (0,-1) to (1.5,0);
    \draw (3,1.5) to (1.5,0);
    \draw (3,-1.2) to (1.5,0);
    \draw (3,-1.2) to (0,-1);
    \draw (3,1.5) to (0,1);
    \draw (3,1.5) to (4,1.5);
    \draw[out=0,in=100, looseness=1.5] (4,1.5) to (6.5,0);
    \draw[out=0,in=-100, looseness=1.5] (4.5,-1) to (6.5,0);


    \node[fill=white,scale=0.7] at (0.6,0.5) {$\sigma_2$};
    \node[fill=white,scale=0.7] at (0.8,-0.4) {$\sigma_3$};
    \node[fill=white,scale=0.7] at (2,-0.3) {$\sigma_4$};
    \node[fill=white,scale=0.7] at (2,0.4) {$\sigma_1$};

    \node[fill=white,scale=0.7] at (3,0.5) {$\tau_1$};
    \node[fill=white,scale=0.7] at (3.5,0.2) {$\tau_2$};
    \node[fill=white,scale=0.7] at (4.15,0.3) {$\tau_3$};

    \draw[red,thick] (1.5,0) to (5,0);
    \draw[out=90,in=90, looseness=0.8, blue, thick] (1,0) to (5.5,0);

    \node[red, fill=white, scale=0.8] at (2.8,0.0) {$\gamma_1$};
    \node[blue, fill=white, scale=0.8] at (3.3,1) {$\gamma_2$};

    \node[fill=white,scale=0.7] at (4.7,0.4) {$\eta_3$};
    \node[fill=white,scale=0.7] at (5.8,0) {$\eta_2$};
    \node[fill=white,scale=0.7] at (4.8,-0.3) {$\eta_1$};
    \draw[out=-90,in=-90, looseness=0.5, blue, thick] (1,0) to (5.5,0);

    \node[scale=0.8] at (1.5,0.2) {$p$};
    \node[scale=0.8] at (5.1,0.2) {$q$};

    \draw[fill=black] (5,0) circle [radius=2pt];
    \draw[fill=black] (6.5,0) circle [radius=2pt];
    \draw[fill=black] (4,1.5) circle [radius=2pt];
    \draw[fill=black] (4.5,-1) circle [radius=2pt];
    \draw[fill=black] (3,-1.2) circle [radius=2pt];
    \draw[fill=black] (3,1.5) circle [radius=2pt];
    \draw[fill=black] (0,1) circle [radius=2pt];
    \draw[fill=black] (0,-1) circle [radius=2pt];
    \draw[fill=black] (1.5,0) circle [radius=2pt];

    \end{tikzpicture}
    \label{fig:arcs}
\end{center}

Fix a triangulation $T$ of a surface $(S,M)$, a plain arc $\gamma$, and an orientation of $\gamma$. If $\gamma \in T$, then  the snake graph $\mathcal{G}_{\gamma,T}$ is a single edge labeled with $\gamma$. Otherwise, $\gamma$ must cross at least one arc in $T$. Let $\tau_{i_1}, \dots, \tau_{i_k}$ be the sequence of arcs of $T$ crossed by $\gamma$, in order. Each intersection of $\gamma$ with an arc in $\tau_{i_j} \in T$ corresponds to a single square tile $G_j$.  The snake graph $\mathcal{G}_{\gamma,T}$ is recursively constructed by gluing together individual tiles $G_{1}, \dots, G_{k}$ such that $G_{1}$ is the southwest-most tile and each pair $G_{j}, G_{j+1}$ shares one edge.

An individual tile $G_j$ is constructed by taking the two triangles bordered by $\tau_{i_j}$ in the triangulated surface and gluing them along $\tau_{i_j}$ such that both have either the same or opposite orientation as in $S$. If the orientation of the triangles in $G_{j}$ matches their orientation in $S$, we say that $G_j$ has relative orientation $+1$ and denote this as $\textrm{rel}(G_j) = +1$. Otherwise, $\textrm{rel}(G_j) = -1$. The dashed lines (diagonals) indicate labels of the tiles and are not considered edges of the graph.

\begin{center}
\begin{tikzpicture}
        \draw[thick] (0,0) to node[below,midway]{$c$} (1,0) to node[right,midway]{$d$} (1,1) to node[midway,above]{$a$} (0,1) to node[left,midway]{$b$} (0,0);
        \draw[thick, dashed, black!60!white] (0,1) to node[above,midway,xshift=4,yshift=-5]{$\tau_{i_j}$} (1,0);
        
        \draw[thick] (3,0) to node[below,midway]{$d$} (4,0) to node[right,midway]{$c$} (4,1) to node[midway,above]{$b$} (3,1) to node[midway,left]{$a$} (3,0);
        \draw[thick, dashed, black!60!white] (3,1) to node[above,midway,xshift=4,yshift=-5]{$\tau_{i_j}$} (4,0);
    \end{tikzpicture}
\end{center}

When $\gamma$ crosses two consecutive arcs $\tau_{i_j}$ and $\tau_{i_{j+1}}$, we glue $G_{j}$ and $G_{j+1}$ along their shared edge labeled $\tau_{[j]}$, where $\tau_{[j]}$ is the arc from $T$ or the boundary arc which forms a triangle with $\tau_{i_j}$ and $\tau_{i_{j+1}}$, using planar embeddings of the tiles so that $\textrm{rel}(G_{j}) \neq \textrm{rel}(G_{j+1})$. 

\emph{Band graphs}, which represent closed curves, are constructed using a similar procedure. Let $\zeta$ be a closed curve on $S$ that is not homotopic to a curve which does not cross any arcs in $T$ and has no contractible kinks. Choose a point $p$ on $\zeta$ such that $p$ is not in the interior of a pending arc and fix an orientation of the curve. We then construct a snake graph by beginning at $p$ and following the chosen orientation of $\zeta$. The first and last tiles will correspond to arcs bordering the same triangle and will always have a unique common edge. We identify these edges to form the \emph{band graph}, $\mathcal{G}_{\zeta,T}^\circ$ corresponding to $\zeta$.

For singly and doubly notched arcs, Wilson constructed \emph{loop graphs} using the previously described hook construction. Let $\gamma$ be a notched arc such that $\gamma^{0} \not\in T$. Orient $\gamma$ so that it is notched at its starting point. Let $\overline{\gamma}$ denote the \emph{hooked arc} obtained from $\gamma$ by replacing each notched end with a hook. While there are two choices for the orientation of the hook, we will see the final output will not depend on the choice. Let $\mathcal{G}_{\overline{\gamma},T} = (G_1, \dots, G_d)$ be the snake graph obtained from $\overline{\gamma}$ via the construction for plain arcs. The snake graph $\mathcal{G}_{\gamma^{0},T} = (G_{k_1}, \dots, G_{k_2})$ corresponding to $\gamma^{0}$ appears as a subgraph of $\mathcal{G}_{\overline{\gamma},T}$ for some $k_1, k_2 \in \{ 1, \dots, d\}$, necessarily with $k_1 \leq k_2$.

The tiles $G_1$ and $G_{k_1}$ contain a common triangle formed by two spokes of the puncture $p$ and the first arc crossed by $\gamma^{0}$. Let $c$ denote the boundary edge in the copy of this triangle that appears in $G_k$ and $c'$ denote the corresponding edge in $G_1$. Let $x'$ denote the southwest vertex of $c'$ in $G_1$ and $y'$ denote the other vertex of $c'$. Denote the copies of these vertices in $G_{k_1}$ by $x$ and $y$. The \emph{loop} with respect to $G_1,k_{1}$ is formed by identifying $c$ and $c'$ such that $x$ and $y$ are, respectively, identified with $x'$ and $y'$. This identified edge is referred to as a \emph{cut edge}. The loop with respect to $G_d,k_2$ has an analogous construction where southwest is replaced with northeast. The graph $\mathcal{G}_{\gamma,T}$ produced by this gluing is the \emph{loop graph} corresponding to $\gamma$.

The diagram below shows the snake graph for $\gamma_1$ (on the left) and the loop graph for $\gamma_1^{(q)}$ (on the right) in our running example. The glued edge in $\mathcal{G}_{\gamma_1^{(q)}}$ is highlighted in orange.

\begin{center}
\begin{tabular}{ccc}
\begin{tikzpicture}
\draw (0,0) to node[midway,below]{$\sigma_1$} (1,0) to (1,1) to node[midway,above]{$\tau_2$} (0,1) to node[midway,left]{$\sigma_4$} (0,0); 
\draw[dashed,gray] (0,1) to node[midway,above,xshift=4,yshift=-2]{$\tau_1$} (1,0); 

\draw (1,0) to node[midway,below]{$\tau_1$} (2,0) to (2,1) to node[midway,above]{$\tau_3$} (1,1); 
\draw[dashed,gray] (1,1) to node[midway,above,xshift=4,yshift=-2]{$\tau_2$} (2,0); 

\draw (2,0) to node[midway,below]{$\tau_2$} (3,0) to node[midway,right]{$\eta_3$} (3,1) to node[midway,above]{$\eta_1$} (2,1); 
\draw[dashed,gray] (2,1) to node[midway,above,xshift=4,yshift=-2]{$\tau_3$} (3,0);
\end{tikzpicture}
&
\qquad\qquad
&
\begin{tikzpicture}
\draw (0,0) to node[midway,below]{$\sigma_1$} (1,0) to (1,1) to node[midway,above]{$\tau_2$} (0,1) to node[midway,left]{$\sigma_4$} (0,0); 
\draw[dashed,gray] (0,1) to node[midway,above,xshift=4,yshift=-2]{$\tau_1$} (1,0); 

\draw (1,0) to node[midway,below]{$\tau_1$} (2,0) to (2,1) to node[midway,above]{$\tau_3$} (1,1); 
\draw[dashed,gray] (1,1) to node[midway,above,xshift=4,yshift=-2]{$\tau_2$} (2,0); 

\draw (2,0) to node[midway,below]{$\tau_2$} (3,0) to node[midway,right,yshift=-6]{$\eta_3$} (3,1) to node[midway,above]{$\eta_1$} (2,1); 
\draw[dashed,gray] (2,1) to node[midway,above,xshift=4,yshift=-2]{$\tau_3$} (3,0);
\draw[ultra thick,orange] (2,1) to (3,1); 

\draw (3,0) to node[midway,below]{$\tau_3$} (4,0) to (4,1) to node[midway,above]{$\eta_2$} (3,1); 
\draw[dashed,gray] (3,1) to node[midway,above,xshift=4,yshift=-2]{$\eta_1$} (4,0); 

\draw (4,0) to node[midway,below]{$\eta_1$} (5,0) to node[midway,right]{$\eta_3$} (5,1) to (4,1); 
\draw[dashed,gray] (4,1) to node[midway,above,xshift=4,yshift=-2]{$\eta_2$} (5,0); 

\draw (4,1) to node[midway,left,xshift=2]{$\eta_2$} (4,2) to node[midway,above]{$\eta_1$} (5,2) to node[midway,right]{$\tau_3$} (5,1); 
\draw[dashed,gray] (4,2) to node[midway,above,xshift=4,yshift=-2]{$\eta_3$} (5,1); 
\draw[ultra thick, orange] (4,2) to (5,2); 
\end{tikzpicture} \\
\highlight{$\mathcal{G}_{\gamma_1}$}
&
& \highlight{$\mathcal{G}_{\gamma_1^{(q)}}$}
\end{tabular}
\end{center}

Given a triangulation $T$, these constructions can be used to recover the expansion $x_{\gamma}^T$ of a cluster algebra element associated to $\gamma$ with respect to the cluster corresponding to $T$. Doing so requires defining several statistics on $\mathcal{G}_{\gamma,T}$.

An arc $\gamma$ that crosses the sequence of arcs $\tau_{i_1}, \dots, \tau_{i_k}$ has \emph{crossing monomial} $\textrm{cross}(\gamma,T) = \prod_{j=1}^{k} x_{\tau_{i_j}}$. The \emph{weight} of an edge labeled by $\tau$ is $x_{\tau}$. Recall that a \emph{perfect matching} of a graph $G=(V,E)$ is a subset of edges $P \subset E$ such that every vertex $v \in V$ is incident to exactly one edge in $P$. If a perfect matching of a loop or band graph can be extended to a perfect matching of the snake graph that results from forgetting the identification of edges,  then we refer to $P$ as a \emph{good matching}. Given a perfect or good matching $P$, let $x(P)$ be the product of the weights of edges used in $P$.

An edge in a snake graph is called a \emph{boundary edge} if it borders only one tile. Every snake graph has two perfect matchings that contain only boundary edges. If $\mathcal{G}$ is a snake graph or band graph, we refer to the boundary matching that does not contain the west edge on $G_1$ as the \emph{minimal matching} $P_{\textrm{min}}$. If $\mathcal{G}$ has a loop with respect to $G_1,k_1$, $P_{\textrm{min}}$ is the matching that does not use the west edge of $G_{k_1}$.

Given any perfect or good matching $P$, the symmetric difference $P \ominus P_{\text{min}} := (P \cup P_{\text{min}}) \backslash (P \cap P_{\text{min}})$ is a collection of cycles of even length. Let $h(P) = \{ h_1, \dots, h_m \}$ denote the set, with multiplicity, of diagonal labels of tiles enclosed by $P \ominus P_{\textrm{min}}$. Then let $y(P) := \prod_{h_i \in h(P)}y_{i}$.

\begin{theorem}[Theorem 4.10 of \cite{musiker2011positivity} and Theorem 5.7 of \cite{wilson2020surface}]
\label{thm:SnakeGraphExpansion}
Let $(S,M)$ be a marked surface and $T^{0}$ be an ideal triangulation with corresponding tagged triangulation $T = \iota(T^{0})$. Let $\mathcal{A}$ be the corresponding cluster algebra with principal coefficients. Let $\gamma$ be a tagged arc such that $\gamma^{(0)} \not\in T$ and let $\mathcal{G}_{\gamma,T}$ be the corresponding snake or loop graph. Then the expansion of $x_{\gamma} \in \mathcal{A}$ with respect to the cluster associated to $T$ is given by
\begin{align}\label{eq:Expansion}
    x_{\gamma} = \frac{1}{\textrm{cross}(T,\gamma)} \sum_{P} x(P)y(P)
\end{align}
where the summation is indexed by either perfect matchings (if $\gamma$ is a plain arc) or good matchings (if $\gamma$ is a notched arc) of $\mathcal{G}_{\gamma,T}$.
\end{theorem}

Generalized arcs and closed curves do not correspond to cluster variables, so one must define what the symbol $x_\gamma$ represents. Some special cases require their own definition, following \cite{musiker2013bases}.

\begin{definition}\label{def:contandkink}
    \begin{enumerate}
        \item If $\gamma$ is a contractible closed loop, then $x_\gamma$ is -2.
        \item Suppose $\gamma$ has a contractible kink. Let $\gamma'$ be the curve with this kink removed. Then, $x_\gamma=(-1)x_{\gamma'}$.
    \end{enumerate}
\end{definition}

If $\gamma$ is a generalized arc or closed curve which is not covered in Definition \ref{def:contandkink}, then we define the symbol $x_\gamma$ to be the result of applying the expansion formula to $\mathcal{G}_{\gamma,T}$. The issue of defining such expressions is discussed in Section 8.2 in \cite{musiker2013bases} and our approach is the ``combinatorial definition'' in Section 8.2 in \cite{musiker2013bases}. In Corollaries \ref{cor:PlainTimesNotchedEqualsLoop} and \ref{cor:MSWThm12.9} we will show that we could have equivalently used the ``algebraic definition'' by defining expressions for generalized notched arcs using arcs with less notched endpoints.

\section{Cluster Expansion Formula}
\label{sec:expansion_formula}

\subsection{Posets}

In the following, we construct a poset $(\calP_\gamma, \preceq)$ from a (possibly generalized) arc $\gamma$ on $(S,M)$ with triangulation $T$.

First, suppose that $\gamma$ is an arc with both endpoints tagged plain. Choose an orientation for $\gamma$. Let $\tau_{i_1},\ldots,\tau_{i_d}$ be the list of arcs of $T$ crossed by $\gamma$, in order determined by the orientation we placed on $\gamma$. We will place a poset structure on $[d]$ in the following way. Two consecutive arcs crossed by $\gamma$, $\tau_{i_j}$ and $\tau_{i_{j+1}}$ will border a triangle $\Delta_j$ which $\gamma$ passes through between these crossings. Let $s_j$ be the shared endpoint of $\tau_{i_j}$ and $\tau_{i_{j+1}}$ which is an endpoint of $\Delta$. If $s_j$ lies to the right of $\gamma$ (with respect to the orientation placed on $\gamma$), then we set ${j} \succ {j+1}$; otherwise, we set ${j} \prec {j+1}$. The resulting poset is sometimes called a \emph{fence poset} since its Hasse diagram is a path graph. Notice this process is the same whether $\gamma$ has self-intersections or not. An example of the plain arc $\gamma_1$ from our running example is highlighted in blue in Table~\ref{table:posetExamples}. By abuse of notation, we will usually refer to an element $j$ of $\calP_\gamma$ instead by $\tau_{i_j}$, the arc it corresponds to, as in Table~\ref{table:posetExamples}.

Next, suppose that $\gamma^{(p)}$ is notched at its starting point $s(\gamma) = p$. Draw the fence poset for the underlying plain arc first. Suppose the first triangle $\gamma^{(p)}$ passes through  is $\Delta_0$. Necessarily, $\Delta_0$ is bordered by $\tau_{i_1}$ and two spokes at $p$. Label the set of all spokes of $T$ at $p$ $\sigma_1,\sigma_m$ in counterclockwise order where $\sigma_1$ is the counterclockwise neighbor of $\tau_{i_1}$. If an arc has both endpoints at $p$, then it will receive two labels depending on where these ends fall in the cyclic order around $p$. We include elements $1^s,\ldots,m^s$ to the poset, and set $m^s \succ (m-1)^s \succ \cdots \succ 1^s$, $1^s \prec 1$ and $m^s \succ 1$. If we have an arc which is instead notched at its terminal point, we repeat this process with elements $1^t,\ldots,m^t$, and we combine these processes for an arc notched at both endpoints. We call the resulting posets \emph{loop fence posets} as they correspond to the loop graphs given by Wilson in \cite{wilson2020surface}. We say that the elements $1^s,\ldots,m^s$ are in a \emph{loop}. The designation of what is and is not in a loop will be important to define the corresponding $\gb$-vector in Section \ref{subsec:g-vec}. If we wish to refer to a loop fence poset $\calP$ with the loop portion removed, we will denote this $\calP^0$. Examples for $\gamma_1^{(p)}$, $\gamma_1^{(q)}$, and $\gamma_1^{(p,q)}$ are shown in Table~\ref{table:posetExamples}.

\begin{table}[h]
\captionsetup{width=0.9\textwidth}
    \centering
    \begin{tabular}{c|c|c}
        \highlight{$\calP_{\gamma_1^{(p)}}$} & \highlight{$\calP_{\gamma_1^{(q)}}$} & \highlight{$\calP_{\gamma_1^{(p,q)}}$} \\
        \begin{tikzpicture}
            \node (s4) at (0,0) {$\sigma_4$};
            \node (s3) at (0.5,-1) {$\sigma_3$};
            \node (s2) at (1,-2) {$\sigma_2$};
            \node (s1) at (1.5,-3) {$\sigma_1$};
            \node[blue] (t1) at (2,-2) {$\tau_1$};
            \node[blue] (t2) at (2.5,-3) {$\tau_2$};
            \node[blue] (t3) at (3,-2) {$\tau_3$};

            \draw (s4) to (t1);
            \draw (s4) to (s3);
            \draw (s3) to (s2);
            \draw (s2) to (s1);
            \draw (s1) to (t1);
            \draw[blue] (t1) to (t2);
            \draw[blue] (t2) to (t3);
        \end{tikzpicture}
        &
        \begin{tikzpicture}
            \node[blue] (t1) at (2,-2) {$\tau_1$};
            \node[blue] (t2) at (2.5,-3) {$\tau_2$};
            \node[blue] (t3) at (3,-2) {$\tau_3$};
            \node (e1) at (3.5,-3) {$\eta_1$};
            \node (e2) at (4,-2) {$\eta_2$};
            \node (e3) at (4.5,-1) {$\eta_3$};

            \draw (t3) to (e1);
            \draw (e1) to (e2);
            \draw (e2) to (e3);
            \draw (t3) to (e3);
            \draw[blue] (t1) to (t2);
            \draw[blue] (t2) to (t3);
            
        \end{tikzpicture}
        &
        \begin{tikzpicture}
            \node (s4) at (0,0) {$\sigma_4$};
            \node (s3) at (0.5,-1) {$\sigma_3$};
            \node (s2) at (1,-2) {$\sigma_2$};
            \node (s1) at (1.5,-3) {$\sigma_1$};
            \node[blue] (t1) at (2,-2) {$\tau_1$};
            \node[blue] (t2) at (2.5,-3) {$\tau_2$};
            \node[blue] (t3) at (3,-2) {$\tau_3$};
            \node (e1) at (3.5,-3) {$\eta_1$};
            \node (e2) at (4,-2) {$\eta_2$};
            \node (e3) at (4.5,-1) {$\eta_3$};

            \draw (s4) to (t1);
            \draw (s4) to (s3);
            \draw (s3) to (s2);
            \draw (s2) to (s1);
            \draw (s1) to (t1);
            \draw[blue] (t1) to (t2);
            \draw[blue] (t2) to (t3);
            \draw (t3) to (e1);
            \draw (e1) to (e2);
            \draw (e2) to (e3);
            \draw (t3) to (e3);
        \end{tikzpicture} \\
        $\begin{aligned}
            \mathbf{a}_{\gamma_1^{(p)}} &= \mathbf{e}_{\tau_2} \\
            \mathbf{b}_{\gamma_1^{(p)}} &= \mathbf{e}_{\tau_1} \\
            \mathbf{r}_{\gamma_1^{(p)}} &= -\mathbf{e}_{\sigma_1} + \mathbf{e}_{\eta_3}
        \end{aligned}$
        &
        $\begin{aligned}
            \mathbf{a}_{\gamma_1^{(q)}} &= \mathbf{e}_{\tau_2} \\
            \mathbf{b}_{\gamma_1^{(q)}} &= \mathbf{e}_{\tau_3} \\
            \mathbf{r}_{\gamma_1^{(q)}} &= \mathbf{e}_{\sigma_4} - \mathbf{e}_{\eta_{1}}
        \end{aligned}$
        &
        $\begin{aligned}
            \mathbf{a}_{\gamma_1^{(p,q)}} &=  \mathbf{e}_{\tau_2}  \\
            \mathbf{b}_{\gamma_1^{(p,q)}} &= \mathbf{e}_{\tau_1} + \mathbf{e}_{\tau_3} \\
            \mathbf{r}_{\gamma_1^{(p,q)}} &= - \eb_{\sigma_1} - \eb_{\eta_1}
        \end{aligned}$
\end{tabular}

\caption{The loop fence posets $\calP_{\gamma_1^{(p)}}$, $\calP_{\gamma_1^{(q)}}$, and  $\calP_{\gamma_1^{(p,q)}}$ for the arc $\gamma_1$ from our running example. The fence poset $\calP_{\gamma_1}$ for the plain arc $\gamma_1$ appears as a subposet of all three, indicated in blue, and has  $\mathbf{a}_{\gamma_1} = \mathbf{e}_{\tau_2}$, $\mathbf{b}_{\gamma_1} = 0$, and $\mathbf{r}_{\gamma_1} = \mathbf{e}_{\sigma_1} + \mathbf{e}_{\eta_1}$}.

\label{table:posetExamples}
\end{table}

Finally, suppose that $\gamma$ is a closed curve. Choose a point $a$ of $\gamma$ which is not a point of intersection between $\gamma$ and $T$ and choose an orientation of $\gamma$. Treat $\gamma$ like an arc with $s(\gamma) = t(\gamma) = a$ and form the fence poset on $[d]$ associated to this arc. It must be that $\tau_{i_1}$ and $\tau_{i_d}$ share an endpoint which is an endpoint of the triangle containing $a$. If this endpoint is to the right of $\gamma$ with the chosen orientation, we set $d \succ 1$; otherwise, we set $d \prec 1$. These posets are called \emph{circular fence posets} since the underlying graph of such a Hasse diagram is a cycle. 

The circular fence poset for the closed curve $\gamma_2$ in our running example is shown in Figure \ref{fig:ExampleCircularFence}.

\begin{figure}[h]
\captionsetup{width=0.9\textwidth}
\centering
\begin{minipage}{0.5\textwidth}
\begin{tikzpicture}
    \node (t2-1) at (0,0) {$\tau_2$};
    \node (t3-1) at (0.5,1) {$\tau_3$};
    \node (s3) at (1,2) {$\eta_3$};
    \node (s2) at (1.5,1) {$\eta_2$};
    \node (s1) at (2,0) {$\eta_1$};
    \node (t3-2) at (2.5,1) {$\tau_3$};
    \node (t2-2) at (3,0) {$\tau_2$};
    \node (t1-2) at (3.5,1) {$\tau_1$};
    \node (e4) at (4,2) {$\sigma_4$};
    \node (e3) at (4.5,1) {$\sigma_3$};
    \node (e2) at (5,0) {$\sigma_2$};
    \node (e1) at (5.5,-1) {$\sigma_1$};
    \node (t1-1) at (6,1) {$\tau_1$};

    \draw (t2-1) to (t3-1);
    \draw (t3-1) to (s3);
    \draw (s3) to (s2);
    \draw (s2) to (s1);
    \draw (s1) to (t3-2);
    \draw (t3-2) to (t2-2);
    \draw (t2-2) to (t1-2);
    \draw (t1-2) to (e4);
    \draw (e4) to (e3);
    \draw (e3) to (e2);
    \draw (e2) to (e1);
    \draw (e1) to (t1-1);
    \draw (t2-1) to (t1-1);
\end{tikzpicture}
\end{minipage}
\begin{minipage}{0.3\textwidth}
$\begin{aligned}
    \mathbf{a}_{\gamma_2} &= \mathbf{e}_{\eta_{1}} + 2\mathbf{e}_{\tau_2} + \mathbf{e}_{\sigma_1} \\
    \mathbf{b}_{\gamma_2} &= \mathbf{e}_{\eta_3} +  \mathbf{e}_{\tau_1}+ \mathbf{e}_{\tau_3} + \mathbf{e}_{\sigma_4} \\
    \mathbf{r}_{\gamma_2} &= 0
\end{aligned}$
\end{minipage}
\caption{The circular fence poset associated to the closed curve $\gamma_2$ from the beginning of Section \ref{subsec:snakegraph}}\label{fig:ExampleCircularFence}
\end{figure}
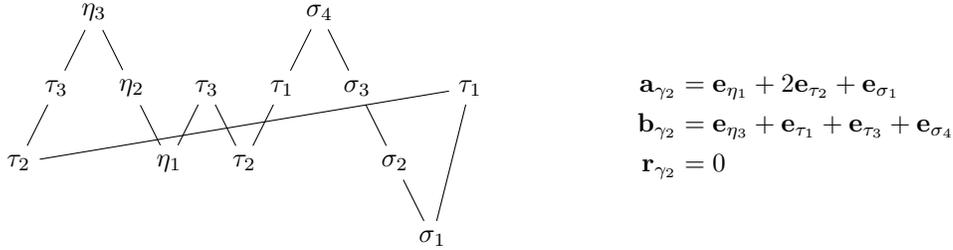

Given a subset $I$ of a poset $\calP_\gamma$, we define the \emph{content} of $I$ to be the multiset with elements from $T$ where $\tau \in T$ appears with multiplicity $m$ if there are exactly $m$ elements of $I$ which correspond to $\tau$. 

The poset $\calP_\gamma$ is related to the snake graph $\mathcal{G}_{\gamma,T}$ in the following sense. Given a poset $\calP$, let $J(\calP)$ denote its lattice of order ideals. Recall a \emph{(lower) order ideal} of a poset $(\calP,\succeq)$ is a subset $I \subseteq \calP$ such that if $x \in I$ and $y \succeq x$, $y \in I$. Given $x \in \calP$, the \emph{principal order ideal} $\langle x \rangle$ is the defined as $\langle x \rangle := \{ y \in \calP : y \succeq x\}$. Dually, an \emph{order filter} is a subset $I \subseteq \calP$ such that if $x \in I$ and $y \preceq x$, $y \in I$; the definition of a principal order filter is dual to principal order ideal.

\begin{theorem}[Theorems 5.4, 5.7 \cite{musiker2011positivity}, Theorem 7.9 \cite{wilson2020surface}]\label{Thm:PMPoset}
Consider a surface $S$ with triangulation $T$. Let $\gamma$ be an arc on $S$ such that $\gamma^0 \notin T$.
\begin{enumerate}
    \item Let $\gamma$ be a plain arc. There exists a lattice bijection between $J(\calP_\gamma)$ and the lattice of perfect matchings of $\mathcal{G}_{\gamma,T}$.
    \item Let $\gamma$ be a tagged arc or a closed curve. There exists a lattice bijection between $J(\calP_\gamma)$ and the lattice of good matchings of $\mathcal{G}_{\gamma,T}$.
\end{enumerate}
\end{theorem}

It will be convenient for us to remove the condition that $\gamma^0 \notin T$ in the above. We will associate a \emph{decorated} poset to arcs whose underlying plain arc is in $T$; this will be a tuple $\widetilde{\calP}_\gamma = (\calP_\gamma, \omega)$ which consists of a poset $\calP_\gamma$ along with a word in the formal letters $\{\tau^\pm: \tau \in T\}$. This decoration will not affect the set of order ideals of the poset, but will affect the monomial $\gb$ we associate to the poset, as in Section \ref{subsec:g-vec}. All our previously defined posets can be considered decorated posets with $\omega = \emptyset$, and in future sections we will often not refer to a decoration when it is $\emptyset$.

Throughout the remainder of the section, we assume $\gamma$ is $\gamma^0 \in T$. First, suppose $\gamma = \gamma^0$ is a plain arc in $T$. Then, $\widetilde{\calP}_\gamma$ will be the tuple $(\emptyset, \gamma)$.

Next, suppose that $\gamma^{(p)}$ is notched at $s(\gamma) = p$ and not at $t(\gamma)$. Then, we have $\widetilde{\calP}_{\gamma^{(p)}} = (\calP_{\ell_p}, \gamma^-)$ where $\ell_p$ is the arc which cuts out a monogon containing $p$ and no other punctures and we regard all of $\calP_{\ell_p}$ as a loop. The poset $\calP_{\ell_p}$ is a chain consisting of all arcs from $T$ incident to $p$ except for $\gamma$. The process is identical to an arc only notched at its terminal point.

Finally, we consider an arc $\gamma^{(p,q)}$ which is notched at both endpoints; it is possible that $p = q$. Form the posets $\calP_{\ell_p}$ and $\calP_{\ell_q}$; these will be chains $1^s \prec 2^s \prec \cdots \prec m^s$ and $1^t \prec 2^t \prec \cdots \prec h^t$. We define $\calP_{\gamma^{(p,q)}}$ to be the poset on $\{1^s,\ldots,m^s,1^t,\ldots,h^t,0^-,0^+\}$ with the relations from $\calP_{\ell_p}$ and $\calP_{\ell_q}$ as well as $0^- \prec 1^s \prec h^t \prec 0^+$ and $0^- \prec 1^t \prec m^s \prec 0^+$,  so that $0^-$ is the minimum and $0^+$ is the maximum. The elements $0^\pm$ of $\calP_{\gamma^{(p,q)}}$ will represent the arc $\gamma^0 \in T$. Then we set the decorated poset to be $\widetilde{\calP}_{\gamma^{(p,q)}} = (\calP_{\gamma^{(p,q)}}, \emptyset)$. We regard the entire poset as a loop.

\begin{center}
\begin{tikzpicture}
    \node[] (k1-) at (0,0){$0^-$};
    \node[] (k) at (-1,1){$1^s$};
    \node[] (dotsl) at (-1,2){$\vdots$};
    \node[] (k2) at (-1,3){$m^s$};
    \node[] (as) at (1,1){$1^t$};
    \node[] (dotsr) at (1,2){$\vdots$};
    \node[] (a1) at (1,3){$h^t$};
    \node[] (k1+) at (0,4){$0^+$};
    \draw (k1+) -- (a1);
    \draw(k1+) -- (k2);
    \draw(k1-) -- (k);
    \draw(k1-) -- (as);
    \draw (as) --(dotsr);
    \draw (dotsr) -- (a1);
    \draw(k) -- (dotsl);
    \draw(dotsl) -- (k2);
    \draw(k2) -- (as);
    \draw(a1) -- (k);
\end{tikzpicture}
\end{center}

We remark that these posets for tagged arcs with an underlying plain arc in $T$ were independently discovered recently by many parties \cite{pilaud2023posets,weng2023f,WilsonsSlides}.

\subsection{Minimal Terms}\label{subsec:g-vec}

Given $T = \{\tau_1,\ldots,\tau_n\}$, we will work in $\mathbb{R}^n$ where each standard basis vector $\eb_i$ is associated to an arc $\tau_i$ from $T$. We sometimes denote this $\eb_{\tau_i}$. 

Given a fence poset $\calP_\gamma$, let $\mathbf{a}_\gamma = \sum_{i=1}^n a_i \eb_i$ where $a_j$ is the number of times there is a minimal element associated to $\tau_j$ in $\calP_\gamma^0$.  Let $\mathbf{b}_\gamma = \sum_{i=1}^n b_i \eb_i$ where $b_j$ is the number of times there is an element associated to $\tau_j$ in $\calP_\gamma^0$ which covers at least two elements and is not in a loop.  It is possible that one or both of the elements which $\tau_{j}$ covers is in a loop. We will refer to a minimal or maximal element as ``strict" if the element covers or is covered by at least two elements.

We define one more vector, $\mathbf{r}_\gamma$. We initialize $\mathbf{r}_\gamma = 0$ and possibly add standard basis vectors as follows. Recall we index the arcs crossed by $\gamma$ as $\tau_{i_1},\ldots,\tau_{i_d}$. We also refer to the first and last triangles which $\gamma$ passes through as $\Delta_0$ and $\Delta_d$ respectively. 

\begin{enumerate}
    \item If $s(\gamma)$ is plain and the clockwise neighbor of $\tau_{i_1}$ in $\Delta_0$ is $\tau \in T$, we add $\eb_\tau$ to $\mathbf{r}_\gamma$. 
    \item If $s(\gamma)$ is notched and the counterclockwise neighbor of $\tau_{i_1}$ in $\Delta_0$ is $\tau' \in T$, we add $-\eb_{\tau'}$ to $\mathbf{r}_\gamma$. 
    \item We consider the same cases for $\tau_{i_d}$ and $\Delta_d$.
\end{enumerate}

Recall we do not have variables associated to the boundary arcs of $(S,M)$. For example, if $s(\gamma)$ is plain and the clockwise neighbor of $\tau_{i_1}$ in $\Delta_0$ is on the boundary, then there is no effect on $\mathbf{r}_\gamma$. See Table \ref{table:posetExamples} and Figure \ref{fig:ExampleCircularFence} for examples of $\mathbf{a}_\gamma, \mathbf{b}_\gamma,$ and $\mathbf{r}_\gamma$.

Finally, suppose we have a decorated poset with nontrivial decoration $\widetilde{\calP}_\gamma = (\calP_\gamma,w)$. If $w = \tau_{i_1}^{\epsilon_1}\cdots \tau_{i_d}^{\epsilon_d}$ for $\epsilon_i \in \{+,-\}$, we let $\mathbf{j}_\gamma = \epsilon_1 \mathbf{e}_{i_1} + \cdots + \epsilon_d \mathbf{e}_{i_d}$. 

Now, any arc or closed curve $\gamma$ with corresponding decorated poset $\widetilde{\calP}_\gamma$, we set $\mathbf{g}_\gamma = -\mathbf{a}_\gamma + \mathbf{b}_\gamma + \mathbf{r}_\gamma + \mathbf{j}_\gamma$.

If $\gamma$ is a plain arc in a surface with triangulation $T = \{\tau_1,\ldots,\tau_n\}$, we already defined $\text{cross}(\gamma,T)$ to be the monomial $x_1^{f_1}\cdots x_n^{f_n}$ where $f_i = e(\gamma,\tau_i)$. We can use the same definition when $\gamma$ is a closed curve. 

Given a puncture $p$, let $X_p = \prod_{i=1}^m x_{\sigma_i}$ be the product of all cluster variables associated to arcs from $T$ incident to $p$, where if an arc has both endpoints at $p$ it contributes twice. If $\gamma$ is notched at one or both endpoints, then we define $\text{cross}(\gamma,T)$ to be the product of $\text{cross}(\gamma^0,T)$ and $X_p$ for each puncture $p$ at which $\gamma$ has a notched endpoint. In each of these cases, one can check that $\text{cross}(\gamma,T)$ records the labels of tiles in the snake graph $\mathcal{G}_{\gamma,T}$. 

\begin{lemma}\label{lem:g-vectorMin}
Let $\gamma$ be a tagged arc or closed curve on a surface with triangulation $T$. Let $\mathcal{G}_{\gamma,T}$ be the corresponding snake, band, or loop graph. If $\min(\mathcal{G}_{\gamma,T})$ is the minimal matching of $\mathcal{G}_{\gamma,T}$, then \[
\mathbf{x}^{\gb_\gamma} = \frac{wt(\min(\mathcal{G}_{\gamma,T}))}{\text{cross}(\gamma,T)}.
\]
\end{lemma}

\begin{proof}
When $\gamma$ is a plain arc or closed curve, this comes from Remark 11.1 and Proposition 10.14 in \cite{geiss2022schemes}. Now, suppose $\gamma$ is notched at $t(\gamma)$ and plain at $s(\gamma)$. Let $\gamma'$ be the result of following $\gamma$ until the last triangle $\gamma$ passes through, $\Delta_d$, crossing each spoke at $t(\gamma)$ and ending at the endpoint of $\Delta_d$ opposite from the last spoke we cross. There are two such options for this arc. For a convention, suppose we choose the version of this arc which goes counterclockwise around $t(\gamma)$ and ends at the endpoint of $\Delta_d$ which is clockwise of $t(\gamma)$.

\begin{center}
 \begin{tikzpicture}
\draw (0,0) -- (1,1) -- (3,0) -- (1,-1) -- (0,0);
\draw[fill=black] (1.5,0) circle [radius=1pt];
\draw (1.5,0) -- (1,1) -- (1,-1) -- (1.5,0) -- (3,0);
\node[] at (4,0){$\rightarrow$};
\draw (5,0) -- (6,1) -- (8,0) -- (6,-1) -- (5,0);
\draw[thick, orange] (0,0) -- (1.5,0);
\draw[fill=black] (6.5,0) circle [radius=1pt];
\draw (6.5,0) -- (6,1) -- (6,-1) -- (6.5,0) -- (8,0);
\draw[thick, orange] (5,0) to [out = 0, in = 180] (6.25,-0.2);
\draw[thick, orange] (6.25,-0.2) to [out = 0, in = 270] (6.75, 0);
\draw[thick, orange] (6.75,0) to [out = 90, in = 90] (6.25,0);
\draw[thick, orange] (6.25,0) -- (6,-1);
\node[orange, rotate=90] at (1.3,0){$\bowtie$};
\end{tikzpicture}
\end{center}

Since $\gamma'$ is a plain arc, we know from the previous cases that $\mathbf{x}^{\gb_{\gamma'}} = \frac{wt(\min(\mathcal{G}_{\gamma',T}))}{\text{cross}(\gamma',T)}$. If $\tau_{i_d}$ is the last arc $\gamma$ crosses, and $\gamma'$ crosses the spokes in order $\sigma_1,\ldots,\sigma_m$, the loop in $\calP_\gamma$ is of the form $\tau_{i_d} \prec \sigma_m \succ \sigma_{m-1} \succ \cdots \succ \sigma_1 \prec \tau_{i_d}$ The poset $\calP' = \calP_{\gamma'}$ is the same as $\calP$ except $\tau_{i_d}$ and $\sigma_m$ are incomparable, and there are no loops in $\calP'$.
Therefore, we can observe that $\gb_{\gamma'} = \mathbf{e}_{\sigma_1} + \gb_{\gamma}$ since $\sigma_1$ contributes to $\mathbf{r}_{\gamma'}$ as the clockwise neighbor of the last arc crossed by $\gamma'$ but not to $\mathbf{r}_{\gamma}$. 

We see that $\text{cross}(\gamma',T) = \text{cross}(\gamma,T)$. Since we form $\mathcal{G}_{\gamma,T}$ by gluing two edges weighted with $x_{\sigma_1}$ in $\mathcal{G}_{\gamma',T}$, we have that $\min(\mathcal{G}_{\gamma',T}) = x_{\sigma_1} \min(\mathcal{G}_{\gamma,T})$ .  Thus, we have arrived at the desired conclusion. If $\gamma$ is instead notched at $s(\gamma)$, the same arguments hold, and we can perform this process twice if $\gamma$ is doubly-notched.

\end{proof}

\begin{remark}
The notation in this section is inspired by the notation for the $\gb$-vector of a string module, as in \cite{palu2021non}. When $\gamma$ is a plain arc or closed curve, $\gb_\gamma$ matches the $\gb$-vector of the associated \emph{arc module} of the Jacobian algebra associated to $T$; the algebra and arc modules were described in \cite{labardini2009quivers, labardini2009quivers2}. If $\gamma$ is a tagged arc, then $\gb_\gamma$ appears to be equal to the $\gb$-vector of the arc module associated to $\gamma$ in \cite{dominguez2017arc}.
\end{remark}

\subsection{Expansion formula via posets}

Given an arc $\tau \in T$, define $x_{CCW}(\tau) = x_{\tau_j}x_{\tau_k}$ if there exist two arcs $\tau_j,\tau_k \in T$ which are counterclockwise neighbors of $\tau$ within the triangles it borders. If one or both of these neighbors is on the boundary, we ignore the contribution. Define $x_{CW}(\tau)$ analogously, using the clockwise neighbors. Then, we set  \[
\hat{y}_{\tau} = \frac{x_{CCW}(\tau)}{x_{CW}(\tau)} y_{\tau}.
\]

Note this is simply the ordinary definition of the $\hat{y}$-variables, specified to the setting of a cluster algebra from a surface. Given $I \in J(\calP_\gamma)$, let $wt(I) = \prod_{\tau \in I} \hat{y}_\tau$, where $I$ is regarded as a multiset. Let $wt(\emptyset) = 1$.

\begin{prop}\label{prop:PosetExpansion}
Let $\gamma$ be an arc or closed curve on a marked surface $(S,M)$ with triangulation $T$. Then, the associated element $x_\gamma$ of the cluster algebra $\mathcal{A}(S,M)$ written with respect to the cluster given by $T$ can be expressed by \[
x_\gamma^T = \mathbf{x}^{\mathbf{g}_\gamma} \sum_{I \in J(\calP_\gamma)} wt(I).
 \]
\end{prop}

\begin{proof}
First, suppose $\gamma$ is not such that $\gamma^0 \in T$. Then, this follows from combining Proposition 3.2 in \cite{ouguz2024cluster} with Lemma \ref{lem:g-vectorMin}. 

If $\gamma = \gamma^0 \in T$, then $\widetilde{\calP}_\gamma = (\emptyset, \gamma)$ and this is immediate. Now let $\gamma \neq \gamma^0 \in T$. Suppose first $\gamma =  \gamma^{(p)}$ is singly-notched. Then, since $\mathbf{x}^{\mathbf{g}_{\gamma^{(p)}}} = \frac{\mathbf{x}^{\mathbf{g}_{\ell_p}}}{x_{\gamma^0}}$ and $\widetilde{\calP}_{\gamma^{(p)}} = (\calP_{\ell_p}, \gamma^-)$, the result follows from the formula $x_{\gamma^{(p)}} = \frac{x_{\ell_p}}{x_{\gamma^0}}$.

Finally, consider $\gamma^{(p,q)}$, a doubly-notched arc such that  $\gamma^0 \in T$.  For convenience, orient $\gamma$ from $p$ to $q$; it is possible that $p = q$. Let the spokes at $p$, excluding $\gamma$ be $\sigma_1,\ldots,\sigma_{m}$, numbered in counterclockwise order around $p$, such that $\sigma_1$ and $\sigma_{m}$ each border a common triangle with $\gamma$. Label the arcs at $q$ excluding $\gamma$ similarly in counterclockwise order $\eta_1,\ldots,\eta_{h}$. Then, from the previous case we know that $x_{\gamma^{(p)}} = \mathbf{x}^{\mathbf{g}_{\gamma^{(p)}}}(1 + \hat{y}_{\sigma_1} + \hat{y}_{\sigma_1}\hat{y}_{\sigma_2} + \cdots + \hat{y}_{\sigma_1}\cdots \hat{y}_{\sigma_{m}}) = \frac{x_{\eta_{h}}}{x_{\sigma_1}}(1 + \hat{y}_{\sigma_1} + \hat{y}_{\sigma_1}\hat{y}_{\sigma_2} + \cdots + \hat{y}_{\sigma_1}\cdots \hat{y}_{\sigma_{m}}) $ and similarly for $x_{\gamma^{(q)}}$. 

Now, we combine Theorem 12.9 from \cite{musiker2011positivity} with our notation to give an explicit expression for $x_{\gamma^{(p,q)}}$, \begin{align*}
x_{\gamma^{(p,q)}} &= \frac{1}{x_\gamma}( 1 - y_{\sigma_1}\cdots y_{\sigma_{m}}y_\gamma - y_{\eta_1} \cdots y_{\eta_{h}}y_\gamma + y_{\sigma_1}\cdots y_{\sigma_{m}} y_{\eta_1} y_{\eta_{h}}y_\gamma^2 + y_\gamma x_{\gamma^{(p)}}x_{\gamma^{(q)}} ) \\
&=\frac{1}{x_\gamma}\bigg( 1 - \hat{y}_{\sigma_1}\cdots \hat{y}_{\sigma_{m}}\hat{y}_\gamma - \hat{y}_{\eta_1} \cdots \hat{y}_{\eta_{h}}\hat{y}_\gamma + \hat{y}_{\sigma_1}\cdots \hat{y}_{\sigma_{m}} \hat{y}_{\eta_1} \hat{y}_{\eta_{h}}\hat{y}_\gamma^2\\ 
&+ y_\gamma \frac{x_{\sigma_{m}}x_{\eta_{h}}}{x_{\sigma_1}x_{\eta_1}}(1 + \hat{y}_{\sigma_1} + \cdots + \hat{y}_{\sigma_1}\cdots \hat{y}_{\sigma_{m}})(1 + \hat{y}_{\eta_1} + \cdots + \hat{y}_{\eta_1} \cdots \hat{y}_{\eta_{h}})\bigg)\\
\end{align*}
where the second equality follows from Lemma \ref{cor:y-hat_spokes} and expanding $x_{\gamma^{(p)}}$ and $x_{\gamma^{(q)}}$. Now, we recognize that $\hat{y}_\gamma = y_\gamma \frac{x_{\sigma_{m}}x_{\eta_{h}}}{x_{\sigma_1}x_{\eta_1}}$ so that the terms $\hat{y}_{\sigma_1}\cdots \hat{y}_{\sigma_{m}}\hat{y}_\gamma + \hat{y}_{\eta_1} \cdots \hat{y}_{\eta_{h}}\hat{y}_\gamma$ cancel. 
The resulting polynomial exactly matches the sum of weights of order ideals of $\calP_{\gamma^{(p,q)}}$. 
Finally, we see that the term without a factor of $\hat{y}_i$ is $\frac{1}{x_\gamma}$, which matches $\mathbf{x}^{\mathbf{g}_{\gamma^{(p,q)}}}$, so we are finished. 
\end{proof}

We remark that we exclude closed surfaces with two punctures because  Theorem 12.9 from \cite{musiker2011positivity} is not known to hold in these cases. 

Next, we introduce a few results concerning products of $\hat{y}$-variables, which will be useful in later proofs.

\begin{lemma}\label{lem:y-hat}
Given an arc $\gamma \notin T$ which crosses arcs $\alpha_1,\ldots,\alpha_m$ from $T$, suppose the first triangle $\gamma$ passes through has edges $\alpha_1, \eta_1,\eta_h$ in clockwise order and similarly the last triangle has edges $\alpha_m,\sigma_1,\sigma_m$. Let $\{\alpha_{i_{1}},\ldots,\alpha_{i_{s}}\}$ be the multiset of arcs associated to strict maximal elements in $J(\calP_{\gamma})$ and $\{\alpha_{j_{1}},\ldots,\alpha_{j_{t}}\}$ be the multiset of arcs associated to strict minimal elements. Then, we have
  \[\hat{y}_{\alpha_1} \cdots \hat{y}_{\alpha_m}= y_{\alpha_1} \cdots y_{\alpha_m} \frac{ x_{\eta_h} x_{\sigma_m}}{ x_{\eta_{1}} x_{\sigma_1}} x_{\alpha_1}^\pm x_{\alpha_m}^\pm\prod_{l=1}^{s}x_{\alpha_{k_{l}}}^{-2}\prod_{v=1}^{t}x_{\alpha_{j_{v}}}^{2},\]
where, the sign on $x_{\alpha_1}$ is $+$ if $\alpha_1 \prec \alpha_2$ and otherwise it is $-$ and similarly for $x_{\alpha_m}$, where we compare $\alpha_m$ with $\alpha_{m-1}$. 
\end{lemma}
\begin{proof}
We proceed by comparing the unique term divisible by $y_{\alpha_1}\cdots y_{\alpha_m}$ in the expansion formulas from Theorem \ref{thm:SnakeGraphExpansion} and Proposition \ref{prop:PosetExpansion}.  From this, we see \[
\frac{wt(\max(\mathcal{G}_{\gamma,T}))}{\text{cross}(\gamma,T)} = \mathbf{x}^{\gb_\gamma} \hat{y}_{\alpha_1} \cdots \hat{y}_{\alpha_m} = \frac{wt(\min(\mathcal{G}_{\gamma,T}))}{\text{cross}(\gamma,T)}  \hat{y}_{\alpha_1} \cdots \hat{y}_{\alpha_m}
\]
where the second equality follows from Lemma \ref{lem:g-vectorMin}. That is, we have \[
\hat{y}_{\alpha_1} \cdots \hat{y}_{\alpha_m} = \frac{wt(\max(\mathcal{G}_{\gamma,T}))}{wt(\min(\mathcal{G}_{\gamma,T}))}.
\]

We can infer from Lemma \ref{lem:g-vectorMin} that $wt(\min(\mathcal{G}_{\gamma,T})) = x_{\eta_1} x_{\sigma_1} x_{\alpha_{i_1}}^2 \cdots x_{\alpha_{i_s}}^2 x_{\alpha_1}^{\delta_{1,p}} x_{\alpha_m}^{\delta_{m,p}}$ where where $\delta_{i,p} = 1$ if $\alpha_i$ is maximal and equals 0 otherwise. A parallel argument, given by switching conventions for minimal and maximal matchings, shows that $wt(\max(\mathcal{G}_{\gamma,T})) = x_{\eta_h}x_{\sigma_m} x_{\alpha_{j_1}}^2 \cdots x_{\alpha_{j_t}}^2  x_{\alpha_1}^{\delta_{1,d}} x_{\alpha_m}^{\delta_{m,d}}$ where where $\delta_{i,d} = 1$ if $\alpha_i$ is minimal and equals 0 otherwise. By evaluating cases for how $\alpha_1$ and $\alpha_2$ are related, how $\alpha_{m-1}$ and $\alpha_m$ are related, and whether any ends of $\gamma$ are tagged,  the statement follows.

Now, suppose $\gamma$ is a closed curve. Then, from \cite{ouguz2024cluster}, we have the same equality of formulas relating band graphs and posets so we can use the same method of proof. 
\end{proof}

The following special case of Lemma \ref{lem:y-hat} will be particularly helpful in later proofs involving arcs which wind around punctures.

\begin{cor}\label{cor:y-hat_spokes}
Suppose the complete list of spokes from a triangulation $T$ incident to a puncture $p$ is $\sigma_1,\ldots,\sigma_m$, listed in counterclockwise order. Let $\sigma_{[i]}$ be the third arc in the triangle bounded by $\sigma_i$ and $\sigma_{i+1}$ For shorthand set $y_i := y_{\sigma_i}$ and $y_{[i]}:= y_{\sigma_{[i]}}$ and similarly for $\hat{y}_i$ and $x_i$. Then, for $i < j$ where $j-i < n$, we have\[
\hat{y}_{i} \cdots \hat{y}_{j} = y_i\cdots y_j \frac{x_{i-1}x_ix_{[j]}}{x_j x_{j+1}x_{[i-1]}}
\]
and in particular, \[
\hat{y}_{1} \cdots \hat{y}_{m} = y_1\cdots y_m.
\]
\end{cor} 

We establish the following symbol. If $p$ is a puncture with incident spokes $\sigma_1,\ldots,\sigma_m$, we set $Y_p = \prod_{i=1}^m y_{\sigma_i}$ and $\hat{Y}_p$ similarly. Then, the last statement of Corollary \ref{cor:y-hat_spokes} can be rephrased as $\hat{Y}_p = Y_p$.
 
\subsection{Format of proofs}
\label{subsec:proofStrategy}
While our method of proving Theorem \ref{thm:main} consists of considering many cases for the interaction of two arcs or an arc with self-intersection, the method of proving each case is the same. Here we give the overarching format of these proofs. 

To lighten notation, we write this section for skein relations of the form $x_1 x_2 = x_3 x_4 + Y x_5x_6$ where $x_i := x_{\gamma_i}$ and the sets $\{\gamma_3,\gamma_4\}$ and $\{\gamma_5,\gamma_6\}$ are resolutions of an intersection between $\gamma_1$ and $\gamma_2$. If there is only one factor in one or multiple terms, the ideas are the same. 

In light of Proposition \ref{prop:PosetExpansion}, we can write $x_1x_2$ as \[
\mathbf{x}^{\gb_1 + \gb_2} \sum_{(I_1,I_2) \in J(\calP_1) \times J(\calP_2)} wt(I_1)wt(I_2)
\]
where $\calP_i := \calP_{\gamma_i}$. The bulk of our proofs concern finding a partition of $J(\calP_1) \times J(\calP_2) = A \cup B$ and bijections between $A$ and $J(\calP_{\gamma_3}) \times J(\calP_{\gamma_4})$ and between $B$ and $J(\calP_5) \times J(\calP_6)$, which combine to form $\Phi: J(\calP_1) \times J(\calP_2) \to (J(\calP_3) \times J(\calP_4)) \cup  (J(\calP_5) \times J(\calP_6))$. It will be necessary that $\emptyset \in A$, $\Phi\vert_A$ is weight-preserving, and $\Phi\vert_B$ is weight-preserving up to a fixed monomial. Specifically,  we require that, if $(I_1,I_2) \in A$, then $wt(\Phi(I_1,I_2)) = wt(I_1,I_2)$ and if $(I_1,I_2) \in B$, then $Z wt(\Phi(I_1,I_2)) = wt(I_1,I_2)$ for a fixed monomial $Z$ where $wt(I,J) := wt(I)wt(J)$. 

Given $Z = (\prod_{i=1}^n x_i^{v_i})(\prod_{i=1}^n y_i^{u_i})$, let $\degx(Z) = \sum_{i=1}^n v_i \eb_i$. The final step of our proof will be to show that, if $\gb_i:= \gb_{\gamma_i}$, then $\gb_1 + \gb_2 = \gb_3 + \gb_4$, and  $\gb_1 + \gb_2 + \degx(Z) = \gb_5 + \gb_6$. Then, if $Y = Z\vert_{x_i = 1} = \prod_{i=1}^n y_i^{u_i}$, we can rewrite $x_1x_2$ as \begin{align*}
&\mathbf{x}^{\gb_1 + \gb_2} \sum_{(I_1,I_2) \in A} wt(I_1)wt(I_2) + \mathbf{x}^{\gb_1 + \gb_2} \sum_{(I_1,I_2) \in B} wt(I_1)wt(I_2) 
\\&= \mathbf{x}^{\gb_1 + \gb_2} \sum_{(I_3,I_4) \in J(\calP_{\gamma_3}) \times J(\calP_{\gamma_4})} wt(I_3)wt(I_4) + \mathbf{x}^{\gb_1 + \gb_2} \sum_{(I_5,I_6) \in J(\calP_5) \times J(\calP_6)} Zwt(I_5)wt(I_6)\\
&= \mathbf{x}^{\gb_3 + \gb_4} \sum_{(I_3,I_4) \in J(\calP_{\gamma_3}) \times J(\calP_{\gamma_4})} wt(I_3)wt(I_4) + \mathbf{x}^{\gb_5 + \gb_6} Y \sum_{(I_5,I_6) \in J(\calP_5) \times J(\calP_6)} wt(I_5)wt(I_6) = x_3x_4 + Y x_5x_6.
\end{align*}

In each proof, we will explicitly identify the monomial $Z$ and will see that it is equal to $\hat{Y}_R\hat{Y}_S$ for the sets $R,S$ defined in Section \ref{sec:introduction}. In Section \ref{subsec:Kissing}, some resolutions will involve a contractible kink, which introduces a negative sign into the equation. This will modify the type of bijections we construct. 

We remark that Pilaud, Reading, and Schroll used a similar proof idea to describe skein relations in the case where $\gamma_1 \notin T$, $\gamma_2 \in T$, and $e(\gamma_1,\gamma_2) = 1$  \cite{pilaud2023posets}. These skein relations are in fact exchange relations from the cluster algebra. 

\subsection{Comparison to snake graph calculus}\label{subsec:CompareToSGC}

Our methods could be rephrased as a generalization of \emph{snake graph calculus}, as in \cite{snakegraphcalculus1,snakegraphcalculus2, snakegraphcalculus3}, to loop graphs. Moreover, several of our proofs use ideas from these articles, rephrased in the language of posets. We will use similar notation to these references in several places, where the meaning has been updated for the poset setting.

In a snake graph, there is a natural ordering on the tiles dictated by following the graph in the north-east direction. We provide an analogous notion for our posets with a \emph{chronological ordering}. When we have a fence poset without loops on $n$ elements, we will refer to these elements as $\calP(1),\calP(2),\ldots,\calP(n)$ such that $\calP(i)$ only has cover relations with $\calP(i-1)$ and $\calP(i+1)$, if these elements exist. Notice that every fence poset with more than one element has exactly two chronological orderings and these are rarely linear extensions. If we have a loop fence poset $\calP$ with $n$ elements in $\calP^0$, we provide a chronological labeling on $\calP^0$ with numbers in $[n]$. Then, if $\calP(1)$ is adjacent to a loop, we label these as $\calP(-a) \succ P(-a+1) \succ \cdots \succ P(0) \prec P(1) \prec P(-a) $; that is, we continue the labeling on the loop passing first to the element covered by $\calP(1)$. Similarly, if $\calP(n)$ is incident to a loop, we add the labels such that $\calP(n+b)\succ \calP(n)\succ \calP(n+1) \prec \calP(n+2) \prec \cdots \prec \calP(n+b)$. We give an example below for a loop fence poset $\calP$ with loops on both sides.

 \begin{center}
 \begin{tikzpicture}
 \node(1) at (0,0) {$\calP(1)$};
 \node(2) at (1,1){$\calP(2)$};
 \node(3) at (2,0){$\calP(3)$};
 \node(0) at (-1,-1){$\calP(0)$};
 \node(-1) at (-2,0){$\calP(-1)$};
 \node(-2) at (-3,1){$\calP(-2)$};
 \node(4) at (3,-1){$\calP(4)$};
\node(5) at (4,1){$\calP(5)$};
\draw(1) -- (2);
\draw(2) -- (3);
\draw(0) -- (1);
\draw(-1) -- (0);
\draw(-2) -- (-1);
\draw(-2) -- (1);
\draw(3) -- (4);
\draw(4) -- (5);
\draw(3) -- (5);
 \end{tikzpicture}
 \end{center}

We will use the following notation for various (induced) subposets of $\calP$, using ordinary $\leq,\geq$ symbols to refer to the indexing via a chronological ordering and using $\preceq, \succeq$ to refer to the inequalities from $\calP$. One can infer from here the meaning of similar expressions, such as $\calP^{\succ a}$. 

\begin{itemize}
    \item $\calP^{\preceq a} :=\{\calP(x)|\calP(x)\preceq \calP(a)\}$
    \item $\calP[\leq a] :=\{\calP(x)|x\leq a\}$
    \item $\calP[a,b] := \calP[\geq a] \cap \calP[\leq b] = \{\calP(x) : a \leq x \leq b\}$
\end{itemize}

When we take such a subset of a poset, we will often use the original chronological labeling inherited from the larger poset. For example, while one could define a chronological labeling on $\calP[a,b]$ using numbers in $\{1,2,\ldots,b-a+1\}$, we will instead use the labels from $\calP$, thus using $\{a,a+1,\ldots,b\}$. 

In \cite{snakegraphcalculus1}, the authors characterize a way to translate (some) intersections of a pair of arcs $\gamma_1$ and $\gamma_2$ cross based on the behavior of the snake graphs $\mathcal{G}_{\gamma_1,T}$ and $\mathcal{G}_{\gamma_2,T}$. This condition can be translated nicely into a condition on the posets $\calP_{\gamma_1}$ and $\calP_{\gamma_2}$. Let a set $S$ of elements in a poset $(\calP,\preceq)$ be on \emph{top} if for any $x \in S, y \in P \backslash S$, $x \not\preceq y$, and we define a set to be on \emph{bottom} similarly. 

\begin{definition}\label{def:Overlaps}
\begin{enumerate}
    \item If two posets $\calP_1$ and $\calP_2$ have isomorphic subposets $R_1 = \calP_1[s,t]$ and $R_2 = \calP_2[s',t']$, we say they have an \emph{overlap}.
    \item If $\calP_1$ and $\calP_2$ have overlaps $R_1$ and $R_2$ such that $R_1$ is on top of $\calP_1^0$ and $R_2$ is on bottom of $\calP_2^0$ or vice versa, and we do not have $s=s'=1$ nor both $t = n_1$ and $t' = n_2$ where $d_i = \vert \calP_i^0 \vert$,  then we say $\calP_1$ and $\calP_2$ have a \emph{crossing overlap}.
    \item A poset $\calP$ has a \emph{self-overlap} if it has two isomorphic subposets of $\calP^0$, $R_1 = \calP^0[s,t]$ and $R_2 = \calP^0[s',t']$.
    \item If a poset $\calP$ has self-overlaps $R_1$ and $R_2$ such that  $R_1$ is on top of $\calP_1^0$ and $R_2$ is on bottom of $\calP_2^0$ or vice versa, and such that, if $s < s'$, we do not have both $s = 1$ and $t' = d$, then we say $\calP$ has a \emph{crossing self-overlap}.
\end{enumerate}
\end{definition}  

These crossing overlaps appear under several other guises in the literature, resembling paths which kiss as in \cite{petersen2010non} and (some) pairs of strings with nontrivial extension groups as in  \cite{brustle2020combinatorics, canakci2017extensions}.

\begin{prop}[Theorem 5.3 \cite{snakegraphcalculus1}, Theorem 6.1 \cite{snakegraphcalculus2}]\label{lem:CrossingOverlapImpliesCrossing}
If $ \calP_{\gamma_1}$ and $\calP_{\gamma_2}$ have a crossing overlap, then $\gamma_1$ and $\gamma_2$ have an intersection. Similarly, if $\calP_\gamma$ has a crossing self-overlap, then $\gamma$ has a self-intersection. 
\end{prop}

As we will explain at the beginning of Section \ref{sec:transverse_crossings}, not every intersection or self-intersection of arcs will result in a crossing overlap of the corresponding posets. These other types of intersections are exactly when the intersection point lies in the first or last triangle passed through by one of the arcs involved. This is referred to as grafting in \cite{snakegraphcalculus1,snakegraphcalculus2} and arrow crossing and 3-cycle crossing in \cite{canakci2017extensions}. 

In Definition \ref{def:CompatibleTaggedArcs}, we see that, on a punctured surface, a pair of arcs can be incompatible without having a point of intersection on their interiors. These incompatibilities concern common endpoints which have differing taggings, as in Definition \ref{def:CompatibleTaggedArcs}. These will be addressed in Section \ref{sec:puncture_incompatibility}.

A useful concept from \cite{snakegraphcalculus1,snakegraphcalculus2} is the idea of a \emph{switching position}. We will frequently use the same concept in the context of order ideals of fence posets. Let $\calP_1$ and $\calP_2$ be isomorphic fence posets with bijection $\Phi$. Fix a compatible chronological ordering on each poset, so that the isomorphism $\Phi$ satisfies $\Phi(\calP_1(i)) = \calP_2(i)$. Let $I_1,I_2$ be order ideals of $\calP_1$ and $\calP_2$ respectively. Then, we say the \emph{switching position between $I_1$ and $I_2$} (with respect to the ordering) is the smallest value $i$ such that $\calP_1(i) \in I_1$ if and only if $\calP_2(i) \in I_2$. There are two compatible pairs of chronological of $\calP_1$ and $\calP_2$ and this choice of ordering will in general affect the switching position. 

\begin{lemma}\label{lem:SwitchingExists}
Given $I_1$ and $I_2$, order ideals of isomorphic fence posets $\calP_1$ and $\calP_2$ respectively, a switching position between $I_1$ and $I_2$ exists unless $I_1 = \calP_1$ and $I_2 = \emptyset$ or vice versa. 
\end{lemma}

\begin{proof}
Suppose that there is no switching position between $I_1$ and $I_2$. Without loss of generality,  $\calP_1(1) \in I_1$ and $\calP_2(1) \notin I_2$. Since either $\calP_1(1) \prec \calP_1(2)$ or $\calP_1(1) \succ \calP_1(2)$, we cannot have that $\calP_1(2) \notin I_1$ and $\calP_2(2) \in I_2$. So, it must be that $\calP_1(2) \in I_1$ and $\calP_2(2) \notin I_2$. Proceeding in this way, we see that in order for there to not exist a switching position, we must have $I_1 = \calP_1$ and $I_2 = \emptyset$. 
\end{proof}

\section{Incompatibility at Punctures}
\label{sec:puncture_incompatibility}

In this section, we address how to resolve the incompatibility of two arcs which have incompatible taggings at one or two punctures, as in Definition \ref{def:CompatibleTaggedArcs}. We also discuss a singly-notched monogon, which is a tagged arc which is incompatible with itself. 

\subsection{Incompatible at one puncture}\label{subsec:IncAtOne}

Consider two arcs, $\gamma_1^{(p)}$ and $\gamma_2$ which are incompatible at a puncture $p$; that is, both arcs are incident to $p$, $\gamma_1$ is notched at $p$, $\gamma_2$ is plain at $p$, and $\gamma_1^0 \neq \gamma_2^0$. Orient $\gamma_1^{(p)}$ and $\gamma_2$ to both begin at $p$. Let the spokes at $p$ from $T$ be labeled $\sigma_1,\ldots,\sigma_m$ in counterclockwise order such that the first triangle that $\gamma_1^{(p)}$ passes through is bounded by $\sigma_1$ and $\sigma_m$. This means that in the poset for $\gamma_1^{(p)}$, the set $\sigma_1,\ldots,\sigma_m$ sit on a loop connected to $\calP_{\gamma_1^{(p)}}(1)$ with $\calP_{\gamma_1^{(p)}}(-i+1)$ corresponding to $\sigma_i$. We will refer to these elements of the loop by the corresponding arcs, even though the interior of $\gamma_1$ may cross some of the $\sigma_i$.  

If $\gamma_2 \notin T$, let $1 \leq k \leq m$ be such that the first triangle $\gamma_2$ passes through is bounded by $\sigma_k$ and $\sigma_{k+1}$, where we interpret $\sigma_{m+1}$ as $\sigma_1$. If $\gamma_2 \in T$, then we let $k$ be such that $\gamma_2 = \sigma_k$.

Draw a small circle $\mathsf{h}$ encompassing $p$ and not crossing any arcs of $T$ except the spokes at $p$. We define $\gamma_1^{-1} \circ_{CCW} \gamma_2$ as the arc which results from following $\gamma_1$ from $t(\gamma_1)$ with reverse orientation until its intersection with $\mathsf{h}$, following $\mathsf{h}$ counterclockwise until its intersection with $\gamma_2$, and then following $\gamma_2$ until $t(\gamma_2)$. We define $\gamma_1^{-1} \circ_{CW} \gamma_2$ similarly.  Define $\gamma_3:= \gamma_1^{-1} \circ_{CCW} \gamma_2$ and $\gamma_4:= \gamma_1^{-1} \circ_{CW} \gamma_2$; in particular, $\gamma_3$ crosses $\sigma_1,\ldots,\sigma_k$ and $\gamma_4$ crosses $\sigma_{k+1},\ldots,\sigma_m$. See Figure~\ref{fig:Incompatible} for a case where $k \neq m$ and Table \ref{tab:PosetsKIsNot0orM} for the associated posets.

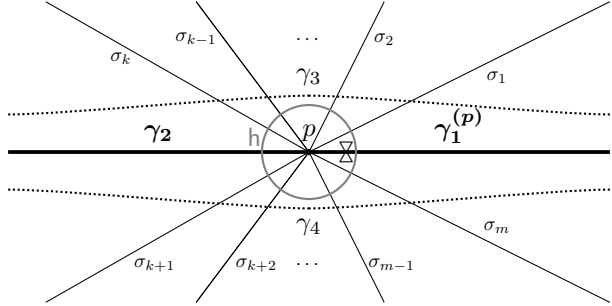
\begin{figure}[H]
\captionsetup{width=0.9\textwidth}
    \centering
    \begin{tikzpicture}[scale=0.5, transform shape]
\draw[line width=0.5mm] (0,0) to (8,0);
\node[line width=0.5mm,scale=2,rotate=90] at (1,0) {$\bowtie$};
\draw[line width=0.5mm] (0,0) to (-8,0);

\draw[out=180,in=0,looseness=0.5, thick, densely dotted] (8,1) to (0,1.5);
\draw[out=0,in=180,looseness=0.5, thick,densely dotted] (-8,1) to (0,1.5);

\draw[out=180,in=0,looseness=0.5, thick,densely dotted] (8,-1) to (0,-1.5);
\draw[out=0,in=180,looseness=0.5, thick, densely dotted] (-8,-1) to (0,-1.5);

\draw (0,0) to (8,4);
\draw (0,0) to (2,4);
\draw (0,0) to (-3,4);
\draw (0,0) to (-3,4);
\draw (0,0) to (-7,4);

\draw (0,0) to (8,-4);
\draw (0,0) to (2,-4);
\draw (0,0) to (-3,-4);
\draw (0,0) to (-3,-4);
\draw (0,0) to (-7,-4);
\draw[fill=black] (0,0) circle [radius=2pt];

\node[scale=2] at (0,0.5) {$p$};
\node[scale=1.5] at (5,2) {$\sigma_1$};
\node[scale=1.5] at (2,3) {$\sigma_2$};
\node[scale=1.5] at (-3,3) {$\sigma_{k-1}$};
\node[scale=1.5] at (-5,2.5) {$\sigma_{k}$};
\node[scale=1.5] at (-4.1,-3) {$\sigma_{k+1}$};
\node[scale=1.5] at (-1.4,-3) {$\sigma_{k+2}$};
\node[scale=1.5] at (2.2,-3) {$\sigma_{m-1}$};
\node[scale=1.5] at (5,-2) {$\sigma_{m}$};

\node[scale=1.5] at (0,3) {$\cdots$};
\node[scale=1.5] at (0,-3) {$\cdots$};

\node[scale=2,font=\boldmath] at (4,0.7) {$\gamma_1^{(p)}$};
\node[scale=2,font=\boldmath] at (-4,0.5) {$\gamma_2$};
\node[scale=2] at (0,2) {$\gamma_3$};
\node[scale=2] at (0,-2) {$\gamma_4$};

\draw[gray, thick] (0,0) circle (1.25);
\node[scale = 2, gray] at (-1.45,0.4){$\mathsf{h}$};
    \end{tikzpicture}
    \caption{An example of two arcs, $\gamma_2$ and $\gamma_1^{(p)}$(thick), that are incompatible at a puncture and the resulting arcs $\gamma_3$ and $\gamma_4$(dotted)}.
    \label{fig:Incompatible}
\end{figure}

If $\gamma_2 = \sigma_k$, then $\calP_{\gamma_3}$ consists of $\calP_{\gamma_1^{(p)}}$ with the principal order filter generated by $\sigma_k$ removed. Similarly, $\calP_{\gamma_4}$ consists of $\calP_{\gamma_1^{(p)}}$ with the order ideal $\langle \sigma_k \rangle$ removed. 

\begin{table}[]
\captionsetup{width=0.9\textwidth}
    \centering
\begin{tabular}{>{\centering\arraybackslash} m{7cm}|>{\centering\arraybackslash}m{6cm}}
    \highlight{$\calP_{\gamma_1^{(p)}}$} & \highlight{$\calP_{\gamma_2}$} \\
    \begin{tikzpicture}[scale = 0.7]
    \node[](dots) at (8.5,0){$\cdots$};
    \node[](alpha) at (7,0){$\calP_1(1)$};
    \node[](sigma1) at (6,-1){$\sigma_1$};
    \node[](sigma2) at (5,0) {$\sigma_2$};
    \node[](dots1) at (4,1){$\ddots$};
    \node[](sigmak) at (3,2){$\sigma_k$};
    \node[](sigmak+1) at (2,3){$\sigma_{k+1}$};
    \node[](dots2) at (1,4){$\ddots$};
    \node[](sigmas) at (0,5){$\sigma_m$};
    \draw(alpha)--(dots);
    \draw(alpha)--(sigma1);
    \draw(sigma1) -- (sigma2);
    \draw (sigma2) -- (dots1);
    \draw (dots1) -- (sigmak);
    \draw(sigmak) -- (sigmak+1);
    \draw(sigmak+1) -- (dots2);
    \draw(dots2) -- (sigmas);
    \draw(sigmas) -- (alpha);
    \end{tikzpicture} &  
    \begin{tikzpicture}[scale = 0.7]
    \node[](r) at (0,0) {$\calP_2(1)$};
    \node[] (dots) at (1.3,0){$\cdots$};
    \end{tikzpicture}\\\hline\hline
    \vspace{2mm}
    \highlight{$\calP_{\gamma_3}$} & 
    \vspace{2mm} 
    \highlight{$\calP_{\gamma_4}$} \\
    \begin{tikzpicture}[scale = 0.7]
    \node[](dots) at (-1.5,0){$\cdots$};
    \node[](tau) at (0,0){$\calP_1(1)$};
    \node[](sigma1) at (1,-1){$\sigma_1$};
    \node[](sigma2) at (2,0){$\sigma_2$};
    \node[] (dots3) at (3,1){$\iddots$};
    \node[](sigmak) at (4,2){$\sigma_k$};
    \node[](etam) at (5,1){$\calP_2(1)$};
    \node[](dots2) at (6.5,1){$\cdots$};
    \draw(tau) -- (dots);
    \draw(tau) -- (sigma1);
    \draw (sigma1) -- (sigma2);
    \draw(sigma2) -- (dots3);
    \draw(dots3) --(sigmak);
    \draw(sigmak) -- (etam);
    \draw(dots2) -- (etam);
    \end{tikzpicture}
    &  
    \begin{tikzpicture}[scale = 0.7]
    \node[](dots) at (-1.5,0){$\cdots$};
    \node[](tau) at (0,0){$\calP_1(1)$};
    \node[](sigma1) at (1,1){$\sigma_m$};
    \node[](sigma2) at (2,0){$\sigma_{m-1}$};
    \node[] (dots3) at (3,-1){$\ddots$};
    \node[](sigmak) at (4,-2){$\sigma_{k+1}$};
    \node[](etam) at (5,-1){$\calP_2(1)$};
    \node[](dots2) at (6.5,-1){$\cdots$};
    \draw(tau) -- (dots);
    \draw(tau) -- (sigma1);
    \draw (sigma1) -- (sigma2);
    \draw(sigma2) -- (dots3);
    \draw(dots3) --(sigmak);
    \draw(sigmak) -- (etam);
    \draw(dots2) -- (etam);
    \end{tikzpicture}\\
\end{tabular}
\caption{Posets for a resolution of a pair of arcs $\gamma_1^{(p)}$ and $\gamma_2$ with incompatible taggings at a puncture $p$ for the case $k \neq 0,m$, $\gamma_2 \notin T$  with $\calP_1:=\calP_{\gamma_1^{(p)}}$ and $\calP_2 := \calP_{\gamma_2}$}\label{tab:PosetsKIsNot0orM}
\end{table}

When $k = m$ and $\gamma_2 \notin T$, so that the first triangle $\gamma_1^{(p)}$ passes through is the same as the first triangle $\gamma_2$ passes through, then we have two additional cases based on whether $\gamma_1^{(p)}$ is clockwise or counterclockwise of $\gamma_2$ at $p$. Since these cases will produce different expansions, we will refer to the case where $\gamma_1^{(p)}$ lies clockwise from $\gamma_2$ as the $k=0$ case. See Figure~\ref{fig:punctureIncompatibilityCases}.

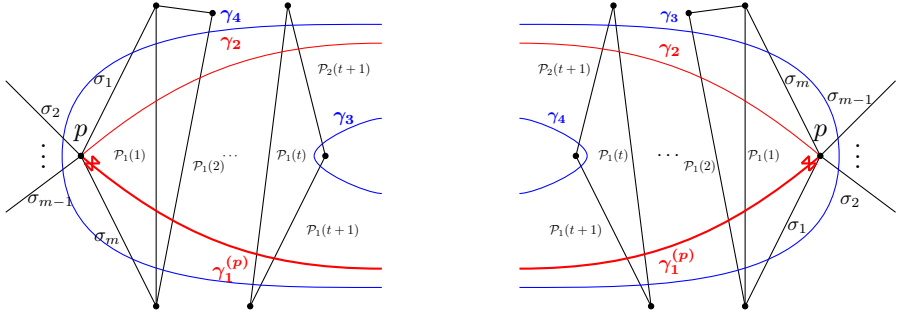
\begin{figure}[H]
\captionsetup{width=0.9\textwidth}
    \centering
    \begin{subfigure}{0.4\textwidth}
        \begin{tikzpicture}[scale = 0.5, transform shape]

    \draw (-8,0) to (-6,4);
    \draw (-6,-4) to (-6,4);
    \draw (-8,0) to (-6,-4);
    \draw (-6,-4) to (-4.5,3.8);
    \draw (-6,4) to (-4.5,3.8);
    \draw (-8,0) to (-10,-1.5);
    \draw (-8,0) to (-10,2);
    \draw (-2.5,4) to (-3.5,-4);
    \draw (-1.5,0) to (-3.5,-4);
    \draw (-2.5,4) to (-1.5,0);

    \draw[out=40,in=180,looseness=1, red] (-8,0) to (0,3);
    \draw[out=-40,in=180,looseness=1,thick, red] (-8,0) to (0,-3);
    \draw[thick,red] (-7.7,0) to (-7.7,-0.4); 
    \draw[thick,red] (-7.9,-0.2) to (-7.5,-0.2); 
    \draw[thick,red] (-7.9,-0.2) to (-7.7,-0.4); 
    \draw[thick,red] (-7.7,0) to (-7.5,-0.2); 

    \draw[out=90,in=180,looseness=0.5, blue] (-1.8,0) to (0,1);
    \draw[out=-90,in=180,looseness=0.5, blue] (-1.8,0) to (0,-1);
    \draw[out=90,in=180,looseness=1, blue] (-8.5,0) to (0,3.5);
    \draw[out=-90,in=180,looseness=1, blue] (-8.5,0) to (0,-3.5);

    \draw[fill=black] (-8,0) circle [radius=2pt];
    \draw[fill=black] (-6,4) circle [radius=2pt];
    \draw[fill=black] (-6,-4) circle [radius=2pt];
    \draw[fill=black] (-2.5,4) circle [radius=2pt];
    \draw[fill=black] (-3.5,-4) circle [radius=2pt];
    \draw[fill=black] (-4.5,3.8) circle [radius=2pt];
    \draw[fill=black] (-1.5,0) circle [radius=2pt];

    \node[scale=2] at (-8,0.6) {$p$};
    \node[scale=2] at (-9,0.2) {$\vdots$};
    \node[scale=1.5] at (-7.4,2) {$\sigma_1$};
    \node[scale=1.5] at (-8.8,1.2) {$\sigma_{2}$};
    \node[scale=1.5] at (-8.8,-1.2) {$\sigma_{m-1}$};
    \node[scale=1.5] at (-7.3,-2.2) {$\sigma_{m}$};
    \node[scale=1] at (-6.7,0) {$\calP_1(1)$};
    \node[scale=1] at (-4.6,-0.3) {$\calP_1(2)$};
    \node[scale=1] at (-2.4,0) {$\calP_1(t)$};

    \node at (-4,0) {$\cdots$};

    \node[scale=1] at (-1,2.3) {$\calP_2(t+1)$};
    \node[scale=1] at (-1.3,-2) {$\calP_{1}(t+1)$};

    \node[scale=1.5,red] at (-4,3) {$\bm{\gamma_2}$};
    \node[scale=1.5,red] at (-4,-3) {$\bm{\gamma_1^{(p)}}$};

    \node[scale=1.5,blue] at (-1,1) {$\bm{\gamma_3}$};
    \node[scale=1.5,blue] at (-4,3.7) {$\bm{\gamma_4}$};

    \end{tikzpicture}
    \end{subfigure}
    \begin{subfigure}{0.4\textwidth}
        \begin{tikzpicture}[scale=0.5,transform shape]

        \draw (8,0) to (6,4);
        \draw (6,-4) to (6,4);
        \draw (8,0) to (6,-4);
        \draw (6,-4) to (4.5,3.8);
        \draw (6,4) to (4.5,3.8);

        \draw (8,0) to (10,-1.5);
        \draw (8,0) to (10,2);

        \draw (2.5,4) to (3.5,-4);
        \draw (1.5,0) to (3.5,-4);
        \draw (2.5,4) to (1.5,0);

        \draw[out=140,in=0,looseness=1, red] (8,0) to (0,3);
        \draw[out=-140,in=0,looseness=1, thick, red] (8,0) to (0,-3);
        \draw[thick,red] (7.7,0) to (7.7,-0.4); 
        \draw[thick,red] (7.9,-0.2) to (7.5,-0.2); 
        \draw[thick,red] (7.9,-0.2) to (7.7,-0.4); 
        \draw[thick,red] (7.7,0) to (7.5,-0.2); 

        \draw[out=90,in=0,looseness=0.5, blue] (1.8,0) to (0,1);
        \draw[out=-90,in=0,looseness=0.5, blue] (1.8,0) to (0,-1);
        \draw[out=90,in=0,looseness=1, blue] (8.5,0) to (0,3.5);
        \draw[out=-90,in=0,looseness=1, blue] (8.5,0) to (0,-3.5);

        \draw[fill=black] (8,0) circle [radius=2pt];
        \draw[fill=black] (6,4) circle [radius=2pt];
        \draw[fill=black] (6,-4) circle [radius=2pt];
        \draw[fill=black] (2.5,4) circle [radius=2pt];
        \draw[fill=black] (3.5,-4) circle [radius=2pt];
        \draw[fill=black] (4.5,3.8) circle [radius=2pt];
        \draw[fill=black] (1.5,0) circle [radius=2pt];

        \node[scale=2] at (8,0.6) {$p$};
        \node[scale=2] at (9,0.2) {$\vdots$};
        \node[scale=1.5] at (7.4,2) {$\sigma_m$};
        \node[scale=1.5] at (8.8,1.5) {$\sigma_{m-1}$};
        \node[scale=1.5] at (8.8,-1.2) {$\sigma_{2}$};
        \node[scale=1.5] at (7.4,-1.9) {$\sigma_{1}$};
        \node[scale=1] at (6.5,0) {$\calP_1(1)$};
        \node[scale=1] at (4.8,-0.3) {$\calP_1(2)$};
        \node[scale=1] at (2.5,0) {$\calP_1(t)$};

        \node[scale=1.5] at (4,0) {$\cdots$};

        \node[scale=1] at (1.2,2.3) {$\calP_2(t+1)$};
        \node[scale=1] at (1.5,-2) {$\calP_{1}(t+1)$};

        \node[scale=1.5,red] at (4,2.8) {$\bm{\gamma_2}$};
        \node[scale=1.5,red] at (4.2,-2.8) {$\bm{\gamma_1^{(p)}}$};

        \node[scale=1.3,blue] at (1,1) {$\bm{\gamma_4}$};
        \node[scale=1.3,blue] at (4,3.7) {$\bm{\gamma_3}$};
    
        \end{tikzpicture}
    \end{subfigure}
    \caption{The local configuration when $k=0$ (left) and when $k=m$ (right).}
    \label{fig:punctureIncompatibilityCases}
\end{figure}

\begin{table}
\captionsetup{width=0.9\textwidth}
    \centering
\begin{tabular}{c|c}
\highlight{$\calP_{\gamma_1^{(p)}}$} & \highlight{$\calP_{\gamma_2}$} \\
     \begin{tikzpicture}[scale = 0.6] 
    \node[scale = 0.8](dots) at (6.2,-1){$\cdots$};
    \node[](taun) at (4.7,-1){$\calP_1(t+1)$};
    \node[](ul) at (2.2,0){$\calP_1(t)$};
    \node[scale = 0.8](dots2) at (1.2,0){$\cdots$};
    \node[](u1) at (0,0){$\calP_1(1)$};
    \node[](sigma1) at (-1,-3){$\sigma_1$};
    \node[](sigma2) at (-1,-1.5){$\sigma_{2}$};
    \node[] (dots3) at (-1,0){$\vdots$};
    \node[](sigmas) at (-1,1.5){$\sigma_m$};
    \draw(ul) -- (taun);
    \draw(u1) -- (sigma1);
     \draw (sigma1) -- (sigma2);
     \draw(sigma2) -- (dots3);
     \draw(dots3) --(sigmas);
     \draw(sigmas) -- (u1);
     \end{tikzpicture} &
    \begin{tikzpicture}[scale = 0.6]
    \node[](u1) at (-2,0){$\calP_2(1)$};
     \node[scale = 0.8](dots2) at (-1,0){$\cdots$};
     \node[](ul) at (0,0){$\calP_2(t)$};
     \node[](r) at (1.5,1.5) {$\calP_2(t+1)$};
     \node[] (dots) at (3.2,1.5){$\cdots$};
     \draw (r) -- (ul);
    \end{tikzpicture}\\\hline\hline
    \highlight{$\calP_{\gamma_3}$} & \highlight{$\calP_{\gamma_4}$} \\
          \begin{tikzpicture}[scale = 0.7]
    \node[scale = 0.8](dots) at (-1.5,0.5){$\cdots$};
    \node[](tau) at (0,0.5){$\calP_1(t+1)$};
    \node[](etam) at (1.5,-1){$\calP_2(t+1)$};
    \node[scale = 0.8](dots2) at (3,-1){$\cdots$};
    \draw(tau) -- (etam);
    \end{tikzpicture}&
     \begin{tikzpicture}[scale = 0.6] 
    \node[scale = 0.8](dots) at (-4.5,-1.5){$\cdots$};
    \node[](taun) at (-3,-1.5){$\calP_1(t+1)$};
    \node[](ul) at (-2,0){$\calP_1(t)$};
    \node[scale = 0.8](dots2) at (-1,0){$\cdots$};
    \node[](u1) at (0,0){$\calP_1(1)$};
    \node[](sigma1) at (1,1){$\sigma_m$};
    \node[](sigma2) at (2,0){$\sigma_{m-1}$};
    \node[] (dots3) at (3,-1){$\ddots$};
    \node[](sigmak) at (4,-2){$\sigma_1$};
    \node[](u12) at (5,-1){$\calP_2(1)$};
    \node[scale = 0.8](dots5) at (6,-1){$\cdots$};
    \node[](ul2) at (7,-1){$\calP_2(t)$};
    \node[](etam) at (8,0.5){$\calP_2(t+1)$};
    \node[scale = 0.8](dots7) at (9.5,0.5){$\cdots$};
    \draw(ul) -- (taun);
    \draw(u1) -- (sigma1);
    \draw (sigma1) -- (sigma2);
    \draw(sigma2) -- (dots3);
    \draw(dots3) --(sigmak);
    \draw(sigmak) -- (u12);
    \draw(sigmak) -- (u12);
    \draw(etam) -- (ul2);
    \end{tikzpicture}
      \\
\end{tabular}
\caption{The posets for all arcs in the resolution of an incompatibility in the $k=0$ case. }
\label{tab:PosetsFork=0}
\end{table}

If $k = 0$ or $k = m$, then for some $t \geq 1$, we have overlaps $\calP_{\gamma_1^{(p)}}[1,t]$ and $\calP_{\gamma_2}[1,t]$ as in Definition \ref{def:Overlaps}. Let $t$ be the largest such $t$, so that  $\calP_{\gamma_1^{(p)}}(t+1) \neq \calP_{\gamma_2}(t+1)$. Since we assume $\gamma_1^0 \neq \gamma_2^0$, such a value $t+1$ must exist. Since these overlaps both begin at 1, these are not crossing overlaps 

When $k = 0$, the poset $\calP_{\gamma_4}$ can be constructed in the same way as above, whereas $\calP_{\gamma_3}$ consists of $\calP_{1}[>t] \cup \calP_{2}[>t]$ where we have the added relation $\calP_{1}(t+1) \prec \calP_{2}(t+1)$ based on the local configuration of these arcs. Similarly, if $k = m$, the poset $\calP_{\gamma_3}$ follows from the $k \neq 0,m$ case and the poset $\calP_{\gamma_4}$ is given by $\calP_{1}[>t] \cup \calP_{2}[>t]$ where we now have $\calP_1(t+1) \succ \calP_2(t+1)$. See Table \ref{tab:PosetsFork=0} for the $k =0$ case; the $k = m$ case is similar.

We can similarly deal with a singly-notched monogon $\gamma^{(p)}$, which is incompatible with itself by Definition~\ref{def:CompatibleTaggedArcs}. If we follow the definition of $\gamma_3$ and $\gamma_4$ now, we will instead get two closed curves. Again we will have that $\gamma_3$ crosses $\sigma_1,\ldots,\sigma_k$ and $\gamma_4$ crosses $\sigma_{k+1},\ldots,\sigma_m$. If the first and last triangles which $\gamma^{(p)}$ passes through are the same, then we will be in the $k = 0$ case if the tagged endpoint of $\gamma$ lies clockwise from the plain endpoint and otherwise we will be in the $k = m$ case.

In the special case that $\gamma^{(p)}$ is contractible, which is a subset of the $k=0$ and $k=m$ cases, we will have that one of the closed curves $\gamma_3$ and $\gamma_4$ will be contractible and the other will be contractible onto $p$. Recall that for $\ell$ contractible, in Definition \ref{def:contandkink} we set $x_\ell = -2$, and one can see from the band graph expansion formula that for $\ell$ a closed curve contractible onto $p$, $x_\ell = 1 + Y_p$ . We also introduce a definition for the element of the cluster algebra associated to a contractible singly-notched monogon. 

\begin{definition}\label{def:ContractibleSinglyTaggedLoop}
Let $\gamma^{(p)}$ be a contractible singly-notched monogon. Suppose the spokes from $T$ at $p$ are $\sigma_1,\ldots,\sigma_m$. Then, 
\begin{align*}
    x_{\gamma^{(p)}} := \begin{cases}
        Y_p-1 & \textrm{if } k = 0 \\
        1- Y_p & \textrm{if } k = m.
    \end{cases}
\end{align*}
\end{definition}

Comparing Definition \ref{def:ContractibleSinglyTaggedLoop} with Proposition \ref{prop:incompatible_puncture_singly_tagged} will show that this definition is compatible with our resolution for a non-contractible, singly-notched monogon.

\begin{figure}[H]
\captionsetup{width=0.9\textwidth}
    \centering
     \begin{subfigure}{0.3\textwidth}
        \includegraphics[scale=0.5]{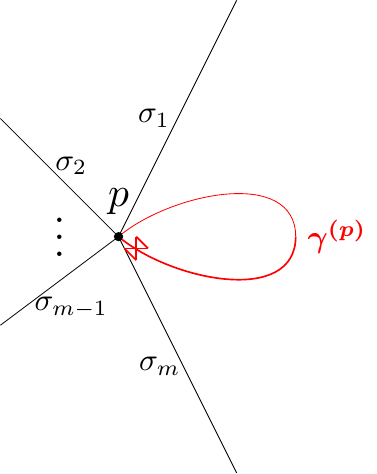}
    \end{subfigure}
    \hspace{0.5cm}
    \begin{subfigure}{0.3\textwidth}
        \includegraphics[scale=0.5]{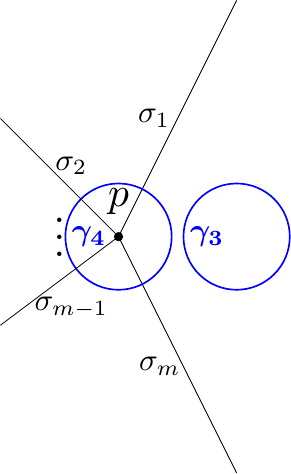}
    \end{subfigure}
        \caption{A contractible singly-notched monogon in the $k=0$ case (left) and the resulting contractible closed curves $\gamma_3$ and $\gamma_4$}.
        \label{fig:monogon}
\end{figure}

For the part of the following proposition concerning a pair of incompatible arcs, we assume the tagging of one endpoint of each arc. An advantage of our approach is that the tagging of the other endpoint of each arc does not affect the relation. Therefore, this proposition encompasses cases 6,7,9, and 11 from \cite{musiker2013bases}. The case where $\gamma_2 \in T$ was also studied by Pilaud, Reading, and Schroll in \cite{pilaud2023posets}. 

\begin{prop}
\label{prop:incompatible_puncture_singly_tagged}
Let $\gamma_1^{(p)}$ and $\gamma_2$ be arcs which are incompatible at a puncture $p$ and let $\gamma^{(p)}$ be a singly-notched non-contractible monogon.

\begin{enumerate}
    \item Suppose $\gamma_2 = \sigma_k \in T$. Then,
    \[
    x_{\gamma_1^{(p)}}x_{\gamma_2}= x_{\gamma_3} + (\prod_{\tau \in \calP_{\gamma_1^{(p)}}^{\preceq \sigma_k}} y_\tau) x_{\gamma_4}.
    \]
    \item Suppose $\gamma_2 \notin T$. Let $t \geq 0$ be the largest number such that $\calP_{1}[1,t]$ and $\calP_{2}[1,t]$ are overlaps, where in particular $t = 0$ if $\calP_{\gamma_1^{(p)}}(1) \neq \calP_{\gamma_2}(1)$. Let $k$ be as defined above.  Set \[
Y_R = \prod_{i=1}^{t} y_{\calP_1(i)}\quad \text{and} \quad Y_{S} = \prod_{i=1}^{k} y_{\calP_1(-i+1)} = \prod_{i=1}^k y_{\sigma_i}
\]
We have $x_{\gamma_1^{(p)}}x_{\gamma_2}=X_{C^{+}}+X_{C^{-}}$ where $X_{C^{+}}$ and $X_{C^{-}}$ are defined as follows.
\begin{center}
\bgroup
\def\arraystretch{1.25}
\begin{tabular}{ |C{2cm} | C{2cm} | C{2cm}| } 
 \hline
 \cellcolor{gray!75} & \cellcolor{gray!25} $X_{C^{+}}$ & \cellcolor{gray!25} $X_{C^{-}}$ \\
 \hline
 \cellcolor{gray!25} $k\neq0, m$ & $x_{\gamma_3}$ & $Y_S x_{\gamma_4}$ \\ 
 \hline
 \cellcolor{gray!25} $k=0$ & $x_{\gamma_4}$ &  $Y_Rx_{\gamma_3}$\\ 
 \hline
 \cellcolor{gray!25} $k=m$ & $x_{\gamma_3}$ & $Y_SY_Rx_{\gamma_4}$ \\ 
 \hline
\end{tabular} 
\egroup
\end{center}
\item Let $n = \vert \calP_{\gamma}^0 \vert$ and let $t \geq 0$ be the largest number such that $\calP_\gamma[1,t]$ and $\calP_\gamma[d-t+1,d]$ are self-overlaps, where in particular $t = 0$ if $\calP(1) \neq \calP(n)$. If we define $Y_R, Y_S, X_{C^+},$ and $X_{C^-}$ as above, then, $x_{\gamma^{(p)}} = X_{C^+} + X_{C^-}$.
\end{enumerate} 
\end{prop}
\begin{proof}

For shorthand in this proof, set $\calP_1 = \calP_{\gamma_1^{(p)}}$ and set $\calP_2,\calP_{3},\calP_{4}$ similarly. 

\textbf{Part 1) $\gamma_2 = \sigma_k \in T$}

Every element in $\calP_1$ has at most one corresponding image in $\calP_{3}$ and in $\calP_{4}$, so that under certain conditions we can consider subsets of $\calP_1$ as subsets of $\calP_{3}$ or of $\calP_{4}$. We define a map $\Phi: J(\calP_1) \to J(\calP_{3}) \cup J(\calP_{4})$ as 
\[
\Phi(I_1) = \begin{cases} I_1 \in J(\calP_{\gamma_3}) & \sigma_k \notin I_1 \\
I_1 \backslash \langle \sigma_k \rangle \in J(\calP_{\gamma_4}) & \sigma_k \in I_1\\
\end{cases}
\]

This map is clearly a bijection.

Recall that, since $\gamma_2$ is a plain arc in $T$, $\gb_2:= \gb_{\gamma_2} = \eb_{\sigma_k}$. We can quickly check that $\gb_{\gamma_3} = \gb_{\gamma_1^{(p)}} + \eb_{\sigma_k}$. This is immediate if $k \neq 1$. If $k = 1$ and $a$ is the smallest integer such that $\calP_1(a)$ is maximal in $\calP_1^0$, then our claim follows from noting that $\calP_{\gamma_3}$ begins with $\calP_1(a+1)$, so that $\eb_{\calP_1(a)}$ contributes to $\mathbf{b}_{\gamma_1^{(p)}}$ and to $\mathbf{r}_{\gamma_3}$. 

Next, we check that $\gb_{\gamma_1^{(p)}} + \eb_{\sigma_k} + \degx(\prod_{\tau \in \calP_{\gamma_1^{(p)}}^{\preceq \sigma_k}} \hat{y}_\tau)= \gb_{\gamma_4}$.

If $k < m$, then by Corollary \ref{cor:y-hat_spokes} we have \[
\degx(\prod_{\tau \in \calP_{\gamma_1^{(p)}}^{\preceq \sigma_k}} \hat{y}_\tau) = (\eb_{\sigma_m} + \eb_{\sigma_1} + \eb_{\sigma_{[k]}}) - (\eb_{\sigma_k} + \eb_{\sigma_{k+1}} +  \eb_{\calP_1(1)}). 
\]

Given a statement $T$, let $\delta_T = 1$ if $X$ is true and otherwise let $\delta_T = 0$. We observe \[
\gb_{\gamma_1^{(p)}} = -\eb_{\sigma_1} + \delta_{\mathcal{P}_1(1) \succ \mathcal{P}_1(2)} \eb_{\calP_1(1)} + \gb'
\]
and \[
\gb_{\gamma_4} = \eb_{\sigma_{[k]}} - \eb_{\sigma_{k+1}} + \eb_{\sigma_m} - \delta_{\mathcal{P}_1(1) \prec \mathcal{P}_1(2)} \eb_{\calP_1(1)} + \gb'
\]

where $\gb'$ refers to the contributions from $\calP[>1]$. Our claim follows after noting $\delta_{\mathcal{P}_1(1) \succ \mathcal{P}_1(2)} - 1 = \delta_{\mathcal{P}_1(1) \prec \mathcal{P}_1(2)}$. 

Now suppose $k = m$ and $a$ is the smallest integer such that $\calP_1(a)$ is minimal in $\calP_1^0$.  Let $[\calP_1(a)]$ denote the third arc in the triangle formed by $\calP_1(a)$ and $\calP_1(a+1)$. Then, using Lemma \ref{lem:y-hat} and noting by Corollary \ref{cor:y-hat_spokes}, $\degx(\hat{y}_{\sigma_1}\hat{y}_{\sigma_1} \cdots \hat{y}_{\sigma_m}) = 0 $, we have \[
\degx(\prod_{\tau \in \calP_{\gamma_1^{(p)}}^{\preceq \sigma_m}} \hat{y}_\tau) = (\eb_{\sigma_m} + \eb_{\calP_1(1)} + \eb_{\calP_1(a+1)}) - (\eb_{\calP_1(a)} + \eb_{[\calP_1(a)]} + \eb_{\sigma_1})
\]

We can update $\gb_{\gamma_1^{(p)}}$ and $\gb_{\gamma_4}$ in this case with our additional knowledge of the first minimal element in $\calP_1$ and again find  $\gb_{\gamma_1^{(p)}} + \eb_{\sigma_k} + \degx(\prod_{\tau \in \calP_{\gamma_1^{(p)}}^{\preceq \sigma_k}} \hat{y}_\tau)= \gb_{\gamma_4}$.

\textbf{Part 2) $\gamma_2 \notin T$}

If $k \neq 0, m$, then we set up the following map $\Phi:  J(\calP_{1}) \times J(\calP_{2}) \to J(\calP_{\gamma_3}) \cup J(\calP_{\gamma_4})$, 
\[\Phi((I_1,I_2)) = \begin{cases} I_1 \cup I_2 \in J(\calP_{\gamma_3}) &  \sigma_{k+1} \notin I_1, \sigma_k \in I_1 \text{ only if } \calP_2(1) \in I_2 \\
I_1 \backslash \langle \sigma_k \rangle \cup I_2 \in J(\calP_{\gamma_4}) & \text{otherwise}\\ \end{cases}
\]

The conditions make it clear that our map constructs order ideals of both posets $\calP_3$ and $\calP_4$.  Once we have verified that this map is well-defined, the fact that it is a bijection is immediate.

Now, we consider the special cases. First, let $k = 0$; recall this means that the first triangles which $\gamma_1^{(p)}$ and $\gamma_2$ pass through are the same and that $\gamma_2$ lies counterclockwise from $\gamma_1^{(p)}$ as on the lefthand side of Figure \ref{fig:punctureIncompatibilityCases}. The relevant posets are given in Table \ref{tab:PosetsFork=0}.

 We partition $J(\calP_1)\times J(\calP_2) = A_1 \cup A_2 \cup A_3 \cup B$ and define a map $\Phi: J(\calP_1) \times J(\calP_2) \to J({P}_{3}) \cup J({P}_{4})$. Since $k = 0$, we  want a bijection $\Phi$ which restricts to a weight-preserving bijection between $A_1 \cup A_2 \cup A_3$ and $J(\calP_{4})$ and such that for each $(I_1,I_2) \in B$, $wt(I_1)wt(I_2) = \hat{Y}_R wt(\Phi(I_1,I_2))$.  Let $R_1 = \calP_1[1,t]$ and $R_2 = \calP_2[1,t]$, and  let $R_3$ ($R_4$) denote the image of $R_1$ ($R_2$) in $\calP_{3}$. Any elements outside $R_1$ and $R_2$ in $I_1$ and $I_2$ each have a unique image in $\calP_{3}$ and $\calP_{4}$, so we will largely ignore these elements.

\begin{enumerate}
    \item Let $A_1$ consist of all pairs $(I_1,I_2)$ such that $\calP_2(1) \in I_2$ only if $\sigma_1 \in I_1$. If $(I_1,I_2) \in A_1$, then $\Phi$ acts by sending elements from $R_1$ to $R_3$ and from $R_2$ to $R_4$. The image $\Phi(A_1)$ is all $I_4 \in \calP_4$ such that $\calP_1(1) \in I_4$ only if $\sigma_1 \in I_4$. 
    \item Let $A_2$ consist of all pairs $(I_1,I_2)$ such that $\calP_2(1) \in I_2$; $\sigma_1 \notin I_1$; and, there exists a switching position between $I_1 \cap R_1$ and $I_2 \cap R_2$.  If $(I_1,I_2) \in A_2$, we know there exists at least one switching position between $I_1 \cap R_1$ and $I_2 \cap R_2$. We map $I_1 \cap R_1$  and $I_2 \cap R_2$ to $R_4$ and $R_3$ respectively until the first switching position, and then we swap. The image $\Phi(A_2)$ is all $I_4 \in \calP_4$ such that $\calP_1(1) \in I_4$; $\sigma_1 \notin I_1$; and, there is a switching position between $R_3$ and $R_4$. 
    \item Let $A_3$ consist of all pairs $(I_1,I_2)$ such that  $R_2 \subseteq I_2$; $R_1 \cap I_1 = \emptyset$; $\sigma_1 \notin I_1$; $\calP_1(t+1) \in I_1$; and,  $\calP_2(t+1) \notin I_2$. If $(I_1,I_2) \in A_3$, we send $R_2$ to $R_3$, and all other elements have clear images. The set $\Phi(A_3)$ is all $I_4 \in \calP_4$ such that $R_3 \subseteq I_4$; $R_4 \cap I_4 = \emptyset$; and, $\sigma_1 \notin I_1$. 
    \item Let $B$  consist of all pairs $(I_1,I_2)$ such that $R_2 \subseteq I_2$; $R_1 \cap I_1 = \emptyset$; $\sigma_1 \notin I_1$; and, $\calP_1(t+1) \in I_1$ only if $\calP_2(t+1) \in I_2$. If $(I_1,I_2) \in B$, we send $I_1 \cup (I_2 \backslash R_2)$ to the corresponding order ideal of $\calP_{3}$.
\end{enumerate}

Now that we understand the image of each subset of $J(\calP_1) \times J(\calP_2)$, an inverse map could be quickly constructed.

The bijection for the case $k = m$ is dual - one can swap $\sigma_1$ with $\sigma_m$ and the order of the conditionals to define a new partitioning $A_1 \cup A_2 \cup A_3 \cup B$ and proceed with similar maps.

When $k \neq 0,m$, the $\gb$-vector calculations are similar to those for the case when $\gamma_2 \in T$. Consider the case for $k = 0$. Let $\gb_R$ be the contribution of $\calP_1[2,t-1]$ to $\gb_{\gamma_1^{(p)}}$. Notice this is equal to the contribution of $\calP_2[2,t-1]$ to $\gb_{\gamma_2}$. Let $\gb_{1}'$ be the contribution of $\calP_1[>t+1]$ to $\gb_{\gamma_1}$ and define $\gb_{2}'$ similarly. We have that \[
\gb_{\gamma_1^{(p)}} = -\eb_{\sigma_1} + \delta_{\calP_1(1) \succ \calP_1(2)} \eb_{\calP_1(1)} + \delta_{\calP_1(t) \succ \calP_1(t-1)} \eb_{\calP_1(t)} - \delta_{\calP_1(t+1) \prec \calP_1(t+2)} \eb_{\calP_1(t+1)} + \gb_R + \gb_{1}', 
\]
\[
\gb_{\gamma_2} = \eb_{\sigma_m} - \delta_{\calP_2(1) \prec \calP_2(2)} \eb_{\calP_2(1)}  - \delta_{\calP_2(t) \prec \calP_2(t-1)} \eb_{\calP_2(t)} + \delta_{\calP_2(t+1) \succ \calP_2(t+2)} \eb_{\calP_2(t+1)} + \gb_R + \gb_2'.
\]
By observing $\calP_4$ in Table \ref{tab:PosetsFork=0}, we see $\gb_{\gamma_1^{(p)}} + \gb_{\gamma_2} = \gb_{\gamma_4}$. Now, by Lemma \ref{lem:y-hat} we have \begin{align*}
\degx(\hat{Y}_R) = \eb_{\sigma_1} - \eb_{\sigma_m} + \eb_{\calP_1(t+1)} - \eb_{\calP_2(t+1)} + (\delta_{\calP_1(1) \prec \calP_1(2)} - \delta_{\calP_1(1) \succ \calP_1(2)}) \eb_{\calP_1(1)} \\+ (\delta_{\calP_1(t) \prec \calP_1(t-1)} - \delta_{\calP_1(t) \succ \calP_1(t-1)})  \eb_{\calP_1(t)}  -2\gb_R.
\end{align*}
Notice that for each term we have the opposite sign on $\eb_{\calP_1(1)} = \eb_{\calP_2(1)}$ and $\eb_{\calP_1(t)} = \eb_{\calP_2(t)}$ as in $\gb_{\gamma_1^{(p)}} + \gb_{\gamma_2}$. 
We conclude \[
\gb_{\gamma_1^{(p)}} + \gb_{\gamma_2}+\degx(\hat{Y}_R) = -\delta_{\calP_2(t+1) \prec \calP_2(t+2)} \eb_{\calP_2(t+1)} + \delta_{\calP_1(t+1) \succ \calP_1(t+2)} \eb_{\calP_1(t+1)} + \gb_1' + \gb_2' = \gb_{\gamma_3}.
\]

\textbf{Part 3) A singly-notched monogon}

Given a singly-notched non-contractible monogon, one can use the same proofs as in the two arc case with only minor adjustments to reach the conclusion of part (3).

\end{proof}
\begin{remark}
In Example 3.16 of \cite{Mandel2023}, Mandel presents a similar result regarding singly-notched monogons in the context of coefficient-free cluster algebras. In this example, the approach involves introducing a covering space that covers the surface three times, referred to as a triple cover. This technique is applied alongside the digon relation, which concerns the interaction between a singly-notched arc and a plain arc, as discussed in \cite{fomin2018cluster}. 
\end{remark}

\begin{remark}
Since resolving incompatible taggings involves a different local move than resolving transverse crossings, the set $S$ behaves slightly differently in this setting. However, one can notice that the spokes at $p$ which contribute to the $y$-monomial are exactly those which lie counterclockwise of $\gamma_1^{(p)}$ and clockwise of $\gamma_2$, mirroring the way arcs contribute to $S$ from the description in the introduction. 
\end{remark}

In Proposition \ref{prop:incompatible_puncture_singly_tagged}, we assumed $\gamma_1^0 \neq \gamma_2^0$, and the resolution had two terms. However, if  $\gamma_1^0 = \gamma_2^0$, we can use similar techniques to prove a relation between $x_{\gamma_1^{(p)}}$ and $x_{\gamma_2}$.

\begin{cor}\label{cor:PlainTimesNotchedEqualsLoop}
Let $\gamma_1^{(p)}$ and $\gamma_2$ be two arcs incompatible at a puncture $p$ such that $\gamma_1^0 = \gamma_2^0$. Let $\gamma_3 = \gamma_1^{-1} \circ_{CCW} \gamma_2$. Then,\[
x_{\gamma_1^{(p)}} x_{\gamma_2} = x_{\gamma_3}
\]
\end{cor}

Note that, up to orientation, $\gamma_3$ is homotopic to $\gamma_1^{-1} \circ_{CW} \gamma_2$. When $\gamma_1$ and $\gamma_2$ are ordinary arcs, so that $x_{\gamma_1^{(p)}}$ and $x_{\gamma_2}$ are cluster variables, this identity is known. Here, we can see this formula holds even for generalized arcs. This shows that our definition of the cluster algebra element $x_{\gamma_1^{(p)}}$ via the poset expansion formula is equivalent to a definition using such algebraic expressions. This equality was suggested in Section 8.2 of \cite{musiker2013bases}. In the next section, we show a similar statement for doubly-notched arcs.

\subsection{Incompatible at two punctures}\label{subsec:IncAtTwo}

Two items from the list of cases of skein relations in \cite{musiker2013bases} concern a pair of arcs $\gamma_1$ and $\gamma_2$ which are incompatible at two punctures. This pair could consist of two singly-notched arcs or a doubly-notched arc and a plain arc. In Section 3.2 of \cite{brustle2015tagged} it is shown that the modules corresponding to such a pair of arcs are related by Auslander-Reiten translation. 

If $\gamma_1^0 \neq \gamma_2^0$, then we can use Proposition \ref{prop:incompatible_puncture_singly_tagged} twice to write an expression for $x_{\gamma_1}x_{\gamma_2}$ in terms of arcs which do not have incompatibilities. This will result in an expression involving the following four closed curves. Suppose $\gamma_1$ and $\gamma_2$ are incident to punctures $p$ and $q$, and let $\mathsf{h}_p$ and $\mathsf{h}_q$ be small circles drawn around $p$ and $q$ respectively such that each does not intersect any arcs other than the spokes at the relevant puncture. Suppose both $\gamma_1$ and $\gamma_2$ are oriented from $q$ to $p$. These curves are also drawn in Figure \ref{fig:loops}; compare with Figure \ref{Two}.

\begin{itemize}
\item Let $\ell_1$ denote the closed curve we have from beginning at the intersection point of $\gamma_1$ and $\mathsf{h}_q$, following $\gamma_1$ until its intersection point with $\mathsf{h}_p$, following $\mathsf{h}_p$ in \emph{clockwise} order until its intersection point with $\gamma_2$, following $\gamma_2^{-1}$ until its intersection point with $\mathsf{h}_q$, and then following $\mathsf{h}_q$ in \emph{clockwise} order until we reach our starting point. 
\item Let $\ell_2$ be the closed curve obtained when we instead follow both $\mathsf{h}_p$ and $\mathsf{h}_q$ in counterclockwise order.
\item Let $\ell_3$ be the closed curve obtained when we follow $\mathsf{h}_p$ in clockwise order and $\mathsf{h}_q$ in counterclockwise order.
\item Let $\ell_4$ be the closed curve obtained when we follow both $\mathsf{h}_p$ in counterclockwise order and $\mathsf{h}_q$ in clockwise order. 
\end{itemize}

If we have $\gamma_1^0 = \gamma_2^0$, then we cannot use the methods from Section \ref{subsec:IncAtOne} since we assume $\gamma_1^0 \neq \gamma_2^0$. We will show the results directly in this case. Note that $\ell_2$ is contractible if $\gamma_1^0 = \gamma_2^0$.

\begin{figure}[H]
\captionsetup{width=0.9\textwidth}
\centering
\begin{tikzpicture}[scale=0.5, transform shape]
\draw (8,0) to (7,4);
\draw (4.5,1) to (7,4);
\draw (4.5,1) to (8,0);

\draw (8,0) to (7,-4);
\draw (5.5,-1) to (7,-4);
\draw (5.5,-1) to (8,0);

\draw (8,0) to (10,-1);
\draw (8,0) to (10,1.5);

\draw (-8,0) to (-7,4);
\draw (-6,1.5) to (-7,4);
\draw (-6,1.5) to (-8,0);

\draw (-6.5,-1) to (-7,-4);
\draw (-6.5,-1) to (-8,0);
\draw (-8,0) to (-7,-4);

\draw (-8,0) to (-10,-1.5);
\draw (-8,0) to (-10,2);

\draw[out=120,in=60,looseness=1, red] (8,0) to (-8,0);
\draw[out=-120,in=-60,looseness=1, red] (8,0) to (-8,0);

\draw[fill=black] (8,0) circle [radius=2pt];
\draw[fill=black] (4.5,1) circle [radius=2pt];\draw[fill=black] (5.5,-1) circle [radius=2pt];

\draw[fill=black] (-8,0) circle [radius=2pt];
\draw[fill=black] (-6,1.5) circle [radius=2pt];\draw[fill=black] (-6.5,-1) circle [radius=2pt];

\node[scale=2] at (8.2,0.5) {$p$};
\node[scale=1.5] at (7.9,2) {$\sigma_m$};
\node[scale=1.5] at (8.8,1.3) {$\sigma_{m-1}$};
\node[scale=1.5] at (8.8,-1) {$\sigma_{k_p+2}$};
\node[scale=1.5] at (8.2,-2.2) {$\sigma_{k_p+1}$};
\node[scale=1.5] at (6.6,-1) {$\sigma_{k_p}$};
\node[scale=1.5] at (6.6,0.7) {$\sigma_{1}$};
\node[scale=1] at (5.6,3) {$\calP_1(1)$};
\node[scale=1] at (5.8,-3) {$\calP_2(1)$};

\node[scale=3] at (0,2) {$\cdots$};

\node[scale=2] at (9,0.28) {$\vdots$};
\node[scale=1.5] at (6.6,0.1) {$\vdots$};
\node[scale=1.5] at (-6.6,0.2) {$\vdots$};

\node[scale=3] at (0,-1.5) {$\cdots$};

\node[scale=2] at (-8.2,0.5) {$q$};
\node[scale=2] at (-9,0.2) {$\vdots$};
\node[scale=1.5] at (-7.9,2) {$\eta_1$};
\node[scale=1.5] at (-8.8,1.2) {$\eta_2$};
\node[scale=1.5] at (-8.8,-1.2) {$\eta_{k_q-1}$};
\node[scale=1.5] at (-7.8,-2.2) {$\eta_{k_q}$};
\node[scale=1.5] at (-6.5,-0.6) {$\eta_{k_q+1}$};
\node[scale=1.5] at (-6.6,0.6) {$\eta_{h}$};
\node[scale=1] at (-6,3) {$\calP_1(n_1)$};
\node[scale=1] at (-6.3,-3) {$\calP_2(n_2)$};

\node[scale=2,red] at (0,3.5) {$\gamma_1^{0}$};
\node[scale=2,red] at (0,-3.5) {$\gamma_2^{0}$};

\end{tikzpicture}
    \caption{The underlying arcs, $\gamma_1^{0}$ and $\gamma_2^{0}$, which will be considered in Section \ref{subsec:IncAtTwo} with several different types of notchings.}
    \label{Two}
\end{figure}
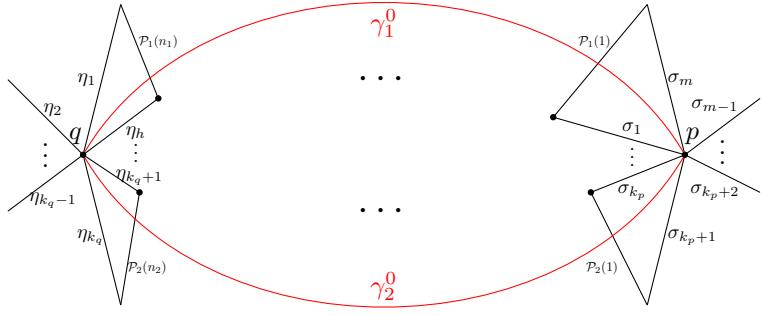

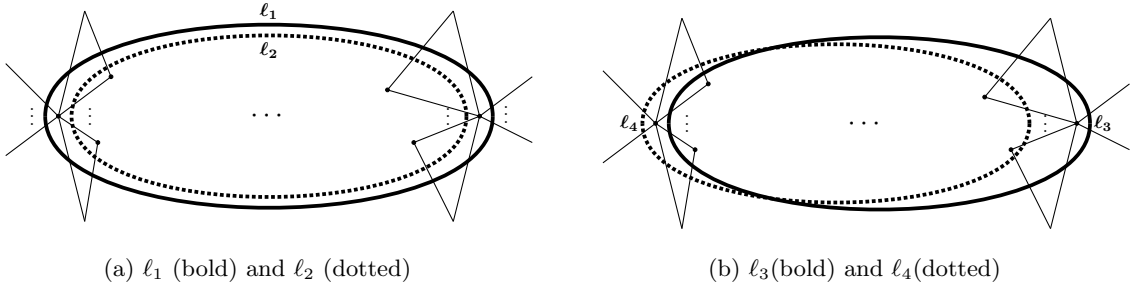
\begin{figure}[H]
\captionsetup{width=0.9\textwidth}
    \centering
     \begin{subfigure}[b]{0.4\textwidth}
        \begin{tikzpicture}[scale=0.35, transform shape]
        \draw (8,0) to (7,4);
        \draw (4.5,1) to (7,4);
        \draw (4.5,1) to (8,0);
        
        \draw (8,0) to (7,-4);
        \draw (5.5,-1) to (7,-4);
        \draw (5.5,-1) to (8,0);

        \draw (8,0) to (10,-1);
        \draw (8,0) to (10,1.5);

        \draw (-8,0) to (-7,4);
        \draw (-6,1.5) to (-7,4);
        \draw (-6,1.5) to (-8,0);
        
        \draw (-6.5,-1) to (-7,-4);
        \draw (-6.5,-1) to (-8,0);
        \draw (-8,0) to (-7,-4);

        \draw (-8,0) to (-10,-1.5);
        \draw (-8,0) to (-10,2);
        
        \draw[out=90,in=90,looseness=0.7,line width=0.5mm] (8.5,0) to (-8.5,0);
        \draw[out=-90,in=-90,looseness=0.7,line width=0.5mm] (8.5,0) to (-8.5,0);
        
        \draw[out=90,in=90,looseness=0.7,line width=0.5mm,densely dotted] (7.5,0) to (-7.5,0);
        \draw[out=-90,in=-90,looseness=0.7,line width=0.5mm,densely dotted] (7.5,0) to (-7.5,0);

        \draw[fill=black] (8,0) circle [radius=2pt];
        \draw[fill=black] (4.5,1) circle [radius=2pt];\draw[fill=black] (5.5,-1) circle [radius=2pt];

        \draw[fill=black] (-8,0) circle [radius=2pt];
        \draw[fill=black] (-6,1.5) circle [radius=2pt];\draw[fill=black] (-6.5,-1) circle [radius=2pt];



        \node[scale=2] at (9,0.28) {$\vdots$};
        \node[scale=2] at (6.8,0.2) {$\vdots$};
        \node[scale=2] at (-6.8,0.2) {$\vdots$};

        \node[scale=3] at (0,0) {$\cdots$};

        \node[scale=2] at (-9,0.2) {$\vdots$};

        \node[scale=2, font=\boldmath] at (0,4) {$\ell_1$};
        \node[scale=2, font=\boldmath] at (0,2.5) {$\ell_2$};

        \end{tikzpicture}
                \caption{$\ell_1$ (bold) and $\ell_2$ (dotted)}
        \label{fig:loops2}
    \end{subfigure}
    \hspace{1cm}
    \begin{subfigure}[b]{0.4\textwidth}
        \begin{tikzpicture}[scale=0.35, transform shape]
        \draw (8,0) to (7,4);
        \draw (4.5,1) to (7,4);
        \draw (4.5,1) to (8,0);
        
        \draw (8,0) to (7,-4);
        \draw (5.5,-1) to (7,-4);
        \draw (5.5,-1) to (8,0);

        \draw (8,0) to (10,-1);
        \draw (8,0) to (10,1.5);

        \draw (-8,0) to (-7,4);
        \draw (-6,1.5) to (-7,4);
        \draw (-6,1.5) to (-8,0);
        
        \draw (-6.5,-1) to (-7,-4);
        \draw (-6.5,-1) to (-8,0);
        \draw (-8,0) to (-7,-4);

        \draw (-8,0) to (-10,-1.5);
        \draw (-8,0) to (-10,2);
        
        \draw[out=90,in=90,looseness=0.7,line width=0.5mm] (8.5,0) to (-7.5,0);
        \draw[out=-90,in=-90,looseness=0.7,line width=0.5mm] (8.5,0) to (-7.5,0);
        
        \draw[out=90,in=90,looseness=0.7,line width=0.5mm, densely dotted] (6.2,0) to (-8.5,0);
        \draw[out=-90,in=-90,looseness=0.7,line width=0.5mm, densely dotted] (6.2,0) to (-8.5,0);
        
        \draw[fill=black] (8,0) circle [radius=2pt];
        \draw[fill=black] (4.5,1) circle [radius=2pt];\draw[fill=black] (5.5,-1) circle [radius=2pt];
        
        \draw[fill=black] (-8,0) circle [radius=2pt];
        \draw[fill=black] (-6,1.5) circle [radius=2pt];\draw[fill=black] (-6.5,-1) circle [radius=2pt];
        

        
        \node[scale=2] at (6.8,0.2) {$\vdots$};
        \node[scale=2] at (-6.8,0.2) {$\vdots$};
        
        \node[scale=3] at (0,0) {$\cdots$};

        \node[scale=2, font=\boldmath] at (9,0) {$\ell_3$};
        \node[scale=2, font=\boldmath] at (-9,0) {$\ell_4$};

        \end{tikzpicture}
                \caption{{$\ell_3$}(bold) and {$\ell_4$}(dotted)}
        \label{fig:loops3}
    \end{subfigure}
        \caption{The four closed curves obtained by resolving arcs $\gamma_1$ and $\gamma_2$ (whose underlying arcs $\gamma_1^{0}$ and $\gamma_2^{0}$ are shown in Figure \ref{Two}) that are incompatible at two punctures.}
        \label{fig:loops}
\end{figure}

We assume the spokes at $p$ are $\sigma_1,\ldots,\sigma_m$ and the spokes at $q$ are $\eta_1,\ldots,\eta_h$. We again index these so that, if $\gamma_1$ is oriented from $p$ to $q$, the first triangle $\gamma_1$ passes through is bordered by $\sigma_1$ and $\sigma_m$ and the last triangle is bordered by $\eta_1$ and $\eta_h$. We let $k_p$ be the quantity $k$ from Section \ref{subsec:IncAtOne} and let $k_q$ refer to this same quantity but with respect to the spokes $\eta_i$ and the puncture $q$.

If $k_p = 0$ or $k_p = m$, let $t_p$ be the largest number such that $\calP_{\gamma_1}[1,t_p]$ and $\calP_{\gamma_2}[1,t_p]$ are  overlaps, and similarly if $k_q = 0$ or $k_q = h$, let $t_q$ be the largest number such that $\calP_{\gamma_1}[n_1 - t_q,n_1]$ and $\calP_{\gamma_2}[n_2 - t_q,n_2]$ are overlaps.

We begin with the resolution of a doubly-notched arc and a plain arc that are incompatible at two punctures.

\begin{prop}\label{prop:TwiceIncDoublyNotched}
If $\gamma_1^{(p,q)}$ and $\gamma_2$ are doubly-notched and plain arcs respectively which are incompatible at both $p$ and $q$ and such that $\gamma_1^0 \neq \gamma_2^0$, then $x_{\gamma_1^{(p,q)}} x_{\gamma_2} =  Y_1x_{\ell_1} + Y_2 x_{\ell_2} + Y_3 x_{\ell_3} + Y_4 x_{\ell_4}$ where the tuple $(Y_1,Y_2,Y_3,Y_4)$ is given below. 

\scalebox{0.9}{
\bgroup
\def\arraystretch{1.25}
\begin{tabular}{|c|c|c|c|} \hline
   \cellcolor{gray!75} & \cellcolor{gray!25} $k_p \neq 0,m$ & \cellcolor{gray!25} $k_p = m$  & \cellcolor{gray!25} $k_p = 0$ \\\hline
 \cellcolor{gray!25} $k_q \neq 0,h$    & $(Y_{S^p},Y_{S^q_1},Y_{S^p}Y_{S^q_1},1)$  & $(Y_pY_{R^p},Y_{S^q_1},Y_pY_{S^q_1}Y_{R^p},1)$ & $(1,Y_{S^q_1}Y_{R^p},Y_{S^q_1},Y_{R^p})$ \\\hline
  \cellcolor{gray!25} $k_q = h$   &$(Y_{S^p},Y_qY_{R^q},Y_{S^p}Y_qY_{R^q},1)$ &$(Y_pY_{R^p},Y_qY_{R^q},Y_pY_qY_{R^q}Y_{R^p},1)$&$(1,Y_qY_{R^p}Y_{R^q},Y_qY_{R^q},Y_{R^p})$ \\\hline
 \cellcolor{gray!25} $k_q = 0$    & $(Y_{S^p}Y_{R^q},1,Y_{S^p},Y_{R^q})$ & $(Y_pY_{R^p}Y_{R^q},1,Y_pY_{R^p},Y_{R^q})$ & $(Y_{R^q},Y_{R^p},1,Y_{R^p}Y_{R^q})$\\\hline

\end{tabular} \egroup}

The $y$-monomials in the above table are as follows, where $\calP_i:= \calP_{\gamma_i}$, \[
Y_{R^p} = \prod_{j=1}^{t_p} y_{\calP_1(i)}  \quad Y_{R^q} = \prod_{j=n_1-t_q}^{t_q} y_{\calP_1(j)} \quad Y_{S^p} = \prod_{j=1}^{k_p} y_{\sigma_j}\quad \text{and} \quad Y_{S^q_1} = \prod_{j=1}^{k_q} y_{\eta_j}
\]

Now let $\gamma_1^{(p,q)}$ and $\gamma_2$ be doubly-notched and plain arcs respectively which are incompatible at both $p$ and $q$ and such that $\gamma_1^0=\gamma_2^0$. 
Let $Y_R = \prod_{j=1}^{n_1} y_{\calP_1(j)}$.  Then, we have $x_{\gamma_1^{(p,q)}} x_{\gamma_2} = x_{\ell_1} + Y_pY_qY_R + Y_R$. 
\end{prop}

\begin{proof}

If $\gamma_1^0 \neq \gamma_2^0$, then these follow from Proposition \ref{prop:incompatible_puncture_singly_tagged}. We illustrate one case. Suppose $k_p \neq 0,m$ and and $k_q = h$. If we first resolve the incompatibility at $p$, from part (2) of Proposition \ref{prop:incompatible_puncture_singly_tagged} we have \[
x_{\gamma_1^{(p,q)}} x_{\gamma_2} = x_{\gamma_3^{(q)}} + Y_{S^p} x_{\gamma_4^{(q)}}.
\]

Now, we apply part (3) of Proposition \ref{prop:incompatible_puncture_singly_tagged} to $\gamma_3^{(q)}$ and $\gamma_4^{(q)}$, \[
x_{\gamma_3^{(q)}} = x_{\ell_4} + Y_{q}Y_{R^q} x_{\ell_2} \qquad x_{\gamma_4^{(q)}} = x_{\ell_1} + Y_qY_{R^q} x_{\ell_3},
\]

and the claim follows.

If $\gamma_1^0 = \gamma_2^0$, then we cannot use Proposition \ref{prop:incompatible_puncture_singly_tagged}. We will instead exhibit a bijection between $J(\calP_1) \times J(\calP_2) \backslash \{(\calP_1,\emptyset),(\emptyset,\calP_2)\}$ and $J(\calP_{\ell})$ and demonstrate a relation on the $\gb$-vectors which will explain the relation. 

We draw $\calP_1$ and $\calP_{\ell_1}$ below, where $R_1 = \calP_1^0$. We know $\calP_1^0 = \calP_2 = \calP_2^0$, and we refer to $\calP_2$ as $R_2$ for shorthand. We have two copies of this poset as a subset in $\calP_{\ell_1}$, denoted $R_3$ and $R_4$. We choose chronological orderings on the three posets so that its restrictions to $R_1,R_2,$ and $R_3$ all agree; consequentially, this ordering on $R_4$ will be opposite, which we emphasize by writing  $\overline{R_4}$.

\begin{center}
\begin{tabular}{ccc}
\begin{tikzpicture}
\node(r1) at (0,0){$\boxed{R_1}$};
\node(s1) at (-1,-1){$\sigma_1$};
\node(s2) at (-1,-0.4){$\sigma_2$};
\node(sm) at (-1,1){$\sigma_m$};
\node(dots) at (-1,0.35){$\vdots$};
\node(n1) at (1,-1){$\eta_1$};
\node(n2) at (1,-0.4){$\eta_2$};
\node(nh) at (1,1){$\eta_h$};
\node(dots2) at (1,0.35){$\vdots$};
\draw(r1) -- (s1);
\draw(r1) -- (sm);
\draw(s1) -- (s2);
\draw(s2) -- (dots);
\draw(dots) -- (sm);
\draw(r1) -- (n1);
\draw(r1) -- (nh);
\draw(n1) -- (n2);
\draw(n2) -- (dots2);
\draw(dots2) -- (nh);
\end{tikzpicture}
\qquad
\qquad
&
\qquad
\begin{tikzpicture}
\node at (0,0){$\boxed{R_2}$};
\end{tikzpicture}
\qquad
&
\qquad
\qquad
\begin{tikzpicture}
\node(r1) at (0,0){$R_3$};
\node(s1) at (-1,-1){$\sigma_1$};
\node(s2) at (-1,-0.4){$\sigma_2$};
\node(sm) at (-1,1){$\sigma_m$};
\node(dots) at (-1,0.35){$\vdots$};
\node(n1) at (1,-1){$\eta_1$};
\node(n2) at (1,-0.4){$\eta_2$};
\node(nh) at (1,1){$\eta_h$};
\node(dots2) at (1,0.35){$\vdots$};
\node(r4) at (2,0){$\boxed{\overline{R_4}}$};
\draw(r1) -- (s1);
\draw(s1) -- (s2);
\draw(s2) -- (dots);
\draw(dots) -- (sm);
\draw(r4) -- (n1);
\draw(n1) -- (n2);
\draw(n2) -- (dots2);
\draw(dots2) -- (nh);
\draw(nh) -- (r1);
\draw(sm) [out = 30,in=90]  to (2.2,0.2);
\end{tikzpicture}\\
\highlight{$\calP_1$} & \highlight{$\calP_2$}&\highlight{$\calP_{\ell_1}$}
\end{tabular}
\end{center}

Let $A_1$ denote all pairs $(I_1,I_2) \in J(\calP_1) \times J(\calP_2)$ such that there is a switching position between $I_1 \vert_{R_1}$ and $I_2 = I_2\vert_{R_2}$. We define the map $\Phi$ on $A_1$ by sending $I_1 \vert_{R_1}$ to $R_3$ and $I_2$ to $R_4$ first and then swapping at the switching position. The $\sigma_i$ and $\eta_i$ will always be sent to their respective images in $\calP_{\ell_1}$. The set $\Phi(A_1)$ is all $I \in \calP_{\ell_1}$ such that there is a switching position between $I\vert_{R_3}$ and $I\vert_{R_4}$.

Let $A_2$ denote all pairs $(I_1,I_2) \in J(\calP_1) \times J(\calP_2)$ such that $I_2 = \calP_2$ and $I_1 \vert_{R_1} = \emptyset$ but $I_1 \neq \emptyset$. If $\eta_1 \in I_1$, $\Phi$ acts on $(I_1,I_2)$ by sending $R_2$ to $R_4$, and otherwise $\Phi$ acts by sending $R_2$ to $R_3$. The set $\Phi(A_2)$ consists of $I \in \calP_{\ell_1}$ such that ($I \vert_{R_3} = \emptyset$; $I\vert_{R_4} = R_4$; and, $\sigma_m \notin I$) or ($I \vert_{R_4} = \emptyset$; $I\vert_{R_3} = R_3$; and, $\eta_1 \notin I$).

Let $A_3$ denote all pairs $(I_1,I_2) \in J(\calP_1) \times J(\calP_2)$ such that $I_2 = \emptyset$ and $I_1 \vert_{R_1} = R_1$ but $I_1 \neq \calP_1$. If $\sigma_m \in I_1$, $\Phi$ acts on $(I_1,I_2)$ by sending $R_1$ to $R_4$ and otherwise $\Phi$ acts on $(I_1,I_2)$ by sending $R_1$ to $R_3$. The set $\Phi(A_2)$ consists of $I \in \calP_{\ell_1}$ such that ($I \vert_{R_3} = \emptyset$; $I\vert_{R_4} = R_4$; and, $\sigma_m \in I$) or ($I \vert_{R_4} = \emptyset$; $I\vert_{R_3} = R_3$; and, $\eta_1 \in I$).

We see $A_1 \cup A_2 \cup A_3$ is a partitioning of $J(\calP_1) \times J(\calP_2) \backslash \{(\calP_1,\emptyset),(\emptyset,\calP_2)\}$ and $\Phi(A_1) \cup \Phi(A_2) \cup \Phi(A_3)$ is a partitioning of $\calP_{\ell_1}$. 

Now, our final step is to show $\gb_1 + \gb_2 + \degx(\hat{Y}_R\hat{Y}_p\hat{Y}_q) = \gb_1 + \gb_2 + \degx(\hat{Y}_R) = 0$ where $\gb_1:= \gb_{\gamma_1^{(p,q)}}$ and $\gb_2:= \gb_{\gamma_2}$. The first equality follows immediately from Corollary \ref{cor:y-hat_spokes}. 

Let $\gb_R$ denote the contributions of $\calP_1[2,n_1-1]$ to $\gb_1$, which is equivalent to the contributions of $\calP_2[2,n_2-1]$ to $\gb_2$

 We compute\[
\gb_1 = -\eb_{\sigma_1} - \eb_{\eta_1} + \delta_{\calP_1(1) \succ \calP_1(2)} \eb_{\calP_1(1)} + \delta_{\calP_1(n_1) \succ \calP_1(n_1-1)} \eb_{\calP_1(n_1)}  + \gb_R,
 \]
 and
 \[
\gb_2 = \eb_{\sigma_m} + \eb_{\eta_h} - \delta_{\calP_1(1) \prec \calP_1(2)} \eb_{\calP_1(1)} -\delta_{\calP_1(n_1) \prec \calP_1(n_1-1)} \eb_{\calP_1(n_1)} + \gb_R.
 \]

Now, Lemma \ref{lem:y-hat} guarantees $\gb_1 + \gb_2 = - \degx(\hat{Y}_R)$. 
\end{proof}

We note that when $k_p=m$, $Y_{S^p} = Y_p$ and similarly for $k_q = h$.

\begin{prop}
\label{prop:two_punctures_1}
    If $\gamma_1^{(p)}$ and $\gamma_2^{(q)}$ are two singly-notched arcs which are incompatible at both $p$ and $q$ and such that $\gamma_1^0 \neq \gamma_2^0$, then $x_{\gamma_1^{(p)}}x_{\gamma_2^{(q)}} = Y_1x_{\ell_1} + Y_2 x_{\ell_2} + Y_3 x_{\ell_3} + Y_4 x_{\ell_4}$ where the tuple $(Y_1,Y_2,Y_3,Y_4)$ is given below. 

\scalebox{0.9}{
\bgroup
\def\arraystretch{1.25}
\begin{tabular}{|c|c|c|c|} \hline
   \cellcolor{gray!75} & \cellcolor{gray!25}$k_p \neq 0,m$ & \cellcolor{gray!25} $k_p = m$  & \cellcolor{gray!25} $k_p = 0$ \\\hline
 \cellcolor{gray!25}$k_q \neq 0,h$    & $(Y_{S^p}Y_{S^q_2},1,Y_{S^p},Y_{S^q_2})$  & $(Y_{p}Y_{S^q_2}Y_{R^p},1,Y_{p}Y_{R^p},Y_{S^q_2})$ & $(Y_{S^q_2},Y_{R^p},1,Y_{S^q_2}Y_{R^p})$ \\\hline
  \cellcolor{gray!25}$k_q = h$   &$(Y_{S^p}Y_{q}Y_{R^q},1,Y_{S^p},Y_{q}Y_{R^q})$ &$(Y_{p}Y_{q}Y_{R^p}Y_{R^q},1,Y_{p}Y_{R^p},Y_{q}Y_{R^q})$&$(Y_{q}Y_{R^q},Y_{R^p},1,Y_{q}Y_{R^p}Y_{R^q})$ \\\hline
 \cellcolor{gray!25} $k_q = 0$    & $(Y_{S^p},Y_{R^q},Y_{S^p}Y_{R^q},1)$ & $(Y_{p}Y_{R^p},Y_{R^q},Y_{p}Y_{R^p}Y_{R^q},1)$ & $(1,Y_{R^p}Y_{R^q},Y_{R^q},Y_{R^p})$\\\hline
\end{tabular} \egroup}

The $y$-monomials $Y_{R^p},Y_{R^q},$ and $Y_{S^p}$ are as in Proposition \ref{prop:TwiceIncDoublyNotched} and $Y_{S^q_2} = \prod_{j=k_q+1}^{h} y_{\eta_j}$. 

If $\gamma_1^{(p)}$ and $\gamma_2^{(q)}$ are two singly-notched arcs which are incompatible at both $p$ and $q$ and such that $\gamma_1^0 = \gamma_2^0$, then $x_{\gamma_1^{(p)}} x_{\gamma_2^{(q)}} =x_{\ell_1}+Y_{p}Y_R+Y_{q}Y_R$.
\end{prop}

Again, if $\gamma_1^{0} \neq \gamma_2^0$, Proposition \ref{prop:two_punctures_1} follows from Proposition \ref{prop:incompatible_puncture_singly_tagged}. The case  $\gamma_1^{0} = \gamma_2^0$ can be shown in a similar manner as this case in Proposition \ref{prop:TwiceIncDoublyNotched}.

Each entry in the table in Proposition \ref{prop:two_punctures_1} results from taking the corresponding entry in the table in Proposition \ref{prop:TwiceIncDoublyNotched}, swapping the entries $Y_1$ and $Y_3$ and the entries $Y_2$ and $Y_4$, and replacing $S_1^q$ with $S_2^q$ where necessary.

By combining Propositions \ref{prop:TwiceIncDoublyNotched} and \ref{prop:two_punctures_1} in the case where $\gamma_1^0 = \gamma_2^0$, we show a generalization of Theorem 12.9 from \cite{musiker2011positivity} for arcs which possibly have self-intersections. Recently, similar relations were studied in \cite{baur2024infinite} with the goal of comparing various frieze patterns arising from twice-punctured discs. 

\begin{cor}\label{cor:MSWThm12.9}
Let $\gamma$ be a (possibly generalized) plain arc incident to two punctures $p$ and $q$. Let $Y_R = \prod_{i=1}^{n_1} y_{\calP_1(i)}$. Then,\[
x_\gamma x_{\gamma^{(p,q)}} - x_{\gamma^{(p)}} x_{\gamma^{(q)}} = (1-Y_p) (1-Y_q) Y_R.
\]
\end{cor}

\section{Transverse Crossings}
\label{sec:transverse_crossings}

Here we consider two tagged arcs $\gamma_1$ and $\gamma_2$ whose interiors have at least one point of intersection. As discussed in Section \ref{subsec:CompareToSGC} and in \cite{snakegraphcalculus1}, there are several ways this intersection can appear. We refer to these as intersections of Type 0, Type 1, and Type 2, where an intersection of Type $i$ occurs if the intersection point occurs in the first or last triangle passed through by $i$ of the two arcs. An intersection of Type 0 is exactly the type of intersection which produces crossing overlaps in the posets $\calP_{\gamma_1}$ and $\calP_{\gamma_2}$.

\begin{center}
\begin{tabular}{ccc}
Type 0 & Type 1 & Type 2\\
\begin{tikzpicture}[scale = 1.3]
 \draw (0,0.8) -- (1,1.6) -- (4,1.6) -- (5,0.8) -- (4,0) -- (1,0) -- (0,0.8);
 \draw (1,0) -- (1,1.6);
 \draw (4,0) -- (4,1.6);
 \draw[red,thick] (0.5,1.5) -- (4.5,0.1);
 \draw[blue,thick] (0.5,0.1) -- (4.5,1.5);
\end{tikzpicture}
&
\begin{tikzpicture}[scale = 1.3]
 \draw[] (0,0) -- (2,0) -- (1,1.6) -- (0,0); 
 \draw[red, thick] (0,0) -- (2.5,1.2);
\draw[blue,thick] (0,1.2) -- (1.5,-0.5);
\end{tikzpicture}
&
\begin{tikzpicture}[scale = 1.3]
 \draw[] (0,0) -- (2,0) -- (1,1.6) -- (0,0); 
 \draw[red, thick] (0,0) -- (2.5,1.2);
\draw[blue,thick] (1,1.6) -- (1,-0.5);
\end{tikzpicture}
\end{tabular}
\end{center}

Our first three subsections deal with, in order, intersections of Type 0,1, and 2, with the assumption that none of the arcs involved are in $T$ nor any version of these arcs' notchings removed. The last subsection is devoted to an intersection between two arcs where the plain underlying version of one arc is in the triangulation.

Notice that this sorting of cases differs from the sorting in Section 8 of \cite{musiker2013bases}. For example, in Section \ref{subsec:Type0}, we work with Type 0 intersections and all possible options for notchings on the endpoints of the two arcs involved.

\subsection{Type 0}\label{subsec:Type0}

Consider two arcs, $\gamma_1$ and $\gamma_2$ not in $T$ and suppose that the posets $\calP_{\gamma_1}$ and $\calP_{\gamma_2}$ have a crossing overlap, so that $\gamma_1$ and $\gamma_2$ cross and near this point of intersection, the two arcs cross the same set of arcs from $T$.  The resolution consists of $\{\gamma_3,\gamma_4\} \cup \{\gamma_5,\gamma_6\}$ where $\gamma_3:= \gamma_1 \circ \gamma_2$, $\gamma_4:= \gamma_1^{-1} \circ \gamma_2^{-1}$, $\gamma_5:=\gamma_1\circ \gamma_2^{-1}$, and $\gamma_6:=\gamma_1^{-1}\circ\gamma_2$, where now the notation $\circ$ refers to switching between the two arcs at their point of intersection.

\begin{figure}[H]
\captionsetup{width=0.9\textwidth}
    \centering
    \includegraphics{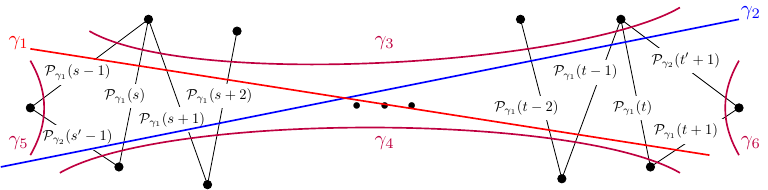}
    \caption{A type $0$ intersection in the case of  $s > 1, s'>1, t< n_1 = \vert \calP_{\gamma_1} \vert, t' < n_2 = \vert \calP_{\gamma_2}\vert$. }
    \label{type0}
\end{figure}

Let $\gamma_1$ and $\gamma_2$ be indexed so that, if their crossing overlaps are $R_1 = \{\calP_{\gamma_1}(s),\ldots,\calP_{\gamma_1}(t)\}$ and $R_2 = \{\calP_{\gamma_2}(s'),\ldots,\calP_{\gamma_2}(t')\}$, then $R_1$ is on the top of $\calP_{\gamma_1}$ and $R_2$ is on the bottom of $\calP_{\gamma_2}$.
Then, $\calP_{\gamma_3} = \calP_{\gamma_1}[\leq t] \cup \calP_{\gamma_2}[>t']$ where from the geometry of the triangulation we have $\calP_{\gamma_1}(t) \prec \calP_{\gamma_2}(t'+1)$. Similarly, $\calP_{\gamma_4} = \calP_{\gamma_2}[\leq t'] \cup \calP_{\gamma_1}[>t]$ where we also know $\calP_{\gamma_1}(t+1) \succ \calP_{\gamma_2}(t)$. 

There will be several cases for $\calP_5$ based on whether $s = 1$ or $s' = 1$ and similarly whether $t = n_1$ or $t' = n_2$. Recall we cannot have $s = s' = 1$ nor $t = n_1$ and $t' = n_2$ simultaneously. Let $\calP_{\gamma_5}' = \calP_{\gamma_1}[<s] \cup \calP_{\gamma_2}[<s']$ where we know from $T$ that we also have $\calP_{\gamma_1}(s-1) \succ \calP_{\gamma_2}(s'-1)$. A chronological ordering on $\calP_{\gamma_5}'$ would have to reverse the ordering on either the component from $\calP_{\gamma_1}$ or the component from $\calP_{\gamma_2}$. Throughout, let $\alpha$ be the counterclockwise neighbor of $\calP_{\gamma_2}(1)$ and let $\beta$ be the clockwise neighbor of $\calP_{\gamma_1}(1)$.

\begin{enumerate}
    \item If $s > 1$ and $s' > 1$, then $\calP_{\gamma_5} = \calP_{\gamma_5}'$.  
    \item If $s = 1$, and $s(\gamma_1)$ is plain, then we have $\calP_{\gamma_5} = \calP_{\gamma_5}' \backslash \calP_{\gamma_2}^{\preceq (s'-1)}$. There are two special cases where we introduce a decoration.
    \begin{enumerate}
        \item If this set is empty, implying that $s(\gamma_2)$ is plain, then we include decoration $\alpha$. 
        \item If this set consists entirely of elements from a loop in $\calP_{\gamma_2}$, implying that $s(\gamma_2)$ is notched, then we include decoration $\alpha^-$.
    \end{enumerate} 
    \item If $s = 1$ and $s(\gamma_1)$ is notched, then we again have $\calP_{\gamma_5} = \calP_{\gamma_5}' \backslash \calP_{\gamma_2}^{\preceq (s'-1)}$. We explain which relations will connect the loop at the beginning of $\calP_{\gamma_1}$ with the rest of $\calP_{\gamma_5}$. 
    \begin{enumerate}
        \item If $\calP_{\gamma_2}[<s'] \backslash \calP_{\gamma_2}^{\preceq (s'-1)} = \emptyset$, then $\calP_{\gamma_5}$ will be the poset for $\alpha^{(s(\gamma_1))}$, and thus must have  decoration $\alpha^-$.
        \item If $\calP_{\gamma_2}[<s'] \backslash \calP_{\gamma_2}^{\preceq (s'-1)}$ consists only of elements of a loop, then $\calP_5$ will be the poset for $\alpha^{(s(\gamma_1), s(\gamma_2))}$ and thus we do not include a decoration. 
        \item If $\calP_{\gamma_2}[<s'] \backslash \calP_{\gamma_2}^{\preceq (s'-1)}$ contains elements not in a loop, then $\calP_{\gamma_5}$ will be the result of attaching the elements of the loop from $\calP_{\gamma_1}$ onto the element of this set with largest index in the chronological ordering. 
    \end{enumerate}
    \item If $s' = 1$ instead, then we have 
    parallel cases, using instead  $\calP_{\gamma_5}' \backslash \calP_{\gamma_1}^{\succeq (s-1)}$ and $\beta$.
\end{enumerate}

If $t = n_1$ or $t = n_2$, we perform an identical operation, beginning with $\calP_{\gamma_6}' = \calP_{\gamma_1}[>t] \cup \calP_{\gamma_2}[>t']$. If $t = n_1$, then we remove $\calP_{\gamma_2}^{\preceq (t'+1)}$ and if $t' = n_2$, then we remove $\calP_{\gamma_1}^{\succeq (t+1)}$; in each case, we have the same subcases. We highlight that in Case iii of the following proof we introduce a method of indexing a chain of spokes which will be used repeatedly in later proofs. This indexing is depicted in Figure \ref{fig:Indexing}.

\begin{prop}\label{prop:GraftingType0}
Let $\gamma_1$ and $\gamma_2$ be two arcs with a Type 0 crossing, indexed so that the crossing overlap $\calP_{\gamma_1}[s,t]$ is on top and $\calP_{\gamma_2}[s',t']$ is on bottom. Let the resolution be $\{\gamma_3, \gamma_4\} \cup \{\gamma_5,\gamma_6\}$. Then,\[
x_{\gamma_1} x_{\gamma_2} = x_{\gamma_3}x_{\gamma_4} + Y_R Y_{S_s} Y_{S_t} x_{\gamma_5} x_{\gamma_6}
\]
where 
\[
Y_R = \prod_{i=s}^t y_{\calP_{\gamma_1}(i)},
\]

\[
Y_{S_s} = \begin{cases} 
\displaystyle \prod_{\substack{\calP_{\gamma_1}(i) \succeq \calP_{\gamma_1}(s-1) \\ i \leq s-1}} y_{\calP_{\gamma_1}(i)} & s' = 1 \text{ and } s(\gamma_2) \text{ is notched} \\ 
\displaystyle \prod_{\substack{\calP_{\gamma_2}(i) \preceq \calP_{\gamma_2}(s'-1) \\ i \leq s'-1}} y_{\calP_{\gamma_2}(i)} & s = 1 \text{ and } s(\gamma_1) \text{ is plain} \\ 
1 & \text{ otherwise,}\\ 
\end{cases}
\]

and 

\[
Y_{S_t} = \begin{cases} 
\displaystyle \prod_{\substack{\calP_{\gamma_1}(i) \succeq \calP_{\gamma_1}(t+1) \\ i \geq t+1}} y_{\calP_{\gamma_1}(i)} & t' = n_2 \text{ and } t(\gamma_2) \text{ is notched} \\ 
\displaystyle \prod_{\substack{\calP_{\gamma_2}(i) \preceq \calP_{\gamma_2}(t'+1) \\ i \geq t'+1}} y_{\calP_{\gamma_2}(i)} & t = n_1 \text{ and } t(\gamma_1) \text{ is plain} \\ 
1 & \text{ otherwise.}\\ 
\end{cases}
\]
\end{prop}

\begin{proof}

Since the definition of the posets $\calP_{\gamma_5}$ and $\calP_{\gamma_6}$ vary based on the notchings of the endpoints of $\gamma_1$ and $\gamma_2$ and whether these endpoints are near the overlap, we will break our proof into cases. 

We set $\calP_i:= \calP_{\gamma_i}$ throughout the proof. Let $R_3$ denote the image of $R_1$ in $\calP_3$ and define $R_4$ similarly. 

\textbf{Case i) $ s>1$; $s'>1$; $t<n_1$; and, $t' < n_2$.}  

The notching of each arc does not affect this case, so here the result follows from the plain case as in \cite{snakegraphcalculus1,musiker2013}. However, we will write the details in our setting of order ideals as they will be used in later cases where the notching does have an effect. We proceed with the usual strategy, providing a partition of $J(\calP_1) \times J(\calP_2)$ and defining $\Phi: J(\calP_1) \times J(\calP_2) \to (J(\calP_3) \times J(\calP_4)) \cup (J(\calP_5) \times J(\calP_6))$.

\begin{tabular}{cccc}
  \begin{tikzpicture}[scale = 0.8]
\node(s1) at (0,0){$\calP_1(s-1)$};
\node(R1) at (1,1){$\boxed{R_1}$};
\node(t1) at (2,0){$\calP_1(t+1)$};
\draw(s1) -- (R1);
\draw(R1) -- (t1);
  \end{tikzpicture}   
  &  \begin{tikzpicture}[scale = 0.8]
\node(s1) at (0,0){$\calP_2(s'-1)$};
\node(R1) at (1,-1){$\boxed{R_2}$};
\node(t1) at (2,0){$\calP_2(t'+1)$};
\draw(s1) -- (R1);
\draw(R1) -- (t1);
  \end{tikzpicture} 
   & \begin{tikzpicture}[scale = 0.8]
\node(s1) at (0,0){$\calP_1(s-1)$};
\node(R1) at (1,1){$\boxed{R_3}$};
\node(t1) at (2,2){$\calP_2(t'+1)$};
\draw(s1) -- (R1);
\draw(R1) -- (t1);
  \end{tikzpicture} 
    &\begin{tikzpicture}[scale = 0.8]
\node(s1) at (0,0){$\calP_2(s'-1)$};
\node(R1) at (1,-1){$\boxed{R_4}$};
\node(t1) at (2,-2){$\calP_1(t+1)$};
\draw(s1) -- (R1);
\draw(R1) -- (t1);
  \end{tikzpicture} \\
\end{tabular}

\begin{enumerate}
\item Let $A_1$ denote the set of pairs $(I_1,I_2) \in J(\calP_1) \times J(\calP_2)$ such that there is a switching position between $I_1 \vert_{R_1}$ and $I_2 \vert_{R_2}$. Given such a pair, we define $\Phi(I_1,I_2)$ to send $R_1$ to $R_3$ and $R_2$ to $R_4$ until the first switching position and then swap. The image of other elements is clear. The set $\Phi(A_1)$ would be all pairs $(I_3,I_4)$ such that there is a switching position between $R_3$ and $R_4$.
\item Let $A_2$ denote the set of $(I_1,I_2) \in J(\calP_1) \times J(\calP_2)$ such that $R_1 \subseteq I_1$ and $R_2 \cap I_2 = \emptyset$. If $(I_1,I_2) \in A_2$, we define $\Phi(I_1,I_2)$ to send $R_1$ to $R_3$. The image of $A_2$ is all pairs $(I_3,I_4)$ such that $R_3 \subseteq I_3$; $R_4 \cap I_4 = \emptyset$; $\calP_1(t+1) \in I_4$; and, $\calP_2(t'+1) \notin I_3$.
\item Let $A_3$ denote the set of $(I_1,I_2)$ such that  $I_1 \cap R_1 = \emptyset$; $R_2 \subseteq I_2$; $\calP_1(t+1) \in I_1$; and, $\calP_2(t'+1) \notin I_2$. Given $(I_1,I_2) \in A_3$, the map $\Phi$ acts by sending $R_2$ to $R_4$. The set $\Phi(A_3)$ is all pairs $(I_3,I_4)$ such that $R_3 \cap I_3 = \emptyset$ and  $R_4 \subseteq I_4$. 
\item Let $A_4$  denote the set of $(I_1,I_2)$ such that $I_1 \cap R_1 = \emptyset$; $R_2 \subseteq I_2$; $\calP_1(t+1) \in I_1$ only if $\calP_2(t'+1) \in I_2$; $\calP_1(s-1)\in I_1$; and, $\calP_2(s'-1) \notin I_2$. Given $(I_1,I_2) \in A_3$, the map $\Phi$ acts by sending $R_2$ to $R_3$. The set $\Phi(A_4)$ is all  all pairs $(I_3,I_4)$ such that $R_3 \subseteq I_3$; $R_4 \cap I_4 = \emptyset$; and, $\calP_1(t+1) \in I_4$ only if $\calP_2(t'+1) \in I_3$. 
\item Let $B$ denote the set of $(I_1,I_2)$ such that $I_1 \cap R_1 = \emptyset$; $R_2 \subseteq I_2$; $\calP_1(s-1) \in I_1$ only if $\calP_2(s'-1) \in I_2$; and,  $\calP_1(t+1) \in I_1$ only if $\calP_2(t'+1) \in I_2$. Given $(I_1,I_2) \in B$, the set $I_1 \cup (I_2 \backslash R_2)$ can be naturally partitioned into an element of $J(\calP_5) \times J(\calP_6)$, and this map is bijective. 
\end{enumerate}

We omit the proof for $\gb$-vectors as this case is already known.

\textbf{Case ii) ($s=1$; $s(\gamma_1)$ is plain; $t<n_1$; and, $t' < n_2$) or
($s'=1$; $s(\gamma_2)$ is plain; $t<n_1$; and, $t' < n_2$).}

The remaining notching of the arcs does not affect the combinatorics for these cases, so these follow from previously known results, such as from \cite{snakegraphcalculus1,musiker2013}. 

\textbf{Case iii) $s=1$; $s(\gamma_1)$ is notched; $t<n_1$; and, $t' < n_2$.}

Here, we introduce some notation which will be used repeatedly in later proofs. Let $\sigma_1,\ldots,\sigma_m$ denote the spokes at $s(\gamma_1)$, indexed in counterclockwise order,  such that the first triangle $\gamma_1$ passes through is bound by $\sigma_1$ and $\sigma_m$. Then, before $\gamma_2$ starts crossing the arcs recorded in the region $R_2$, it will cross some of these spokes; in particular, $\calP_2(s'-1) = \sigma_m$. $\gamma_2$ may wind around $s(\gamma_1)$ so that it could cross some spokes multiple times. This creates a chain, which we index $\sigma_m^{(w)} \succ \sigma_{m-1}^{(w)} \succ \cdots \succ \sigma_1^{(w)} \succ \sigma_m^{(w-1)} \succ \cdots \succ \sigma_k^{(1)}$ for some $w \geq 1$ and $k \in [m]$. We assume that this is the maximal, saturated chain of these spokes so that if $\gamma_2$ crosses an arc before this, it is not a spoke at $s(\gamma)$. Let $\mu$ denote the element before (with respect to the chronological ordering) $\sigma_k^{(1)}$ in $\calP_2$, if it exists. We will assume we are in situation 3.b or 3.c in the description before the statement; situation 3.a is simpler so we do not provide the details.

We draw $\calP_1, \calP_2, \calP_5,$ and $\calP_6$ in Figure \ref{fig:Indexing}. The posets $\calP_3$ and $\calP_4$ come from taking $\calP_1$ and $\calP_2$ respectively and swapping $\calP_1(t+1)$ and $\calP_2(t'+1)$. 

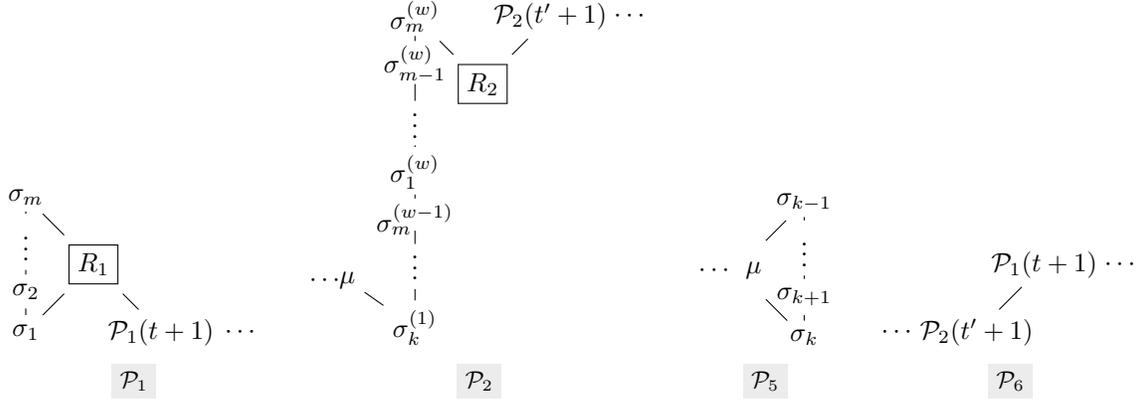
\begin{figure}
\captionsetup{width=0.9\textwidth}
    \centering
\begin{tabular}{cccc}
\begin{tikzpicture}[scale = 0.9]
\node(r1) at (0,0){$\boxed{R_1}$};
\node(s1) at (-1,-1){$\sigma_1$};
\node(s2) at (-1,-0.4){$\sigma_2$};
\node(sm) at (-1,1){$\sigma_m$};
\node(pt1) at (1,-1){$\calP_1(t+1)$};
\node(dots) at (-1,0.35){$\vdots$};
\node(dots2) at (2.2,-1){$\cdots$};
\draw(r1) -- (s1);
\draw(r1) -- (sm);
\draw(r1) -- (pt1);
\draw(s1) -- (s2);
\draw(s2) -- (dots);
\draw(dots) -- (sm);
\end{tikzpicture}
&
\begin{tikzpicture}[scale = 0.9]
\node(dots3) at (2.2,1.2){$\cdots$};
\node(pt1) at (1,1.2){$\calP_2(t'+1)$};
\node(r2) at (0,0.2){$\boxed{R_2}$};
\node(sm) at (-1,1.2){$\sigma_m^{(w)}$};
\node(sm1) at (-1,0.5){$\sigma_{m-1}^{(w)}$};
\node(dots1) at (-1,-0.3){$\vdots$};
\node(s1w) at (-1,-1.1){$\sigma_1^{(w)}$};
\node(smw1) at (-1,-1.8){$\sigma_m^{(w-1)}$};
\node(dots2) at (-1,-2.4){$\vdots$};
\node(sk1) at (-1,-3.4){$\sigma_k^{(1)}$};
\node(mu) at (-2,-2.7){$\mu$};
\node(dots4) at (-2.3,-2.7){$\cdots$};
\draw(pt1) -- (r2);
\draw(r2) -- (sm);
\draw(sm) -- (sm1);
\draw(-1,-0.05) -- (-1,0.2);
\draw(dots1) -- (s1w);
\draw(s1w) -- (smw1);
\draw(smw1) -- (dots2);
\draw(dots2) -- (sk1);
\draw(sk1) -- (mu);
\end{tikzpicture}
&
\begin{tikzpicture}[scale = 0.9]
\node(dots2) at (-0.3,0){$\cdots$};
\node(mu) at (0.25,0){$\mu$};
\node(s1) at (1,-1){$\sigma_k$};
\node(s2) at (1,-0.4){$\sigma_{k+1}$};
\node(sm) at (1,1){$\sigma_{k-1}$};
\node(dots) at (1,0.35){$\vdots$};
\draw(r1) -- (s1);
\draw(r1) -- (sm);
\draw(s1) -- (s2);
\draw(s2) -- (dots);
\draw(dots) -- (sm);
\end{tikzpicture}
&
\begin{tikzpicture}[scale = 0.9]
\node(dots1) at (-1.15,0){$\cdots$};
\node at (2.15,1){$\cdots$};
\node(pt1) at (0,0){$\calP_2(t'+1)$};
\node(ptt1) at (1,1){$\calP_1(t+1)$};
\draw(pt1) -- (ptt1);
\end{tikzpicture}\\
\highlight{$\calP_1$} & \highlight{$\calP_2$} & \highlight{$\calP_5$} & \highlight{$\calP_6$}
\end{tabular}
    \caption{These posets are used in the proof of Proposition \ref{prop:GraftingType0}. The posets $\calP_1$ and $\calP_2$ showcase an indexing system which we use frequently in later sections.}
    \label{fig:Indexing}
\end{figure}

The sets $A_1,A_2,$ and $A_3$ and the definition of $\Phi$ restricted to each is identical to the general case. Thus, we have that the complement of $\Phi(A_1 \cup A_2 \cup A_3)$ in $J(\calP_3) \times J(\calP_4)$ consists of pairs $(I_3,I_4)$ such that $R_3 \subseteq I_3, R_4 \cap I_4 = \emptyset$, and $\calP_1(t+1) \in I_4$ only if $\calP_2(t'+1) \in I_3$. We partition the remainder of $J(\calP_1) \times J(\calP_2)$ below and complete the definition of $\Phi$.

Given an arbitrary element $(I_1,I_2) \in J(\calP_1)\times J(\calP_2)$, let $\sigma_x$ denote the maximal element of $I_1$ in the loop and let $\sigma_y^{(z)}$ denote the maximal element of $I_2$ on the chain of spokes.  It is of course possible one or both of these do not exist. We will use these symbols in the same way when referring to $(I_3,I_4) \in J(\calP_3) \times J(\calP_4)$.  

\begin{enumerate}
\item[4.] Let $A_4$ consist of pairs $(I_1,I_2)$ such that $R_1 \cap I_1 = \emptyset$; $R_2 \subseteq I_2$; $\calP_1(t+1) \in I_1$ only if $\calP_2(t'+1) \in I_2$; $\calP_1(s-1) = \sigma_1 \in I_1$; and, $\sigma_y^{(z)} \prec \sigma_m^{(1)}$. Given $(I_1,I_2) \in A_4$, we define $\Phi(I_1,I_2)$ as the map which sends $R_2$ to $R_3$, and the elements in the loop and chain of spokes to the respective elements in $I_3$ and $I_4$ respectively. The set $\Phi(A_4)$ is all pairs $(I_3,I_4) \in J(\calP_3) \times J(\calP_4)$ such that $R_3 \subseteq I_3$; $R_4 \cap I_4 = \emptyset$; $\calP_1(t+1) \in I_4$ only if $\calP_2(t'+1) \in I_3$; $\sigma_y^{(z)} \prec \sigma_m^{(1)}$; and, $x < m$. 
\item[5.] Let $A_5$ consist of pairs $(I_1,I_2)$ such that $R_1 \cap I_1 = \emptyset$; $R_2 \subseteq I_2$; $\calP_1(t+1) \in I_1$ only if $\calP_2(t'+1) \in I_2$; $\sigma_y^{(z)} = \sigma_m^{(1)}$; and, $x \geq k-1$ with strict inequality if $\mu \in I_2$. Given $(I_1,I_2) \in A_5$, we define $\Phi(I_1,I_2)$ as the map which sends $R_2$ to $R_3$ and which sends $(\langle \sigma_x \rangle, \langle \sigma_m^{(1)} \rangle)$ to $(\langle \sigma_m \rangle, \langle \sigma_x \rangle)$. The set $\Phi(A_5)$ is all pairs $(I_3,I_4) \in J(\calP_3) \times J(\calP_4)$ such that $R_3 \subseteq I_3$; $R_4 \cap I_4 = \emptyset$; $\calP_1(t+1) \in I_4$ only if $\calP_2(t'+1) \in I_3$; $\sigma_y^{(z)} \prec \sigma_m^{(1)}$; and, $x = m$. 
\item[6.] Let $A_6$ consist of pairs $(I_1,I_2)$ such that $R_1 \cap I_1 = \emptyset$; $R_2 \subseteq I_2$; $\calP_1(t+1) \in I_1$ only if $\calP_2(t'+1) \in I_2$; and, $z > 1$. Given $(I_1,I_2) \in A_6$, we define $\Phi(I_1,I_2)$ as the map which sends $R_2$ to $R_3$ and which sends $(\langle \sigma_x \rangle, \langle \sigma_m^{(1)} \rangle)$ to $(\langle \sigma_y \rangle, \langle \sigma_x^{(z)} \rangle)$ where we set $\sigma_x^{(z)} := \sigma_m^{(z-1)}$ if $\sigma_x$ does not exist. The set $\Phi(A_6)$ is all pairs $(I_3,I_4) \in J(\calP_3) \times J(\calP_4)$ such that $R_3 \subseteq I_3$; $R_4 \cap I_4 = \emptyset$; $\calP_1(t+1) \in I_4$ only if $\calP_2(t'+1) \in I_3$; and, $\sigma_y^{(z)} \succeq  \sigma_m^{(1)}$. 
\item[7.] Let $B$ consist of pairs $(I_1,I_2)$ such that $R_1 \cap I_1 = \emptyset$; $R_2 \subseteq I_2$; $\calP_1(t+1) \in I_1$ only if $\calP_2(t'+1) \in I_2$; $x \leq k-1$; $z = 1$; $\sigma_1 \in I_1$ only if $\sigma_m^{(1)} \in I_2$; and, $\sigma_{k-1} \in I_1$ only if $\mu \in I_2$. If $\mu$ does not exist, we ignore the last criterion. Given $(I_1,I_2) \in B$, there is a natural way to break $I_1 \cup (I_2 \backslash R_2)$ into a pair of elements of $J(\calP_5) \times J(\calP_6)$, and this describes a bijective map. 
\end{enumerate}

For our computations relating to $\gb$-vectors, we assume $\mu$ exists( equivalently, $\gamma_5^0 \notin T$). If this was not the case, the proof would be similar but would use the idea of a decorated poset. 

We begin by listing $\gb_1:=\gb_{\gamma_1^{(s(\gamma))}}$ and $\gb_2:=\gb_{\gamma_2}$. Let $\gb_1'$ refer to the contributions of $\calP_1[>t+1]$ to $\gb_1$, $\gb_2'$ refer to the contribution of elements in  $\calP_2$ before $\sigma_k^{(1)}$ to $\gb_2$, $\gb_2''$ refer to the contribution of $\calP_2[>t'+1]$ to $\gb_2$, and $\gb_R$ refer to the contribution of elements in $\calP_1[s+1,t-1]$ to $\gb_1$ (equivalently of elements in $\calP_2[s'+1,t'-1]$ to $\gb_2$). Then, we have 
\begin{align*}
\gb_1 = -\eb_{\sigma_1} + \delta_{\calP_1(1) \succ \calP_1(2)} \eb_{\calP_1(1)} + \delta_{\calP_1(t) \succ \calP_1(t-1)} \eb_{\calP_1(t)} + (\delta_{\calP_1(t+1) \succ \calP_1(t+2)} - 1) \eb_{\calP_1(t+1)} + \gb_1' + \gb_R,
\end{align*}
and 
\begin{align*}
\gb_2 &= -\eb_{\sigma_k} + \eb_{\sigma_m} - \delta_{\calP_2(s') \prec \calP_2(s'+1)} \eb_{\calP_2(s')} - \delta_{\calP_2(t') \prec \calP_2(t'-1)} \eb_{\calP_2(t')} \\
&+ \delta_{\calP_2(t'+1) \succ \calP_2(t'+2)} \eb_{\calP_2(t'+1)} + \gb_2' + \gb_2'' + \gb_R.
\end{align*}
We require $\calP_1(t+2)$ to exist in order for ``$\calP_1(t+1) \succ \calP_1(t+2)$'' to be true. 

Setting $\gb_i:= \gb_{\gamma_i}$ for remaining values of $i$, one can observe that \[
\gb_3 = \gb_1 - \eb_{\calP_1(t)} - (\delta_{\calP_1(t+1) \succ \calP_1(t+2)} - 1) \eb_{\calP_1(t+1)} + \delta_{\calP_2(t'+1) \succ \calP_2(t'+2)} \eb_{\calP_2(t'+1)} -\gb_1' + \gb_2'' \]
and similarly \[
\gb_4 = \gb_2 + \eb_{\calP_2(t')} - \delta_{\calP_2(t'+1) \succ \calP_2(t'+2)} \eb_{\calP_2(t'+1)} + (\delta_{\calP_1(t+1) \succ \calP_1(t+2)} - 1) \eb_{\calP_1(t+1)} - \gb_2'' + \gb_1'
\]
and since $\calP_1(t)$ and $\calP_2(t')$ correspond to the same arc in $T$, we conclude $\gb_1 + \gb_2 = \gb_3 + \gb_4$. 

Now, we have $\gb_5 = \gb_2' - \eb_{\sigma_k}$ and $\gb_6 = (\delta_{\calP_2(t'+1) \succ \calP_2(t'+2)} -1 )\eb_{\calP_2(t'+1)} + \delta_{\calP_1(t+1) \succ \calP_1(t+2)} \eb_{\calP_1(t+1)} + \gb_1' + \gb_2''$. By Lemma \ref{lem:y-hat}, we have \begin{align*}
\degx(\hat{Y}_R) = \eb_{\sigma_1} - \eb_{\sigma_m} + \eb_{\calP_1(t+1)} - \eb_{\calP_2(t'+1)} + (\delta_{\calP_1(1) \prec \calP_1(2)} - \delta_{\calP_1(1) \succ \calP_1(2)}) \eb_{\calP_1(1)} \\+ (\delta_{\calP_1(t) \prec \calP_1(t-1)} - \delta_{\calP_1(t) \succ \calP_1(t-1)}) \eb_{\calP_1(t)}+ \gb_R,
\end{align*}

and after noting that $\calP_1(1)$ and $\calP_2(s')$ also correspond to the same arc in $T$, we conclude $\gb_1 + \gb_2 + \degx(\hat{Y}_R)= \gb_5 + \gb_6$. 

\textbf{Case iv) $s' = 1$; $s(\gamma_2)$ is notched; $t < n_1$; and, $t' < n_2$.}

We use a similar indexing as in Case iii. Let the spokes at $s(\gamma_2)$ again be $\sigma_1,\ldots,\sigma_m$ with the same convention. Here, $\calP_1(s-1)$ corresponds to $\sigma_1$. The chain of spokes immediately before $R_1$ in $\calP_1$ is of the form $\sigma_1^{(1)} \prec \sigma_2^{(1)} \prec \cdots \prec \sigma_m^{(1)} \prec \sigma_1^{(2)} \prec \cdots \prec \sigma_k^{(w)}$ for some $w \geq 1$ and $k \in [m]$. Let $\rho \in T$ correspond to the element immediately before $\sigma_k^{(w)}$ in $\calP_1$, if it exists.  As in previous cases, we let the restriction of $I_1$ and $I_2$ to the chain of spokes and loop in $\calP_1$ and $\calP_2$ be $\langle \sigma_y^{(z)} \rangle$ and $\langle \sigma_x \rangle$ respectively. If $I_1$ is not supported on the chain of spokes, it will be convenient to set $\sigma_y^{(z)} := \sigma_m^{(0)}$.

We again use the same sets $A_1,A_2$, and $A_3$ and the definition of $\Phi$ as in the previous case. We will define 4 more sets $A_4,A_5,A_6$, and $B$. We define these such that $A_4 \cup A_5 \cup A_6 \cup B$ consists of all sets such that  $R_1 \cap I_1 = \emptyset$; $R_2 \subseteq I_2$; and, $\calP_1(t+1) \in I_1$ only if $\calP_2(t'+1) \in I_2$. For all three new sets, our bijection will send $R_2$ to $R_3$ and will send elements outside $R_i$, the loop, and the chain of spokes to their corresponding elements in $\calP_3$ or $\calP_4$.

\begin{enumerate}
    \item[4.] Let $A_4$ consist of pairs satisfying the above condition and such that $y \neq m$ and if $\sigma_y^{(z)} \succ \sigma_m^{(w-1)}$, then $\sigma_x \preceq \sigma_k$ with equality only if $\rho \in I_1$. Given $(I_1,I_2) \in A_4$, $\Phi$ sends $(\langle \sigma_y^{(z)} \rangle,\langle \sigma_x \rangle)$ to $(\langle \sigma_x^{(z)} \rangle, \langle \sigma_y \rangle)$. As before, we interpret $\langle \sigma_x^{(z)} \rangle = \langle \sigma_m^{(z-1)} \rangle$ if $\sigma_x$ does not exist. The set $\Phi(A_4)$ is all pairs $(I_3,I_4)$ such that $R_3 \subseteq I_3$; $R_4 \cap I_4 = \emptyset$; $\calP_1(t+1) \in I_4$ only if $\calP_2(t'+1) \in I_3$; $\sigma_1 \in I_4$; and, if $z = w$, then $x \leq k$, with equality only if $\rho \in I_3$. 
    \item[5.] Let $A_5$ consist of pairs satisfying the above condition and such that $y = m$ (including $\sigma_m^{(0)}$) and if $z = w$, then $x \leq k$ with equality only if $\rho \in I_1$. Given $(I_1,I_2) \in A_5$, $\Phi$ sends $(\langle \sigma_m^{(z)} \rangle,\langle \sigma_x \rangle)$ to $(\langle \sigma_x^{(z+1)} \rangle, \emptyset)$. The set $\Phi(A_5)$ is all pairs $(I_3,I_4)$ such that $R_3 \subseteq I_3$; $R_4 \cap I_4 = \emptyset$; $\calP_1(t+1) \in I_4$ only if $\calP_2(t'+1) \in I_3$; and, $\sigma_1 \notin I_4$. 
    \item[6.] Let $A_6$ consist of pairs satisfying the above condition and such that $z = w$ and $x \geq k$ with equality only if $\rho \notin I_3$. Given $(I_1,I_2) \in A_6$, $\Phi$ sends $(\langle \sigma_y^{(z)} \rangle,\langle \sigma_x \rangle)$ to $(\langle \sigma_y^{(z)} \rangle,\langle \sigma_x \rangle)$. The set $\Phi(A_6)$ is all pairs $(I_3,I_4)$ such that $R_3 \subseteq I_3$; $R_4 \cap I_4 = \emptyset$; $\calP_1(t+1) \in I_4$ only if $\calP_2(t'+1) \in I_3$; $z = w$; and, $x \geq k$ with equality only if $\rho \notin I_3$.
    \item[7.] Let $B$ be the set of $(I_1,I_2)$ such that $ \sigma_y^{(z)} \succeq \sigma_m^{(w-1)}$; $\sigma_x \succeq \sigma_k$; $\calP_1(t+1) \in I_1$ only if $\calP_2(t'+1) \in I_2$; $\rho \in I_1$ only if $\sigma_{k+1} \in I_2$; and, $\sigma_1^{(w)} \in I_1$ only if $\sigma_m \in I_2$. We interpret $\langle \sigma_m^{(0)}\rangle = \emptyset$, and if $k = m$, the condition ``$\rho \in I_1$ only if $\sigma_{k+1} \in I_2$'' is replaced with ``$\rho \in I_1$ only if $\sigma_1^{(w)} \in I_1$''. Given $(I_1,I_2) \in B$, the set $(I_1 \backslash \langle \sigma_m^{(w-1)} \rangle) \cup (I_2 \backslash (\langle \sigma_k \rangle \cup R_2))$ can naturally be partitioned into elements of $J(\calP_5) \times J(\calP_6)$, and this forms a bijective map. 
\end{enumerate}

We turn to the $\gb$-vectors, focusing on the case where $\rho$ exists and $\sigma_k^{(w)} \neq \sigma_1^{(1)}$; one can make small adjustments if one or both of these statements is/are not true.  Let our specific element $\rho$ from be poset element $ \calP_1(b)$. Let $\gb_1'$ concern the contributions of $\calP_1[<b]$ to $\gb_1$, and define $\gb_1'',\gb_2''$ and $\gb_R$ as in the previous case. In particular, we do not have a term $\gb_2'$ as there is no portion of $\calP_2$ before the region we are focusing on. We have \begin{align*}
\gb_1 = (\delta_{\calP_1(b) \succ \calP_1(b-1)}-1) \eb_{\calP_1(b)} + \eb_{\sigma_k} - \eb_{\sigma_1} + \delta_{\calP_1(s) \succ \calP_1(s+1)} \eb_{\calP_1(s)}  + \delta_{\calP_1(t) \succ \calP_1(t-1)} \eb_{\calP_1(t)} \\ + (\delta_{\calP_1(t+1) \succ \calP_1(t+2)} - 1) \eb_{\calP_1(t+1)} + \gb_1' + \gb_1''+ \gb_R
\end{align*}
and 
\[
\gb_2 = -\eb_{\sigma_1} + \delta_{\calP_2(1) \succ \calP_2(2)} \eb_{\calP_2(1)}   - \delta_{\calP_2(t') \prec \calP_2(t'-1)} \eb_{\calP_2(t')} + \delta_{\calP_2(t'+1) \succ \calP_2(t'+2)}\eb_{\calP_2(t'+1)} + \gb_2'' + \gb_R.
\]

Showing $\gb_1 + \gb_2 = \gb_3 + \gb_4$ is akin to the previous case. We turn to comparing $\gb_1 + \gb_2$ with $\gb_5 + \gb_6$. Recall the monomial $Y_{S_s}$ is not just 1 in this case, so we again use Lemma \ref{lem:y-hat} to compute its $x$-degree vector, \begin{align*}
\degx(\hat{Y}_R\hat{Y}_{S_s}) &= - \eb_{\sigma_{k+1}} + \eb_{\calP_1(b)} + \eb_{\calP_1(t+1)} - \eb_{\calP_2(t'+1)} - \eb_{\sigma_k} + 2\eb_{\sigma_1} -2\delta_{\calP_1(s) \succ \calP_1(s+1)} \eb_{\calP_1(s)} \\
&+ (\delta_{\calP_1(t) \prec \calP_1(t-1)} - \delta_{\calP_1(t) \succ \calP_1(t-1)}) \eb_{\calP_1(t)}-2 \gb_R.
\end{align*}

Now, similar to the previous case, we have $\gb_5 = -\eb_{\sigma_{k+1}} + \delta_{\calP_1(b) \succ \calP_1(b-1)} \eb_{\calP_1(b)} + \gb_1'$, and $\gb_6$ is identical to the previous case. We can conclude $\gb_1 + \gb_2 + \degx(\hat{Y}_R\hat{Y}_{S_s}) = \gb_5 + \gb_6$. 

\textbf{All other cases)}

If we are in the case where either $t = n_1$ or $t' = n_2$ and $s,s'$ are both larger than 1, we can switch the chronological ordering of both posets and use the previous cases. 

If we are in a case where $s$ or $s'$ is 1 and $t = n_1$ or $t' = n_2$, we can use the ideas from the previous cases in two different ways to construct similar bijections with a finer partitioning of $J(\calP_1) \times J(\calP_2)$. 

\end{proof}

\subsection{Type 1}\label{subsec:Type1}
Here we consider an intersection between two arcs $\gamma_1$ and $\gamma_2$ which occurs before the first crossing between $\gamma_2$ and $T$ or dually after the last crossing but between the first crossing and last crossings of $\gamma_1$ and $T$. As a convention, we will choose an orientation of $\gamma_2$ so that the intersection is before the first crossing of $\gamma_2$ and $T$. We will assume $s(\gamma_2)$ is notched; if this is not the case, then the skein relation and proofs will be the same as in the unpunctured case.  An example is shown in Figure \ref{fig:Type1Grafting}.

\begin{figure}
\captionsetup{width=0.9\textwidth}
\centering
\begin{tikzpicture}
\draw(0,0) to node[right]{$\sigma_1$} (1,-2);
\draw(0,0) to node[left]{$\sigma_m$} (-1,-2);
\draw(0,0) to node[left, yshift = 14pt, xshift = 11pt]{$\sigma_{\ell+1}$} (1,2);
\draw(0,0) to node[right]{$\sigma_\ell$}(1.9,1.2);
\draw(1,-2) to node[below, xshift = -13pt]{$\calP_2(1)$} (-1,-2);
\draw(0,0) to node[right, yshift = 10pt, xshift = -8pt]{$\sigma_{k-1}$} (-1,2);
\draw(0,0) to node[left]{$\sigma_{k}$} (-1.9,1.2);
\draw(-1,2) to node[left]{$\calP_1(a)$} (-1.9,1.2);
\draw(1,2) to node[right]{$\calP_1(b)$} (1.9,1.2);
\node[] at (1.2,0){$\vdots$};
\node[] at (-1.2,0){$\vdots$};
\node[] at (0,0.8){$\cdots$};
\draw[thick, orange,->](0,0)  to node[right]{$\gamma_2$} (0,-2.5);
\node[scale = 0.7, orange] at (0,-0.4){$\bowtie$};
\draw[thick, Dgreen] (-1.4,2) to [out = -45, in = 90] (-1,0);
\node[Dgreen] at (-0.75,0.8){$\gamma_1$};
\draw[thick, Dgreen,->] (-1,0) to [out = 270, in = 270, looseness = 1.4](0.9,0);
\draw[thick, Dgreen](0.9,0) to [out = 90, in = 90, looseness = 1.4](-0.8,0);
\draw[thick, Dgreen] (-.8,0) to [out = 270, in = 270, looseness = 1.4](0.7,0);
\draw[thick, Dgreen] (0.7,0) to [out = 90, in = 225] (1.4,2);
\end{tikzpicture}
\caption{An illustration of the type of intersections we study in Section \ref{subsec:Type1} as well as the notation we will use. We can have $k \geq \ell$ as well and it is possible that $\gamma_1$ does not intersect the arcs labeled $\calP_1(a)$ and $\calP_1(b)$. }\label{fig:Type1Grafting}
\end{figure}
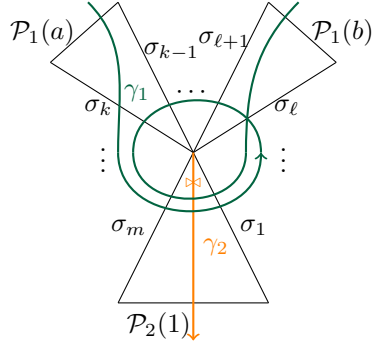

Let $p = s(\gamma_2)$, and suppose the spokes in $T$ which are incident to $p$ are $\sigma_1,\ldots,\sigma_m$, indexed in counterclockwise order and such that the first triangle that $\gamma_2$ passes through is bounded by $\sigma_1$ and $\sigma_m$. Since $\gamma_1$ crosses $\gamma_2$ in a Type 1 intersection, $\gamma_1$ must cross a consecutive sequence of at least two spokes incident to $p$, which, up to the choice of orientation, is of the form $\sigma_k,\ldots,\sigma_m,\sigma_1,\ldots \sigma_\ell$, for some $1 \leq k, \ell \leq m$.

It is possible that $\gamma_1$ winds around this puncture multiple times, thus crossing $\gamma_2$ in multiple Type 1 intersections. Let $w \geq 2$ be such that $\gamma_1$ and $\gamma_2$ intersect $w-1$ times in the first triangle $\gamma_2$ passes through. We let $a < b$ be such that $\calP_{\gamma_1}(a)$ and $\calP_{\gamma_1}(b)$ correspond to the arcs from $T$ which $\gamma_1$ crosses immediately before and after this sequence of crossings with spokes at $p$ respectively. However, these values might not exist. 

We will use indexing for $\calP_1:= \calP_{\gamma_1}$ similar to that depicted in Figure \ref{fig:Indexing}. So the sequence of spokes crossed will produce a chain \[
\sigma_k^{(1)} \prec \sigma_{k+1}^{(1)} \prec \cdots \prec \sigma_m^{(1)} \prec \sigma_1^{(2)} \prec \cdots \prec \sigma_{\ell}^{(w)}.
\]

We also can see from Figure \ref{fig:Type1Grafting} that $\sigma_k^{(1)} \prec \calP_1(a)$ and $\sigma_{\ell}^{(w)} \succ \calP_1(b)$.

Label the Type 1 intersection points between $\gamma_1$ and $\gamma_2$ as $1,2,\ldots,w-1$, which are ordered with respect to the orientation of $\gamma_1$. The resolution of $\gamma_1$ and $\gamma_2$ at a given intersection point consists of two sets of size 2, $\{\gamma_3,\gamma_4\} \cup \{\gamma_5, \gamma_6\}$. We index these so that $\gamma_3$ is the result of following $\gamma_1$ with opposite orientation from $t(\gamma_1)$ to the chosen intersection point and then following $\gamma_2$ with positive orientation, which we denote $\gamma_1^{-1}\circ \gamma_2$. Then, $\gamma_4 = \gamma_1 \circ \gamma_2^{-1}$, $\gamma_5 = \gamma_1 \circ \gamma_2$ and $\gamma_6 = \gamma_1^{-1} \circ \gamma_2^{-1}$. The arcs $\gamma_4$ and $\gamma_6$ are homotopic to arcs which do not cross any spokes since they each have an endpoint at $p$. The posets for these arcs are depicted in Table \ref{tab:indexingType1}, whereas before we set $\calP_i:= \calP_{\gamma_i}$.

\begin{prop}\label{prop:GraftingType1}
 Given $\gamma_1$ and $\gamma_2$ as above, when we resolve at the $r$-th intersection point, we have that \[
x_{\gamma_1}x_{\gamma_2} = Y_{\preceq \sigma_m^{(r)}} x_{\gamma_3}x_{\gamma_4} + x_{\gamma_5}x_{\gamma_6}
 \]
 and 
 \[
x_{\gamma_1}x_{\gamma_2^{(p)}} =  x_{\gamma_3}x_{\gamma_4^{(p)}} + Y_{\succeq \sigma_1^{(r+1)}} x_{\gamma_5}x_{\gamma_6^{(p)}}
\]
where \[
Y_{\preceq \sigma_m^{(r)}} = \prod_{\calP_1(i) \preceq \sigma_m^{(r)}} y_{\calP_1(i)} \qquad Y_{\succeq \sigma_1^{(r+1)}} = \prod_{\calP_1(i) \succeq \sigma_1^{(r+1)}} y_{\calP_1(i)}.
 \]
\end{prop}

\begin{table}
\captionsetup{width=0.9\textwidth}
\centering
\begin{tabular}{|c|c|}\hline
    \highlight{$\calP_1$} & \highlight{$\calP_2$} \\
    \begin{tikzpicture}[scale = 1]
    \node[] at (-0.75,0){$\cdots$};
     \node[](alpha) at (0,0){$\calP_1(a)$};
     \node[](k0) at (1,-1){$\sigma_k^{(1)}$};
     \node[](dots1) at (2,0.1) {$\iddots$};
     \node[](ellw) at (3,1){$\sigma_\ell^{(w)}$};
     \node[](beta) at (4,0){$\calP_1(b)$};
      \node[] at (4.75,0){$\cdots$};
     \draw(alpha)--(k0);
     \draw (k0) -- (dots1);
     \draw (dots1) -- (ellw);
     \draw(ellw) -- (beta);
    \end{tikzpicture} &  
    \begin{tikzpicture}[scale = 1]
     \node[](r) at (0,0) {$\sigma_1$};
     \node[] (r-1) at (1,1){$\sigma_2$};
     \node[] (dots) at (2,2) {$\iddots$};
     \node[] (1) at (3,3){$\sigma_m$};
     \node[] (delta) at (4,1.5){$\calP_2(1)$};
      \node[] at (4.75,1.5){$\cdots$};
    \draw (r) -- (r-1);
    \draw (r-1) -- (dots);
    \draw (dots) -- (1);
    \draw(1) -- (delta);
    \draw(r) -- (delta);
    \end{tikzpicture}\\\hline
    \highlight{$\calP_{3}$} & \highlight{$\calP_{4}$} \\
     \begin{tikzpicture}[scale = 1]
      \node[] at (-0.75,0){$\cdots$};
     \node[](delta) at (0,0){$\calP_2(1)$};
     \node[](rw-m) at (1,-1){$\sigma_1^{(r+1)}$};
     \node[] (dots3) at (2,0){$\iddots$};
    \node[](ellw) at (3,1){$\sigma_\ell^{(w)}$};
     \node[](beta) at (4,0){$\calP_1(b)$};
      \node[] at (4.75,0){$\cdots$};
    \draw(ellw) -- (beta);
     \draw (dots3) -- (ellw);
     \draw(rw-m) --(dots3);
     \draw(delta) -- (rw-m);
     \end{tikzpicture}
     &  
      \begin{tikzpicture}[scale = 1]
     \node[](r) at (0,0) {$\sigma_k$};
     \node[] (r-1) at (1,1){$\sigma_{k+1}$};
     \node[] (dots) at (2,2) {$\iddots$};
     \node[] (1) at (3,3){$\sigma_{k-1}$};
     \node[] (delta) at (4,1.5){$\calP_1(a)$};
      \node[] at (4.75,1.5){$\cdots$};
    \draw (r) -- (r-1);
    \draw (r-1) -- (dots);
    \draw (dots) -- (1);
    \draw(1) -- (delta);
    \draw(r) -- (delta);
    \end{tikzpicture}\\\hline
    \highlight{$\calP_5$} & \highlight{$\calP_6$} \\
         \begin{tikzpicture}[scale = 1]
     \node[] at (-0.75,0){$\cdots$};
     \node[](delta) at (0,0){$\calP_1(a)$};
     \node[](rw-m) at (1,-1){$\sigma_k^{(1)}$};
     \node[] (dots3) at (2,0){$\iddots$};
    \node[](ellw) at (3,1){$\sigma_m^{(r)}$};
     \node[](beta) at (4,0){$\calP_2(1)$};
      \node[] at (4.75,0){$\cdots$};
    \draw(ellw) -- (beta);
     \draw (dots3) -- (ellw);
    \draw (dots3) -- (rw-m);
     \draw(delta) -- (rw-m);
     \end{tikzpicture} &  
      \begin{tikzpicture}[scale = 1]
     \node[](r) at (0,0) {$\sigma_{\ell+1}$};
     \node[] (r-1) at (1,1){$\sigma_{\ell+2}$};
     \node[] (dots) at (2,2) {$\iddots$};
     \node[] (1) at (3,3){$\sigma_\ell$};
     \node[] (delta) at (4,1.5){$\calP_1(b)$};
      \node[] at (4.75,1.5){$\cdots$};
    \draw (r) -- (r-1);
    \draw (r-1) -- (dots);
    \draw (dots) -- (1);
    \draw(1) -- (delta);
    \draw(r) -- (delta);
    \end{tikzpicture}\\\hline
\end{tabular}
\caption{We depict the indexing used in the proof of Proposition \ref{prop:GraftingType1}, with $\calP_i:= \calP_{\gamma_i}$. As in Figure \ref{fig:Indexing}, the subscripts increase modulo $m$ and the superscripts increase between $\sigma_m$ and $\sigma_1$. If $\gamma_2$ is plain at $s(\gamma)$, as in the first part of Proposition \ref{prop:GraftingType1}, then the posets for the resolution would be the result of removing the loops in $\calP_2,\calP_4,$ and $\calP_6$}\label{tab:indexingType1}
\end{table}

\begin{proof}

The first result can be found in \cite{snakegraphcalculus1, musiker2013}, so we  assume $s(\gamma_2)$ is notched throughout. 

\textbf{Case i) $\calP_1(a)$ and $\calP_1(b)$ do not lie in loops.}
 
First, we assume we do not have that both $s(\gamma_1)$ is notched and the first crossing between $\gamma_1$ and $T$ is in this chain of spokes which $\gamma_1$ crosses as it crosses $\gamma_2$, nor do we have the parallel configuration at $t(\gamma_1)$. This means that if the elements $\calP_1(a)$ and $\calP_1(b)$ in Table \ref{tab:indexingType1} exist, they are elements of $\calP_1^0$.

Any element of $\calP_1$ and $\calP_2$ outside the chain of spokes and loop will correspond to a unique element in $\calP_{\gamma_3} \cup \calP_{\gamma_4}$ and also in $\calP_5\cup \calP_6$. Therefore, we will focus on how $\Phi$ acts on the elements in the chain of spokes and loop in each $\calP_i$, and its action on the remaining elements will always be the natural one from this correspondence. 

Given $(I_1,I_2) \in J(\calP_1) \times J(\calP_2)$, let $\sigma_y^{(z)}$  be the largest value in $I_1$ in the chain of spokes, if it exists. Similarly, let $\sigma_x$ be the largest value in $I_2$ in the chain of spokes. By abuse of notation, we will also use $\sigma_y^{(z)}$ to denote the maximal elements on the chain of spokes in $\calP_3$ and $\calP_5$ and similarly for $\sigma_x$ in $\calP_4$ and $\calP_6$. 

We can write any $I_1 \in J(\calP_1)$ as $\langle \sigma_y^{(z)} \rangle \cup I_1'$, where even if $\sigma_y^{(z)} = \sigma_{\ell}^{(w)}$, we include in $I_1'$ everything not in the chain of spokes. We similarly can write $I_2 = \langle \sigma_x \rangle \cup I_2'$. The elements in $I_1'$ and $I_2'$ have clear images in the other posets, so we will ignore them when we define the bijection. 

We interpret any statement ``$\sigma_y^{(z)} \succeq \sigma_{k-1}^{(0)}$'' as always true. If $\calP_1(a)$ or $\calP_1(b)$ do not exist, then we ignore any statement regarding these. 

\textbf{Subcase) $r > 1$.}

\begin{enumerate}
\item Let $A_1$ denote the set of $(I_1,I_2) \in J(\calP_1) \times J(\calP_2)$ such that $\sigma_y^{(z)} \prec \sigma_{k-1}^{(w-r+1)}$;  we do not have both $\calP_1(a) \in I_1$ and $y = k-1$; and,  if $\sigma_y^{(z)} \succeq \sigma_{k-1}^{(w-r)}$, then $\sigma_x \preceq \sigma_\ell$, with equality only if $\calP_1(b) \in I_1$. Given $(I_1,I_2) \in A_1$, we set \[
    \Phi(\langle \sigma_y^{(z)} \rangle, \langle \sigma_x \rangle) = \begin{cases}
    (\langle \sigma_x^{(r+z-1)} \rangle, \langle \sigma_y \rangle) & y < k-1\\
     (\langle \sigma_x^{(r+z)} \rangle, \emptyset) & y = k-1 \\
    (\langle \sigma_x^{(r+z)} \rangle, \langle \sigma_y \rangle) & y > k-1\\
    \end{cases}
    \]
    where if $\sigma_x$ does not exist, we set $\sigma_x^{(z)}:= \sigma_{k-1}^{(z)}$. The set $\Phi(A_1)$ is all $(I_3,I_4) \in J(\calP_3) \times J(\calP_4)$ such that $\sigma_{k-1} \notin I_4$.
\item $A_2$ is the set of  $(I_1,I_2) \in J(\calP_1) \times J(\calP_2)$ such that $\calP_1(a) \in I_1$; $y = k-1$; and, $z \leq w-r+1$ where if $z = w-r+1$, then $\sigma_x \preceq \sigma_\ell$, with equality only if $\calP_1(b) \in I_1$. Given $(I_1,I_2) \in A_2$, we set $\Phi(\langle \sigma_{k-1}^{(z)} \rangle, \langle \sigma_x \rangle) = (\langle \sigma_x^{(r+z-1)} \rangle, \langle  \sigma_{k-1} \rangle)$ where again $\sigma_x^{(r+z-1)}:= \sigma_m^{(r+z-2)}$ if $\sigma_x$ does not exist and we also set $\langle \sigma_m^{(r)} \rangle:= \emptyset$ in $\calP_3$. The set $\Phi(A_2)$ is all $(I_3,I_4)$ such that $\sigma_{k-1} \in I_4$. 
\item $B_1$ is the set of $(I_1,I_2) \in J(\calP_1) \times J(\calP_2)$ such that $\sigma_y^{(z)} \succeq \sigma_{k-1}^{(w-r+1)}$ with strict inequality if $\calP_1(a) \in I_1$ and $\sigma_x \preceq \sigma_\ell$, with equality only if $\calP_1(b) \in I_1$. Given $(I_1,I_2) \in B_1$, we set $\Phi(\langle \sigma_y^{(z)} \rangle, \langle \sigma_x \rangle) = (\langle \sigma_y^{(z-(w-r))} \rangle , \langle \sigma_x \rangle)$ if $\sigma_x$ exists and otherwise set the image to $(\langle \sigma_y^{(z-(w-r))} \rangle , \langle \sigma_m \rangle)$. Notice that if $\ell \leq k-1$, the image of $(I_1,I_2)$ has the same content as $(I_1 \backslash (\langle \sigma_{k-1}^{(w-r)} \rangle \cup \{ \sigma_1^{(w-r+1)},\ldots,\sigma_\ell^{(w-r+1)} \}))\cup I_2$ and similarly if $\ell > k-1$, the image has the same content as $(I_1 \backslash (\langle \sigma_{\ell}^{(w-r)} \rangle \cup \{ \sigma_1^{(w-r+1)},\ldots,\sigma_{k-1}^{(w-r+1)}))\cup I_2$.  The set $\Phi(B_1)$ consists of all pairs $(I_5,I_6) \in J(\calP_5) \times J(\calP_6)$ such that $\sigma_y^{(z)} \preceq \sigma_\ell^{(r)}$ with equality only if $\calP_1(b) \in I_6$ and $\sigma_x \succeq \sigma_m$ with strict inequality if $\calP_2(1) \in I_5$.
\item $B_2$ is the set of $(I_1,I_2) \in J(\calP_1) \times J(\calP_2)$ such that $\sigma_y^{(z)} \succeq \sigma_{k-1}^{(w-r)}$ with strict inequality if $\calP_1(a) \in I_1$; $\sigma_y^{(z)} \preceq \sigma_\ell^{(w-1)}$ with equality only if $\calP_1(b) \in I_1$; and, $\sigma_x \succeq \sigma_\ell$, with strict inequality if $\calP_1(b) \in I_1$.  Given $(I_1,I_2) \in B_2$, we set the image to be $(\langle \sigma_y^{(z-(w-r))}\rangle, \langle \sigma_x \rangle)$. Notice that the image of $(I_1,I_2)$ has the same content as $(I_1 \backslash \langle \sigma_{k-1}^{(w-r)} \rangle) \cup (I_2 \backslash \langle \sigma_\ell \rangle)$. The set $\Phi(B_2)$ consists of all pairs $(I_5,I_6)$ such that $\sigma_y^{(z)} \preceq \sigma_\ell^{(r)}$ with equality only if $\calP_1(b) \in I_6$ and $\sigma_x \preceq \sigma_m$ with equality only if $\calP_2(1) \in I_5$.
\item $B_3$ is the set of $(I_1,I_2) \in J(\calP_1) \times J(\calP_2)$ such that $\sigma_y^{(z)} \succeq \sigma_\ell^{(w-1)}$ with strict inequality if $\calP_1(b) \in I_1$ and $\sigma_x \succeq \sigma_\ell$ with strict inequality if $\calP_1(b) \in I_1$. Given $(I_1,I_2) \in B_3$, if $\sigma_y^{(z)} \succ \sigma_\ell^{(w-1)}$, we set the image to be $(\langle \sigma_x^{(r)} \rangle , \langle \sigma_y \rangle)$ and otherwise we set the image to be $(\langle \sigma_x^{(r)} \rangle , \emptyset)$. The content of the image of $(I_1,I_2)$ in this case is again the same as $(I_1 \backslash \langle \sigma_{k-1}^{(w-r)} \rangle) \cup (I_2 \backslash \langle \sigma_\ell \rangle)$. The set $\Phi(B_3)$ consists of pairs $(I_5,I_6)$ such that $\sigma_y^{(z)} \succeq \sigma_\ell^{(r)}$ with strict inequality if $\calP_1(b) \in I_6$.
\end{enumerate}

Now, let $r =1 $. If $k-1 \leq \ell$, then the same maps as above will work, so let $\ell < k-1$.  We will partition the complement of $A_1 \cup A_2$ in $J(\calP_1) \times J(\calP_2)$ now by $B_2^0 \cup B_3^0$, which are defined as follows. Set $B_2^0$ to be the set of pairs $(I_1,I_2)$ such that $\sigma_{k-1}^{(w-1)} \preceq \sigma_y^{(z)} \preceq \sigma_m^{(w-1)}$ where the second inequality is equality only if $\calP_2(1) \in I_2$, and $\sigma_\ell \preceq \sigma_x \prec \sigma_{k-1}$ where the first inequality is strict if $\calP_1(b) \in I_1$. Set $B_3^0$ to be the set of pairs $(I_1,I_2)$ such that $\sigma_y^{(z)} \succeq \sigma_{k-1}^{(w-1)}$ and $\sigma_x \succeq \sigma_k$. We set $\Phi\vert_{B_2^0}(\langle \sigma_y^{(z)} \rangle, \langle\sigma_x\rangle) = (\langle \sigma_y^{(0)} \rangle, \langle\sigma_x \rangle)$ and $\Phi\vert_{B_3^0}(\langle \sigma_y^{(z)} \rangle, \langle\sigma_x\rangle) = (\langle \sigma_x^{(0)} \rangle, \langle \sigma_y \rangle)$. The image of $B_2^0$ is pairs $(I_5,I_6)$ such that $\sigma_{k-1} \notin I_6$ and the image of $B_3^0$ is the complement. From this, one can construct a reverse map and observe that $\Phi: B_2^0 \cup B_3^0 \to J(\calP_5) \times J(\calP_6)$ is a bijection.

Let $\mathbf{g}_i := \mathbf{g}_{\gamma_i}$. We compare the quantities $\mathbf{g}_1 + \mathbf{g}_2$, $\mathbf{g}_3 + \mathbf{g}_4$, and $\mathbf{g}_5 + \mathbf{g}_6$. Recall that for a statement $X$, $\delta_X = 1$ if $X$ is true and otherwise $\delta_X = 0$, and if a statement involves an inequality, we require both objects in the comparison to exist for the statement to be true. We have the following $\gb$-vectors, where $\gb_i'$ denotes the contribution for the portion of $\calP_i$ which is not pictured in Table \ref{tab:indexingType1}. 

\begin{itemize}
    \item $\mathbf{g}_1 = -\eb_{\sigma_k} + \eb_{\sigma_\ell} - (1-\delta_{\calP_1(b)>\calP_1(b+1)}) \eb_{\calP_1(b)} + \delta_{\calP_1(a) > \calP_1(a-1)} \eb_{\calP_1(a)} + \gb_1'$
    \item $\mathbf{g}_2 = -\eb_{\sigma_1} + \delta_{\calP_2(1) > \calP_2(2)} \eb_{\calP_2(1)} + \gb_2'$
    \item $\mathbf{g}_3 = -\eb_{\sigma_1}+ \eb_{\sigma_\ell} - (1 - \delta_{\calP_1(b)>\calP_1(b+1)} ) \eb_{\calP_1(b)} + \delta_{\calP_2(1) > \calP_2(2)} \eb_{\calP_2(1)} + \gb_3'$.
    \item $\mathbf{g}_4 =   -\eb_{\sigma_k} + \delta_{\calP_1(a) > \calP_1(a-1)} \eb_{\calP_1(a)} + \gb_4'$.
    \item $\mathbf{g}_5 = -\eb_{\sigma_k} + \eb_{\sigma_m} + \delta_{\calP_1(a) > \calP_1(a-1)} \eb_{\calP_1(a)} - (1-\delta_{\calP_2(1) > \calP_2(2)}) \eb_{\calP_2(1)} + \gb_5'$
    \item $\mathbf{g}_6 = -\eb_{\sigma_{\ell+1}} + \delta_{\calP_1(b)>\calP_1(b+1)} \eb_{\calP_1(b)} + \gb_6'$
\end{itemize}

From the definition of the posets $\calP_i$, we can see that $\gb_1' + \gb_2' = \gb_3' + \gb_4' = \gb_5' + \gb_6'$. Thus, we immediately have $\mathbf{g}_1 + \mathbf{g}_2 = \mathbf{g}_3 + \mathbf{g}_4$. We also can see that \[
(\mathbf{g}_5 + \mathbf{g}_6) - (\mathbf{g}_1 + \mathbf{g}_2) = \eb_{\sigma_1} + \eb_{\sigma_m} -\eb_{\sigma_\ell} - \eb_{\sigma_{\ell+1}} + \eb_{\calP_1(b)} - \eb_{\calP_2(1)},
\]
and by Corollary \ref{cor:y-hat_spokes}, this vector is exactly $\deg_{\boldsymbol{x}}(\hat{Y}_{\succeq \sigma_1^{(r+1)}}) = \deg_{\boldsymbol{x}}(\hat{Y}_p^{w-r} \hat{y}_{\sigma_1} \hat{y}_{\sigma_1}\cdots \hat{y}_{\sigma_\ell}) =\deg_{\boldsymbol{x}}( \hat{y}_{\sigma_1} \hat{y}_{\sigma_1}\cdots \hat{y}_{\sigma_\ell})$. 

Now suppose $\gamma_1$ is notched at $s(\gamma_1) = q_s$ and $\gamma_1$ has no crossings with $T$ before its chain of crossings with spokes. Then, by our choice of labeling, $s(\gamma_4)$ and $s(\gamma_5)$ are also notched at $q_s$, and $\sigma_{k-1}$ is also a spoke incident to $q_s$. Let the spokes at $q$ be $\eta_0 = \sigma_{k-1},\ldots,\eta_h$, labeled in counterclockwise order. This implies that  $\eta_h$ forms a triangle with $\sigma_{k-1}$ and $\sigma_{k}$. In $\calP_1$ and $\calP_5$, we update the poset by replacing the cover relation $\calP_1(a) \succ \sigma_k^{(1)}$ with the portion of a Hasse diagram shown below to the left.
\begin{center}
\begin{tabular}{cc}
\begin{tikzpicture}
\node[] (k0) at (0,0){$\sigma_{k}^{(0)}$};
\node[] (a1) at (-1,1){$\eta_h$};
\node[rotate=10] (dots) at (-2,0){$\iddots$};
\node[] (as) at (-3,-1){$\eta_1$};
\node[] (k-1) at (-1,-2){$\sigma_{k-1}^{(0)}$};
\node[] (dotsu) at (1,1){$\iddots$};
\draw (k0) -- (a1);
\draw(k-1) -- (k0);
\draw(k-1) -- (as);
\draw (as) --(dots);
\draw (dots) -- (a1);
\draw(k0) -- (dotsu);
\end{tikzpicture}
&
\begin{tikzpicture}
\node[] (k1-) at (0,0){$\sigma_{k-1}^-$};
\node[] (k) at (-1,1){$\sigma_k$};
\node[] (dotsl) at (-1,2){$\vdots$};
\node[] (k2) at (-1,3){$\sigma_{k-2}$};
\node[] (as) at (1,1){$\eta_1$};
\node[] (dotsr) at (1,2){$\vdots$};
\node[] (a1) at (1,3){$\eta_h$};
\node[] (k1+) at (0,4){$\sigma_{k-1}^+$};
\draw (k1+) -- (a1);
\draw(k1+) -- (k2);
\draw(k1-) -- (k);
\draw(k1-) -- (as);
\draw (as) --(dotsr);
\draw (dotsr) -- (a1);
\draw(k) -- (dotsl);
\draw(dotsl) -- (k2);
\draw(k2) -- (as);
\draw(a1) -- (k);
\end{tikzpicture}
\end{tabular}
\end{center}

Notice that $\gamma_4$ is a doubly-notched arc whose underlying untagged version, $\gamma_4^0$, is $\sigma_{k-1} \in T$. Its corresponding poset is above to the right.

To construct the bijections between the products of order ideals, we separately deal with sets $(I_i,I_{i+1})$ such that $\eta_h \notin I_i \cup I_{i+1}$ and those which contain $\eta_h$. The same bijections as in the untagged case work for going between the sets of tuples which do not contain $\eta_h$, where occurrences of $\sigma_k$ or $\sigma_k^{(z)}$ are replaced with $\sigma_{k-1}$ and $\sigma_{k-1}^{(z)}$ depending on which is appropriate and $\eta_1$ acts as $\calP_1(a)$. If $\eta_h \in I_1$, then we can use our maps above where instances of $\calP_1(a)$ are replaced with $\eta_h$. 

Now, we consider the case where $\gamma_1$ is notched at $t(\gamma_1) = q_t$  and $\gamma_1$ does not have further crossings with $T$ after its crossings in the chain of spokes. Then, we have a parallel situation where the spokes at $q_t$ are $\beta_1, \beta_2,\ldots,\beta_t = \sigma_{\ell+1}$. Then, we replace the cover relation $\sigma_\ell^{(w)} \succ \calP_1(b)$ in $\calP_1$ and $\calP_{3}$ with the following portion of a Hasse diagram on the left, and $\calP_6$ is given by the Hasse diagram on the right as $\gamma_6$ is now $\sigma_{\ell}^{(s(\gamma_2),q_t)}$.

\begin{center}
\begin{tabular}{cc}
\begin{tikzpicture}
\node[] (lw) at (0,0){$\sigma_{\ell}^{(w)}$};
\node[] (l-1) at (1,2){$\sigma_{\ell+1}^{(w)}$};
\node[] (bt) at (1,-1){$\beta_1$};
\node[] (dots) at (2,0){$\iddots$};
\node[] (b1) at (3,1){$\beta_{t-1}$};
\node[] (dotsd) at (-0.9,-0.7){$\iddots$};
\draw (lw) -- (bt);
\draw(lw) -- (l-1);
\draw(bt) -- (dots);
\draw (b1) --(dots);
\draw (b1) -- (l-1);
\draw(lw) -- (-0.5,-0.5);
\end{tikzpicture}
&
\begin{tikzpicture}
\node[] (k1-) at (0,0){$\sigma_{\ell+1}^-$};
\node[] (k) at (-1,1){$\sigma_{\ell+2}$};
\node[] (dotsl) at (-1,2){$\vdots$};
\node[] (k2) at (-1,3){$\sigma_{\ell}$};
\node[] (as) at (1,1){$\beta_1$};
\node[] (dotsr) at (1,2){$\vdots$};
\node[] (a1) at (1,3){$\beta_{t-1}$};
\node[] (k1+) at (0,4){$\sigma_{\ell+1}^+$};
\draw (k1+) -- (a1);
\draw(k1+) -- (k2);
\draw(k1-) -- (k);
\draw(k1-) -- (as);
\draw (as) --(dotsr);
\draw (dotsr) -- (a1);
\draw(k) -- (dotsl);
\draw(dotsl) -- (k2);
\draw(k2) -- (as);
\draw(a1) -- (k);
\end{tikzpicture}
\end{tabular}
\end{center}

We adjust the definitions of $A_1,A_2$ and $B_1$ by replacing the condition ``$\sigma_x \preceq \sigma_\ell$ with equality only if $\calP_1(b) \in I_1$'' with ``$\sigma_x \preceq \sigma_{\ell+1}$ where $\sigma_x = \sigma_{\ell+1}$ only if $\beta_{t-1} \in I_1$ and  $\sigma_x = \sigma_{\ell}$ only if $\beta_1 \in I_1$''. In the definition of $B_2$, we replace ``$\sigma_x \succeq \sigma_\ell$ with strict inequality if $\calP_1(b) \in I_1''$ with  `` $\sigma_x \succeq \sigma_\ell$, with $\sigma_x \succ \sigma_\ell$ if $\beta_1 \in I_1$ and $\sigma_x \succ \sigma_{\ell+1}$ if $\beta_{t-1} \in I_1$'', and in the definition of $B_3$ we replace ``$\sigma_y^{(z)} \succeq \sigma_\ell^{(w-1)}$ with strict inequality if $\calP_1(b) \in I_1$'' with ``$\sigma_y^{(z)} \succeq \sigma_\ell^{(w-1)}$ with  $\sigma_y^{(z)} \succ \sigma_\ell^{(w-1)}$ if $\beta_1 \in I_1$ and $\sigma_y^{(z)} \succeq \sigma_{\ell+1}^{(w-1)}$ if $\beta_{t-1} \in I_1$''.

If both endpoints of $\gamma_1$ are tagged and $\gamma_1$ has no crossings with $T$ besides in this set of spokes, then we combine the above methods. In all such cases, the $\gb$-vector computations are similar. 

\end{proof} 

\subsection{Type 2}\label{subsec:Type2}

Now, we consider an intersection between two arcs $\gamma_1$ and $\gamma_2$ that, up to orientation, occurs before both arcs' first intersection with $T$.

Let $p = s(\gamma_2)$ and $q = s(\gamma_1)$. Let the edge between $p$ and $q$ be labeled $\tau$. Following previous sections, let the remaining spokes incident to $p$ be labeled $\sigma_1, \dots, \sigma_m$, indexed in counterclockwise order such that $\sigma_1$ follows counterclockwise from $\tau$ and $\tau$ follows counterclockwise from $\sigma_m$. Similarly, label the remaining spokes incident to $q$ with $\eta_1, \dots, \eta_h$, indexed in counterclockwise order such that $\eta_1$ follows counterclockwise from $\tau$ and $\tau$ follows counterclockwise from $\eta_h$. In this configuration, the triangle where $\gamma_1$ and $\gamma_2$ intersect has edges labeled, in clockwise order, by $\tau, \sigma_1,$ and $\eta_h$. The arc $\gamma_1$ must cross some consecutive sequence $\sigma_{1}, \dots, \sigma_{k}$ of spokes incident to $p$ and $\gamma_2$ must cross a sequence $\eta_\ell,\ldots,\eta_h$. It is possible for either sequence to consist of a single spoke, implying $k = 1$ or $\ell = m$, or for the arcs to wind several times around each puncture. Let $a$ be such that $\calP_{\gamma_1}(a)$ denotes the arc (if it exists) that $\gamma_1$ crosses immediately after this sequence of spokes, if it exists, and similarly let $\calP_{\gamma_2}(b)$ denote such an arc for $\gamma_2$. 

In the figure below, only one of the intersections is an intersection of Type 2 and all others, which occur when at least one of the arcs winds around one of the punctures, are of Type 1. 

\begin{center}
\begin{tikzpicture}
    \draw[thick] (0,0) to node[midway,below]{$\tau$} (3,0) to node[left,midway]{$\eta_{h}$} (1.5,2.6) to node[midway,right]{$\sigma_1$} (0,0) to node[midway, right]{$\sigma_m$}(1.5,-2.6) to node[midway, left]{$\eta_1$} (3,0);

    \draw[thick] (0,0) to node[midway,right,yshift=15,xshift=-5]{$\sigma_2$} (-0.75,2);
    \draw[thick] (0,0) to node[midway,below,xshift=-10,yshift=-8]{$\sigma_{k+1}$} (-2,-1);
    \draw[thick] (0,0) to node[midway,above,xshift=-15,yshift=5]{$\sigma_{k}$} (-2.25,0.75);

    \draw[thick] (-2.25,0.75) to node[midway,left,yshift=10]{$\calP_{\gamma_1}(a)$} (-2,-1);

    \draw[thick] (3,0) to node[midway,right,yshift=15,xshift=5]{$\eta_{h-1}$} (3.75,2);
    \draw[thick] (3,0) to node[midway,below,xshift=7,yshift=-5]{$\eta_{\ell-1}$} (5,-1);
    \draw[thick] (3,0) to node[midway,above,xshift=15,yshift=5]{$\eta_{\ell}$} (5.25,0.75);

    \draw[thick] (5,-1) to node[midway,right,yshift=10]{$\calP_{\gamma_2}(b)$} (5.25,0.75);

    \draw[thick,black!20!green,out=135,in=90] (3,0) to (-1.2,0);
    \draw[thick,black!20!green,out=-90,in=180] (-1.2,0) to (0,-0.75);
    \draw[thick,black!20!green,out=0,in=0,looseness=2] (0,-0.75) to (0,0.9);
    \draw[thick,black!20!green,out=180,in=180,looseness=1.5] (0,0.9) to (-0.4,-0.5);
    \node[black!20!green] at (-0.1,-0.5) {$\dots$};
    \node[black] at (-0.4,-1.2) {$\dots$};
    \node[black] at (3.4,-1.2) {$\dots$};
    \node[black] at (-1,1.2) {$\iddots$};
    \node[black] at (4.3,1.2) {$\ddots$};

    \draw[thick,black!20!green,out=0,in=0,looseness=1.75] (0.1,-0.5) to (0,0.75);
    \draw[thick,black!20!green,out=180,in=45] (0,0.75) to (-0.6,0);
    \draw[thick,black!20!green,out=225,in=0] (-0.6,0) to (-2.5,-0.4);
    \node[black!20!green] at (-2.75,-0.4) {$\gamma_1$};

    \draw[thick,orange,out=45,in=90,looseness=1] (0,0) to (4.2,0);
    \draw[thick,orange,out=-90,in=0] (4.2,0) to (3,-0.75);
    \draw[thick,orange,out=180,in=180,looseness=2] (3,-0.75) to (3,0.9);
    \draw[thick,orange,out=0,in=0,looseness=1.5] (3,0.9) to (3.4,-0.5);
    \node[orange] at (3.2,-0.5) {$\dots$};
    \draw[thick,orange,out=180,in=180,looseness=1.75] (2.9,-0.5) to (3,0.75);
    \draw[thick,orange,out=0,in=135] (3,0.75) to (3.6,0);
    \draw[thick,orange,out=-45,in=180] (3.6,0) to (5.5,-0.4);
    \node[orange] at (5.75,-0.4) {$\gamma_2$};

    \node at (-0.1,-0.25) {$p$};
    \node at (3.1,-0.25) {$q$};
\end{tikzpicture}
\end{center}

Suppose that $\gamma_1$ winds $w-1 \geq 0$ complete times around $p$ and $\gamma_2$ winds $v-1 \geq 0$ complete times around $q$. For notational convenience, we set $\tau = \eta_0 = \sigma_{0}$. Then, on the chain of $\sigma_i^{(j)}$, we have $\sigma_1^{(1)} \prec \cdots \prec \sigma_m^{(1)} \prec \sigma_{0}^{(2)} \prec \sigma_1^{(2)} \prec \cdots \prec \sigma_k^{(w)}$ and similarly in the chain of $\eta_i^{(j)}$. 

Let $\gamma_3 = \gamma_1^{-1} \circ \gamma_2$, where we switch between these arcs at the intersection point of Type 2. Similarly, set $\gamma_5 = \gamma_1 \circ \gamma_2$ and $\gamma_6 = \gamma_2 \circ \gamma_1$. Note that $\gamma_2 \circ \gamma_1^{-1}$ is isotopic to $\tau$.

We draw all relevant posets in Table \ref{table:typeII_wind}. We draw these in the case where $s(\gamma_1)$ and $s(\gamma_2)$ are both notched; if one of these is plain, one obtains the correct posets by deleting the corresponding loops (which are all drawn in blue). We will recall the structure poset for the arc $\tau$ with appropriate notchings within each proof. 

\begin{table}
\captionsetup{width=0.9\textwidth}
    \begin{center}
    \begin{tabular}{|c|c|} \hline
        \highlight{$\mathcal{P}_{1}$} & \highlight{$\mathcal{P}_{2}$} \\
       \begin{tikzpicture}
\node(s1) at (0,0){$\sigma_1^{(1)}$};
\node(dots) at (1,1){$\iddots$};
\node(sm) at (2,2){$\sigma_k^{(w)}$};
\node(a) at (3,1){$\calP_{1}(a)$};
\node[] at (3.75,1){$\cdots$};
\draw(s1) -- (dots);
\draw(dots) --(sm);
\draw(sm) -- (a);
\node[blue](t) at (-1,-1.3){$\tau$};
\node[blue](n1) at (-1,-0.5){$\eta_1$};
\node[blue](vdots) at (-1,0.25){$\vdots$};
\node[blue](nh) at (-1,1){$\eta_h$};
\draw[blue](s1) -- (t);
\draw[blue](t) -- (n1);
\draw[blue](n1) -- (vdots);
\draw[blue](vdots) -- (nh);
\draw[blue](s1) -- (nh);
\end{tikzpicture}     
     &
\begin{tikzpicture}
\node(n1) at (0,0){$\eta_{\ell}^{(1)}$};
\node(dots) at (-1,1){$\ddots$};
\node(nh) at (-2,2){$\eta_h^{(v)}$};
\node(b) at (1,1){$\calP_{2}(b)$};
\node[] at (1.75,1){$\cdots$};
\draw(n1) -- (dots);
\draw(dots) --(nh);
\draw(n1) -- (b);
\node[blue](s1) at (-3,0.7){$\sigma_1$};
\node[blue](vdots) at (-3,1.5){$\vdots$};
\node[blue](sm) at (-3,2.25){$\sigma_m$};
\node[blue](t) at (-3,3){$\tau$};
\draw[blue](nh) -- (s1);
\draw[blue](nh) -- (t);
\draw[blue](s1) -- (vdots);
\draw[blue](vdots) -- (sm);
\draw[blue](sm) -- (t);
\end{tikzpicture}          
    \\\hline
        \multicolumn{2}{|c|}{\highlight{$\calP_{3}$}} \\
        \multicolumn{2}{|c|}{\begin{tikzpicture}
            \node[] at (-0.75,0){$\cdots$};
            \node[](A) at (0,0) {$\calP_1(a)$};
            \node[](Rj) at (1,1) {$\sigma_k^{(w)}$};
            \node[](dots1) at (2,0) {$\ddots$};
            \node[](R1) at (3,-1) {$\sigma_1^{(1)}$};
            \node(nh) at (4,0){$\eta_h^{(v)}$};
            \node(dots2) at (5,-1){$\ddots$};
            \node(n1) at (6,-2){$\eta_\ell^{(1)}$};
            \node(BO) at (7,-1){$\calP_2(b)$};
            \node[] at (7.75,-1){$\cdots$};
            \draw (A) to (Rj);
            \draw (Rj) to (dots1);
            \draw (dots1) to (R1);
            \draw (R1) to (nh);
            \draw (nh) to (dots2);
            \draw (dots2) to (n1);
            \draw (n1) to (BO);
        \end{tikzpicture}} \\ \hline
        \highlight{$\calP_5$} & \highlight{$\calP_6$} \\
        \begin{tikzpicture}
        \node[] at (0.75,0){$\cdots$};
        \node[](A) at (0,0) {$\calP_2(b)$};
        \node[blue](t) at (-1,-1.3){$\eta_{\ell}$};
        \node[blue](n1) at (-1,-0.5){$\eta_{\ell+1}$};
        \node[blue](vdots) at (-1,0.25){$\vdots$};
        \node[blue](nh) at (-1,1){$\eta_{\ell-1}$};
        \draw[blue] (A) -- (t);
        \draw[blue] (t) -- (n1);
        \draw[blue] (n1) -- (vdots);
        \draw[blue] (vdots) -- (nh);
        \draw[blue] (nh) -- (A);
        \end{tikzpicture} &
        \begin{tikzpicture}
        \node[] at (0.75,0){$\cdots$};
        \node[](A) at (0,0) {$\calP_1(a)$};
        \node[blue](t) at (-1,-1.3){$\sigma_{k+1}$};
        \node[blue](n1) at (-1,-0.5){$\sigma_{k+2}$};
        \node[blue](vdots) at (-1,0.25){$\vdots$};
        \node[blue](nh) at (-1,1){$\sigma_k$};
        \draw[blue] (A) -- (t);
        \draw[blue] (t) -- (n1);
        \draw[blue] (n1) -- (vdots);
        \draw[blue] (vdots) -- (nh);
        \draw[blue] (nh) -- (A);
        \end{tikzpicture} \\ \hline
    \end{tabular}
    \end{center}

    \caption{The structure of the posets used in the proofs in Section \ref{subsec:Type2}. The depicted posets are for the case where $s(\gamma_1)$ and $s(\gamma_2)$ are both notched. If $s(\gamma_1)$ is plain, one deletes the loops in $\calP_1$ and $\calP_5$ and similarly if $s(\gamma_2)$ is plain.} 
    \label{table:typeII_wind}
\end{table}

We include the statement for the case when $s(\gamma_1)$ and $s(\gamma_2)$ are both plain for the sake of comparison. However, even if the other ends are notched, this fact follows from the work done for plain arcs \cite{snakegraphcalculus1,musiker2013}. 

\begin{prop}
    \label{prop:typeIIgraft_wind_plain}
    Given $\gamma_1$ and $\gamma_2$ as described above, let $\gamma_3 = \tau \circ \gamma_1$, $\gamma_4 = \tau \circ \gamma_2$, and $\gamma_5 = \gamma_2^{-1} \circ \gamma_1$, where $\gamma_5$ is spliced at the intersection point of $\gamma_1$ and $\gamma_2$ in their mutual last triangle. Then,
   \[ x_{\gamma_1}x_{\gamma_2} = x_{\gamma_3}x_{\tau} + Y_{\preceq \eta_h^{(v)}} x_{\gamma_5}x_{\gamma_6} \]
    where $ Y_{\preceq \eta_h^{(v)}} = \prod_{\calP_{2}(i) \preceq \eta_h^{(v)}} y_{\calP_2(i)}$.
\end{prop}

Now we turn to our first new result.

\begin{prop}
\label{prop:typeII_singlytagged_winding_1}
 Given $\gamma_1$ and $\gamma_2$ which intersect in a Type 2 intersection, we have \[
x_{\gamma_1}x_{\gamma_2^{(p)}} = x_{\gamma_3} x_{\tau^{(p)}} + Y_{\succeq \sigma_1^{(1)}} Y_{\preceq \eta_h^{(v)}}^{\geq 1} x_{\gamma_5}x_{\gamma_6^{(p)}}
\]

where \[
Y_{\succeq \sigma_1^{(1)}} = \prod_{\calP_1(i) \succeq \sigma_1^{(1)}} y_{\calP_1(i)}\quad \text{and} \quad Y_{\preceq \eta_h^{(v)}}^{\geq 1} = \prod_{\substack{\calP_2(i) \preceq \eta_h^{(v)} \\ i \geq 1}} y_{\calP_2(i)}.
\]

\end{prop}

\begin{proof}
We follow our usual approach, partitioning $J(\calP_1) \times J(\calP_2)$ as $A_1 \cup A_2 \cup A_3 \cup B$ and constructing bijections between $A_1 \cup A_2 \cup A_3$ and $J(\calP_3) \times J(\calP_{\tau^{(p)}})$ and between $B$ and $J(\calP_5) \times J(\calP_6)$. We set $\calP_4:= \calP_{\tau^{(p)}}$. Recall the poset for $\tau^{(p)}$ is the chain $\sigma_1 \prec \sigma_2 \prec \cdots \prec \sigma_m$. 

In the following, let $\tau^{(v+1)}$ be the element $\tau$ in the loop of $\calP_2$. We repeat our notation from previous proofs, setting $\sigma_y^{(z)}$ to be the maximal element of a given order ideal of $\calP_1$ on the chain of spokes and $\sigma_x$ to be the maximal element on the loop of $\calP_2$. We let $\eta_c^{(d)}$ similarly denote the maximal element of an order ideal of $\calP_2$ on the chain of spokes. We also use this notation for the other posets in the resolution.

\begin{enumerate}
    \item Let $A_1 \subseteq J(\calP_1) \times J(\calP_2)$ consist of pairs $(I_1,I_2)$ such that $\tau^{(v+1)} \notin I_2$ and if $z = w$, then $x \leq k$, with equality only if $\alpha \in I_1$. Given $(I_1,I_2) \in A_1$, we set
\[
\Phi(\langle \sigma_y^{(z)} \rangle, \langle \sigma_x \rangle \cup \langle \eta_c^{(d)} \rangle) = \begin{cases}
(\langle \sigma_x^{(z)} \rangle \cup \langle \eta_c^{(d)} \rangle, \langle \sigma_y \rangle) & y \neq 0 \\
(\langle \sigma_x^{(z)} \rangle \cup \langle \eta_c^{(d)} \rangle, \emptyset) & y =0
\end{cases}
\]

where in each case we set $\sigma_x^{(z)}:= \sigma_0^{(z)}$ if $\sigma_x$ does not exist. The set $\Phi(A_1)$ is the set of all $(I_3,I_4)$ such that if $z = w$, then $x \leq k$ with equality only if $\calP_1(a) \in I_3$ and if $\eta_c^{(d)} = \eta_h^{(v)}$, then $y \neq 0$. 
\item $A_2$ consists of pairs $(I_1,I_2)$ such that $\tau^{(v+1)} \in I_2$ and  $z < w$. Given $(I_1,I_2) \in A_2$, we define
\[
\Phi(\langle \sigma_y^{(z)} \rangle, \langle \tau^{(v+1)} \rangle) = \begin{cases}
(\langle \sigma_0^{(z+1)} \rangle \cup \langle \eta_h^{(v)} \rangle, \langle \sigma_y \rangle) & y \neq 0 \\
(\langle \sigma_0^{(z+1)} \rangle \cup \langle \eta_h^{(v)} \rangle, \emptyset) & y =0
\end{cases}
\]
with the same convention when $\sigma_x$ does not exist. The set $\Phi(A_2)$ is the set of all $(I_3,I_4)$ such that $\eta_c^{(d)} = \eta_h^{(v)}$ and $y = 0$. 
\item $A_3$ consists of pairs $(I_1,I_2)$ such that $\tau^{(v+1)} \notin I_2$; $\sigma_y^{(z)} \succeq \sigma_0^{(w)}$ which is strict if $\eta_h^{(v)} \in I_2$; and, $x \geq k $, which is strict if $\calP_1(a) \in I_1$. The map $\Phi$ is defined on the spokes by $\Phi(\langle \sigma_y^{(z)} \rangle, \langle \sigma_x \rangle \cup \langle \eta_c^{(d)} \rangle) = (\langle \sigma_y^{(z)} \rangle \cup \langle \eta_c^{(d)} \rangle, \langle \sigma_x \rangle )$. The set $\Phi(A_3)$ is the set of $(I_3,I_4)$ such that $\sigma_y^{(z)} \succeq \sigma_0^{(w)}$ which is strict if $\eta_c^{(d)} = \eta_h^{(v)}$ and $x \geq k$, which is strict if $\calP_1(a) \in I_3$.
\item $B$ consists of sets with $\sigma_y^{(z)} \succeq \sigma_0^{(w-1)}$, which is only strict if $\tau^{(v+1)} \in I_2$; $x \geq k$, which is strict if $\calP_1(a) \in I_1$; and, $\eta_c^{(d)} = \eta_h^{(v)}$. Given $(I_1,I_2) \in B$, the set $(I_1 \backslash \langle \sigma_0^{(w-1)} \rangle) \cup (I_2 \backslash (\langle \sigma_k \rangle \cup \langle \eta_h^{(v)} \rangle))$ can be naturally partitioned into order ideals of $\calP_5$ and $\calP_6$ and this map is bijective.
\end{enumerate}

Finally, we must consider the $\gb$-vectors. Let $\gb_1', \gb_2'$ denote contribution of the parts of $\calP_1[>a]$ and $\calP_2[>b]$ to $\gb_1, \gb_2$ respectively. We have
\begin{align*}
    \gb_1 + \gb_2 &= \gb_1' + \gb_2' +  \eb_{\sigma_k} + 2\eb_{\eta_{h}} - 2\eb_{\sigma_1} - \eb_{\eta_{\ell}} + (\delta_{\calP_1(a) \succ \calP_1(a+1)}-1) \eb_{\calP_1(a)} + \delta_{\calP_2(b) \succ \calP_2(b+1)}\eb_{\calP_2(b)} \\
    \gb_3 &= \gb_1' + \gb_2' + \eb_{\sigma_k} + \eb_{\eta_{h}}  - \eb_{\sigma_1} - \eb_{\eta_{\ell}} + (1-\delta_{\calP_1(a) \succ \calP_1(a+1)}) \eb_{\calP_1(a)} + \delta_{\calP_2(b) \succ \calP_2(b+1)}\eb_{\calP_2(b)}\\
    \gb_{5} + \gb_6 &= \gb_1' + \gb_2' + \eb_{\eta_{\ell-1}} - \eb_{\sigma_{k+1}} + \delta_{\calP_1(a) \succ \calP_1(a+1)} \eb_{\calP_1(a)} + (\delta_{\calP_2(b)\succ \calP_2(b+1)}-1) \eb_{\calP_2(b)}.  \\
\end{align*}

Since the poset for $\tau^{(p)}$ comes with decoration $\tau^-$, we have the $\gb$ vector associated to it is $-\eb_{\sigma_1} + \eb_{\eta_h}$, so we have $\gb_1 + \gb_2 = \eb_{\tau^{(p)}} + \gb_3$. 

 Observe that
\begin{align*}
\left(\gb_5 + \gb_6 \right) - \left( \gb_{1} + \gb_2 \right) = -\eb_{\sigma_k} - 2\eb_{\eta_{h}} + 2\eb_{\sigma_1} + \eb_{\eta_{\ell}} +\eb_{\eta_{\ell-1}} - \eb_{\sigma_{k+1}} + \eb_{\calP_1(a)} - \eb_{\calP_2(b)}
\end{align*}
which exactly matches the $x$-degree of  $Y_{\succeq \sigma_1^{(1)}} Y_{\preceq \eta_h^{(v)}}^{\geq 1} $, as can be verified by using Corollary \ref{cor:y-hat_spokes} two times. 
\end{proof}

Recall that the content of an order ideal $I$ of some poset $\calP_\gamma$, denoted  $\cont(I)$, is the multiset of elements from $T$ where $\tau \in T$ shows up with multiplicity equal to the number of distinct elements of $I$ correspond to $\tau$. In the following proof, concerning $\gamma_1^{(q)}$ and $\gamma_2$, we can avoid the usual strategy by using the skein relation of $\gamma_1$ and $\gamma_2^{(p)}$ and noting that the two situations are related, up to relabeling, by reversing the posets. This also will explain why the $y$-monomial appears on a different term than in the rest of the section. We will employ a similar tactic in the proof of Proposition \ref{prop:IntArcFromTDoubleTag}.

\begin{prop}
\label{prop:typeII_singlytagged_winding_2}
Given $\gamma_1$ and $\gamma_2$ which intersect in a Type 2 intersection, we have \[
x_{\gamma_1^{(q)}}x_{\gamma_2} = y_\tau x_{\gamma_3} x_{\tau^{(q)}} + x_{\gamma_5^{(q)}}x_{\gamma_6}.
\]
\end{prop}

\begin{proof}
Let $\gamma_4:= \tau^{(q)}$ and set $\calP_i:= \calP_{\gamma_i}$, using the versions from the statement of the proposition. For each $i$, let $\overline{\calP_i}$ be the dual poset to $\calP_i$. That is, this is a poset on the same elements where each relation is reversed. One can observe that, up to different labeling conventions,  this set of posets resembles those used in Proposition \ref{prop:typeII_singlytagged_winding_1}. 
Therefore, from the proof of that proposition, there exists a partition $\overline{A} \cup \overline{B} =  J(\overline{\calP_1}) \cup J(\overline{\calP_2})$ and bijections between $\overline{A}$ and $J(\overline{\calP_3}) \cup J(\overline{\calP_4})$ and between $\overline{B}$ and $J(\overline{\calP_5}) \cup J(\overline{\calP_6})$.

In any poset, the complement of an order ideal is an order filter, so the above bijections could instead be phrased in terms of order filters. Now, the order filters of $\overline{\calP_i}$ are exactly the order ideals of $\calP_i$, so we recover a bijection $J(\calP_1) \times J(\calP_2) \to (J(\calP_3) \times J(\calP_4)) \cup (J(\calP_5) \times J(\calP_6))$
 Now we check that this bijection has the desired properties to give the skein relation.

Given any order ideal $\overline{I}_i \in \overline{\calP}_i$, let $I_i$ be its complement viewed as an order ideal in $\calP_i$. Let $A = \{(I_1,I_2) \in J(\calP_1) \times J(\calP_2): (\overline{I_1},\overline{I_2}) \in \overline{A}\}$ and similarly for $B$.  
Given $(\overline{I_1},\overline{I_2}) \in \overline{A}$ with $\overline{\Phi}(\overline{I_1},\overline{I_2}) = (\overline{I_3},\overline{I_4})$, we have $\cont(\overline{I_1} \cup \overline{I_2}) = \cont(\overline{I_3}\cup \overline{I_4})$. Therefore, $\cont(\overline{I_1} \cup \overline{I_2})$ is bounded above by $\cont(\overline{\calP_3} \cup \overline{\calP_4}) = \cont(\overline{\calP_1} \cup \overline{\calP_2})\backslash \{\tau\}$.  Equivalently, $\{\tau\} \subseteq \cont(I_1 \cup I_2)$ for all $(I_1,I_2)  \in A$. The fact that $\overline{\Phi}: \overline{A} \to J(\overline{\calP_3}) \times J(\overline{\calP_4})$ preserved content will mean that there is a bijection $\Phi: A \to J(\calP_3) \times J(\calP_4)$ such that for all $(I_1,I_2) \in A$, if $ \Phi(I_1,I_2) = (I_3,I_4) \in J(\calP_3) \times J(\calP_4)$, then $\cont(I_1 \cup I_2) \backslash \{\tau\} = \cont(I_3 \cup I_4)$. 

Let $S$ be the multiset $\{\sigma_1^w, \sigma_2^w,\ldots,\sigma_k^w, \sigma_{k+1}^{w-1},\ldots,\eta_1^{v-1},\ldots,\eta_{\ell-1}^{v-1},\eta_{\ell}^{v},\ldots,\eta_h^{v}, \tau^{w+v-2}\}$ and notice this multiset corresponds to the indices of the $Y$-monomial in the skein relation in Proposition \ref{prop:typeII_singlytagged_winding_1}. There are some special cases when $\sigma_k = \tau$ or $\eta_\ell = \tau$ but these are handled similarly. Now, given $(\overline{I_1},\overline{I_2}) \in \overline{B}$, we know that $S \subseteq \cont(\overline{I_1} \cup \overline{I_2})$, so $\cont(I_1 \cup I_2)$ is bounded above by $\cont(\calP_1 \cup \calP_2) \backslash S = \{\tau,\eta_1,\ldots,\eta_h\}$. Here we are using a relabeling, since in the case where $\gamma_2$ is notched, the set $S$ contains both chains of spokes but not the loop $\sigma_1 \prec \cdots \prec \sigma_m \prec \tau$, and then we relabel these as $\eta_i$. Since $(\overline{\calP_1},\overline{\calP_2}) \in \overline{B}$, we know $(\emptyset, \emptyset) \in B$. We can conclude that $\Phi$ is a weight-preserving bijection between $B$ and $J(\calP_5) \times J(\calP_6)$. 

Now, we have shown the desired bijection (notice the sets $A$ and $B$ swapped roles with the duality). The $\gb$-vector computations would also follow a dual version of those in Proposition \ref{prop:typeII_singlytagged_winding_1}, so we omit this.
\end{proof}

We turn to the final and most complicated scenario of this section, which is where $s(\gamma_1)$ and $s(\gamma_2)$ are both notched.  

\begin{prop}
\label{prop:typeII_bothTagged}
Given $\gamma_1$ and $\gamma_2$ as previously described, let $\gamma_3 = \tau \circ \gamma_1$, $\gamma_4 = \tau \circ \gamma_2$, and $\gamma_5 = \gamma_2^{-1} \circ \gamma_1$, where $\gamma_5$ is spliced at the intersection point of $\gamma_1$ and $\gamma_2$ in their mutual last triangle. Then,
\begin{align*}
x_{\gamma_1^{(q)}}x_{\gamma_2^{(p)}} &= x_{\tau^{(p,q)}}x_{\gamma_5} + Y_{\succeq \sigma_1^{(1)}} x_{\gamma_5^{(q)}}x_{\gamma_6^{(p)}} + 
\end{align*}
where \[
Y_{\succeq \sigma_1^{(1)}} = \prod_{\calP_1(i) \succeq \sigma_1^{(1)}} y_{\calP_1(i)}.
\]
\end{prop}

\begin{proof}
In this case, $\calP_4=\mathcal{P}_{\tau^{(p,q)}}$, which is shown below using the notation from this section. 

    \begin{center}
    \begin{tikzpicture}[scale = 0.7]
        \node[](taubot) at (0,0) {$\tau^{-}$};
        \node[](sigma1) at (-1,1) {$\sigma_1$};
        \node[](dots1) at (-1,2.15) {$\vdots$};
        \node[](sigman) at (-1,3) {$\sigma_m$};
        \node[](eta1) at (1,1) {$\eta_1$};
        \node[](dots2) at (1,2.15) {$\vdots$};
        \node[](etam) at (1,3) {$\eta_h$};
        \node[](tautop) at (0,4) {$\tau^{+}$};
        \draw (taubot) to (sigma1);
        \draw (taubot) to (eta1);
        \draw (tautop) to (sigman);
        \draw (tautop) to (etam);
        \draw (sigman) to (eta1);
        \draw (etam) to (sigma1);
        \draw (sigman) to (-1,2.35);
        \draw (dots1) to (sigma1);
        \draw (etam) to (1,2.35);
        \draw (dots2) to (eta1);
    \end{tikzpicture}
    \end{center}

    We set $\sigma_0 = \eta_0 = \tau$. We will sometimes denote the minimal element $\tau$ in $\calP_1$ as $\tau^{(1)}$ or $\sigma_0^{(1)}$ and similarly the maximal element $\tau$ in $\calP_2$ as $\tau^{(v+1)}$ or $\eta_0^{(v+1)}$.
    
    The partitions we use will depend on whether $w = 1$ or $w > 1$ and similarly for $v$. We begin with the most general case, following the usual proof strategy.

    \textbf{Case i) $w > 1, v> 1$.}

    We partition $J(\calP_1) \times J(\calP_2)$ and define maps from each part to $J(\calP_3) \times J(\calP_4)$ or $ J(\calP_5) \times J(\calP_6)$ as follows. 
    
    \begin{enumerate}
        \item $A_1$ consists of pairs $(I_1,I_2) \in J(\calP_1) \times J(\calP_2)$ such that if $\eta_h \in I_1$, then $\sigma_0^{(2)} \in I_1$ or $\sigma_1 \in I_2$; if $y=m$, then $c \neq 0$; if $c = h, y \neq 0$; $\tau^{(v+1)} \notin I_2$; $d \geq 2$; and, $z \leq w-1$.  Given $(I_1,I_2) \in A_1$, $\Phi(\langle \sigma_y^{(z)} \rangle \cup \langle \eta_f \rangle, \langle \eta_c^{(d)} \rangle \cup \langle \sigma_x \rangle) = (\langle \sigma_x^{(z)} \rangle \cup \langle \eta_c \rangle, \langle \eta_f^{(d)} \rangle \cup \langle \sigma_y \rangle)$, so that $\Phi(A_1)$ consists of all $(I_3,I_4) \in J(\calP_3) \times J(\calP_4)$ satisfying if $c = h$, then $\sigma_1 \in I_4$ or $\sigma_0^{(2)} \in I_3$; if $I_4 = \emptyset$, then $z = 1$; if $d = v$ and $f = h$, then $y \neq 0$; $\tau^+ \notin I_4$; $d \geq 2$; and, $z \leq w-1$. 
        \item $A_2$ consists of pairs $(I_1,I_2)$ satisfying $y = m$; $c = 0$; and, $z \leq w-1$ with equality only if $x \leq k$ and this latter is only equality if $\calP_1(a) \in I_1$. Given $(I_1,I_2) \in A_2$, $\Phi(\langle \sigma_m^{(z)} \rangle \cup \langle \eta_f \rangle, \langle \sigma_x \rangle \cup \eta_0^{(d)} \rangle) = (\langle \sigma_x^{(z)} \rangle \cup \langle \eta_f^{(d)} \rangle)$ so that $\Phi(A_2)$ consists of $(I_3,I_4)$ such that $I_4 = \emptyset$ and $z \geq 2$. 
        \item $A_3$ consists of pairs $(I_1,I_2)$ satisfying $d = 1$; $z \leq w-1$; if $f = h$, then $y \neq 0$; and, if $y = m$, then $f \neq 0$. Given $(I_1,I_2) \in A_3$, $\Phi(\langle \sigma_y^{(z)} \rangle \cup \langle \eta_f \rangle, \langle \sigma_x \rangle \cup \eta_c^{(1)} \rangle) = (\langle \sigma_x^{(z)} \rangle \cup \langle \eta_c^{(1)} \rangle, \langle \sigma_y \rangle \cup \langle \eta_f \rangle)$ so that $\Phi(A_3)$ consists of $(I_3,I_4)$ satisfying $d = 1$; $z \leq w-1$; and, $I_4 = \emptyset$. 
        \item $A_4$ consists of pairs $(I_1,I_2)$ satisfying $d \geq 2$; $z = w$; if $x = w$, then $c \neq 0$; and, if $c = h$, then $x \neq 0$. Given $(I_1,I_2) \in A_4$, $\Phi(\langle \sigma_y^{(w)} \rangle \cup \langle \eta_f \rangle, \langle \sigma_x \rangle \cup \langle \eta_c^{(d)} \rangle) = (\langle \sigma_y^{(w)} \rangle \cup \langle \eta_f^{(d)} \rangle, \langle \sigma_x \rangle \cup \langle \eta_c \rangle)$ so that $\Phi(A_4)$ consists of $(I_3,I_4)$ satisfying $d \geq 2$; $z = w$; and, $I_4 \neq 0$. 
        \item $A_5$ consists of pairs $(I_1,I_2)$ such that $\sigma_y^{(z)} \succeq \sigma_m^{(w-1)}$ and $\eta_c^{(d)} \preceq \eta_0^{(2)}$ where (1) if $\sigma_y^{(z)} = \sigma_m^{(w-1)}$ and $d =1$, then $\sigma_x \preceq \sigma_k$ with equality only if $\calP_1(a) \in I_1$, (2) if $\sigma_y^{(z)} = \sigma_m^{(w-1)}$ and $\eta_c^{(d)} = \eta_0^{(2)}$, then $\sigma_x \succeq \sigma_k$, which is strict if $\calP_1(a) \in I_1$, and (3) if $z = w$ and $d = 1$, then $x = m$ only if $\eta_1 \in I_2$, $f = h$ only if $\sigma_1 \in I_1$, and $y =0$ and $c = h$ only if  $\sigma_x \preceq \sigma_k$ with equality only if $\calP_1(a) \in I_1$, and (4) we do not allow $d = 2$ and $z = w$. Given $(I_1,I_2) \in A_5$, we have \[
        \Phi(\langle \sigma_y^{(z)} \rangle \cup \langle \eta_f \rangle, \langle \eta_c^{(d)} \rangle \cup \langle \sigma_x \rangle) = \begin{cases} 
        (\langle \sigma_x^{(w)} \rangle \cup \langle \eta_c^{(1)} \rangle, \emptyset) & z = w-1 \text{ and } d = 1\\
        (\langle \sigma_0^{(w)} \rangle \cup \langle \eta_h^{(1)} \rangle, \langle \sigma_x \rangle) & z = w-1 \text{ and } d = 2\\
        (\langle \sigma_y^{(w)} \rangle \cup \langle \eta_c^{(1)}\rangle, \langle  \eta_f \rangle \cup \langle \sigma_x \rangle) & z = w \text{ and } d = 1.\\
        \end{cases}
        \]
        We have $\Phi(A_5)$ consists of $(I_3,I_4)$ satisfying $d = 1$; $z = w$; and, $\tau^+ \notin I_4$.
        \item $A_6$ consists of pairs $(I_1,I_2)$ satisfying $y = 0$; if $d = 1$, then $f  = h$; and, if $d > 1$, then $c = h$. Given $(I_1,I_2) \in A_6$, we have \[
            \Phi(\langle \sigma_0^{(z)} \rangle \cup \langle \eta_f \rangle, \langle \eta_c^{(d)} \rangle \cup \langle \sigma_x \rangle) = \begin{cases}
            (\langle \eta_c^{(1)} \rangle \cup \langle \sigma_x^{(z-1)} \rangle, \langle \tau^+ \rangle) & d = 1 \\
            (\langle \eta_f^{(d)} \rangle \cup \langle \sigma_x^{(z-1)} \rangle, \langle \tau^+ \rangle) & d > 1. \\
        \end{cases}\]
        We have $\Phi(A_6)$ consists of $(I_3,I_4)$ such that $\tau^+ \in I_4$; $z < w$; and, if $d = v$, then $y \neq 0$. 
        \item $A_7$ consists of pairs $(I_1,I_2)$ such that $x = m$; $z = w$; $c = 0$; and, if $d = 2$, then $f \geq \ell -1$, which is strict if $\calP_2(b) \in I_2$. Given $(I_1,I_2) \in A_7$, we have $\Phi(\langle \sigma_y^{(w)} \rangle \cup \langle \eta_f \rangle, \langle \eta_0^{(d)} \rangle \cup \langle \sigma_m \rangle) = (\langle \eta_f^{(d-1)} \rangle \cup \langle \sigma_y^{(w)} \rangle, \langle \tau^+ \rangle)$ so $\Phi(A_7)$ consists of $(I_3,I_4)$ satisfying $\tau^+ \in I_4$; $z = w$; and, $d < v$.
        \item $A_8$ consists of pairs $(I_1,I_2)$ such that $\tau^{(v+1)} \in I_2$. Given $(I_1,I_2) \in A_8$, we have \[
        \Phi(\langle \sigma_y^{(z)} \rangle \cup \langle \eta_f \rangle, \langle \tau^{(v+1)} \rangle) = \begin{cases} 
        (\langle \eta_f^{(v)} \rangle \cup \langle \sigma_y^{(z)} \rangle, \langle \tau^+ \rangle) & y =0 \text{ or } z = w\\
        (\langle \eta_f^{(v)} \rangle \cup \langle \sigma_0^{(z+1)} \rangle, \langle \sigma_y \rangle \cup \langle \eta_h \rangle) & y \neq 0 \text{ and } z<w. \end{cases}
        \]
        We have $\Phi(A_8)$ consists of $(I_3,I_4)$ such that ($\tau^+ \in I_4$; $d = v$; and, $y = 0$), ($\tau^+ \in I_4$; $d = v$; and, $z = w$), or ($d = v$; $f = h$; and, $y = 0$).
        \item $A_9$ consists of pairs $(I_1,I_2)$ such that $c =h$ and $\sigma_1^{(1)} \notin I_1$. Given $(I_1,I_2) \in A_9$, $\Phi( \langle \eta_f \rangle, \langle \eta_h^{(d)} \rangle \cup \langle \sigma_x \rangle) = (\langle \eta_h^{(d)} \rangle \cup \langle \sigma_x^{(1)} \rangle, \langle \eta_f \rangle)$ so that $\Phi(A_9)$ consists of $(I_3,I_4)$ with $c = h$; $\sigma_0^{(2)} \notin I_3$; and, $\sigma_1 \notin I_4$. 
        \item $B$ consists of pairs $(I_1,I_2)$ such that  $x \geq k$ which is strict if $\calP_1(a) \in I_1$; $\eta_c^{(d)} \preceq \eta_0^{(2)}$ which is equality only if $\sigma_0^{(w)} \in I_1$; $f \leq \ell - 1$ which is equality only if $\calP_2(b) \in I_2$;  $\sigma_y^{(z)} \succeq \sigma_m^{(w-1)}$ which is only strict if $\sigma_m \in I_2$ or $\eta_h^{(1)} \in I_2$; $\sigma_1^{(w)} \in I_1$ only if $\sigma_m \in I_2$; $\eta_1 \in I_1$ only if $\eta_h^{(1)} \in I_2$; and, $\{\sigma_1^{(w)}, \eta_1\} \subset I_1$ only if $\eta_0^{(2)} \in I_2$. There is a straightforward bijection between this set and $J(\calP_5) \times J(\calP_6)$ where the content of the image is the same as the content of $(I_1 \backslash \langle \sigma_m^{(w-1)} \rangle) \cup (I_2 \backslash \langle \sigma_k \rangle)$. 
    \end{enumerate}

    \textbf{Case ii) $w = 1, v > 1$.}

We now describe how to adjust the sets and maps in the general case to this case. Let $A_1$ be the analogue of $A_4$ and $A_2$ be the analogue of $A_7$ from the general case, and define $\Phi$ in a similar way on these.  Here, $\Phi(A_1)$ consists of all $(I_3,I_4) \in J(\calP_3) \times J(\calP_4)$ such that $\tau^+ \notin I_4$; if $\sigma_1^{(1)} \in I_3$ then $\tau^- \in I_4$; and, if $c = h$, then $\sigma_1^{(1)} \in I_3$, and $\Phi(A_2)$ consists of all $(I_3,I_4)$ such that $\tau^+ \in I_4$. The set $B$ in this case is obtained by interpreting the set from the general case in an appropriate way for $w = 1$. The remaining parts of $J(\calP_1) \times J(\calP_2)$ can be divided in a simpler way than in the general case.
    
    \begin{enumerate}
        \item[3.] $A_3$ consists of pairs $(I_1,I_2)$ such that $\tau^{(1)} \notin I_1$; $\sigma_x$ exists; and, $x \leq k$ with equality only if $\calP_1(a) \in I_1$. Given such a pair, we set $\Phi(\emptyset, \langle \sigma_x \rangle \cup \langle \eta_c^{(d)} \rangle) = (\langle \sigma_x^{(1)} \rangle \cup \langle \eta_c^{(d)} \rangle, \emptyset)$ so that the image of $A_3$ is all pairs $(I_3,I_4)$ with $\sigma_1^{(1)} \in I_3$ and $I_4 = \emptyset$. 
        \item[4.] $A_4$ consists of pairs $(I_1,I_2)$ such that $\tau^{(1)} \notin I_1$; $x \geq k$ which is strict if $\calP_1(a) \in I_1$; and, $d > 1$. Given such a pair, we set $\Phi(\emptyset, \langle \sigma_x \rangle \cup \langle \eta_c^{(d)} \rangle) = (\langle \eta_h^{(d-1)} \rangle, \langle \eta_f \rangle \cup \langle \sigma_x \rangle)$ so that the image of $A_4$ is all pairs $(I_3,I_4)$ such that $c = h$; $\sigma_1^{(1)} \notin I_3$; and, $x \geq k $ which is strict if $\calP_1(a) \in I_3$.
        \item[5.] $A_5$ consists of pairs $(I_1,I_2)$ such that $\sigma_x$ does not exist and $c = h$. Given such a pair, we set $\Phi(\langle \eta_f \rangle \cup \langle \sigma_x^{(1)} \rangle, \langle \eta_h^{(d)} \rangle) = (\langle \eta_h^{(d)} \rangle, \langle \eta_f \rangle \cup \langle \sigma_x \rangle)$ so that the image of $A_5$ is all pairs $(I_3,I_4)$ such that $c = h$; $\sigma_1^{(1)} \notin I_3$; and, $x \leq k $ which is only equal if $\calP_1(a) \in I_3$.
    \end{enumerate}

    \textbf{All other cases)}

    Adjusting to $w > 1$ and $v =1$ can be done similarly to adjusting from Case i to Case ii. When $w = 1$ and $v = 1$, one can use a simpler partitioning, which we omit here. One can see a similar strategy used in the proof of Proposition \ref{prop:IntArcFromTDoubleTag} where the full details are provided.

    The $\gb$-vector computations are straightforward and need not be broken into these cases, so we omit them.
\end{proof}

\subsection{Intersection with arcs in the triangulation}\label{subsec:IntersectWithT}

In this section, we deal with intersections between $\gamma \notin T$ and $\tau$ such that $\tau^0 \in T$. We can recognize these intersections by the presence of $\tau^0$ in the poset $\calP_\gamma^0$; if $\tau^0$ appears in the loop portion of $\calP_\gamma$, then this is an incompatibility which was dealt with in Section \ref{sec:puncture_incompatibility}. 

The resolution of $\gamma$ and $\tau$ will result in two multicurves, each containing two curves. If we ignore possible notchings, then we can describe the posets of these four curves as the following. 

Let $\calP_{\low}^{<i} = \calP_{\gamma}[<i] \backslash \calP_{\gamma}^{\succeq \calP_{\gamma}(i)}$ and $\calP_{\low}^{>i} = \calP_{\gamma}[>i] \backslash \calP_{\gamma}^{\succeq \calP_{\gamma}(i)}$. 
Define $\calP_{\upper}^{<i}$ and $\calP_{\upper}^{>i}$ similarly, using ``$\preceq \calP_{\gamma}(i)$'' instead in the superscripts. If $\{\gamma_5,\gamma_6\}$ is the multicurve with a $Y$-coefficient in the expansion of $x_\gamma x_\tau$, we will either have $\calP_{\gamma_3} = \calP_{\low}^{<i}, \calP_{\gamma_4} = \calP_{\low}^{>i}, \calP_{\gamma_5} = \calP_{\upper}^{<i}$ and $\calP_{\gamma_6} = \calP_{\upper}^{>i}$ or the roles of lower and upper will be swapped. This will depend on the notching on $\tau$ and the local configuration of $\gamma$ near its intersection point with $\tau$.

Throughout this section, let $i$ be such that $\calP_\gamma(i) = \tau^0$. We define a set $F$ which will determine most of the $y$-monomial in the relations presented here. Begin with $F = \emptyset$. For both possible signs, we decide whether to add certain sets to $F$. There are several cases of $\calP_\tau$ based on the possible notchings. 

\begin{itemize}
    \item If $\calP_\gamma(i \pm 1) \prec \calP_\gamma(i)$, $\calP_\gamma(i\pm 1) \in \calP_\gamma^0$, and $\calP_\gamma(i \pm 1) \notin \calP_\tau$, include the order ideal generated by $\calP_\gamma(i \pm 1)$ to $F$.
    \item  If $\calP_\gamma(i \pm 1) \succ \calP_\gamma(i)$, $\calP_\gamma(i\pm 1) \in \calP_\gamma^0$, and $\calP_\gamma(i \pm 1) \in \calP_\tau$, include the order filter generated by $\calP_\gamma(i \pm 1)$ to $F$.
\end{itemize}

The following Lemma is used implicitly throughout this section.

\begin{lemma}
For any pair $(\gamma,\tau)$ such that $e(\gamma,\tau) > 0$, $F$ is either empty, an order ideal, or an order filter.
\end{lemma}

\begin{proof}
Assume $F \neq \emptyset$. Up to reindexing the chronological ordering,  $F$ would fail to be an order ideal or order filter  if we have $\calP_\gamma(i-1) \prec \calP_\gamma(i) \prec \calP_\gamma(i+1)$, $\calP_\gamma(i-1) \notin \calP_\tau$, and $\calP_\gamma(i+1) \in \calP_\tau$. Clearly, the three arcs $\calP_\gamma(i-1),\calP_\gamma(i),\calP_\gamma(i+1)$ must share a vertex, call it $p$. Since $\calP_\gamma(i+1) \in \calP_\tau$, it must be that $\tau$ is notched at $p$. However, then $\calP_\gamma(i-1) \in \calP_\tau$ as well, a contradiction. 
\end{proof}

First, let $\tau = \tau^0$ so that $\tau$ is a plain arc. This type of intersection was also studied by Pilaud, Reading, and Schroll in \cite{pilaud2023posets}. Here $\calP_\tau = \emptyset$ and $F \cup \{\calP_\gamma(i)\} = \langle \calP_{\gamma}(i) \rangle$. In this case, we have that $\calP_{\gamma_3} = \calP_{\low}^{<i}, \calP_{\gamma_4} = \calP_{\low}^{>i}, \calP_{\gamma_5} = \calP_{\upper}^{<i}$ and $\calP_{\gamma_6} = \calP_{\upper}^{>i}$. Throughout, we set $Y_F = \prod_{\calP_\gamma(j) \in F} y_{\calP_\gamma(j)}$.

\begin{prop}\label{prop:IntArcFromTSingleTag}
Let $\tau \in T$ and let $\gamma$ be an arc such that $e(\gamma,\tau) > 0$. Pick $i$ such that $\calP_{\gamma}(i) = \tau$. Let $Y_{\langle \calP_\gamma(i) \rangle} = \prod_{\calP_\gamma(j) \in \langle \calP_\gamma(i)\rangle} y_{\calP_\gamma(j)}$. Then,\[
x_\gamma x_\tau = x_{\gamma_3}x_{\gamma_4} + Y_{F}y_\tau x_{\gamma_5} x_{\gamma_6} = x_{\gamma_3}x_{\gamma_4} + Y_{\langle \calP_\gamma(i) \rangle} x_{\gamma_5} x_{\gamma_6}.
\]
\end{prop}

\begin{proof}
We describe a bijection $\Phi: J(\calP_\gamma) \to (J(\calP_{\gamma_3}) \times J(\calP_{\gamma_4})) \cup (J(\calP_{\gamma_5}) \times J(\calP_{\gamma_6}))$ as follows. Let $A$ be the subset of $J(\calP_\gamma)$ given by all order ideals which do not contain $\calP_\gamma(i)$ and let $B$ be the subset of all order ideals containing $\calP_{\gamma}(i)$. We see that an element of $A$ restricts to order ideals of  $\calP_{\gamma_3}$ and $\calP_{\gamma_4}$ while if $I \in B$, $I \backslash \langle \calP_\gamma(i) \rangle$ restricts to an order ideal of $\calP_{\gamma_5}$ and $\calP_{\gamma_6}$. Combining these maps gives the desired bijection, with a clear inverse map. 

There are several cases to consider when comparing the $\gb$-vectors. We provide one case in-depth; the others are similar or simpler, and many follow from skein relations in the unpunctured case. Suppose that $\calP_\gamma(1) \prec \calP_\gamma(i) \prec \calP_\gamma(i+1)$ so that in particular $\calP_\gamma(i)$ is neither minimal nor maximal. Suppose moreover that $s(\gamma)$ is notched, so that there is a loop $\calP_\gamma(1) \succ \calP_\gamma(0) \prec \calP_\gamma(-1) \prec \cdots \prec \calP_\gamma(-m) \succ \calP_\gamma(1)$.  

Recall the $\gb$-vector for $\tau \in T$ is $\mathbf{e}_\tau$. We verify that $\gb_{\gamma_3} + \gb_{\gamma_4} = \gb_\gamma + \gb_\tau$. We have a summand $\mathbf{e}_\tau$ in the vector $\mathbf{r}_{\gamma_3}$. The maximal element closest to $\calP_{\gamma}(i)$, $\calP_\gamma(j)$, does not appear in either $\calP_{\gamma_3}$ or $\calP_{\gamma_4}$, but we have a summand $\mathbf{e}_{\calP_\gamma(j)}$ in $\mathbf{r}_{\gamma_4}$. All other terms appear for the same reasons in $\mathbf{g}_\gamma$ and $\mathbf{g}_{\gamma_3} + \mathbf{g}_{\gamma_4}$. 

Now we compare $\gb_\gamma + \gb_\tau$ and $\gb_{\gamma_5} + \gb_{\gamma_6}$. The arc $\gamma_5$ as defined will be $\calP_\gamma(0)^{(s(\gamma))}$. Therefore, $\gb_{\gamma_5} = \mathbf{e}_\alpha - \mathbf{e}_{\calP_\gamma(-1)}$ where $\alpha$ is the third arc in the triangle formed by $\calP_{\gamma}(0)$ and $\calP_{\gamma}(-1)$. Let $\beta$ be the third arc in the triangle formed by $\calP_{\gamma}(i)$ and $\calP_{\gamma}(i+1)$. Then, $\gb_{\gamma_6} = \mathbf{e}_{\beta} - \mathbf{e}_{\calP_{\gamma}(i+1)} + \gb'$ where $\mathbf{g'}$ consists of all contributions to $\mathbf{g}$ from values $\calP_\gamma[>i]$. There are implicitly two cases: $\calP_\gamma(i+1) \succ \calP_\gamma(i+2)$ and $\calP_\gamma(i+1) \prec \calP_\gamma(i+2)$. 

From Lemma \ref{lem:y-hat}, we know that \[
\degx(\hat{Y}_{\langle \calP_{\gamma}(i) \rangle}) = \degx(\prod_{j=0}^{i-1} \hat{y}_{\calP_\gamma(j)}) =\mathbf{e}_{\alpha} + \mathbf{e}_\beta - \mathbf{e}_{\calP_{\gamma(-1)}} - \mathbf{e}_{\calP_{\gamma(i+1)}}  + \mathbf{e}_{\calP_{\gamma}(0)}  - \mathbf{e}_{\calP_{\gamma(i)}}, 
\] and we can conclude that this agrees with $(\gb_{\gamma_5} + \gb_{\gamma_6}) - (\gb_\gamma + \gb_\tau)$.
\end{proof}

Now we consider an intersection with $\gamma$ and $\tau = \tau^{(p)}$ singly-notched. Let $\alpha,\beta,\sigma_1,$ and $\sigma_m$ be the arcs in the quadrilateral around $\tau$ in $T$, written in clockwise order such that $\sigma_1$ and $\sigma_m$ are incident to the notched endpoint $p$. The elements in $\calP_\tau$ will be $\{\sigma_1,\ldots,\sigma_m\}$.

\begin{center}
\begin{tikzpicture}[scale = 1.5]
\draw (0,0) to node[left]{$\beta$} (1,1)  to node[right]{$\alpha$}  (2,0) to node[right]{$\sigma_m$} (1,-1) to node[left]{$\sigma_1$} (0,0);
\draw (1,1) to node[right]{$\tau$} node[pos=0.85,sloped,rotate=90]{$\bowtie$} (1,-1);
\node[below] at (1,-1){$p$};
\end{tikzpicture}
\end{center}

In order to describe the posets that arise in the various cases here, we introduce a term. Suppose there is a puncture $p$ with spokes $\sigma_1,\ldots,\sigma_m$, and let $\gamma$ be an arc such that the first arc crossed by $\gamma$ forms a triangle with two consecutive spokes $\sigma_k$ and $\sigma_{k+1}$. Then, we say that $\calP_\gamma$ has \emph{a loop at the beginning based at $p$}. We could similarly say a poset has \emph{a loop at the end based at $p$}. Trivially, a poset that only consists of a loop both begins and ends at a loop.

Orient $\tau$ from its notched endpoint to its plain endpoint. We know that, up to a choice of orientation, either $\gamma$ crosses $\sigma_1$ immediately before crossing $\tau$, $\gamma$ crosses $\beta$ immediately before crossing $\tau$, or $\tau$ is the first arc which $\gamma$ crosses, so that $i = 1$. If we are in the first case or $\tau$ is the first arc which $\gamma$ crosses and $s(\gamma)$ is plain, we set $\gamma_3 = \gamma \circ \tau$, $\gamma_4 = \tau\circ \gamma$, $\gamma_5 = \gamma \circ \tau^{-1}$ and $\gamma_6 = \tau^{-1} \circ \gamma$ where the resolution is  $\{\gamma_3,\gamma_4\} \cup \{\gamma_5,\gamma_6\}$. The corresponding posets are as follows.
\begin{itemize}
    \item $\calP_{\gamma_3} = \calP_{\upper}^{<i}$.
    \item $\calP_{\gamma_4}^0 = (\calP_{\upper}^{>i})^0$ and $\calP_{\gamma_4}$ has a loop at the beginning based at $p$. 
    \item $\calP_{\gamma_5}^0 = (\calP_\low^{<i})^0$ and $\calP_{\gamma_5}$ has a loop at the end based at $p$. 
    \item $\calP_{\gamma_6} = \calP_\low^{>i}$.
\end{itemize}

It is possible for $(\calP_{\upper}^{>i})^0$ or $(\calP_\low^{<i})^0$ to be empty, in which case $\gamma_4$ or $\gamma_5$ would be notched versions of arcs from $T$ and $\calP_{\gamma_4}$ or $\calP_{\gamma_5}$ will be decorated posets.

If $\gamma$ crosses $\beta$ immediately before crossing $\tau$, or $\tau$ is the first arc which $\gamma$ crosses and $s(\gamma)$ is notched, then we flip the definition of $\gamma_3$ and $\gamma_5$ and flip $\gamma_4$ and $\gamma_6$. This is to reflect which term in the expansion is multiplied by a monomial in $y$-variables.

In this case, $F$ will either be an order ideal or an order filter based on which arcs near $\tau$ $\gamma$ crosses. The $y$-monomial will also possibly have an extra factor, depending on whether $\gamma_5^0$ or $\gamma_6^0$ is in $T$ and is notched at $p$. 

\begin{prop}
Let $\tau \in T$ and let $\gamma$ be an arc such that $e(\gamma,\tau) > 0$. Pick $i$ such that $\calP_{\gamma}(i) = \tau$. Set $Y_F = \prod_{\calP_{\gamma}(j) \in F} y_{\calP_\gamma(j)}$. If $\calP_{\gamma_5}^0 = (\calP_{\low}^{<i})^0$ and this set is empty, set $\delta_5 = 1$. Similarly if $\calP_{\gamma_6}^0 = (\calP_\upper^2)^0$ and this set is empty, set $\delta_6 = 1$. If $\delta_i$ is not yet specified, set $\delta_i = 0$. Then,
\[
x_\gamma x_{\tau^{(p)}} = x_{\gamma_3}x_{\gamma_4} + Y_F y_{\gamma_5^0}^{\delta_5} y_{\gamma_6^0} ^{\delta_6} x_{\gamma_5}x_{\gamma_6}
\]
\end{prop}

\begin{proof}

Throughout let $\calP_1:= \calP_\gamma, \calP_2:= \calP_{\tau^{(p)}}$, and for $3 \leq i \leq 6$, $\calP_i:= \calP_{\gamma_i}$.

\textbf{Case i) $\gamma$ crosses $\beta$ immediately before crossing $\tau$}.

In this case, there will never be $j$ such that $\calP_\gamma(j) \succ \calP_\gamma(i)$ and $\calP_\gamma(j) \in \calP_\tau$, so $F$ will be $F = \langle \beta \rangle$. 

\textbf{Subcase a) $\gamma$ crosses $\beta$; $\tau$ is the last arc crossed by $\gamma$; and, $t(\gamma)$ is plain.} 
 
Here $\gamma_6 = \sigma_m^{(p)}$. In this case, $\calP_{6}^0 = (\calP_\upper^{>i})^0 = \emptyset$, so that $\delta_6 = 1$. We partition $J(\calP_1) \times J(\calP_2) = A_1 \cup A_2 \cup B$ where  $A_1$ is the set of pairs $(I_1,I_2)$ such that $\tau \notin I_1$ and $\beta \in I_1$ only if $\sigma_1 \in I_2$, $A_2$ is the set of pairs such that $\tau \in I_1$ and $\sigma_m \in I_2$, and $B$ is the set of pairs such that $\beta \in I_1$, $\sigma_m \notin I_2$, and $\sigma_1 \in I_2$ only if $\tau \in I_1$.

We see that $A_1 \cup A_2$ is in bijection with $J(\calP_{3})$, where we send each element from $\calP_1 \times \calP_2$ to its image in $\calP_{\gamma_3}$ (note that $\widetilde{\calP}_4 = (\emptyset,\alpha)$ here), and $B$ is in bijection with $J(\calP_5) \times J(\calP_6)$ where $(I_1 \backslash \langle \beta \rangle) \cup I_2$ restricts to order ideals of each.

For the $\gb$-vectors, suppose $n = \vert \calP_1^0 \vert$, so that $\beta = \calP_1(n-1)$. Set $\gb_i$ to correspond to $\calP_i$. We have $\gb_1 = \eb_\alpha + \delta_{\calP_1(n-1) \succ \calP_1(d-2)}\eb_\beta + \gb_1'$, $\gb_2 = -\eb_{\sigma_1} + \eb_\beta$, $\gb_3 = \delta_{\calP_1(n-1) \succ \calP_1(d-2)} \eb_\beta - \eb_{\sigma_1} + \gb_1'$ and $\gb_4 = \eb_\alpha$ where $\gb_1'$ records the contributions of $\calP_1[<n-1]$ to $\gb_1$. We immediately see $\gb_1 + \gb_2 = \gb_3 + \gb_4$.

Let $x$ be the largest integer such that $\calP_1(x+1)$ is minimal in $\calP_1$. If $x+1 > 1$, let $\mu$ be the third arc in the triangle formed by $\calP_1(x)$ and $\calP_1(x+1)$. Then we have $\degx(\hat{Y}_F) = \eb_{\calP_1(x+1)} - \eb_{\calP_1(x)} + \eb_\mu - \eb_\beta - \eb_\tau + \eb_{\sigma_1}$. If $x+1 = 1$, replace $\mu$ ($\calP_1(x)$) with the counterclockwise (clockwise) neighbor of $\calP_1(x+1)$ in the first triangle $\gamma$ passes through.  We have $\gb_{\gamma_5} = -\eb_{\calP_1(x)} + \eb_\mu + \gb_1''$ where $\gb_1''$ now refers to the contribution of $\calP_1[\leq x]$ to $\gb$, and $\gb_{\gamma_6} = \gb_{\sigma_m^{(p)}} = \eb_\alpha - \eb_{\tau}$. If we rewrite $\gb_1$ where we only set aside $\gb_1''$,  we can see that $\gb_1 + \gb_2 + \degx(\hat{Y}_F) = \gb_5 + \gb_6$.

 \textbf{Subcase b) $\gamma$ crosses $\beta$; $\tau$ is the last arc crossed by $\gamma$; and, $t(\gamma)$ is notched.} 

Notice here we have $\gamma_4 = \alpha^{(t(\gamma))}$ and $\gamma_6 = \sigma_m^{(p,t(\gamma))}$. We already know $\sigma_m$ and $\alpha$ are spokes at $t(\gamma)$. Let the remaining spokes be labeled $\eta_1,\ldots,\eta_h$. In particular, the loop at the end of $\calP_1$ is $\tau \succ \sigma_m \prec \eta_1 \prec \cdots \prec \eta_h \prec \alpha \succ \tau$ where $\tau \in (\calP_1)^0$, and the poset $\calP_4$ is the chain $\sigma_m \prec \eta_1 \prec \cdots \prec \eta_h$.  

\begin{enumerate}
\item We define $A_1$ in the same way as the previous case, and given $(I_1,I_2) \in A_1$, we map these to a pair $(I_3,I_4) \in J(\calP_3) \times J(\calP_4)$ where elements from the loop in $\calP_1$ are sent to the loop in $\calP_4$  elements from the loop in $\calP_2$ are sent to the loop in $\calP_3$. The image of $A_1$ is all pairs $(I_3,I_4) \in J(\calP_3) \times J(\calP_4)$ such that $\tau \notin I_3$. 
\item Let $A_2$ denote the subset of $J(\calP_1) \times J(\calP_2)$ such that $\tau \in I_1$; $\alpha \notin I_1$; $\sigma_{m-1} \in I_2$; and, $\eta_1 \in I_1$ only if $\sigma_m \in I_2$. Pairs $(I_1,I_2) \in A_2$ can be naturally mapped to pairs $(I_3,I_4) \in J(\calP_3) \times J(\calP_4)$ such that $\tau \in I_3$. This completes our bijection between $A_1 \cup A_2$ and $J(\calP_3) \times J(\calP_4)$.
\item Let $B$ be the complement of $A_1 \cup A_2$ in $J(\calP_1) \times J(\calP_2)$ and notice that $B$ can be characterized as all pairs $(I_1,I_2)$ such that $\langle \beta \rangle \subseteq I_1$; $\sigma_1 \in I_2$ only if $\tau \in I_1$; $\sigma_{m-1} \in I_2$ only if $\eta_1 \in I_1$; and, $\sigma_m \in I_2$ only if $\alpha \in I_1$. There is a natural bijection between $B$ and $J(\calP_5) \times J(\calP_6)$ which completes our overall bijection. 
The $\gb$-vector computation is similar enough to the plain case that we suppress it. 
\end{enumerate}

\textbf{Subcase c) $\gamma$ crosses $\beta$ immediately before and $\alpha$ immediately after crossing $\tau$}. 

We will assume $\delta_6 = 0$; if this is not the case, one can combine the following bijections with the ideas from this first case. We also automatically know $\delta_5 = 0$ because $\calP_{5} = \calP_\upper^{<i}$.

To be consistent with our description of this case, we choose a chronological ordering such that $\calP_1(i+1) = \alpha$ and $\calP_1(i-1) = \beta$. Therefore, the set $F$ in this case is $\langle \calP_1(i-1) \rangle = \langle \beta \rangle$. Even though $\gamma$ could cross $\beta$ and $\alpha$ multiple times, below we write $\beta$ and $\alpha$ to signify the entries $\calP_1(i \pm 1)$.

\begin{enumerate}
    \item $A_1$ is the set of tuples $(I_1,I_2) \in J(\calP_1) \times J(\calP_2)$ such that $\beta \notin I_1$.
    \item $A_2$ is the set of tuples $(I_1,I_2)$ such that $\beta \in I_1$; $\tau \notin I_1$; and, $\sigma_1 \in I_2$. 
    \item $A_3$ is the set of tuples $(I_1,I_2)$ such that $\tau \in I_1$; $\alpha \notin I_1$; and, $\sigma_m \in I_2$.
    \item $B$ is the set of tuples $(I_1,I_2)$ such that $\beta \in I_1$; $\sigma_1 \in I_2$ only if $\tau \in I_1$; and, $\sigma_m \in I_2$ only if $\alpha \in I_1$.
\end{enumerate}

Then, we can map any tuple $(I_1,I_2) \in A_1 \cup A_2 \cup A_3$ to a tuple $(I_3,I_4) \in J(\calP_{\gamma_3}) \times J(\calP_{\gamma_4})$, where the map in each case is clear since each element of $\calP_1$ and $\calP_2$ has a clear image in $\calP_{3}$ or $\calP_{4}$. Similarly, for any $(I_1,I_2) \in B$, there is a natural way to partition $(I_1 \backslash \langle \beta \rangle) \cup I_2$ into order ideals  of $\calP_5$ and $\calP_6$ In each case, these operations are clearly invertible. 
The $\gb$-vector calculations follow similarly from the previous case. 

\textbf{Subcase d) $\gamma$ crosses $\beta$ immediately before and $\sigma_m$ immediately after crossing $\tau$}. 

We again specify the chronological ordering such that $\beta = \calP_1(i-1)$ and $\sigma_m = \calP_1(i+1)$. Let $k \in [m]$ be such that the last spoke that $\gamma$ crosses as it winds around $p$ immediately after crossing $\tau$ is $\sigma_k$. If this spoke is $\tau$, set $k = 0$. 

If $k \neq 0$, let $i' \geq i$ be the largest value such that $\calP_1(i') = \tau$ and $\calP_1(i') \preceq \calP_1(i)$; it is possible that $i' = i$. If $k = 0$, let $i'$ be the second largest value. By our assumption that $\gamma$ crosses $\sigma_m$ immediately after crossing $\tau = \sigma_0$, this value will exist.

\begin{enumerate}
    \item $A_1$ is the set of tuples $(I_1,I_2) \in J(\calP_1) \times J(\calP_2)$ such that $\beta  \notin I_1$.
    \item $A_2$ is the set of tuples such that $\beta \in I_1$; $\sigma_1 \in I_2$; and, $\tau = \calP_1(i) \in I_1$ only if $\sigma_m \in I_2$. 
    \item $B$ is the set of tuples such that $\beta \in I_1$; $\sigma_1 \in I_2$ only if $\calP_1(i') \in I_1$; $\sigma_k \notin I_2$; and, if $i' > i$, $\calP_1(i'-1) \notin I_1$. 
\end{enumerate}

The bijections and the $\gb$-vector calculations are similar to the first case, so we again omit them. We stress that again, given $(I_1,I_2) \in B$, we partition $(I_1 \backslash \langle \beta \rangle) \cup I_2$ into order ideals of $\calP_5$ and $\calP_6$.

\textbf{Case ii) $\mathbf{\gamma}$ crosses $\mathbf{\sigma_1}$ immediately before crossing $\tau$}.

In this case, $F$ will be the order filter generated by $\sigma_1 = \calP_1(i-1)$. 

Some parts of this case resemble Type 1 grafting (Section \ref{subsec:Type1}), so we will streamline some sections, including $\gb$-vectors.

We index the chain $\calP_1[<i]^{\succ \calP_1(i)}$ as in previous sections (see for example Figure \ref{fig:Indexing}): $\sigma_k^{(w)} \succ \sigma_{k-1}^{(w)} \succ \cdots \succ \sigma_1^{(w)} \succ \tau^{(w)} \succ \sigma_m^{(w-1)} \succ \cdots \succ \sigma_1^{(1)}$ for some $w \geq 1$ and $0 \leq k \leq m$. Again, we set $\sigma_0:= \tau$, and if $k =0$, then the maximal element would be $\sigma_0^{(w+1)}$ with our convention for changing superscripts. We extend the indexing and set  $\calP(i) = \sigma_0^{(1)}$. As in previous sections, we let $\sigma_y^{(z)}$ be the largest element from the chain $\calP_1[<i]^{\succ \calP_1(i)}$ in $I_1$, if it exists, and let $\sigma_x$ be the largest element in $I_2$, if it exists. We also let these symbols refer to the largest elements in chain of spokes and loop in the posets from the resolution.  Finally, we let $a<i$ be the largest integer such that $\calP_1(a) \not\succ \calP_1(i)$, if it exists. 

\textbf{Subcase a) $\gamma$ crosses $\sigma_1$ immediately before and $\alpha$ immediately after crossing $\tau$}.

We partition $J(\calP_1) \times J(\calP_2)$ and describe appropriate maps.  In the following, we assume $\delta_5 = \delta_6 = 0$; this is equivalent to assuming $\calP_1(a)$ exists.

\begin{enumerate}
    \item $A_1$ is the set of tuples $(I_1,I_2) \in J(\calP_1) \times J(\calP_2)$ such that we do not have both $y = m$ and $\alpha \in I_1$ and if $\sigma_y^{(z)} \succeq \sigma_m^{(w-1)}$, then $\sigma_x \preceq \sigma_k$ with equality only if $\calP_1(a) \in I_1$. Given $(I_1,I_2) \in A_1$, define \[
    \Phi(\langle \sigma_y^{(z)} \rangle, \langle \sigma_x \rangle) = \begin{cases}
    (\langle \sigma_x^{(z)} \rangle, \langle \sigma_y \rangle) & y \neq m \\
    (\langle \sigma_x^{(z+1)} \rangle, \emptyset) & y = m,
    \end{cases}
    \]
    where if $\sigma_x$ does not exist, we set $\sigma_x^{(z)}:= \sigma_0^{(z)}$. 
    For any other element of $\calP_1 \cup \calP_2$, there is a clear corresponding element in $\calP_3 \cup \calP_4$ so we define $\Phi$ to map between these. The set $\Phi(A_1)$ is the collection of all $(I_3,I_4) \in J(\calP_3) \times J(\calP_4)$ such that $x < m$ and if $z = w$, then $x \leq k $ with equality only if $\calP_1(a) \in I_3$. 
   \item $A_2$ is the set of tuples $(I_1, I_2)$ such that $y = m$; $\alpha \in I_1$; and, $z \leq w-1$. Given $(I_1,I_2) \in A_2$, we set $\Phi(\langle \sigma_m^{(z)} \rangle, \langle \sigma_x \rangle) = (\langle \sigma_x^{(z)} \rangle, \langle \sigma_m \rangle)$. The set $\Phi(A_2)$ is all pairs $(I_3,I_4)$ such that $\sigma_y^{(z)} \preceq \sigma_m^{(w-1)}$ and 
   $\sigma_m \in I_4$.
    \item $A_3$ is the set of tuples $(I_1,I_2)$ such that $z = w$; $x \geq k$ with equality only if $\calP_1(a) \notin I_1$; and, $x = m$ only if $\alpha \in I_1$. Given $(I_1,I_2) \in A_3$, we set $\Phi(\langle \sigma_y^{(z)} \rangle, \langle \sigma_x \rangle) = (\langle \sigma_y^{(z)} \rangle, \langle \sigma_x \rangle)$. The set $\Phi(A_3)$ is all pairs $(I_3,I_4)$ such that $z = w$, and $x \geq k$ with equality only if $\calP_1(a) \notin I_3$, which is the complement of $A_1 \cup A_2$ in $J(\calP_3) \times J(\calP_4)$.
    \item $B$ is the set of tuples $(I_1,I_2)$ such that $\sigma_y^{(z)} \succeq \sigma_m^{(w-1)}$, which is only strict if $\sigma_m \in I_2$; $x \geq k$, which is strict if $\calP_1(a) \in I_1$; and, $\alpha \notin I_1$. Given $(I_1,I_2) \in B$, there is a clear way to partition $(I_1 \backslash \langle \sigma_m^{(w-1)} \rangle) \cup (I_2 \backslash \langle \sigma_k \rangle)$ into order ideals of $\calP_5$ and $\calP_6$ and this map is bijective. 
\end{enumerate}

 Notice that as sets $\langle \sigma_m^{(w-1)} \rangle \cup  \langle \sigma_k \rangle$ has the same content as $F$, which in this case is the principal order filter generated by $ \sigma_1^{(1)}$.

 \textbf{Subcase b) $\gamma$ crosses $\sigma_1$; $\tau$ is the last arc crossed by $\gamma$; and, $t(\gamma)$ is plain.} 

 This case follows from the previous (Case ii.a) by ignoring $A_2$ and any conditions involving $\alpha$. 

 \textbf{Subcase c) $\gamma$ crosses $\sigma_1$; $\tau$ is the last arc crossed by $\gamma$; and, $t(\gamma)$ is notched.} 

In this case, $\calP_1$ has a loop at its end consisting of the chain $\sigma_m \prec \eta_1 \prec \cdots \prec \eta_h \prec \alpha$, attached to $\tau$, and $\gamma_4 = \sigma_m^{(p , t(\gamma))}$. As in other places, we let $\sigma_m^-$ and $\sigma_m^+$ denote the minimum and maximum respectively of $\calP_4$. 

\begin{enumerate}
    \item We use the same set $A_1$ as in Case ii.a with the added condition that if $y = m-1$, then $\eta_1 \in I_1$. This allows us to use the same map. Now, $\Phi(A_1)$ consists of all pairs $(I_3,I_4)$ such that $\sigma_m^+ \notin I_4$; if $z = w$, then $ x \leq k$, with equality only if $\calP_1(a) \in I_3$; and, if $z > 1$,  then $I_4 \neq \emptyset$. 
    \item We can nearly use the same set $A_2$, but we include a special case where $\sigma_y^{(z)}$ is the minimal element $\sigma_m$ in the loop in $\calP_1$. In this case, we send $(\langle \eta_f \rangle,\langle \sigma_x \rangle)$ to $(\emptyset, \langle \sigma_x \rangle \cup \langle \eta_f \rangle)$. The image $\Phi(A_2)$ consists of all pairs $(I_3,I_4)$ such that $\sigma_m^+ \in I_4$ and $\sigma_y^{(z)} \preceq \sigma_m^{(w-1)}$. 
    \item We modify $A_3$ to include the stipulation that if $\sigma_x \preceq \sigma_{m-1}$, then $\eta_1 \in I_1$. The map and image of this set are the same as in Case ii.a. 
    \item We define $A_4$ to be the set of tuples $(I_1,I_2) \in J(\calP_1) \times J(\calP_2)$ satisfying $y = m-1$, and if $z = w-1$, then $x \leq k$, with equality only if $\calP_1(a) \in I_1$. We set $\Phi(\langle \sigma_{m-1}^{(z)} \rangle, \langle \sigma_x \rangle) = (\langle \sigma_x^{(z+1)} \rangle,\emptyset)$. The image of $A_4$ is the set of $(I_3,I_4)$ where $I_4 = \emptyset$ and $z > 1$. 
\end{enumerate}

The complement of $A_1 \cup A_2 \cup A_3 \cup A_4$ will be called $B$ and this can be characterized by $(I_1,I_2)$ such that $\alpha \notin I_1$; $\sigma_y^{(z)} \succeq \sigma_{m-1}^{(w-1)}$, which is strict if $\eta_1 \in I_1$; $z = w$ only if $\sigma_m \in I_2$; and, $x \geq k$, which is strict if $\calP_1(a) \in I_1$. As before, there is a way to partition $(I_1 \backslash \langle \sigma_m^{(w-1)} \rangle) \cup (I_2 \backslash \langle \sigma_k \rangle)$ into order ideals of $\calP_5$ and $\calP_6$ and this map is bijective.

\textbf{Subcase d) $\gamma$ crosses $\sigma_1$ immediately before and $\alpha$ immediately after crossing $\tau$}.

This entire situation is nearly identical to the calculations for Type 1 resolutions in Section \ref{subsec:Type1}. One can for example see the similarity in the definition of the $y$-monomial in the relation.

\textbf{Case iii) $\tau$ is the first arc crossed by $\gamma$}.

The labeling of the resolution in this case depends on the notching of $s(\gamma)$. In particular, $F = \emptyset$ regardless of the notching, but it is possible $\delta_5 = 1$ or $\delta_6 = 1$. The $\gb$-vector computations are all straight forward in this case since the $y$-monomial is degree 1, so we will not address them.

\textbf{Subcase a) $\tau$ is the first arc crossed by $\gamma$; $\alpha$ is the second arc; and, $s(\tau)$ is plain.}

In this case as well as the next, $\delta_6 = 1$, so the $y$-monomial will be $y_{\gamma_6^0} = y_{\sigma_1}$.
This case is solved by setting $A = (I_1,I_2)$ such that $\sigma_1 \in I_2$ only if $\tau \in I_1$ and $\sigma_m \in I_2$ only if $\alpha \in I_2$. Then, the complement is characterized by $(I_1,I_2)$ such that $\sigma_1 \in I_2$; $\tau \in I_1$ only if $\sigma_m \in I_2$; and, $\alpha \notin I_1$.

\textbf{Subcase b) $\tau$ is the first arc crossed by $\gamma$; $\sigma_m$ is the second arc; and, $s(\tau)$ is plain.}

We set $A$ to be the set of pairs $(I_1,I_2)$ such that $\sigma_1 \in I_2$ only if $\sigma_0^{(2)} \in I_1$ and $\sigma_x \leq \sigma_{k-1}$ with equality only if $\calP_1(a) \in I_2$. The map from $A$ to $J(\calP_4)$ is clear; note that $\calP_3 = (\emptyset, \beta^+)$.

We will divide the complement of $A$ into three smaller sets. Note that in order to have $(I_1,I_2) \notin A$, it must be $\sigma_1 \in I_2$.

\begin{enumerate}
    \item Let $B_1$ consist of pairs $(I_1,I_2)$ such that $\sigma_y^{(z)} \prec \sigma_0^{(1)}$. We define $\Phi(\langle \sigma_y^{(z)} \rangle, \langle \sigma_x \rangle) = (\langle \sigma_y^{(z)} \rangle, \langle \sigma_x \rangle)$, so that the image consists of $(I_5,I_6)$ with $\sigma_y^{(z)} \prec \sigma_0^{(1)}$ and $\sigma_x \prec \tau$. 
    \item Let $B_2$ consist of pairs $(I_1,I_2)$ where $\sigma_y^{(z)} = \sigma_0^{(1)}$ and $\sigma_x \succeq \sigma_{k-1}$, which is strict if $\calP_1(a) \in I_1$. For such a pair, we define $\Phi(\langle \sigma_0^{(1)} \rangle, \langle \sigma_x \rangle) = (\langle \tau \rangle, \langle \sigma_x^{(1)} \rangle)$ so that the image consists of $(I_5,I_6)$ with $\sigma_y^{(z)} \prec \sigma_0^{(1)}$ and $\sigma_x = \tau$.
    \item Let $B_3$ consist of pairs $(I_1,I_2)$ where $\sigma_y^{(z)} \succeq \sigma_1^{(2)}$. If $y \neq 1$, we set $\Phi(\langle \sigma_y^{(z)} \rangle, \langle \sigma_x \rangle) = (\langle \sigma_y \rangle, \langle \sigma_x^{(z)}\rangle)$ and if $y = 1$ we set $\Phi(\langle \sigma_1^{(z)} \rangle, \langle \sigma_x \rangle) = (\emptyset, \langle \sigma_x^{(z)} \rangle)$. The image consists of all $(I_5,I_6)$ with $\sigma_y^{(z)} \succeq \sigma_1^{(2)}$, which completes our bijection. 
\end{enumerate}

\textbf{Subcase c) $\tau$ is the first arc crossed by $\gamma$; $\alpha$ is the second arc; and, $s(\tau)$ is notched.}

In this case and in the next case, $\delta_5 = 1$, so the $y$-monomial will be $y_{\gamma_6^0} = y_\beta$. Here, we set $A = (I_1,I_2)$ such that $\beta \in I_1$ only if $\sigma_1 \in I_2$; $\tau \in I_1$ only if $\sigma_m \in I_2$; and,  $\alpha \notin I_1$. The complement, $B$, can be characterized as pairs $(I_1,I_2)$ such that $\beta \in I_1$; $\sigma_1 \in I_2$ only if $\tau \in I_1$; and, $\sigma_m \in I_2$ only if $\alpha \in I_1$. In each case, we see that there is a clear bijection between the subset of $J(\calP_1) \times J(\calP_2)$ and the appropriate product of lattices from the resolution. 

\textbf{Subcase d) $\tau$ is the first arc crossed by $\gamma$; $\sigma_m$  is the second arc; and, $s(\tau)$ is notched.}

We let $A = A_1 \cup A_2 \cup A_3$, as described below. 

\begin{enumerate}
    \item Let $A_1$ consist of pairs $(I_1,I_2)$ such that $f = h$; $y = 1$; $z > 1$; and, if $z = 2$, then $x \geq k-1$ which is strict if $\calP_1(a) \in I_1$. We set $\Phi(\langle \sigma_1^{(z)} \rangle \cup \langle \eta_h \rangle, \langle \sigma_x \rangle) = (\langle \sigma_1^+ \rangle, \sigma_x^{(z)})$ where if $\sigma_x$ does not exist we set $\sigma_x^{(z)} = \sigma_0^{(z)}$. The image of $A_1$ is all pairs $(I_3,I_4)$ such that $\sigma_1^+ \in I_3$. 
    \item Let $A_2$ consist of pairs $(I_1,I_2)$ such that $\sigma_y^{(z)} \succeq \sigma_0^{(2)}$ where $y \neq 1$ if $f = h$ and $y \neq 0$ if $\beta \in I_1$. We define \[
    \Phi(\langle \sigma_y^{(z)} \rangle \cup \langle \eta_f \rangle, \langle \sigma_x \rangle) = \begin{cases} (\langle \sigma_y \rangle \cup \langle \eta_f \rangle, \langle \sigma_x^{(z)} \rangle & y \neq 0\\ (\langle \eta_f \rangle, \langle \sigma_x^{(z)} \rangle) & y = 0\\ \end{cases}
    \]
    where as in other cases we set $\sigma_x^{(z)} = \sigma_0^{(z)}$ if $\sigma_x$ does not exist. The image of $A_2$ is all pairs $(I_3,I_4)$ such that $\sigma_1^+ \notin I_3$ and $z > 1$.
    \item Let $A_3$ consist of pairs $(I_1,I_2)$ such that $z = 1$;if $\beta \in I_1$, then $\sigma_x$ exists; and, if $\eta_h \in I_1$, then $x \geq 2$. We set $\Phi(\langle \sigma_y^{(1)} \rangle \cup \langle \eta_f \rangle, \langle \sigma_x \rangle) = (\langle \sigma_x \rangle \cup \langle \eta_f \rangle, \langle \sigma_y^{(z)} \rangle)$.  The image of $A_3$ is all pairs $(I_3,I_4)$ such that $\sigma_1^+ \notin I_3$ and $z = 1$.
\end{enumerate}

The complement of $A$ can be characterized as pairs $(I_1,I_2)$ such that $\beta \in I_1$; $\sigma_1 \in I_2$ only if $\eta_h \in I_1$ or $\sigma_0^{(2)} \in I_1$; $\sigma_2 \in I_2$ only if $\sigma_0^{(2)} \in I_1$; $\sigma_{k-1} \in I_2$ only if $\calP_1(a) \in I_1$; $\sigma_y^{(z)} = \sigma_1^{(2)}$ only if $\eta_h \in I_1$; $\sigma_2^{(2)} \notin I_1$; and, $\sigma_k \notin I_2$. The map in this case is fairly clear, but since there are two elements associated to $\sigma_1$, a little care is needed to be consistent. We set $\Phi(\langle \eta_h \rangle \cup \langle \sigma_1^{(2)} \rangle, \emptyset) = (\langle \sigma_1 \rangle,\langle \tau \rangle)$ and $\Phi(\langle \eta_h \rangle \cup \langle \sigma_0^{(2)} \rangle, \emptyset) = (\langle \eta_h \rangle,\langle \sigma_1 \rangle)$. 

\end{proof}

Now, we turn to an intersection between $\gamma \notin T$ and $\tau^{(p,q)}$ such that $\tau^0 \in T$. We now have that the endpoints of $\tau$ are punctures $p$ and $q$; these could possibly be equal but it suffices to consider the case where they are distinct.  As before, we will have the spokes at $p$ being $\sigma_1,\ldots,\sigma_m$, and now let the spokes at $q$ be $\eta_1,\ldots,\eta_h$, as below. In this case, $\{\calP_\gamma(j): \calP_\gamma(j) \prec \calP_\gamma(i) \text{ and } \calP_\gamma(j) \notin \calP_\tau \cup \{\tau^0\} \}$ is empty since all elements comparable to $\calP_\gamma(i)$ will also correspond to elements in $ \calP_\tau \cup \{\tau^0\}$, so $F$ is an order filter. 

\begin{center}
\begin{tikzpicture}[scale = 1.5]
\draw (0,0) to node[left]{$\eta_h$} (1,1)  to node[right]{$\eta_1$}  (2,0) to node[right]{$\sigma_m$} (1,-1) to node[left]{$\sigma_1$} (0,0);
\draw (1,1) to node[right]{$\tau$} (1,-1);
\node[below] at (1,-1){$p$};
\node[] at  (1,0.8){$\bowtie$};
\node[] at (1,-0.8){$\bowtie$};
\end{tikzpicture}
\end{center}

Since the set $F \cup \{\tau\}$ is always an order filter, there is no ambiguity in labeling the posets in the resolution such as there was in the case for $\tau^{(p)}$. We set $\calP_{\gamma_3}^0 = \calP_{\upper}^{<i}$ with a loop at the end based at $q$, $\calP_{\gamma_4}^0 = \calP_{\upper}^{>i}$ with a loop at the beginning based at $p$, $\calP_{\gamma_5}^0 = \calP_{\low}^{<i}$ with a loop at the end based at $p$, and $\calP_{\gamma_6}^0 = \calP_{\low}^{>i}$ with a loop at the beginning based at $p$.

\begin{prop}\label{prop:IntArcFromTDoubleTag}
Let $\tau \in T$ and let $\gamma$ be an arc such that $e(\gamma,\tau) > 0$. Pick $i$ such that $\calP_{\gamma}(i) = \tau$. If $\gamma_5^0 \in T$, let $\delta_5 = 1$ and similarly for $\gamma_6^0$ and $\delta_6$. If $\delta_i$ is not yet specified, set $\delta_i = 0$. Then, , \[
x_{\gamma} x_{\tau^{(p,q)}} = x_{\gamma_3} x_{\gamma_4} + Y_F y_\tau y_{\gamma_5^0}^{\delta_5}y_{\gamma_6^0}^{\delta_6} x_{\gamma_5} x_{\gamma_6}
\]
\end{prop}

\begin{proof}
Throughout let $\calP_1:= \calP_\gamma, \calP_2:= \calP_{\tau^{(p)}}$, and for $3 \leq i \leq 6$, $\calP_i:= \calP_{\gamma_i}$.

\textbf{Case i) $\gamma$ crosses $\eta_h$ and $\sigma_m$ immediately before and after crossing $\tau$}.

Throughout this case, we will have $F =  \emptyset$ and $\delta_5 = \delta_6 = 0$, so the $y$-monomial in the skein relation will always be $y_\tau$. We draw diagrams for $\calP_1, \calP_{3}$, and $\calP_{5}$ below, using similar labeling to many previous proofs. In particular, $\sigma_0 := \tau$ and $\eta_0:= \tau$, and the superscripts change between $\sigma_0$ and $\sigma_m$, $\sigma_0^{(i)} \succ \sigma_m^{(i-1)}$ and similarly in the chain of $\eta_c^{(d)}$. 

\begin{center}
\begin{tabular}{ccc}
\highlight{$\calP_1$} & \highlight{$\calP_{3}$}& \highlight{$\calP_5$}\\
\begin{tikzpicture}
\node(a) at (4,0){$\calP_1(b)$};
\node(sk) at (3.5,-1){$\sigma_k^{(1)}$};
\node[rotate=-30](dots1) at (3,0){$\ddots$};
\node(s1) at (2.5,1){$\sigma_m^{(w)}$};
\node(t) at (2,2){$\tau^{top}$};
\node(n1) at (1.5,1){$\eta_h^{(v)}$};
\node[rotate=30](dots2) at (1,0){$\iddots$};
\node(nl) at (0.5,-1){$\eta_\ell^{(1)}$};
\node(b) at (0,0){$\calP_1(a)$};
\draw (a) -- (sk);
\draw(sk) -- (dots1);
\draw (dots1) -- (s1);
\draw(s1) -- (t);
\draw (t) --(n1);
\draw(n1) -- (dots2);
\draw (dots2) -- (nl);
\draw(nl) -- (b);
\end{tikzpicture}&
\begin{tikzpicture}
\node(a) at (0,0){$\calP_1(a)$};
\node(sk) at (0.75,0.75){$\eta_{\ell-1}$};
\node(sk1) at (0.75,-0.75){$\eta_{\ell}$};
\draw(a) -- (sk);
\draw(a) -- (sk1);
\node(dots) at (0.75,0){$\vdots$};
\draw(sk) -- (dots);
\draw(sk1) -- (dots);
\end{tikzpicture}&
\begin{tikzpicture}
\node(a) at (0,0){$\calP_1(a)$};
\node(sk) at (0.5,-1){$\eta_\ell^{(1)}$};
\node[rotate=30](dots1) at (1,0){$\iddots$};
\node(s1) at (1.5,1){$\eta_h^{(v)}$};
\node(t) at (2.25,1.75){$\tau^{top}$};
\node(nn) at (2.25,0.25){$\sigma_1$};
\node(dots2) at (2.25,1){$\vdots$};
\draw (a) -- (sk);
\draw(sk) -- (dots1);
\draw (dots1) -- (s1);
\draw(s1) -- (t);
\draw(s1) -- (nn);
\draw(dots2) -- (nn);
\draw(dots2) -- (t);
\end{tikzpicture}
\end{tabular}
\end{center}

The poset $\calP_2$ is exactly $\calP_{\tau^{(p,q)}}$ and $\calP_4$ and $\calP_6$ are similar to $\calP_3$ and $\calP_5$ respectively, with the roles of $\sigma$'s and $\eta$'s reversed. Notice that we have $\tau^{top}$ in $\calP_1,\calP_5,\calP_6$. In $\calP_1$, $\tau^{top}$ can also be called $\eta_0^{(v+1)}$ or $\sigma_0^{(w+1)}$, and we have similar options for labeling in $\calP_5$ and $\calP_6$.

Note that $\calP_1(a)$ and $\calP_1(b)$ need not exist. If it is the case that one of these does not exist, then throughout we ignore any condition involving them. 

As in previous proofs,  let a generic element $I_1 \in J(\calP_1)$ with $\tau^{top} \notin I_1$ be $\langle \eta_c^{(d)} \rangle \cup \langle \sigma_y^{(z)} \rangle$ where we as always ignore elements outside of the loop and chain of spokes since it is clear how $\Phi$ should act on them. If there is no element in $I_1$ from the chain from $q$, we write $\eta_c^{(d)} = \eta_{\ell-1}^{(1)}$ and similarly for the chain from $p$ we write $\sigma_y^{(z)} = \sigma_{k-1}^{(1)}$. We let a generic element $I_2 \in J(\calP_2)$ be $\langle \eta_f \rangle \cup \langle \sigma_x \rangle$, where if $I_2 \neq \emptyset$, we set $\eta_v = \eta_0$ if $\eta_1 \notin I_2$ and similarly for $\sigma$. By abuse of notation, we also let the support of an arbitrary set  $I_5$ on the chain of spokes and loop be denoted $\langle \eta_c^{(d)} \rangle \cup \langle \sigma_x \rangle$ and similarly for $I_6$, $\langle \sigma_y^{(z)} \rangle \cup \langle \eta_f \rangle$.  Recall $F = \{\tau\}$, so each map $\Phi: B_i \to J(\calP_5) \times J(\calP_6)$ will be such that $(I_1,I_2) \in B_i$ has the same content as $\Phi(I_1,I_2) \cup \{\tau\}$. 

The proof becomes simpler when we  this case into more cases based on how much $\gamma$ winds around each puncture. 

\textbf{Subcase) $v > 1$ and $w > 1$}.

\begin{enumerate}
    \item $A$ is the set of $(I_1,I_2) \in J(\calP_1) \times J(\calP_2)$ such that $I_2 = \emptyset$; $\eta_c^{(d)} \preceq \eta_{\ell-1}^{(2)}$, with equality only if $\calP_1(a) \in I_1$; and, $\sigma_y^{(z)} \preceq \sigma_{k-1}^{(2)}$ with equality only if $\calP_1(b) \in I_1$. There is a clear bijection between $A$ and $J(\calP_3) \times J(\calP_4)$. 
    \item $B_1$ is the set of $(I_1,I_2)$ such that $I_2 = \emptyset$; $\tau^{top} \notin I_1$; and, $(\sigma_y^{(z)} \succeq \sigma_{k-1}^{(2)}$, strict if $\calP_1(b) \in I_1$) or $(\eta_c^{(d)} \succeq \eta_{\ell-1}^{(2)}$, strict if $\calP_1(a) \in I_1$). If it is the case that $\sigma_y^{(z)} \succeq \sigma_{k-1}^{(2)}$ with strict inequality if $\calP_1(b) \in I_1$, then we set $\Phi( \langle \eta_c^{(d)} \rangle \cup \langle \sigma_y^{(z)} \rangle, \emptyset) = (\langle \eta_c^{(d)} \rangle \cup \langle \sigma_m \rangle, \langle \sigma_y^{(z-1)} \rangle)$. If the above is not true, then it must be that $\eta_c^{(d)} \succeq \eta_{\ell-1}^{(2)}$ with strict inequality if $\calP_1(a) \in I_1$, and  we set $\Phi( \langle \eta_c^{(d)} \rangle \cup \langle \sigma_y^{(z)} \rangle, \emptyset) = (\langle \eta_c^{(d-1)} \rangle, \langle \sigma_y^{(z)} \rangle \cup \langle \eta_h \rangle)$.  The set $\Phi(B_1)$ consists of pairs $(I_5,I_6)$ such that $\tau^{top} \notin I_5,I_6$, and either ($\sigma_m \in I_5$; $\eta_1 \notin I_6$; and, $\sigma_y^{(z)} \prec \sigma_0^{(w)}$)  or ($\eta_h \in I_6$; $\sigma_1 \notin I_5$; $\eta_c^{(d)} \prec \eta_0^{(v)}$; and, $\sigma_y^{(z)} \preceq \sigma_{k-1}^{(2)}$ with equality only if $\calP_1(b) \in I_6 $). 
    \item $B_2$ is the set of $(I_1,I_2)$ such that $I_2 \neq \emptyset$; $\tau^{top} \notin I_1$; $\tau^+ \notin I_2$; $\eta_h^{(v)} \in I_1$ only if $\sigma_1 \in I_2$; and, $\sigma_m^{(w)} \in I_1$ only if $\eta_1 \in I_2$. If $(I_1,I_2) \in B_2$, then we set $\Phi(\langle \eta_c^{(d)} \rangle \cup \langle \sigma_y^{(z)} \rangle, \langle \sigma_x \rangle \cup \langle \eta_f \rangle) = (\langle \eta_c^{(d)} \rangle \cup \langle \sigma_x \rangle, \langle \sigma_y^{(z)} \rangle \cup \langle \eta_f \rangle)$. The set $\Phi(B_2)$ consists of all pairs $(I_5,I_6)$ such that $\tau^{top} \notin I_5$; $\tau^{top} \notin I_6$; $\sigma_m \in I_5$ only if $\eta_1 \in I_6$; and, $\eta_h \in I_6$ only if $\sigma_1 \in I_5$. 
    \item $B_3$ is the set of $(I_1,I_2)$ such that $I_2 \neq \emptyset$; $\tau^{top} \notin I_1$; $\tau^+ \notin I_2$; and, $(\sigma_m^{(w)} \in I_1$ and $\eta_1 \notin I_2)$ or $(\eta_h^{(v)} \in I_1$ and $\sigma_1 \notin I_2$). If $\sigma_m^{(w)} \in I_1$ and $\eta_1 \notin I_2$, then we set $\Phi(\langle \sigma_m^{(w)} \rangle \cup \langle \eta_c^{(d)} \rangle, \langle \sigma_x \rangle) = (\langle \eta_c^{(d)} \rangle \cup \langle \sigma_m \rangle, \langle \sigma_x^{(w)} \rangle)$. Otherwise, if that is not true, then it must be that $\eta_h^{(v)} \in I_1$ and $\sigma_1 \notin I_2$, and we set $\Phi(\langle \eta_h^{(v)} \rangle \cup \langle \sigma_y^{(z)} \rangle, \langle \eta_f \rangle ) = (\langle \eta_f^{(d)} \rangle, \langle \sigma_y^{(z)} \rangle \cup \langle \eta_h \rangle)$.  The set $\Phi(B_3)$ consists $(I_5,I_6)$ such that $\tau^{top} \notin I_5,I_6$ and either ( $\sigma_m \in I_5$; $\eta_1 \notin I_6$; and, $\sigma_y^{(z)} \succeq \sigma_0^{(w)}$) or ($\eta_h \in I_6$; $\sigma_1 \notin I_5$; and, $\eta_c^{(d)} \succeq \eta_0^{(v)}$, which is strict if $\sigma_y^{(z)} = \sigma_m^{(w)}$). 
    \item $B_4$ is the set of $(I_1,I_2)$ such that $\tau^+ \in I_2$; $\eta_c^{(d)} \prec \eta_0^{(v)}$; and, $\sigma_y^{(z)} \prec \sigma_0^{(w)}$. Given $(I_1,I_2) \in B_4$, we set $\Phi(\langle \eta_c^{(d)} \rangle \cup \langle \sigma_y^{(z)} \rangle, \langle \sigma_x \rangle \cup \langle \eta_f \rangle) = (\langle \eta_c^{(d)} \rangle, \langle \sigma_y^{(z+1)} \rangle \cup \langle \eta_h \rangle)$. The set $\Phi(B_4)$ consists of all pairs $(I_5,I_6)$ such that $\tau^{top} \notin I_5,I_6$; $\eta_h \in I_6$; $\sigma_1 \notin I_5$; $\sigma_y^{(z)} \succeq \sigma_{k-1}^{(2)}$ with strict inequality if $\calP_1(b) \in I_5$; and, $\eta_c^{(d)} \prec \eta_0^{(v)}$.
    \item $B_5$ is the set of $(I_1,I_2)$ such that $\tau^+ \in I_2$; $\tau^{top} \notin I_1$; and $(\eta_c^{(d)} \succeq \eta_0^{(v)}$, strict if $\sigma_y^{(z)} = \sigma_m^{(w)})$ or $(\sigma_y^{(z)} \succeq \sigma_0^{(w)}$, strict if $\eta_c^{(d)} = \eta_h^{(v)})$. Given $(I_1,I_2) \in B_5$, if we have $\eta_c^{(d)} \succeq \eta_0^{(v)}$, strict if $\sigma_y^{(z)} = \sigma_m^{(w)}$,  then we set $\Phi(\langle \eta_c^{(v)} \rangle \cup \langle \sigma_y^{(w)} \rangle, \langle \tau^+ \rangle) = (\langle \tau^{top} \rangle, (\langle \sigma_y^{(z)} \rangle \cup \langle \eta_c \rangle)$. Otherwise, we have $\eta_c^{(d)} \preceq \eta_0^{(v)}$, only equal if  $\sigma_y^{(z)} = \sigma_m^{(w)}$, and $\sigma_y^{(z)} \succeq \sigma_0^{(w)}$ and here we set $\Phi(\langle \eta_c^{(d)} \rangle \cup \langle \sigma_y^{(w)} \rangle, \langle \tau^+ \rangle) = (\langle \eta_c^{(d)} \rangle \cup \langle \sigma_y \rangle \langle \tau^{top} \rangle)$. The set $\Phi(B_5)$ consists of pairs $(I_5,I_6)$ such that ($\tau^{top} \in I_5$ and  $\tau^{top} \notin I_6$) or ($\tau^{top} \in I_6$ and $\eta_c^{(d)} \preceq \eta_0^{(v)}$ which is only equal if $\sigma_m \in I_5$). 
    \item $B_6$ is the set of $(I_1,I_2)$ such that $\tau^{top} \in I_1$. For the special cases $I_2 = \emptyset$ and $I_2 = \langle \tau^+ \rangle$, we set $\Phi(\langle \tau^{top} \rangle, \emptyset) = (\langle \eta_0^{(v)} \rangle, \langle \sigma_m^{(w)} \rangle \cup \langle \eta_h \rangle)$ and $\Phi(\langle \tau^{top} \rangle, \langle \tau^+ \rangle) = (\langle \tau^{top} \rangle, \langle \tau^{top} \rangle)$. For any other nonempty $I_2 = \langle \eta_f \rangle \cup \langle \sigma_z \rangle$, we have $\Phi(I_1,I_2) = (\langle \eta_f^{(v)} \rangle \cup \langle \sigma_x \rangle \langle \tau^{top} \rangle, )$. The image $\Phi(B_6)$ contains $(\langle \eta_0^{(v)} \rangle, \langle \sigma_m^{(w)} \rangle \cup \langle \eta_h \rangle)$  and pairs $(I_5,I_6)$ such that $\tau^{top} \in I_6$ and $\eta_c^{(d)} \succeq \eta_0^{(v)}$ with strict inequality if $x = m$. 
\end{enumerate}

\textbf{Subcase) $w > 1$ and $v = 1$}.

Now, we consider the case where the arc $\gamma$ only winds multiple times around one of the punctures, which is without  loss of generality chosen to be $p$. We explain how we can modify the partition and maps from the previous case. 

\begin{enumerate}
    \item $A$ is now the set of $(I_1,I_2) \in J(\calP_1) \times J(\calP_2)$ such that $\sigma_y^{(z)} \preceq \sigma_{k-1}^{(2)}$ with equality only if $\calP_1(b) \in I_1$; $\eta_f \preceq \eta_{\ell-1}$ with equality only if $\calP_1(a) \in I_1$; and,  $\tau^- \in I_2$ only if $\eta_h^{(1)} \in I_1$. The bijection between $A$ and $J(\calP_{\gamma_3}) \times J(\calP_{\gamma_4})$ is obvious.
    \item To define $B_1$, we only use the condition involving $\sigma_y^{(z)}$. The image, using the same map, is the set of $(I_5,I_6)$ such that $ \tau^{top} \notin I_5$; $\sigma_m \in I_5$; $\eta_1 \notin I_6$; and, $\sigma_y^{(z)} \prec \sigma_0^{(w)}$.
    \item $B_2$ and $\Phi \vert_{B_2}$ are defined the same way as the general case, and so the image is also described identically.
    \item We set $B_3$ to be the set of pairs $(I_1,I_2)$ such that $I_2 \neq \emptyset$; $\tau^{top} \notin I_1$; $\tau^+ \notin I_2$; and one of the following is true:
    \begin{enumerate}
        \item[(a)] $\sigma_m^{(w)} \in I_1$ and $\eta_1 \notin I_2$;
        \item[(b)] $\eta_c^{(1)} = \eta_h^{(1)}$; $\sigma_1 \notin I_2$; and, $\sigma_y^{(z)} \succeq \sigma_{k-1}^{(2)}$, with strict inequality if $\calP_1(b) \in I_1$;
        \item[(c)] $\eta_c^{(1)} = \eta_h^{(1)}$; $\sigma_1 \notin I_2$; and, $\eta_f \succeq \eta_{\ell-1}$, with strict inequality if $\calP_1(a) \in I_1$.
    \end{enumerate}
    If a set can be described by property (a), then we use a map as in the general case, and the image of this subset is the set of $(I_5,I_6)$ such that $\sigma_m \in I_5$; $\tau^{top} \notin I_5$; $\eta_1 \notin I_6$; and, $\sigma_y^{(z)} \succeq \sigma_0^{(w)}$. 
    If a set cannot be described by (a) but can be described by (b), then we set the image to be $(\langle \tau^{top} \rangle , \langle \sigma_y^{(z-1)} \rangle \cup \langle \eta_f \rangle)$, so that the image is the set of $(I_5,I_6)$ where $\tau^{top} \in I_5$; $\sigma_y^{(z)} \preceq \sigma_m^{(w-1)}$ with equality only if $\eta_1 \in I_6$; and, $\eta_f \prec \eta_h$. If a set cannot be described by (a) or (b) but can be described by (c), then we set the image to be $(\langle \eta_f^{(1)} \rangle, \langle \eta_h \rangle \cup \langle \sigma_y^{(z)} \rangle)$, so that the image consists of $(I_5,I_6)$ such that $\tau^{top} \notin I_5$; $\tau^{top} \notin I_6$; $\eta_h \in I_6$; $\sigma_1 \notin I_5$; and, $\sigma_y^{(z)} \preceq \sigma_{k-1}^{(2)}$ with equality only if $\calP_1(b) \in I_6$.
    \item To define $B_4$, we replace the condition involving $\eta_c^{(d)}$ with $\eta_c^{(1)} \prec \eta_h^{(1)}$. We define the map the same way, so that the image can be described also in the same way except with no condition on $\eta_c^{(d)}$. 
    \item $B_5$ is defined by now using the conditions ($\sigma_y^{(z)} \preceq \sigma_0^{(w)}$ and  $\eta_c^{(d)} = \eta_h^{(1)}$) or ( $\sigma_y^{(z)} \succeq \sigma_0^{(w)}$ which is strict if $\eta_c^{(1)}$). If we are in the first case, we set the image to be $(\langle \tau^{top} \rangle, \langle \sigma_c^{(d)} \rangle \cup \langle \eta_h \rangle)$, so that the image of this portion can be described as all pairs $(I_5,I_6)$ such that $\tau^{top} \in I_5$;  $\eta_f = \eta_h$; and, $\sigma_y^{(z)} \preceq \sigma_0^{(w)}$. If we are in the second  case, then we use the same map as in the general case, and the image of this subset is all pairs $(I_5,I_6)$ such that $\tau^{top} \in I_6$ and $\tau^{top} \notin I_5$. 
    \item $B_6$ is defined the same way as the general case. For the special case where $I_2 = \emptyset$, we define $\Phi(\langle \tau^{top} \rangle, \emptyset) = (\langle \tau^{top} \rangle, \langle \sigma_m^{(w-1)} \rangle)$ and for all other cases we use the same map as the general case. Thus, the set $\Phi(B_6)$ consists of pairs $(I_5,I_6)$ such that $\tau^{top} \in I_5$; $\sigma_y^{(z)} \succeq \sigma_m^{(w-1)}$ with equality only if $\eta_f$ does not exist; and, $\sigma_y^{(z)} \succ \sigma_0^{(w)}$ if $\eta_f = \eta_h$. 
\end{enumerate}

\textbf{Subcase) $w = v = 1$}. 

Since this subcase is much more specialized, not every set from the general case will have a clear analogue here.

\begin{enumerate}
    \item $A$ is the set of $(I_1,I_2)$ such that $\tau^- \in I_2$ only if $\eta_h^{(1)}$ or $\sigma_m^{(1)}$; $\sigma_1 \in I_2$ only if $\sigma_m^{(1)} \in I_1$; $\eta_1 \in I_2$ only if $\eta_h^{(1)} \in I_1$; $\{\sigma_1,\eta_1\} \subset I_2$ only if $\tau^{top} \in I_1$; $f \leq \ell - 1$ with equality only if $\calP_1(a) \in I_1$; and, $x \leq k-1$, with equality only if $\calP_1(b) \in I_1$. There is a small ambiguity here between the images of $(\langle \tau^{top} \rangle, \emptyset)$ and $(\langle \eta_h^{(1)} \rangle \cup \langle \sigma_m^{(1)} \rangle, \langle \tau^- \rangle)$; we arbitrarily set the former to be ($\langle \eta_h \rangle, \langle \sigma_0 \rangle)$ and the latter to be $(\langle \eta_0 \rangle, \langle \sigma_m \rangle)$.
    \item  Here, we set $B_2$ to be the set of pairs $(I_1,I_2)$ such that $\tau^- \in I_2$; $\tau^{top} \notin I_1$; $\tau^+ \notin I_2$; $\eta_h^{(1)} \in I_2$ only if $\sigma_1 \in I_2 $; and, $\eta_h^{(1)}$ and $\sigma_m^{(1)}$ are not in $I_1$. We use a map as before, so that the image is all pairs $(I_5,I_6)$ such that $\eta_h \in I_6$ only if $\sigma_1 \in I_5$ and $\sigma_m \in I_5$ only if $\eta_1 \in I_6$.
    \item We let $B_3$ be all pairs $(I_1,I_2)$ such that on the following is true: 
    \begin{enumerate}
        \item[(a)] $\eta_c^{(1)} = \eta_h^{(1)}$; $\sigma_1 \notin I_2$; and, $f \geq \ell - 1$, with strict inequality if $\calP_1(a) \in I_1$;
        \item[(b)] $\sigma_y^{(1)} = \sigma_m^{(1)}$; $\eta_1 \notin I_2$; and, $x \geq k -1$, with strict inequality $\calP_1(b) \in I_1$.
    \end{enumerate}
    If (a) describes a pair $(I_1,I_2)$, then we set the image to be $(\langle \eta_f^{(1)} \rangle, \langle \sigma_y^{(z)} \rangle \cup \langle \eta_h \rangle)$, so that the image is all pairs $(I_5,I_6)$ such that $\eta_h \in I_6$; $\tau^{top} \notin I_6$; and, $\sigma_1 \notin I_5$. The image is defined dually for a pair described by (b); notice that no pair can be described by both (a) and (b). 
    \item Let $B_{5}$ be the set of pairs $(I_1,I_2)$ such that $\tau^+ \in I_2$ and $\tau^{top} \notin I_1$. If $\eta_\ell^{(1)} \in I_1$, we set the image of such a pair to be $(\langle \tau^{top} \rangle, \langle \eta_c \rangle \cup \langle \sigma_y^{(1)} \rangle)$ and otherwise we set the image to be $(\langle \tau^{top} \rangle, \langle \eta_{\ell-1} \rangle \cup \langle \sigma_y^{(1)} \rangle) $so that the image is the set of pairs $(I_5,I_6)$ such that $\tau^{top} \in I_5$ and $f\geq \ell - 1$, with strict inequality if $\calP_1(a) \in I_5$.
    \item Let $B_6$ be the set of pairs $(I_1,I_2)$ such that $\tau^{top} \in I_5$ and ($f \geq \ell - 1$, with strict inequality if $\calP_1(a) \in I_1)$ or ($x \geq k-1$, with strict inequality if $\calP_1(b) \in I_1$). If the first condition is true, then we set the image to be $(\langle \eta_f^{(1)} \rangle \cup \langle \sigma_x \rangle, \langle \tau^{top} \rangle)$ and if only the second condition is true then we set the image to be $(\langle \tau^{top} \rangle, \langle \sigma_x^{(1)} \rangle \cup \langle \eta_f \rangle)$. Thus, the image is pairs $(I_5,I_6)$ such that $\tau^{top} \in I_6$ or $\tau^{top} \in I_5$ and $f \leq \ell - 1$, with equality only if $\calP_1(a) \in I_5$. 
\end{enumerate}

Now, we consider the $\gb$-vectors, which by Corollary \ref{cor:y-hat_spokes} will not depend on the winding numbers. We have $\gb_{\gamma} = \eb_\tau - \eb_{\sigma_k} - \eb_{\eta_\ell} + \delta_{\calP_1(a) \succ \calP_1(a-1)} \eb_{\calP_1(a)} + \delta_{\calP_1(b) \succ \calP_1(b+1)} \eb_{\calP_1(b)} + \gb'$ where $\gb'$ is the contribution of elements in $\calP_1[<a] \cup \calP_1[>b]$ to $\gb_\gamma$. Then, noting $\gb_{\tau^{(p,q)}} = -\eb_\tau$, we can quickly see $\gb_{\gamma} + \gb_{\tau^{(p,q)}} = \gb_{\gamma_3} + \gb_{\gamma_4}$. 
Similarly, $Y_F = y_\tau$ and $\degx(\hat{y}_\tau) = \eb_{\sigma_m} + \eb_{\eta_h} - \eb_{\sigma_1} - \eb_{\eta_1}$, and we can see $\gb_{\gamma} + \gb_{\tau^{(p,q)}} + \degx(\hat{y}_\tau) = \gb_{\gamma_5} + \gb_{\gamma_6}$. 

\textbf{Cases ii and iii) $\gamma$ crosses $\sigma_1$ immediately before crossing $\tau$ and $\tau$ is not the last arc crossed by $\gamma$}. 
 
There are two cases here, based on which arc $\gamma$ crosses after crossing $\tau$, and each strongly resembles Type 1 resolutions (Section \ref{subsec:Type1}).

\textbf{Case iv) $\gamma$ crosses $\sigma_1$ and $\eta_1$ immediately before and after crossing $\tau$}.
 
This can be derived from Case i in the same way that we could prove Proposition \ref{prop:typeII_singlytagged_winding_2} from Proposition \ref{prop:typeII_singlytagged_winding_1}. That is, the posets involved here are duals of the posets for the first case, and one can check that the dual version of the bijections follow the desired format to show the skein relation, as is detailed in Section \ref{subsec:proofStrategy}. 
  
\textbf{Case v) $\tau$ is the first or last arc crossed by $\gamma$}. 

There are a variety of subcases within this case. However, each can be derived from a more general case in the same manner as we dealt with such a configuration in the proof of Proposition \ref{prop:IntArcFromTSingleTag}.
\end{proof}

\section{Self-Crossings}
\label{sec:self_crossings}

We consider here skein relations of the form $x_\gamma = x_{C^+} + Y x_{C^-}$ which correspond to to an arc which is incompatible with itself. Proposition \ref{prop:incompatible_puncture_singly_tagged} dealt with arcs $\gamma$ with incompatible tagging, so here we only need to consider arcs with a self-intersection. An arc can have a self-intersection of Type 0, 1, or 2. 

If $\gamma$ has a self-intersection of Type 1, let $p$ be the starting or ending point of $\gamma$ near this point of intersection, and suppose $\gamma$ winds counterclockwise around $p$ at this point of intersection. Then, Proposition \ref{prop:GraftingType1} holds when we label the resolution $\{\gamma_{34}\} \cup \{\gamma_5,\gamma_6\}$ and replace $x_{\gamma_3}x_{\gamma_4}$ with $x_{\gamma_{34}}$. 

If $\gamma$ has a self-intersection of Type 2, then let $\gamma_3$ be the closed curve. Then, the results from Section \ref{subsec:Type2} hold when we label the resolution $\{\gamma_3,\gamma_4\} \cup \{\gamma_{56}\}$ and replace $\gamma_5 \gamma_6$ with $\gamma_{56}$. The untagged version of these two cases come from Theorem 7.4 in \cite{snakegraphcalculus2}. 

There are multiple ways in which an arc can have a self-intersection of Type 0. Recall this will mean that the poset will have  a pair of isomorphic subposets, such that one lies at the bottom and one lies at the top of the whole poset. Because these are now subposets of the same poset, there are added complications that can arise. One such scenario is when, given a chronological ordering of the poset $\calP$ and a pair of isomorphic subposets $\{\calP(s),\ldots,\calP(t)\}$ and $\{\calP(s'),\ldots,\calP(t')\}$, the isomorphism between these posets reverses them, sending $\calP(s+i)$ to $\calP(t'-i)$.We call this a ``reverse overlap''. It is also possible for the subposets to ``kiss'' or intersect, so that for some $i$, $\calP_\gamma(i)$ is in one subposet and $\calP_\gamma(i+1)$ is in the other.   It is not geometrically possible for a reverse overlap to be kissing or intersecting as well, so these complications will thankfully not compound. 

In Section \ref{subsec:OrdinarySelfInt}, we deal with overlaps which do not kiss or intersect, including reverse overlaps. In these cases, we can use methods similar to those for two distinct intersecting arcs. In Section \ref{subsec:Kissing} we address a self-intersection which results in a kissing overlap. We do not address overlaps which intersect as taggings on arcs do not affect these. That is, even in a punctured surface, Theorem 4.12 from \cite{snakegraphcalculus2} holds as stated.

In general, the resolution of a self-intersection results in one set containing an arc and a closed curve and another set containing a single arc. We will refer to these as $\{\gamma_3,\gamma_4\} \cup \{\gamma_{56}\}$ or $\{\gamma_{34}\} \cup \{\gamma_3,\gamma_4\}$ depending on which scenario we are in and which term gets a $y$-monomial in the expansion.

\subsection{Non-intersecting and non-kissing overlaps}\label{subsec:OrdinarySelfInt}

Suppose $\gamma$ has a self-intersection resulting in $R_1 = \calP_\gamma[s,t]$ and $R_2 = \calP_\gamma[s',t']$, isomorphic subsets of $\calP_\gamma$ which have a crossing overlap. Here, we suppose that $R_1 \cap R_2 = \emptyset$ and that $t \neq s'-1$ nor $t' = s-1$. We have two cases, based on whether, with a chronological ordering which has $s<t$, we have $s'<t'$ (\emph{a forward overlap}) or $t' < s'$ (\emph{a reverse overlap}).

Suppose first that $\calP_\gamma$ has a forward overlap, so that with the same chronological ordering we have $s<t$ and $s'<t'$. Then, we set $C^+ = \{\gamma_3,\gamma_4\}$ where $\gamma_4$ is a closed curve and $C^- = \{\gamma_{56}\}$. Constructing the corresponding posets is similar to the two arc case, detailed in Section \ref{subsec:Type0}. For instance, to describe $\calP_{\gamma_4}$, one takes $\calP_{\gamma}[\leq t'] \cap \calP_{\gamma}[>t]$, so that one has $\calP_{\gamma}[t+1,t']$ with added relation $\calP_{\gamma}(t') \succ \calP_\gamma(t+1)$, resulting in a circular fence poset. One can similarly combine the descriptions of $\calP_{\gamma_5}$ and $\calP_{\gamma_6}$ in Section \ref{subsec:Type0} to describe $\calP_{\gamma_{56}}$ here.

\begin{prop}\label{prop:Relations:ForwardOverlap}
Let $\gamma$ be an arc with a self-intersection which appears as a  forward, non-intersecting, and non-kissing crossing overlap in $\calP_\gamma$. Let $\{\gamma_3,\gamma_4\} \cup \{\gamma_{56}\}$ be the resolution of this intersection. Then, $x_\gamma = x_{\gamma_3} x_{\gamma_4} + Y x_{\gamma_{56}}$ where $Y$ is defined as follows ,\[
Y = \begin{cases} Y_R & (s > 1 \textbf{ or } s=1 \textbf{ and } s(\gamma)\text{ is notched}) \textbf{ and } (t<n \textbf{ or } t=n \textbf{ and }t(\gamma) \text{ is plain})\\
 Y_R Y^{\leq (s'-1)}_{\preceq (s'-1)}& s = 1 \textbf{ and }s(\gamma) \text{ is plain} \textbf{ and } (t<n \textbf{ or } t=n \textbf{ and }t(\gamma) \text{ is plain})\\
 Y_R Y^{\geq(t+1)}_{\succeq (t'+1)} & (s > 1 \textbf{ or } s=1 \textbf{ and } s(\gamma)\text{ is notched}) \textbf{ and } t=n \textbf{ and }t(\gamma) \text{ is notched}\\
Y_R Y^{\leq (s'-1)}_{\preceq (s'-1)} Y^{\geq(t+1)}_{\succeq (t'+1)} & s = 1 \textbf{ and }s(\gamma) \text{ is plain} \textbf{ and } t=n \textbf{ and }t(\gamma) \text{ is notched}\\
\end{cases}
\]
where 
\[
Y_R = \prod_{\calP_\gamma(i) \in R_1} y_{\calP_\gamma(i)}\qquad Y^{\leq (s'-1)}_{\preceq(s'-1)} = \prod_{\substack{\calP_\gamma(i) \preceq \calP_\gamma(s'-1)\\ i \leq s'-1}} y_{\calP_\gamma(i)} \quad \text{and} \quad Y^{\geq (t+1)}_{\succeq(t'+1)} = \prod_{\substack{\calP_\gamma(i) \succeq \calP_\gamma(t'+1) \\ i \geq t+1}} y_{\calP_\gamma(i)}.
\]
\end{prop}

The proof of these relations follows exactly the format of the proof of Proposition \ref{prop:GraftingType0}.

Now, we turn to reverse overlaps. While the type of arcs which appear will be the same, we will see that the $y$-monomial will now appear with the product of an arc and closed curve.

We again let  $R_1 = \calP_\gamma[s,t]$ and $R_2 = \calP_\gamma[s',t']$ be isomorphic posets with chronological ordering and indexing chosen so that $s < t$, $R_1$ is on top, and $R_2$ is on bottom. We also assume $s<s'-1$; if we have the opposite case, the construction would be similar. Since this is a reverse overlap, we have $t' < s'$. When the overlap is in reverse, one can check that it is impossible to have $R_1 \cap R_2$ be nonempty. Abstractly, we could have the case $t' = t+1$, but since we consider triangulations without self-folded triangles, we would never have an arc cross the same arc in the triangulation two times consecutively. Therefore, we assume that these overlaps do not kiss; that is, $t+1 < t'$.

We define $\calP_{\gamma_{34}}$ to be $\calP_\gamma$ with the chronological ordering on $\calP_{\gamma}[t+1,t'-1]$ reversed. That is, when listing the elements using the chronological ordering, this has the form $\cdots \calP_\gamma(t) \prec \calP_\gamma(t'-1) \cdots \calP_\gamma(t+1) \prec \calP_\gamma(t') \cdots$. Notice this poset has the same content as $\calP_\gamma$, but the sets $R_1$ and $R_2$ no longer form a crossing overlap. 

The poset $\calP_{\gamma_5}$ will be the circular fence poset on $\calP_\gamma[t+1,t'-1]$ where, since we know $\calP_\gamma(t),\calP_\gamma(t+1)$ and $\calP_\gamma(t'-1)$ form a triangle, we know there is a relation $\calP_\gamma(t+1) \succ \calP_\gamma(t'-1)$. The poset $\calP_{\gamma_6}$ can be defined similarly to how the poset $\calP_{\gamma_6}$ was defined in Section \ref{subsec:Type0}. 

\begin{prop}\label{prop:Relations:ReverseSelfOverlap}
Let $\gamma$ be an arc with a self-intersection which appears as a reverse crossing overlap in $\calP_\gamma$. Let $\{\gamma_{34} \} \cup \{\gamma_5,\gamma_6\}$ be the resolution of this intersection. Then, $x_\gamma = x_{\gamma_{34}} + Yx_{\gamma_5} x_{\gamma_6}$ where $Y$ is defined as follows, using analogous definitions to Proposition \ref{prop:Relations:ForwardOverlap},
\[
Y = \begin{cases} Y_R & (s > 1 \textbf{ or } s=1 \textbf{ and } s(\gamma)\text{ is notched}) \textbf{ and } (t<n \textbf{ or } t=n \textbf{ and }t(\gamma) \text{ is plain})\\
 Y_R Y^{\geq (s'+1)}_{\preceq (s'+1)}& s = 1 \textbf{ and }s(\gamma) \text{ is plain} \textbf{ and } (t<n \textbf{ or } t=n \textbf{ and }t(\gamma) \text{ is plain})\\
 Y_R Y^{\leq(s-1)}_{\succeq(s-1)} & (s > 1 \textbf{ or } s=1 \textbf{ and } s(\gamma)\text{ is notched}) \textbf{ and } t=n \textbf{ and }t(\gamma) \text{ is notched}\\
Y_R Y^{\geq (s'+1)}_{\preceq (s'+1)} Y^{\leq(s-1)}_{\succeq(s-1)} & s = 1 \textbf{ and }s(\gamma) \text{ is plain} \textbf{ and } t=n \textbf{ and }t(\gamma) \text{ is notched.}\\
\end{cases}
\]
\end{prop}

The proof of Proposition \ref{prop:Relations:ReverseSelfOverlap} also directly follows the ideas in the proofs of Proposition \ref{prop:GraftingType0}. 

It is also possible for $\gamma$ to have a self-intersection which resembles the Type 1 and Type 2 crossings discussed in Sections \ref{subsec:Type1} and \ref{subsec:Type2} respectively.

If $\gamma$ has a self-intersection of Type 1, let $p$ be the starting or ending point of $\gamma$ near this point of intersection, and suppose $\gamma$ winds counterclockwise around $p$ at this point of intersection. Then, Proposition \ref{prop:GraftingType1} holds when we label the resolution $\{\gamma_{34}\} \cup \{\gamma_5,\gamma_6\}$ and replace $x_{\gamma_3}x_{\gamma_4}$ with $x_{\gamma_{34}}$. 

If $\gamma$ has a self-intersection of Type 2, then let $\gamma_3$ be the closed curve. Then, the results from Section \ref{subsec:Type2} hold when we label the resolution $\{\gamma_3,\gamma_4\} \cup \{\gamma_{56}\}$ and replace $\gamma_5 \gamma_6$ with $\gamma_{56}$. The untagged version of these two cases come from Theorem 7.4 in \cite{snakegraphcalculus2}.

\subsection{Kissing Overlaps}\label{subsec:Kissing}

In this section, we consider an arc $\gamma$ with self-intersection such that its poset $\calP_\gamma$ has a kissing overlap. Precisely, $\calP_\gamma$ has isomorphic subposets $R_1$ and $R_2$ such that   $R_1 = \calP_\gamma[s,t]$, $R_2 = \calP_\gamma[s',t']$, and $t+1 = s'$. Without loss of generality, we suppose $R_1$ is on top and $R_2$ is on bottom, so that $\calP_{\gamma}(t) \succ \calP_{\gamma}(s')$. Let $\alpha$ denote the third arc in the triangle formed by $\calP_{\gamma}(t)$ and $\calP_{\gamma}(s')$. We necessarily have that $s > 1$ and $t' < n = \vert \calP_\gamma^0 \vert$; otherwise, this intersection would be covered by the Type 1 or Type 2 intersections, and the proof would follow immediately from the two arc case. Moreover,  notice that $\calP_\gamma(s-1)$ and $\calP_\gamma(t'+1)$ both correspond to $\alpha$. 

Let $b \geq 1$ be the largest value such that $\calP_\gamma(s-b)$ and $\calP_\gamma(t'+b)$ correspond to the same arc in $T$; call this arc $\beta$. Set $\beta^1 = \calP_\gamma(s-b), \beta^2 = \calP_\gamma(t'+b), \alpha^1 = \calP_\gamma(s-1)$ and $\alpha^2 = \calP_\gamma(t'+1)$. If $\calP_\gamma(s-b-1)$ and $\calP_\gamma(t'+b+1)$ exist, call these $\epsilon_1$ and $\epsilon_2$ respectively. We emphasize that $\epsilon_1$ and $\epsilon_2$ are distinct arcs while $\alpha^1$ and $\alpha^2$ stand for two copies of the same arc in $\calP_\gamma$ and similarly for $\beta^1,\beta^2$. Moreover, there are two options for the configuration of $\epsilon_1$ and $\epsilon_2$, depending on $\gamma$. We draw below a diagram of the relevant arcs from $T$, as well as the two ways that $\gamma$ can achieve this kissing overlap. 

\begin{center}
\begin{tikzpicture}[scale = 1]
\draw (0,0) circle (1);
\draw (0,0) circle (2);
\draw (2*-.94, 2*.34) -- (-1,0) -- (2*-.94,2*-.34);
\draw (2*-.94, 2*.34) -- (-5,.68) -- (-5,-.68) -- (-1.88, -.68);
\draw (-5,.68) --(-6,0) --  (-5,-.68);
\node[right] at (-5,0){$\beta$};
\node[left] at (-2,0){$\alpha$};
\node[above] at (-1.3,0.5){$\calP_\gamma(t)$};
\node[below] at (-1.3,-0.5){$\calP_\gamma(s)$};
\node[] at (-5.7,0.6){$\epsilon_i$};
\node[] at (-5.7,-0.6){$\epsilon_j$};
\end{tikzpicture}
\end{center}

\begin{center}
\begin{tabular}{cc}
\begin{tikzpicture}[scale = 0.75]
\draw (0,0) circle (1);
\draw (0,0) circle (2);
\draw (2*-.94, 2*.34) -- (-1,0) -- (2*-.94,2*-.34);
\draw (2*-.94, 2*.34) -- (-5,.68) -- (-5,-.68) -- (-1.88, -.68);
\draw (-5,.68) --(-6,0) --  (-5,-.68);
\draw[thick, orange, ->] (-5.7,-.45) to (-4.8,-.45);
\draw[thick, orange] (-4.8,-.45) to [out = 0, in = 160, looseness = 0.8] (-1.6,-.5);
\draw[thick, orange] (-1.6,-.5) to [out = 300, in = 180] (0,-1.8) to [out = 0, in = 270] (1.7,0) to [out = 90, in = 0] (0,1.6) to [out = 180, in = 90] (-1.5,0) to [out = 270, in = 180] (0,-1.4) to [out = 0, in = 270] (1.3,0) to [out = 90, in = 0] (0,1.2);
\draw[thick, orange] (0,1.2) to [out = 180, in = 0] (-2.5,0.45);
\draw[thick, orange, ->] (-2.5,0.45) to (-5.7,0.45);
\end{tikzpicture}
&
\begin{tikzpicture}[scale = 0.75]
\draw (0,0) circle (1);
\draw (0,0) circle (2);
\draw (2*-.94, 2*.34) -- (-1,0) -- (2*-.94,2*-.34);
\draw (2*-.94, 2*.34) -- (-5,.68) -- (-5,-.68) -- (-1.88, -.68);
\draw (-5,.68) --(-6,0) --  (-5,-.68);
\draw[thick, orange, ->] (-5.7,.45) to (-2.5,-.45);
\draw[thick, orange] (-2.5,-.45) to [out = -20, in = 160, looseness = 0.8] (-1.6,-.5);
\draw[thick, orange] (-1.6,-.5) to [out = 300, in = 180] (0,-1.8) to [out = 0, in = 270] (1.7,0) to [out = 90, in = 0] (0,1.6) to [out = 180, in = 90] (-1.5,0) to [out = 270, in = 180] (0,-1.4) to [out = 0, in = 270] (1.3,0) to [out = 90, in = 0] (0,1.2);
\draw[thick, orange] (0,1.2) to [out = 180, in = 0] (-2.5,0.45);
\draw[thick, orange, ->] (-2.5,0.45) to [out = 180, in = 45] (-5.7,-0.45);
\end{tikzpicture}\\
\end{tabular}
\end{center}

Set $M_1 = \calP_\gamma[s-b,s-1]$ and $M_2 = \calP_\gamma[t'+1,t'+b]$. The two choices for the orientation of $\gamma$ correspond to whether $M_1$ and $M_2$ are non-intersecting or intersecting overlaps, as below. 

\begin{center}
\begin{tabular}{c|c}
   \begin{tikzpicture}[scale=0.6]
   \node(e1) at (0,3){$\epsilon_1$};
    \node(m1) at (1.5,1.5){$\boxed{M_1}$};
    \node(r1) at (3,0){$\boxed{R_1}$};
    \node(r2) at (4.5,1.5){$\boxed{R_2}$};
    \node(m2) at (6,0){$\boxed{M_2}$};
    \node(e2) at (7.5,-1.5){$\epsilon_2$};
    \draw (m1) -- (r1);
    \draw(r1) -- (r2);
    \draw (r2) -- (m2);
    \draw (e1) -- (m1);
    \draw (m2) -- (e2);
   \end{tikzpicture}  &  
      \begin{tikzpicture}[scale = 0.6]
    \node(e1) at (0,0){$\epsilon_1$};
    \node(m1) at (1.5,1.5){$\boxed{M_1}$};
    \node(r1) at (3,0){$\boxed{R_1}$};
    \node(r2) at (4.5,1.5){$\boxed{R_2}$};
    \node(m2) at (6,0){$\boxed{M_2}$};
    \node(e2) at (7.5,1.5){$\epsilon_2$};
    \draw (m1) -- (r1);
    \draw(r1) -- (r2);
    \draw (r2) -- (m2);
    \draw (e1) -- (m1);
    \draw (m2) -- (e2);
   \end{tikzpicture}\\
\end{tabular}
\end{center}

Since either end of $\gamma$ may be notched, let $n' \leq 1$ and $n'' \geq d$ be such that $\calP_\gamma = \calP_\gamma[n',n'']$.

We draw below the resolution in the case where $M_1$ and $M_2$ are non-intersecting (i.e. the left picture above), with $s-b > 1$, and $t'+b < n$. The intersecting case is identical except for the fact that $\gamma_{56}$ now has a contractible kink. Recall from Definition \ref{def:contandkink} that we define $x_{\gamma_{56}}$ to be negative of the expression associated to $\gamma_{56}$ with the kink removed. 

\begin{center}
\begin{tabular}{cc}
\begin{tikzpicture}[scale = 0.75]
\draw (0,0) circle (1);
\draw (0,0) circle (2);
\draw (2*-.94, 2*.34) -- (-1,0) -- (2*-.94,2*-.34);
\draw (2*-.94, 2*.34) -- (-5,.68) -- (-5,-.68) -- (-1.88, -.68);
\draw (-5,.68) --(-6,0) --  (-5,-.68);
\draw[thick, orange, ->] (-5.7,-.45) to (-4.8,-.45);
\draw[thick, orange] (-4.8,-.45) to [out = 0, in = 160, looseness = 0.8] (-1.6,-.5);
\draw[thick, orange] (-1.6,-.5) to [out = 300, in = 180] (0,-1.8) to [out = 0, in = 270] (1.7,0) to [out = 90, in = 0] (0,1.6);
\draw[thick, orange] (0,1.6) to [out = 180, in = 0] (-2.5,0.45);
\draw[thick, orange, ->] (-2.5,0.45) to (-5.7,0.45);
\draw[thick, blue] (0,0) circle (1.35);
\end{tikzpicture}
&
\begin{tikzpicture}[scale = 0.75]
\draw (0,0) circle (1);
\draw (0,0) circle (2);
\draw (2*-.94, 2*.34) -- (-1,0) -- (2*-.94,2*-.34);
\draw (2*-.94, 2*.34) -- (-5,.68) -- (-5,-.68) -- (-1.88, -.68);
\draw (-5,.68) --(-6,0) --  (-5,-.68);
\draw[thick, orange, ->] (-5.7,-.45) to [out = 75, in = 285] (-5.7,0.45);
\end{tikzpicture}\\
\end{tabular}
\end{center}

Let the resolution of $\gamma$ be denoted $\{\gamma_3,\gamma_4\} \cup \{\gamma_{56}\}$ where $\gamma_3$ and $\gamma_{56}$ are arcs and $\gamma_4$ is a closed curve, as in the figures above. Then, $\calP_{\gamma_3}$ is the poset $\calP_\gamma[\leq t] \cup \calP_{\gamma}[> t']$, where we have the relation $\calP_\gamma(t) \succ \calP_\gamma(t'+1)$, which matches the orientation of the associated arcs. The poset $\calP_{\gamma_4}:= \calP_{\gamma_4}$ is the circular fence poset on $\calP_\gamma[s,t]$ with added relation $\calP_1(s)\succ\calP_1(t)$. 

The arc $\gamma_{56}$, with poset $\calP_{\gamma_{56}}$ can vary based on $s-b$ and $t'+n$ as well as the tagging at each endpoint of $\gamma$.

\begin{enumerate}
    \item If $s-b > 1$ and $t' + b < n$, then $\calP_{\gamma_{56}}$ is the poset on $\calP_\gamma[<s-b] \cup \calP_{\gamma}[>t'+b]$  where the relation between $\calP_1(s-b)$ and $\calP_1(t'+b)$ follows from the configuration of $\epsilon_1$ and $\epsilon_2$. 
    \item If $s-b = 1$ and $t'+b < n$, let $b< a \leq n-t'$ be the largest integer such that $\calP_\gamma[t'+b+1,t'+a]$ correspond to a set of arcs which are all incident to $s(\gamma)$. If $t'+a = n$ and $\calP_\gamma(t'+b) \succ \calP_\gamma(t'+b+1)$ ($\calP_\gamma(t'+b) \prec \calP_\gamma(t'+b+1)$), let $\zeta \in T$ be the arc which is clockwise (counterclockwise) of $\calP_\gamma(d)$ in the last triangle $\gamma$ passes through. 
    \begin{enumerate}
        \item If $t'+a < n$, then $\calP_{\gamma_{56}}$ is $\calP_\gamma[> t'+a]$.
        \item Suppose $t'+a = n$ and $s(\gamma)$ is plain.  If $t(\gamma)$ is plain, then $\calP_{\gamma_{56}}$ is the poset associated $\zeta$ and if $t(\gamma)$ is notched, then $\calP_{\gamma_{56}}$ is the poset associated to $\zeta^{(t(\gamma))}$.
        \item Suppose $t'+a = n$ and $s(\gamma)$ is notched. If $t(\gamma)$ is plain, then $\calP_{\gamma_{56}}$ is the poset for $\zeta^{(s(\gamma))}$ and if $t(\gamma)$ is notched, then $\calP_{\gamma_{56}}$ is the poset for $\zeta^{(s(\gamma), t(\gamma))}$ .
    \end{enumerate}
    \item If $s-b>1$ and $t'+b=n$, we have a dual set of conditions to above to describe $\calP_{\gamma_{56}}$. 
    \item If $s-b = 1$ and $t'+b = n$, then $\gamma_{56}$ is contractible and we do not associate a poset to $\gamma_{56}$. If the taggings on $s(\gamma)$ and $t(\gamma)$ are the same, then  $x_{\gamma_{56}} = 0$. If the taggings are different, then $x_{\gamma_{56}}$ is given by Definition \ref{def:ContractibleSinglyTaggedLoop}.
\end{enumerate}

We will use the letters $a$ and $b$ as described here in the following Proposition and proof.

\begin{prop}
Let $\gamma$ be an arc such that $\calP_{\gamma}$ has a crossing overlap in regions $R_1 = \calP_\gamma[s,t]$ and $R_2 = \calP_\gamma[s',t']$ with $t+1 = s'$. Let $M_1$ and $M_2$ be the overlaps which come before $R_1$ and after $R_2$. Let $Y_R = \prod_{i=s}^t y_{\calP_\gamma(i)}$ and $Y_M = \prod_{i = s-b}^{s-1} y_{\calP_\gamma(i)}$. 
Then, $x_\gamma = x_{\gamma_3}x_{\gamma_4} \pm Y x_{\gamma_{56}}$ where the sign is $+$ when $M_1$ and $M_2$ are non-intersecting and is $-$ when $M_1$ and $M_2$ are intersecting, and the monomial $Y$ is determined by the following conditions.
\begin{enumerate}
    \item If $s-b > 1$ and $t'+b <n$, $Y = Y_R Y_M$.
    \item If $s-b = 1$ and $t'+b < n$ and
    \begin{enumerate}
        \item $M_1$ and $M_2$ are non-intersecting and $s(\gamma)$ is plain, then $Y = Y_R Y_M$.
        \item $M_1$ and $M_2$ are intersecting and $s(\gamma)$ is plain, then if $t(\gamma)$ is plain, $Y = Y_R Y_M\prod_{i=t'+b+1}^{t'+a} y_{\calP_\gamma(i)}$ and if $t(\gamma)$ is notched, $Y = Y_R Y_M\prod_{i=t'+b+1}^{t'+a+1} y_{\calP_\gamma(i)}$.
        \item $M_1$ and $M_2$ are non-intersecting and $s(\gamma)$ is notched, then $Y = Y_R Y_M \prod_{i=t'+b+1}^{t'+a} y_{\calP_\gamma(i)}$ .
        \item $M_1$ and $M_2$ are intersecting and $s(\gamma)$ is notched, then $Y = Y_R Y_M$.
    \end{enumerate}
    \item If $s-b  > 1$ and $t'+b = n$, $Y$ is defined in a parallel way to part 2. 
    \item If $s-b = 1$ and $t'+b = n$, $Y = Y_R Y_M$.
\end{enumerate}
\end{prop}

\begin{proof}

Set $\calP:= \calP_\gamma$ and for $i = 3,4, 56$, $\calP_i:= \calP_{\gamma_i}$. For shorthand, we will refer to the set $\calP_\gamma[s,t]$ in $\calP_{3}$ as $R_3$, the set $\calP_\gamma[s-b,s-1]$ as $M_3$, and the set $\calP_\gamma[t'+1,t'+b]$ in $\calP_{\gamma_3}$ as $M_4$. The posets $\calP_{\gamma_3}$ for $M_1$ and $M_2$ non-intersecting (left) and intersecting (right) are drawn below.

\begin{center}
\begin{tabular}{c|c}
    \begin{tikzpicture}[scale = 0.7]
    \node(e1) at (0,2){$\epsilon_1$};
    \node(m3) at (1,1){$\boxed{M_3}$};
    \node(r3) at (2,-0.5){$\boxed{R_3}$};
    \node(m4) at (3,-2){$\boxed{M_4}$};
    \node(e2) at (4,-1){$\epsilon_2$};
    \draw (m3) -- (r3);
    \draw (e1) -- (m3);
    \draw(r3) -- (m4);
    \draw (m4) -- (e2);
   \end{tikzpicture} & 
      \begin{tikzpicture}[scale = 0.7]
    \node(e1) at (0,0.5){$\epsilon_1$};
    \node(m3) at (1,1.5){$\boxed{M_3}$};
    \node(r3) at (2,0){$\boxed{R_3}$};
    \node(m4) at (3,-1.5){$\boxed{M_4}$};
    \node(e2) at (4,-0.5){$\epsilon_2$};
    \draw (m3) -- (r3);
    \draw (e1) -- (m3);
    \draw(r3) -- (m4);
    \draw (m4) -- (e2);
   \end{tikzpicture}
\end{tabular}
\end{center}

We will begin with part 1. Even when $\gamma$ is notched at one or both endpoints, this case follows from translating the proof mechanics in Section 4.5 of \cite{snakegraphcalculus2} to our poset setting. We will still go into detail for this case as it will streamline our remaining cases, which are distinct from the work in \cite{snakegraphcalculus2}.

\textbf{Case i) $s-b>1$; $t'+b < n$; and, $M_1$ and $M_2$ are non-intersecting.}

Let $A_1 \subset J(\calP)$ be the set of all $I_1$ such that we do not have $R_1 \subseteq I_1$ and  $(R_2 \cup \{\alpha^2\}) \cap I_1 = \emptyset$. Given $I_1 \in A_1$, $\Phi$ sends the elements from $R_1$ to the corresponding elements in $R_3$ and the elements from $R_2$ to $\calP_{4}$ until the first switching position, if it exists, working from $\calP(s)$ and $\calP(s')$ to $\calP(t)$ and $\calP(t')$ and then it swaps where these go. If this switching position does not exist, it must have been $R_1 \cup \{\alpha^2\} \subseteq I_1$ and $R_2 \cap I_1 = \emptyset$. In this case, we send $R_1$ to $R_3$. The elements of $M_1$ go to the corresponding elements in $M_3$ and similarly for $M_2$ and $M_4$, and all other elements have clear images in $\calP_{3}$. The image of this map is all tuples $(I_3,I_4) \in J(\calP_{3}) \times J(\calP_{4})$ except those of the form $(I_3,I_4)$ with $(R_3 \cup \{\alpha^2\}) \cap I_3 = \emptyset$ and $I_4 = \calP_4$. There is a clear inverse map from this subset of $J(\calP_3) \times J(\calP_4)$ back to $A_1$, which sends elements in $R_3$ to $R_1$ and in $\calP_{4}$ to $R_2$ until the first switching position (using the same ordering on these as we did for $R_1$ and $R_2$). 

Now, let $A_2 \subset J(\calP)$ be the set of all $I_1$ such that that $R_1 \subseteq I_1$; $(R_2 \cup \{\alpha^2\}) \cap I_1 = \emptyset$; and, if we have both $M_1 \subseteq I_1$ and $M_2 \cap I_1 = \emptyset$, then $\epsilon_2 \in I_1$ and $\epsilon_1 \notin I_1$.  The map $\Phi \vert_{A_2}$ acts as follows. If there is a switching position between $M_1$ and $M_2$, $\Phi \vert_{A_2}$ sends the elements of $M_1$ to $M_3$ and the elements of $M_2$ to $M_4$ until the first switching position, working from $\calP(s-b)$ to $\calP(s-1)$ and from  $\calP(t'+b)$ to $\calP(t'+1)$. If $M_1 \subseteq I_1$ and $M_2 \cap I_1 = \emptyset$, $\Phi$ send the elements of $M_1$ to $M_4$. We set $\Phi(R_1) = \calP_{4}$, and the other elements have clear images. The image of this subset of $J(\calP)$ is tuples of the form $(I_3,I_4) \in J(\calP_{3}) \times J(\calP_{4})$ where $R_3 \cap I_3 = \emptyset$ and $I_4 = \calP_4$. By instead considering the first switching position between $M_3$ and $M_4$, we can form an inverse map from this subset of $J(\calP_{3}) \times J(\calP_{4})$ to $A_2$.

Finally, let $B = J(\calP) \backslash (A_1 \cup A_2)$, which can be characterized as the set of $I_1$ such that $R_1 \cup M_1 \subseteq I_1$; $(R_2 \cup M_2) \cap I_1 = \emptyset$; and, $\epsilon_2 \in I_1$ only if $\epsilon_1 \in I_1$. Then, $\Phi(I_1)$ sends $I_1$ to the order ideal $I_1 \backslash (R_1 \cup M_1)$ in $J(\calP_{\gamma_{56}})$.  The inverse map is clear in this case.

Next, we check the $\gb$-vectors. By comparing the shape of $\calP_\gamma$ with the expansion in Lemma \ref{lem:y-hat}, we see that

\begin{align}
\begin{split}
\gb_\gamma &= 2\sum_{\substack{x \in (s,t)\\ \calP(x) \text{ maximal }}} \eb_{\calP(x)} - 2 \sum_{\substack{x \in (s,t)\\ \calP(x) \text{ minimal }}} + 2\sum_{\substack{x \in (s-b,s-1)\\ \calP(x) \text{ maximal }}} \eb_{\calP(x)} - 2\sum_{\substack{x \in (s-b,s-1)\\ \calP(x) \text{ minimal }}} \eb_{\calP(x)} \\
&+ (\delta_{\calP(s) \succ \calP(s+1)} - \delta_{\calP(s) \prec \calP(s+1)}) \eb_{\calP(s)} + (\delta_{\calP(t) \succ \calP(t-1)} - \delta_{\calP(t) \prec \calP(t-1)}) \eb_{\calP(t)}\\
&+ (\delta_{\calP(s-1) \succ \calP(s-2)} - \delta_{\calP(s-1) \prec \calP(s-2)}) \eb_{\alpha} + (\delta_{\calP(s-b) \succ \calP(s-b-1)} - \delta_{\calP(s-b) \prec \calP(s-b-1)}) \eb_{\beta}\\
& + \delta_{\calP(s-b-1) \succ \calP(s-b-2)}\eb_{\epsilon_1} - (1-\delta_{\calP(t'+b+1) \succ \calP(t'+b+2)}) \eb_{\epsilon_2} + \gb_\gamma'
\end{split}\label{eqn:self-int-gvec}
\end{align}

where $\gb_\gamma'$ reflects the contributions of elements in $\calP[<s-b-1] \cup \calP[>t'+b+1]$.

Similarly, we have 
\begin{align*}
\gb_{\gamma_3} &= \sum_{\substack{x \in (s,t)\\ \calP(x) \text{ maximal }}} \eb_{\calP(x)} -  \sum_{\substack{x \in (s,t)\\ \calP(x) \text{ minimal }}} + 2\sum_{\substack{x \in (s-b,s-1)\\ \calP(x) \text{ maximal }}} \eb_{\calP(x)} - 2\sum_{\substack{x \in (s-b,s-1)\\ \calP(x) \text{ minimal }}} \eb_{\calP(x)} \\
& - \delta_{\calP(s) \prec \calP(s+1)} \eb_{\calP(s)} + \delta_{\calP(t) \succ \calP(t-1)} \eb_{\calP(t)}\\
&+ (\delta_{\calP(s-1) \succ \calP(s-2)} - \delta_{\calP(s-1) \prec \calP(s-2)}) \eb_{\alpha} + (\delta_{\calP(s-b) \succ \calP(s-b-1)} - \delta_{\calP(s-b) \prec \calP(s-b-1)}) \eb_{\beta}\\
& + \delta_{\calP(s-b-1) \succ \calP(s-b-2)}\eb_{\epsilon_1} - (1-\delta_{\calP(t'+b+1) \succ \calP(t'+b+2)}) \eb_{\epsilon_2} + \gb_\gamma'
\end{align*}\label{eq:g-vector}
and
\[
\mathbf{g}_{\gamma_4} = \delta_{\calP(s) \succ \calP(s+1)} \eb_{\calP(s)} - \delta_{\calP(t') \prec \calP(t'-1)}\eb_{\calP(t')} + \sum_{\substack{x \in (s,t)\\ \calP(x) \text{ maximal }}} \eb_{\calP(x)} -  \sum_{\substack{x \in (s,t)\\ \calP(x) \text{ minimal }}} \eb_{\calP(x)}
\]

so we have $\gb_\gamma = \gb_{\gamma_3} + \gb_{\gamma_4}$. 

From Lemma \ref{lem:y-hat}, the first three lines of Equation \ref{eq:g-vector} are $\deg_x(\hat{Y}_R \hat{Y}_M)$. Since $\gb_{\gamma_{56}} = \delta_{\calP(s-b-1) \prec \calP(s-b-2)}\eb_{\epsilon_1} - (1-\delta_{\calP(t'+b+1) \succ \calP(t'+b+2)}) \eb_{\epsilon_2} + \gb_\gamma'$, we can see that $\gb_{\gamma} + \degx(\hat{Y}_R\hat{Y}_M) = \gb_{\gamma_{56}}$.

\textbf{Case ii) $s-b > 1$; $t'+b < n$; and, $M_1$ and $M_2$ are intersecting.} 

Since our relation is now of the form $x_\gamma = x_{\gamma_3}x_{\gamma_4} - Yx_{\gamma_{56}}$, we construct a bijection $\Phi: J(\calP) \cup J(\calP_{56}) \to J(\calP_3) \times J(\calP_4)$ such that $\Phi$ restricted to $\calP$ preserves content, and given $I_{56} \in \calP_{56}$, $\cont(I_{56}) = \cont(\Phi(I_{56})) \cup R_1 \cup M_1$.

\begin{enumerate}
\item Let $A_1 \subseteq J(\calP)$ and the definition of $\Phi \vert_{A_1}$ be as in case i. 
\item Let $A_2 \subseteq J(\calP)$ be the set of $I_1 \in J(\calP)$ such that $R_1 \subseteq I_1$; $(R_2 \cup \{\alpha^2\}) \cap I_1 = \emptyset$; and, such that we do not have both $M_1 \subset I_1$ and $M_2 \cap I_1 = \emptyset$. We describe $\Phi\vert_{A_2}$. The set $R_1$ is sent to $\calP_{\gamma_4}$. If we have a switching position between $I_1\vert_{M_1}$ and $I_1 \vert_{M_2}$, then we map the elements of $M_1$ to $M_4$ and the elements of $M_2$ to $M_3$ until the first switching position where we switch, working from $\calP(s-b)$ to $\calP(s-1)$ and from $\calP(t'+b)$ to $\calP(t'+1)$. The only way we cannot have a switching position is when $(M_1 \cup M_2) \cap I_1 = \emptyset$, in which case the image is clear.  The image of $\Phi\vert_{A_2}$ is all $(I_3,I_4) \in J(\calP_{3}) \times J(\calP_{4})$ such that $I_4 = \calP_{4}$; $I_3 \cap R_3 = \emptyset$; and, we do not have both $M_4 \subseteq I_3$ and $M_3 \cap I_3 = \emptyset$.
\item Let $A_3 \subseteq J(\calP)$ be the set of $I_1 \in J(\calP)$ such that $(R_1 \cup M_1) \subseteq I_1$ and $(R_2 \cup M_2) \cap I_1 = \emptyset$. Here, we send $M_1$ to $M_4$ and $R_1$ to $\calP_{4}$. The image of $\Phi\vert_{A_3}$ is all $(I_3,I_4) \in J(\calP_{3}) \times J(\calP_{4})$ such that $I_4 = \calP_{4}$; $(M_3 \cup R_3) \cap I_3 = \emptyset$; $M_4 \subset I_3$; $\epsilon_1 \in I_3$; and, $\epsilon_2 \notin I_3$.
\item Finally, we construct a map $J(\calP_{56}) \to J(\calP_{3}) \times J(\calP_{4})$. Given $I_{56} \in J(\calP_{56})$, we send $I_{56}$ to $(I_{56} \cup M_4, \calP_{4})$. The image of this map is all $(I_3,I_4) \in J(\calP_{3}) \times J(\calP_{4})$ such that $I_4 = \calP_{4}$; $(M_3 \cup R_3) \cap I_3 = \emptyset$; $M_4 \subset I_3$;  and $\epsilon_1 \in I_3$ only if $\epsilon_2 \in I_3$.
\end{enumerate}

Each described map is bijective, with clear inverse from the described image. The combinatorics of the $\gb$-vector computation is similar to the non-intersecting case, with the roles of $\calP(s-b-1)$ and $\calP(t'+b+1)$ flipped. 

\textbf{Case iii) $s-b = 1$; $t'+b < n$; and, $s(\gamma)$ is plain.}

If $M_1$ and $M_2$ are non-intersecting, the computations are essentially the same as in case i. So we suppose that $M_1$ and $M_2$ are intersecting. Recall from the exposition beforehand that $a>b$ is the largest integer such that $\calP_\gamma(t'+b+1),\ldots,\calP_\gamma(t'+a)$ correspond to a set of arcs from $T$ incident to $s(\gamma)$.  For shorthand, in the diagram below we denote $\calP_\gamma(t'+i) = \lambda_i$; in particular, $\lambda_{b+1}$ corresponds to $\epsilon_2$ from the exposition above the Proposition, and $\lambda_{a+1}$ is not in this set of arcs with common endpoint $s(\gamma)$. It is possible that $\lambda_{a+1}$ does not exist or is part of a loop. We have a diagram of $\calP$ on the left, $\calP_{3}$ in the middle, and $\calP_{56}$ on the right; $\calP_{4}$ looks the same in all our cases.

\begin{center}
\begin{tabular}{cc} 
      \begin{tikzpicture}[scale = 0.67]
    \node(m1) at (1,1.5){$\boxed{M_1}$};
    \node(r1) at (2,0){$\boxed{R_1}$};
    \node(r2) at (3,1.5){$\boxed{R_2}$};
    \node(m2) at (4,0){$\boxed{M_2}$};
    \node(q1) at (5,1.5){$\lambda_{b+1}$};
    \node(ddots) at (6,0){$\ddots$};
    \node(qk) at (7,-1){$\lambda_{a}$};
    \node[](next) at (8,0){$\lambda_{a+1}$};
    \node[] at (9,0){$\cdots$};
    \draw (m1) -- (r1);
    \draw(r1) -- (r2);
    \draw (r2) -- (m2);
    \draw (m2) -- (q1);
    \draw (q1) -- (ddots);
    \draw (ddots) -- (qk);
    \draw (qk) -- (next);
   \end{tikzpicture} & 
      \begin{tikzpicture}[scale = 0.67]
    \node(m3) at (1,1.5){$\boxed{M_3}$};
    \node(r3) at (2,0){$\boxed{R_3}$};
    \node(m4) at (3,-1.5){$\boxed{M_4}$};
    \node(q1) at (4,0){$\lambda_{b+1}$};
    \node(ddots) at (5,-1.5){$\ddots$};
    \node(qk) at (6,-2.5){$\lambda_{a}$};
    \node[](next) at (7,-1.5){$\lambda_{a+1}$};
    \node[] at (8,-1){$\cdots$};
    \draw (m3) -- (r3);
    \draw(r3) -- (m4);
    \draw (m4) -- (q1);
    \draw (q1) -- (ddots);
    \draw (ddots) -- (qk);
    \draw (qk) -- (next);
    \node[] at (11.5,-1){$\lambda_{a+1}$};
    \node[] at (12.5,-1){$\cdots$};
   \end{tikzpicture}
\end{tabular}
\end{center}

\textbf{Subcase) $t'+a<n$}

In this case, $\lambda_{a+1}$ is in $\calP_\gamma^0$.
 We use the same map from $J(\calP_1)$ to $J(\calP_{\gamma_3}) \times J(\calP_{\gamma_4})$ as in the general case for $M_1$ and $M_2$ intersecting. The cokernel of the map in this case consists of all $(I_3,I_4) \in J(\calP_{3}) \times J(\calP_{4})$ such that $(M_3 \cup R_3) \cap I_3 = \emptyset$; $I_4 = \calP_4$; and $(M_4 \cup \calP_1[t'+b+1, t'+a]) \subseteq I_3$. These sets are in bijection with $J(\calP_{56})$, where we send $I_{56}$ to $(I_{56} \cup (M_4 \cup \calP_1[t'+b+1, t'+a]), \calP_{4})$.

 We check the $\gb$-vectors in this case. It is clear that $\mathbf{g}_\gamma = \mathbf{g}_{\gamma_3} + \mathbf{g}_{\gamma_4}$ as in the general case, so we check only that $\gb_{\gamma} + \degx(\hat{Y}_R\hat{Y}_M\prod_{i=b+1}^{a}\hat{y}_{\lambda_i}) = \gb_{\gamma_{56}}$.

 The expression for $\gb_\gamma$ can be derived by taking Equation \ref{eq:g-vector} and replacing the coefficient of $\eb_\beta$ with $-2\delta_{\calP(1) \prec \calP(2)}$, replacing the coefficient of $\eb_{\epsilon_2}$ with 1, and deleting the term $\eb_{\epsilon_1}$. We can then replace $\gb_\gamma'$ with $-\eb_{\lambda_a} + \delta_{\calP(t'+a+1) \succ \calP(t'+a+2)} \eb_{\lambda_{a+1}} + \gb_\gamma''$ where $\gb_\gamma''$ reflects the contributions of elements in $\calP[>t'+a+1]$. 

 From Corollary \ref{cor:y-hat_spokes}, we compute $\degx(\prod_{i=b+1}^a \hat{y}_{\lambda_i}) = -\eb_{\epsilon_2} + \eb_{\lambda_a} - \eb_{\lambda_{a+1}} +\eb_\beta -\eb_{\epsilon_1} + \eb_{\zeta}$ where $\zeta$ is the counterclockwise neighbor of $\lambda_a$ in the triangle which is also bordered by $\lambda_{a+1}$ and $\epsilon_1$ is as in the general case.  Now we have  $\gb_{\gamma_{56}}=  \eb_\zeta + (-1 + \delta_{\calP(t'+a+1) \succ \calP(t'+a+2)}) \eb_{\lambda_{a+1}} + \gb_\gamma''$ and one can check $\gb_\gamma + \degx(\hat{Y}_R \hat{Y}_M \prod_{i=b+1}^a \hat{y}_{\lambda_i}) = \gb_{\gamma_{56}}$.

\textbf{Subcase) $t'+a = n$ and $t(\gamma)$ is plain.}

If we use the same map as in the previous case, the cokernel has only one element. The poset $J(\calP_{56})$ is $\emptyset$ with decoration  $\{\zeta\}$. So we map the unique element of $J(\calP_{56})$ to the unique element in the cokernel of $\Phi$. The $\gb$-vector computation is very similar to the previous. 

\textbf{Subcase) $t'+a = n$ and $t(\gamma)$ is notched.}

Let $\sigma_1,\ldots,\sigma_m$ be the spokes at $t(\gamma)$, such that $\sigma_1$ and $\sigma_m$ bound the last triangle $\gamma$ passes through and $\sigma_1$ is the counterclockwise neighbor of $\lambda_a$ in the last triangle $\gamma$ passes through. Then, $\calP_1$ and $\calP_{3}$ are as below. By our chronological ordering, $\calP[t'+a+1] = \sigma_1$. Since $\gamma_{56} = \sigma_1^{(t(\gamma))}$, we immediately know that the poset $\calP_{56}$ is the chain $\sigma_2 \prec \sigma_3 \prec \cdots \prec \sigma_m$. 

\begin{center}
\begin{tabular}{cc}
      \begin{tikzpicture}[scale = 0.7]
    \node(m1) at (1,1.5){$\boxed{M_1}$};
    \node(r1) at (2,0){$\boxed{R_1}$};
    \node(r2) at (3,1.5){$\boxed{R_2}$};
    \node(m2) at (4,0){$\boxed{M_2}$};
    \node(q1) at (5,1.5){$\lambda_{b+1}$};
    \node(ddots) at (6,0){$\ddots$};
    \node(qk) at (7,-1){$\lambda_{a}$};
    \node(z) at (11,-2){$\sigma_1$};
    \node(sq) at (8.5,0.5){$\sigma_m$};
    \node(s1) at (10,-1){$\sigma_2$};
    \node[xshift = -3pt](ddots2) at (9.5,0){$\ddots$};
    \draw (m1) -- (r1);
    \draw(r1) -- (r2);
    \draw (r2) -- (m2);
    \draw (m2) -- (q1);
    \draw (q1) -- (ddots);
    \draw (ddots) -- (qk);
    \draw (qk) -- (z);
    \draw(qk) -- (sq);
    \draw (sq) -- (ddots2);
    \draw (ddots2) -- (s1);
    \draw(s1) -- (z);
   \end{tikzpicture}&
      \begin{tikzpicture}[scale = 0.7]
    \node(r1) at (2,3){$\boxed{M_3}$};
    \node(r2) at (3,1.5){$\boxed{R_3}$};
    \node(m2) at (4,0){$\boxed{M_4}$};
    \node(q1) at (5,1){$\lambda_{b+1}$};
    \node(ddots) at (6,0){$\ddots$};
    \node(qk) at (7,-1){$\lambda_{a}$};
    \node(z) at (11,-2){$\sigma_1$};
    \node(sq) at (8.5,0.5){$\sigma_m$};
    \node(s1) at (10,-1){$\sigma_2$};
    \node[xshift = -3pt](ddots2) at (9.5,0){$\ddots$};
    \draw(r1) -- (r2);
    \draw (r2) -- (m2);
    \draw (m2) -- (q1);
    \draw (q1) -- (ddots);
    \draw (ddots) -- (qk);
    \draw (qk) -- (z);
    \draw(qk) -- (sq);
    \draw (sq) -- (ddots2);
    \draw (ddots2) -- (s1);
    \draw(s1) -- (z);
   \end{tikzpicture}
\end{tabular}   
\end{center}

In this case, when we use the same map $J(\calP) \to J(\calP_{3}) \times J(\calP_{4})$ as in the other subcases of this case. The images of the $\sigma_i$ are clear. The cokernel consists of sets $(I_3,I_4)$ where $(M_3 \cup R_3) \cap I_3 = \emptyset$; $I_4 = \calP_4$; and $M_4 \cup \{\lambda_{b+1},\ldots, \lambda_a,\sigma_1\} \subseteq I_3$. This set is in bijection with order ideals of $\calP_{56}$.

We check that $\gb_\gamma+ \degx(\hat{Y}_R\hat{Y}_M \hat{y}_{\sigma_1}\prod_{i=b+1}^{a} \hat{y}_{\lambda_i}) = \gb_{\gamma_{56}}$, noting that the equality $\gb_\gamma = \gb_{\gamma_3} + \gb_{\gamma_4}$ is straightforward. The monomial $\hat{Y}_R\hat{Y}_M$ is the same as before, and $ \degx(\hat{y}_{\sigma_1}\prod_{i=b+1}^a \hat{y}_{\lambda_i}) = \eb_{\sigma_1} - \eb_{\lambda_{b+1}} +\eb_\beta - \eb_{\epsilon_1} - \eb_{\sigma_2} + \eb_{[\sigma_1]}$ where $[\sigma_1]$ is the third arc in the triangle formed by $\sigma_1$ and $\sigma_2$. The vector $\gb_\gamma$ is the result of taking Equation \ref{eq:g-vector}, replacing the coefficient of $\eb_\beta$ with $-2\delta_{\calP(1) \prec \calP(2)}$ and the coefficient of $\eb_{\epsilon_2}$ with 1, and adding $-\eb_{\epsilon_1} - \gb_{\gamma'}- \eb_{\sigma_1}$.

Therefore, we have  
\[
\gb_{\gamma} +\degx(\hat{Y}_R) + \degx(\hat{Y}_M) + \degx(\hat{y}_{\sigma_1}\prod_{i=b+1}^a \hat{y}_{\lambda_i})= \eb_{[\sigma_1]} - \eb_{\sigma_2} = \gb_{\gamma_{56}}.
\]

\textbf{Case iv) $s-b = 1$; $t'+b < n$; $s(\gamma)$ is notched; and, $M_1$ and $M_2$ are non-intersecting.}

Let $\epsilon_1,\ldots,\epsilon_h$ be the set of spokes at $s(\gamma)$ ordered in counterclockwise direction, where $\epsilon_1$ and $\epsilon_2$ are as before. The chain  $\calP[t'+b+1,t'+a]$ is of the form  $\epsilon_2^{(1)} \prec \epsilon_3^{(1)} \prec \cdots \prec \epsilon_m^{(1)} \prec \epsilon_1^{(1)} \prec \epsilon_2^{(2)} \prec \cdots \prec \epsilon_k^{(w)}$. We first assume $t'+a < d$, so that the element $\lambda_{a+1}$ exists in $\calP^0$. For this more general case, we draw $\calP$ and $\calP_{56}$ below. The poset $\calP_3$ is obtained by replacing the set $M_1 \cup R_1 \cup R_2 \cup M_2$ with $M_3 \cup R_3 \cup M_4$, and $\calP_{\gamma_4}$ is the same as always.

\begin{center}
\begin{tabular}{cc}
\begin{tikzpicture}[scale = 0.7]
    \node(e1) at (-3,2){$\epsilon_1$};
    \node(eq) at (-2,1){$\epsilon_h$};
    \node(ddots0) at (-1,0){$\ddots$};
    \node(e2) at (0,-1){$\epsilon_2$};
    \node(m1) at (1,1){$\boxed{M_1}$};
    \node(r1) at (2,-0.5){$\boxed{R_1}$};
    \node(r2) at (3,1){$\boxed{R_2}$};
    \node(m2) at (4,-0.5){$\boxed{M_2}$}; 
    \node(e21) at (5,-2){$\epsilon_2^{(1)}$};
    \node(ddots2) at (6,0){$\iddots$};
    \node(ek) at (7,1){$\epsilon_k^{(w)}$};
    \node(ptc) at (8.5,0){$\lambda_{a+1}$};
    \node(moredots) at (9.5,0){$\cdots$};
    \draw (m1) -- (e1);
    \draw (e1) -- (eq);
    \draw (eq) -- (ddots0);
    \draw (ddots0) -- (e2);
    \draw (m1) -- (e2);
    \draw (m1) -- (r1);
    \draw(r1) -- (r2);
    \draw (r2) -- (m2);
    \draw(m2) -- (e21);
    \draw (e21) -- (ddots2);
    \draw (ddots2) -- (ek);
    \draw (ek) -- (ptc);
\end{tikzpicture}
&
\begin{tikzpicture}
\node(ptc) at (0.5,0){$\lambda_{a+1}$};
\node(ek1) at (-1,-1){$\epsilon_{k+1}$};
\node(iddots) at (-2,0){$\ddots$};
\node(ek) at (-3,1){$\epsilon_k$};
\node[] at (1.25,0){$\cdots$};
\draw(ptc) -- (ek1);
\draw(ptc) -- (ek);
\draw (ek) -- (iddots);
\draw (iddots) -- (ek1);
\end{tikzpicture}
\end{tabular}
\end{center}

\textbf{Subcase) If $t'+a = n$, then $t(\gamma)$ is plain.} 

In the following, if $t'+a = n$, we replace mentions of $\epsilon_k$ with $\epsilon_{k+1}$ and ignore any condition involving $\lambda_{a+1}$. Given $I_1 \in J(\calP)$, let $\epsilon_y^{(z)}$ and $\epsilon_x$ be the maximal elements of $I_1$ on the chain of spokes and loop respectively.

\begin{enumerate}
\item Let $A_1$ be the set of all $I_1 \in J(\calP)$ such that we do not have $(M_1 \cup R_1) \subseteq I_1$ and  $(M_2 \cup R_2) \cap I_1 = \emptyset$. Given $I_1 \in A_1$, we define $\Phi(I_1)$ as in case i.  The set $\Phi(A_1)$ is all tuples $(I_3,I_4) \in J(\calP_3) \times J(\calP_4)$ such that, if $R_3 \cap I_3 = \emptyset$ and $I_4 = \calP_4$, then we do not have $M_4 \subseteq I_3$ and $M_3 \cap I_3 = \emptyset$.
\item Let $A_2$ denote the set of $I_1 \in J(\calP)$ such that $(M_1 \cup R_1) \subseteq I_1$; $(M_2 \cup R_2) \cap I_1 = \emptyset$; and, if $\epsilon_y^{(z)} \succeq  \epsilon_1^{(w-1)}$, then $\lambda_{a+1} \in I_1$ and $\epsilon_x \preceq \epsilon_{k}$. If $w = 1$, we set $\langle \epsilon_1^{(w-1)} \rangle = \emptyset$.  Then, we define 
\[
\Phi(M_1 \cup R_1 \cup \langle \epsilon_x \rangle \cup \langle \epsilon_y^{(z)} \rangle) = \begin{cases}
(M_4 \cup \langle \epsilon_y \rangle \cup \langle \epsilon_x^{(z)} \rangle, \calP_4) & y \neq 1\\
(M_4 \cup \langle \epsilon_x^{(z+1)} \rangle, \calP_4)  & y = 1.\\
\end{cases}
\]
The image of this subset of $J(\calP)$ is pairs $(I_3,I_4)$ such that $M_4 \subseteq I_3$; $R_3 \cap I_3 = \emptyset$; $I_4 = \calP_4$; and, if $\epsilon_2^{(w)} \in I_3$, then $\epsilon_x \preceq \epsilon_k$, with equality  only if $\lambda_{a+1} \in I_3$, where here $\epsilon_x$ and $\epsilon_y^{(z)}$ now signify the maximal elements on the loop and chain of spokes respectively in $I_3$.
\item Let $A_3$ be the set of $I_1 \in J(\calP)$ such that $(M_1 \cup R_1) \subseteq I_1$; $(M_2 \cup R_2) \cap I_1 = \emptyset$; $\epsilon_1 \succ \epsilon_x \succeq \epsilon_k$, where the latter is strict inequality if $\lambda_{a+1} \in I_1$; and, $\epsilon_y^{(z)} \succeq \epsilon_2^{(w)}$. 
Then, we define $\Phi(I_1) = (M_4 \cup \langle \epsilon_x \rangle \cup \langle \epsilon_{y}^{(z)} \rangle, \calP_{\gamma_4})$.  The image of $A_3$ is pairs $(I_3,I_4)$ such that $M_4 \in I_3$; $R_3 \cap I_3 = \emptyset$; $I_4 = \calP_4$; $\epsilon_y^{(z)} \succeq \epsilon_1^{(w-1)} \in I_3$; and, $\epsilon_x \succeq \epsilon_k \in I_3$, with strict inequality  if $\lambda_{a+1} \in I_3$. 
\item Let $B \subseteq J(\calP)$ be the set of all order ideals $I_1$ such that $(M_1 \cup R_1) \subseteq I_1$; $(M_2 \cup R_2) \cap I_1 = \emptyset$; $\epsilon_x \succeq \epsilon_k$, which is strict if $\lambda_{a+1} \in I_1$; and, $\epsilon_y^{(z)} \succeq \epsilon_1^{(w-1)}$, which is only strict if $\epsilon_x = \epsilon_1$.  Given $I_1 \in B$, the set $I_1 \backslash (M_1 \cup R_1 \cup \langle \epsilon_k \rangle \cup \langle \epsilon_1^{(w-1)}\rangle)$ naturally forms an order ideal of $\calP_{56}$, and we see this map is bijective. 
\end{enumerate}

In the special case where $k = 1$, we remove the conditions involving $\epsilon_{k+1}$.

\textbf{Subcase) $t'+a = n$ and $t(\gamma)$ is notched.}

The partition and map for this case can be derived from the $t'+a < n$ case in a parallel way to how we compared such cases in previous proofs, such as the proof of Proposition \ref{prop:IntArcFromTSingleTag}, so we do not provide the details here.

\textbf{Case v) $s-b = 1$; $t'+b < n$; $s(\gamma)$ is notched; and, $M_1$ and $M_2$ are intersecting.}

Now, we assume $M_1$ and $M_2$ are intersecting. Recall we change our indexing of the $\epsilon_i$ in this case, as depicted below in $\calP_\gamma$. The chain of $\epsilon_y^{(z)}$ now has the form $\epsilon_2^{(w)} \succ \epsilon_h^{(w)} \succ \cdots \succ \epsilon_3^{(w)} \succ \epsilon_1^{(w)} \succ \epsilon_2^{(w-1)} \succ \cdots \succ \epsilon_k^{(1)}$, changing superscripts between $\epsilon_1$ and $\epsilon_2$. We draw below $\calP$ for the general case $t' + c < n$. 

\begin{center}
\begin{tikzpicture}[scale = 0.7]
    \node(e1) at (-3,2){$\epsilon_2$};
    \node(eq) at (-2,1){$\epsilon_h$};
    \node(ddots0) at (-1,0){$\ddots$};
    \node(e2) at (0,-1){$\epsilon_1$};
    \node(m1) at (1,1){$\boxed{M_1}$};
    \node(r1) at (2,-0.5){$\boxed{R_1}$};
    \node(r2) at (3,1){$\boxed{R_2}$};
    \node(m2) at (4,-0.5){$\boxed{M_2}$}; 
    \node(e21) at (5,1){$\epsilon_2^{(w)}$};
    \node(ddots2) at (6,0){$\ddots$};
    \node(ek) at (7,-1){$\epsilon_k^{(1)}$};
    \node(ptc) at (8,0){$\lambda_{a+1}$};
    \draw (m1) -- (e1);
    \draw (e1) -- (eq);
    \draw (eq) -- (ddots0);
    \draw (ddots0) -- (e2);
    \draw (m1) -- (e2);
    \draw (m1) -- (r1);
    \draw(r1) -- (r2);
    \draw (r2) -- (m2);
    \draw(m2) -- (e21);
    \draw (ek) -- (ptc);
    \draw (e21) -- (ddots2);
    \draw(ddots2) -- (ek);
\end{tikzpicture}
\end{center}

\textbf{Subcase) If $t'+a = n$, then $t(\gamma)$ is plain.} 

If $t'+a = n$, we ignore all conditions involving $\lambda_{a+1}$. 

We define $A_1$ and $A_2$ and the bijection $\Phi$ restricted to these sets as in case ii.  Recall $A_1 \cup A_2$ consists of all $I_1 \in J(\calP)$ except those such that $(R_1 \cup M_1) \subseteq I_1$ and $(R_2 \cup M_2) \cap I_1 = \emptyset$, and $\Phi(A_1 \cup A_2)$ consists of all pairs $(I_3,I_4)$ except those such that $(R_3 \cup M_3) \cap I_3 = \emptyset$; $M_4 \subseteq I_3$; and, $I_4 = \calP_4$.

Let $A_3 \subseteq J(\calP)$ denote all sets $I_1$ such that $(M_1 \cup R_1) \subseteq I_1$; $(M_2 \cup R_2) \cap I_1 = \emptyset$; and, $\epsilon_y^{(z)} \succeq \epsilon_2^{(1)}$. The map $\Phi$ is defined analogously to the map $\Phi \vert_{A_2}$ in case iv, but where the cases are now $y \neq 2$ and $y = 2$. The set $\Phi(A_3)$ is all pairs $(I_3,I_4)$ such that $(R_3 \cup M_3) \cap I_3 = \emptyset$; $M_4 \subseteq I_3$; $I_4 = \calP_4$; and, $\epsilon_y^{(z)} \succ \epsilon_2^{(1)}$. 

Let $A_4 \subseteq J(\calP)$ denote all sets $I_1$ such that $(M_1 \cup R_1) \subseteq I_1$; $(M_2 \cup R_2) \cap I_1 = \emptyset$; and, $\epsilon_y^{(z)} \prec \epsilon_2^{(1)}$. The map $\Phi$ is defined by \[
\Phi(M_1 \cup R_1 \cup \langle \epsilon_x \rangle \cup \langle \epsilon_y^{(1)} \rangle) = \begin{cases}
(M_4 \cup \langle \epsilon_x \rangle \cup \langle \epsilon_y^{(1)} \rangle, \calP_4) & x \neq 2 \\
(M_4 \cup \langle \epsilon_y \rangle \cup \langle \epsilon_2^{(1)} \rangle, \calP_4) & x = 2
\end{cases}
\]
where we set $\langle \epsilon_y \rangle = \langle \epsilon_{k-1} \rangle$ when $\epsilon_y^{(z)}$ does not exist. The set $\Phi(A_4)$ is all pairs $(I_3,I_4)$ such that $(R_3 \cup M_3) \cap I_3 = \emptyset$; $M_4 \subseteq I_3$; $I_4 = \calP_4$; and, $\epsilon_y^{(z)} \preceq \epsilon_2^{(1)}$, with equality only if $\epsilon_x \succeq \epsilon_{k-1}$, and moreover if $\epsilon_y^{(z)}= \epsilon_2^{(1)}$ and $\lambda_{a+1} \in I_3$, then $\epsilon_x \succ \epsilon_{k-1}$. 

We have that $A_1 \cup A_2 \cup A_3 \cup A_4 = J(\calP_1)$, and the cokernel of $\Phi: J(\calP_1) \to J(\calP_{3}) \times J(\calP_{4})$ consists of pairs $(I_3,I_4)$ with  $I_4 = \calP_{4}$; $M_4 \subseteq I_3$; $(M_3 \cup R_3) \cap I_3 = \emptyset$; $\epsilon_y^{(z)} \preceq \epsilon_2^{(1)}$ with equality only if $\epsilon_1 \in I_3$; and, $\epsilon_x \preceq \epsilon_{k-1}$, with equality only if $\lambda_{a+1} \in I_3$. These sets are in natural bijection with $J(\calP_{56})$. 

In the special case $k = 1$, we interpret statements ``$\epsilon_x \preceq \epsilon_{k-1}$'' regarding $\calP_3$ to mean  ``$\epsilon_x$ does not exist''.

\textbf{Subcase) $t'+a = n$ and $t(\gamma)$ is notched.}

We can use the same partitioning $A_1 \cup A_2 \cup A_3 \cup A_4 = J(\calP)$ and same map $\Phi: J(\calP) \to J(\calP_3) \times J(\calP_4)$. The cokernel of this map has a slightly different description, but it will again correspond bijectively to elements of $J(\calP_{56})$.

The $\gb$-vector computations for all of case v follows quickly from previous cases. 

\textbf{Case vi) $s-b > 1$ and  $t'+b = n$}.

We can switch the chronological ordering on $\calP$ to return to the $s-b = 1, t'+b < n$ case. 

\textbf{Case vii) $s-b = 1$ and $t'+b = n$}.

If $s-b=1$ and $t'+b = n$, then $\gamma_{56}$ is contractible. If $s(\gamma)$ and $t(\gamma)$ are both plain, then the ideas from case i give a direct bijection between $J(\calP)$ and $J(\calP_{3}) \times J(\calP_{4})$. If $s(\gamma)$ and $t(\gamma)$ are both notched, then we construct a similar bijection; in the special case where $(R_1 \cup M_1) \subseteq I_1, (R_2 \cup M_2) \cap I_1 = \emptyset$, it will be necessary to send the portion of $I_1$ at the beginning loop and send it to the end loop in $\calP_{3}$ and similarly take the portion of $I_1$ on the end loop and send it to the first loop in $\calP_{3}$. 

If the taggings at $s(\gamma)$ and $t(\gamma)$ do not agree, then $\gamma_{56}$ is a contractible singly-notched monogon. The expression $x_{\gamma_{56}}$ is defined in Definition \ref{def:ContractibleSinglyTaggedLoop}. Since $M_1$ and $M_2$ cannot be crossing in this case, we always have $x_\gamma = x_{\gamma_3}x_{\gamma_4} + Y_RY_M x_{\gamma_{56}} = x_{\gamma_3}x_{\gamma_4} + Y_RY_M(\pm 1 \mp Y_{s(\gamma)})$.

Suppose $s(\gamma)$ is notched and $t(\gamma)$ is plain.  Let the spokes at $s(\gamma)$ again be $\epsilon_1,\epsilon_2,\ldots,\epsilon_h$. Following the same ideas as the general case, we see there is a bijection between $J(\calP_1) \backslash \{M_1 \cup R_1 \cup \{\epsilon_1,\ldots,\epsilon_h\}\}$ and $(J(\calP_{\gamma_3}) \times J(\calP_{\gamma_4}))\backslash \{(M_4,\calP_{\gamma_4})\}\}$, but there is no possible image for $M_1 \cup R_1 \cup \{\epsilon_1,\ldots,\epsilon_h\}$, and there is no preimage to $(M_4,\calP_{4})$  Thus, we conclude $x_\gamma + Y_RY_M = x_{\gamma_3}x_{\gamma_4} + Y_RY_MY_{s(\gamma)}$, implying $x_\gamma = x_{\gamma_3}x_{\gamma_4} + Y_RY_M x_{\gamma_{56}} $. The case where $s(\gamma)$ is plain and $t(\gamma)$ is notched follows from a similar, dual argument.
\end{proof}

\section{Closed Curves}
\label{sec:closed_curves}

We consider intersections involving a closed curve. The presence of punctures does not affect the intersection of two closed curves, so these can be computed using methods from \cite{snakegraphcalculus3} or \cite{musiker2013}.

Intersections between a closed curve $\xi$ and an arc $\gamma$  can be characterized into Type 0 and Type 1, as in Section \ref{sec:transverse_crossings} based on whether there exists a representative from the isotopy class of $\gamma$ such that the point of intersection lies in the first or last triangle $\gamma$ passes through. There is no analogue for a Type 2 crossing since $\xi$ does not have a well-defined first or last triangle it passes through. 

When the intersection is of Type 0, we orient $\xi$ so that in the region corresponding to the crossing overlap, the orientation of $\gamma$ and $\xi$ agree. When the intersection is of Type 1, suppose the point of intersection is in the first triangle $\gamma$ passes through, and orient $\xi$ so that near the point of intersection, $s(\gamma)$ lies to its right. Then, we label the resolution of the intersection between $\xi$ and $\gamma$ by $\{\gamma_{34}\} \cup \{\gamma_{56}\}$ where $\gamma_{34} = \gamma \circ \xi \circ \gamma$ and $\gamma_{56} = \gamma \circ \xi^{-1} \circ \gamma$. This closely resembles the resolution $\{\gamma_3,\gamma_4\} \cup \{\gamma_5,\gamma_6\}$ for Type 0 and 1 intersections of a pair of arcs respectively.

\begin{prop}
Let $\xi$ be a closed curve and let $\gamma$ be an arc.
\begin{enumerate}
    \item Suppose $\xi$ and $\gamma$ intersect in a Type 0 crossing and let $R_1 = \calP_{\xi}[s,t]$ and $R_2 = \calP_\gamma[s',t']$ be the corresponding crossing overlap in the posets. Then, \[
    x_\xi x_\gamma = x_{\gamma_{34}} + Y_R Y_{S_s} Y_{S_t} x_{\gamma_{56}}
    \]
    where $Y_R = \prod_{i=s}^t y_{P_\xi(i)}$,\[
    Y_{S_s} = \begin{cases}
    \displaystyle \prod_{\substack{\calP_\xi(i) \succeq \calP_\xi(s-1)\\ i \leq s-1}} & s' = 1; s(\gamma) \text{ is notched; and, } \calP_\gamma[s',t'] \text{ is on bottom}\\
    \displaystyle \prod_{\substack{\calP_\xi(i) \preceq \calP_\xi(s-1)\\ i \leq s-1}} & s' = 1; s(\gamma) \text{ is plain; and, } \calP_\gamma[s',t'] \text{ is on top}\\
    1 & \text{otherwise}
    \end{cases}
    \]
    and for $n = \vert \calP_\gamma^0 \vert$,
    \[
    Y_{S_t} = \begin{cases}
    \displaystyle \prod_{\substack{\calP_\xi(i) \succeq \calP_\xi(t+1)\\ i \geq t+1}} & t' = d; t(\gamma) \text{ is notched; and, } \calP_\gamma[s',t'] \text{ is on bottom}\\
    \displaystyle \prod_{\substack{\calP_\xi(i) \preceq \calP_\xi(t+1)\\ i \geq t+1}} & t' = d; t(\gamma) \text{ is plain; and, } \calP_\gamma[s',t'] \text{ is on top}\\
    1 & \text{otherwise}.
    \end{cases}
    \]
    \item Suppose $\xi$ and $\gamma$ intersect in a Type 1 crossing near $s(\gamma)$. Let $i$ be such that this point of intersection occurs between the crossings between $\xi$ and $T$ corresponding to $\calP_\xi(i)$ and $\calP_\xi(i+1)$. Then,\[
    x_\xi x_\gamma = Y_{\preceq \calP_{\xi(i)}} x_{\gamma_{34}} + x_{\gamma_{56}}
    \]
    and
    \[
    x_\xi x_{\gamma^{(s(\gamma))}} = x_{\gamma_{34}} +  Y_{\succeq \calP_{\xi(i+1)}} x_{\gamma_{56}}
    \]
where $Y_{\preceq \calP_{\xi(i)}} = \prod_{\calP_\xi(j) \preceq \calP_\xi(i)} y_{\calP_\xi(j)}$ and $Y_{\succeq \calP_{\xi(i+1)}} = \prod_{\calP_\xi(j) \succeq \calP_\xi(i+1)} y_{\calP_\xi(j)}$.
\end{enumerate}
\end{prop}

The proof of part 1 follows from the proof of Proposition \ref{prop:GraftingType0} while part 2 follows from Proposition \ref{prop:GraftingType1}. By our assumption about the orientation of $\xi$, in part 2 $\calP_\xi(i) \prec \calP_\xi(i+1)$.

\section{Implications}
\label{sec:implication}

\subsection{Main Theorem}

We have now shown Theorem \ref{thm:main} by showing it is true for any possible pair of incompatible curves as well as for any self-incompatible curve. One can compare our list of cases with the list found in Section 8.4 of \cite{musiker2013bases} to show that all possible configurations are discussed. Moreover, in each case, one can see that the description of the $y$-monomial matches the more general description provided in the introduction. In the unpunctured case, our description matches that of \cite{musiker2013}.

\subsection{Bangles and Bracelets}

Having explicit skein relations for punctured surfaces has immediate corollaries regarding the analogues of the bracelet and bangles bases from \cite{musiker2013bases}. We first recall the definitions of \emph{bangles} and \emph{bracelets}.

\begin{definition}[Definition 3.16, \cite{musiker2013bases}]
    Given an essential loop $\xi$ on $(S,M)$, the \emph{bangle} $\bang{k}{\xi}$ is the union of $k$ loops that are isotopic to $\xi$, with no self-crossings.The \emph{bracelet} $\brac{k}{\xi}$ is constructed by concatenating $\xi$ exactly $k$ times, producing a closed loop with $k-1$ self-crossings.
\end{definition}

\begin{definition}[Definition 8.2 \cite{musiker2013bases}]
\begin{itemize}
    \item  A collection $C$ of tagged arcs and essential loops on $(S,M)$ is called \emph{$\mathcal{C}^{\circ}$-compatible} if its tagged arcs are pairwise compatible. Let $\mathcal{C}^{\circ}(S,M)$ to denote the set of all $\mathcal{C}^{\circ}$-compatible collections on $(S,M)$.
    \item A collection $C$ of tagged arcs and bracelets is \emph{$C$-compatible} if its tagged arcs are pairwise compatible; no two elements of $C$ cross each other except for self-crossings of bracelets; and, for any essential loop $\gamma$ in $(S,M)$, there exists at most one $k \geq 1$ such that $\textrm{Brac}_{k}\gamma$ lies in $C$ (and there is at most one copy of this bracelet in $C$).
\end{itemize}
\end{definition}

From sets of $\mathcal{C}^{\circ}$ and $\mathcal{C}$-compatible curves, we define two sets of monomials in the cluster algebra $\mathcal{A}(S,M)$ associated to $(S,M)$. The sets $\calBcirc$ and $\calB$ are then defined as
\begin{align*}
    \calBcirc &:= \left\{ \prod_{c \in C} x_c : C \in \mathcal{C}^{\circ}(S,M) \right\}, \\
    \calB &:= \left\{ \prod_{c \in C} x_c : C \in \mathcal{C}(S,M) \right\}. \\
\end{align*}
In the case of an unpunctured surface (in which case one omits the adjective ``tagged'' in the description), these two sets of monomials are shown to be bases of $\mathcal{A}(S,M)$ in Theorem 1.1 of \cite{musiker2013bases}. The basis $\calBcirc$ is called the bangles basis and $\calB$ is called the bracelet basis.

Having skein relations allows us to quickly make progress towards showing that these sets also provide bases for a cluster algebra from a punctured surface. The most immediate corollary shows that these sets span $\mathcal{A}(S,M)$.

\begin{cor}\label{cor:Span}
Given a possibly punctured surface $(S,M)$, the sets $\calBcirc$ and $\calB$ span $\mathcal{A}(S,M)$.
\end{cor}

From Theorem \ref{thm:main}, the proof of Corollary \ref{cor:Span} follows the proof of the result in the unpunctured case, which appears as Lemma 4.7 in \cite{musiker2013bases}. The relationship between bracelets and bangles via Chebyshev polynomials, given in Proposition 4.2, of the same article still holds in a punctured surface.

We have a second immediate corollary, which results from the description of the $y$-monomial appearing in the skein relations. This result would be a first step towards showing linear independence of $\calBcirc$ and $\calB$ if one were to follow the same strategy as in \cite{musiker2013bases}. 

\begin{cor}\label{cor:KeyLemmaStillHolds}
Lemma 6.3 from \cite{musiker2013bases} holds for a punctured surface as well.
\end{cor}

\begin{proof}
We can follow the same proof strategy as for the original lemma, combining the following two observations. From Theorem \ref{thm:main}, we know that each skein relation has one term without a $y$-monomial. The homogeneity of elements of $\calBcirc$ follows in the punctured case can be shown in the same way as in the unpunctured case using Wilson's lattice structure on good matchings of loop graphs \cite{wilson2020surface}. 
\end{proof}

Armed with Corollary \ref{cor:KeyLemmaStillHolds}, it is straightforward to follow the recipe in Section 8.5 of \cite{musiker2013bases} to show that the elements of $\calBcirc$ are linearly independent and similarly for $\calB$. This argument uses skein relations to show that we can realize the $\gb$-vectors as giving an injection from $\calBcirc$ or $\calB$ to $\mathbb{Z}^n$; if the surface is not once-punctured and closed, this injection will in fact be a bijection. If we have a surface with at least two marked points on the boundary or a closed surface of genus 0, we can then use the suggestions in Section 8.6 of the same article to show that $\calBcirc$ and $\calB$ are both subsets of $\mathcal{A}(S,M)$. Recall in the latter setting, we require there to be at least four marked points, as is standard. Combining these suggestions with Corollary \ref{cor:Span} allows us to conclude that $\calBcirc$ and $\calB$ are bases in such settings.

If our surface has only one marked point on the boundary, we must use another method. Recall that the \emph{upper cluster algebra} $\mathcal{U}$ of a cluster algebra $\mathcal{A}$ consists of all elements which can be expressed as a Laurent polynomial in any cluster of the cluster algebra $\mathcal{A}$. The snake graph construction ensures that all elements of $\calBcirc$ and $\calB$ are indeed in $\mathcal{U}(S,M)$.  Proposition 12 of \cite{canakci2015cluster}, shows that, if $(S,M)$ is a surface with exactly one marked point on the boundary, then $\mathcal{A}(S,M) = \mathcal{U}(S,M)$. Therefore, we can again conclude $\calBcirc \subset \mathcal{A}(S,M)$ and $\calB \subset \mathcal{A}(S,M)$. With these ingredients, we can that each set is indeed a basis in a wide array of settings.

\begin{upstatement}{Theorem 2}
Let $(S,M)$ be such that $S$ has a non-empty boundary, or $S$ has genus 0 and $\vert M \vert > 3$. Then, $\calBcirc$ is a basis for $\mathcal{A}(S,M)$ and similarly, $\calB$ is a basis for $\mathcal{A}(S,M)$.
\end{upstatement}

As mentioned in the introduction, this general statement of the theorem includes some cases where the result was already known. The fact that the bracelets $\mathcal{B}$ form bases for surfaces with at least two marked points on the boundary follows from results of  Mandel-Qin \cite{Mandel2023} and Muller \cite{Muller2016}. Gei{\ss}, Labardini-Fragoso, and Wilson~\cite{geiss2023bangle} showed that the bangles $\mathcal{B}^{\circ}$ form bases in the coefficient-free case for surfaces with non-empty boundary.

The only obstruction to extending the main result of \cite{musiker2013bases} to surfaces not covered in the hypothesis of Theorem \ref{thm:Bases} is showing that the expressions associated to essential loops are elements of $\mathcal{A}(S,M)$ when $(S,M)$ is closed and has nonzero genus. For example, recall that the Markov cluster algebra is the cluster algebra associated to a once-punctured torus. Bernstein, Fomin, and Zelevinsky show that the Markov cluster algebra is properly contained in its upper cluster algebra by exhibiting a specific element in the complement \cite{berenstein2005cluster}.  This proof can be slightly modified to show that the bracelet function (equivalently, bangle function) associated to a certain essential loop on the torus does not lie in the cluster algebra; as previously noted, it will  lie in the upper cluster algebra.

We note that in general cluster algebras from closed surfaces of positive genus are known to differ from their upper cluster algebras (shown in the once-punctured case by Ladkani \cite{ladkani2013cluster} and by Moon and Wong for a general number of punctures \cite{moon2022consequences}). Conversely, for surfaces with boundary, Theorem \ref{thm:Bases} applies to all cases where the cluster algebra equals its upper cluster algebra (for proofs see \cite{ccanakci2015cluster, muller2016skein}). 

\subsection{Punctured Orbifolds}

Analogous to cluster algebras of surface type, a subset of Chekhov and Shapiro's \emph{generalized cluster algebras} arise from triangulated orbifolds \cite{Chekhov-Shapiro}. In fact, the definition of a generalized cluster algebra was inspired by how lambda lengths in an orbifold mutate. In \cite{banaian2020snake}, the first and third authors extended the snake graph construction of Musiker, Schiffler, and Williams for generalized cluster algebras from \emph{unpunctured} orbifolds and further follow \cite{musiker2013} to provide matrix formulas and skein relations in this setting. 

Our study of skein relations for punctured orbifolds was originally motivated by a desire to extend the generalized snake graphs of \cite{banaian2020snake} to the punctured setting. While ordinary tagged arcs on a triangulated orbifold behave similarly to the surface case, generalized tagged arcs which wind nontrivially around an orbifold point require new and intricate constructions. Proving that the snake graphs corresponding to generalized arcs produce ``sensible'' expansions in the generalized cluster algebra requires an understanding of skein relations in the presence of punctures.

In upcoming work, we will collaborate with the authors of \cite{ouguz2024cluster} to use the proof methods demonstrated in this paper to generalize their poset expansion formulae to the orbifold setting. This work will also allow us to give matrix formulae for generalized cluster algebras in the \cite{ouguz2023rank}, which we emphasize are distinct from those found in \cite{banaian2020snake}. 

\subsection*{Acknowledgements}
This project benefited from conversations with {\.I}lke {\c{C}}anak{\c{c}}{\i}, Philippe Di Francesco, Rinat Kedem, Gregg Musiker, Michael Tsironis, and Jon Wilson and helpful comments on the initial preprint from Nathan Reading and Ralf Schiffler. E.B. was supported by Research Project 2 from the Independent Research Fund Denmark (grant no. 1026-00050B). W.K. was partially supported by the Susan C. Morisato and R.H. Schark Scholarships from the University of Illinois Urbana Champaign Department of Mathematics. E.K. was supported by NSF Grant No. DMS-1937241.

\nocite{*} 

\end{document}